\documentclass{article}
\usepackage{amsmath,amssymb,amsthm,graphicx,epsfig,subfigure,float,url}
\usepackage{pdfsync}

\usepackage{verbatim}
\usepackage[applemac]{inputenc}
 
\usepackage[usenames,dvipsnames]{pstricks}
\usepackage{pst-grad} 
\usepackage{pst-plot} 

\def\R{\textrm{I\kern-0.21emR}}
\def\N{\textrm{I\kern-0.21emN}}

\newcommand{\C} {\mathbb{C}}

\renewcommand{\geq}{\geqslant}
\renewcommand{\leq}{\leqslant}

\newtheorem{theorem}{Theorem}
\newtheorem*{thmnonumbering}{Theorem}
\newtheorem{proposition}{Proposition}
\newtheorem{corollary}{Corollary}

\newtheorem{lemma}{Lemma}

\theoremstyle{definition}\newtheorem{remark}{Remark}

\title{Optimal observability of the multi-dimensional wave and Schr\"odinger equations in quantum ergodic domains}

\author{Yannick Privat\footnote{IRMAR, ENS Cachan Bretagne, Univ. Rennes 1, CNRS, UEB, av. Robert Schuman, 35170 Bruz, France
		({\tt yannick.privat@bretagne.ens-cachan.fr}).}
	\and Emmanuel Tr\'elat\footnote{Universit\'e Pierre et Marie Curie (Univ. Paris 6) and Institut Universitaire de France,
CNRS UMR 7598, Laboratoire Jacques-Louis Lions, F-75005, Paris, France ({\tt emmanuel.trelat@upmc.fr}).} 
        \and Enrique Zuazua\footnote{BCAM - Basque Center for Applied Mathematics, Mazarredo, 14 E-48009 Bilbao-Basque Country-Spain}\ \footnote{Ikerbasque, Basque Foundation for Science, Alameda Urquijo 36-5, Plaza Bizkaia, 48011, Bilbao-Basque Country-Spain (\texttt{zuazua@bcamath.org}).}}

\date{}

\begin{document}

\maketitle

\begin{abstract}
We consider the wave and Schr\"odinger equations on a bounded open connected subset $\Omega$ of a Riemannian manifold, with Dirichlet, Neumann or Robin boundary conditions whenever its boundary  is nonempty. We observe the restriction of the solutions to a measurable subset $\omega$ of $\Omega$ during a time interval $[0, T]$ with $T>0$. It is well known that, if the pair $(\omega,T)$ satisfies the Geometric Control Condition ($\omega$ being an open set), then an observability inequality holds guaranteeing that the total energy of solutions can be estimated in terms of the energy localized in $\omega \times (0, T)$. 

We address the problem of the optimal location of the observation subset $\omega$ among all possible subsets of a given measure or volume fraction. We solve it in two different situations.
First, when a specific choice of the initial data is given and therefore we deal with a particular  solution, we show that the problem always admits at least one solution that can be regular or of fractal type depending on the regularity of the initial data.

This first problem of finding the optimal $\omega$ for each initial datum is a mathematical benchmark but, in view of applications, it is important to define a relevant criterion, not depending on the initial conditions and to choose the observation set in an uniform way, independent of the data and solutions under consideration.
Through spectral decompositions, this leads to a second problem which consists of maximizing a spectral functional that can be viewed as a measure of eigenfunction concentration. Roughly speaking, the subset $\omega$ has to be chosen so to maximize the minimal trace of the squares of all eigenfunctions. This spectral criterion can be obtained and interpreted in two ways: on the one hand, it corresponds to a time asymptotic observability constant as the observation time interval tends to infinity, and on the other hand, to a randomized version of the deterministic observability inequality. We also consider the convexified formulation of the problem. We prove a no-gap result between the initial problem and its convexified version, under appropriate quantum ergodicity assumptions on $\Omega$, and compute the optimal value.

We also give several examples in which a classical optimal set exists, although, as it happens in 1D, generically with respect to the manifold $\Omega$ and the volume fraction, one expects relaxation to occur and therefore classical optimal sets not to exist. We then provide spectral approximations and present some numerical simulations that fully confirm the theoretical results in the paper and support our conjectures.

Our results highlight precise connections between optimal observability issues and quantum ergodic properties of the domain under consideration.
\end{abstract}

\noindent\textbf{Keywords:} wave equation, Schr\"odinger equation, observability inequality, optimal design, spectral decomposition, ergodic properties, quantum ergodicity.

\medskip

\noindent\textbf{AMS classification:} 
35P20, 
93B07, 
58J51, 
49K20 

\tableofcontents
\newpage

\section{Introduction}\label{secintro}
\subsection{Presentation of the problems}\label{secobservabilityintro}
Let  $(M,g)$ be a smooth $n$-dimensional Riemannian manifold, $n \ge 1$.
Let $T$ be a positive real number and $\Omega$ be an open bounded connected subset of $M$. 
In this article we consider both the wave equation
\begin{equation}\label{waveEqobs}
\partial_{tt}y=\triangle_g y,
\end{equation}
and the Schr\"odinger equation
\begin{equation}\label{schroEqobs}
i \partial_t y=\triangle_g y,
\end{equation}
in $(0,T)\times\Omega$. 
Here, $\triangle_g$ denotes the usual Laplace-Beltrami operator on $M$ for the metric $g$. 
If the boundary $\partial\Omega$ of $\Omega$ is nonempty, then we consider boundary conditions
\begin{equation}\label{condBC}
By=0\quad\textrm{on}\ (0,T)\times\partial\Omega,
\end{equation}
where $B$ can be either:
\begin{itemize}
\item the usual Dirichlet trace operator, $By=y_{\vert\partial\Omega}$,
\item or Neumann, $By=\frac{\partial y}{\partial n}_{\vert\partial\Omega}$, where $\frac{\partial}{\partial n}$ is the outward normal derivative on the boundary $\partial\Omega$,
\item or mixed Dirichlet-Neumann, $By=\chi_{\Gamma_{0}}y_{\vert\partial\Omega}+\chi_{\Gamma_{1}}\frac{\partial y}{\partial n}_{\vert\partial\Omega}$, where $\partial\Omega=\Gamma_{0}\cup\Gamma_{1}$ with $\Gamma_{0}\cap\Gamma_{1}=\emptyset$, and $\chi_{\Gamma_{i}}$ is the characteristic function of $\Gamma_{i}$, $i=0,1$,
\item or Robin, $By=\frac{\partial y}{\partial n}_{\vert\partial\Omega}+\beta y_{\vert\partial\Omega}$, where $\beta$ is a nonnegative bounded measurable function defined on $\partial\Omega$, such that $\int_{\partial\Omega}\beta>0$.
\end{itemize}
Our study encompasses the case where $\partial\Omega=\emptyset$: in this case, \eqref{condBC} is unnecessary and $\Omega$ is a compact connected $n$-dimensional Riemannian manifold.
The canonical Riemannian volume on $M$ is denoted by $V_g$, inducing the canonical measure $dV_g$. Throughout the paper, measurable sets\footnote{If $M$ is the usual Euclidean space $\R^n$ then $dV_g=dx$ is the usual Lebesgue measure.} are considered with respect to the measure $dV_g$.

In the boundaryless or in the Neumann case, the Laplace-Beltrami operator is not invertible on $L^2(\Omega,\C)$ but is invertible in 
$$
L^2_{0}(\Omega,\C)=\{y\in L^2 (\Omega,\C) \ \vert \ \int_{\Omega}y(x)\, dV_{g}=0\}.
$$
In what follows, the notation $X$ stands for the space $L^2_{0}(\Omega,\C)$ in the boundaryless or in the Neumann case and for the space $L^2 (\Omega,\C) $ otherwise. We denote by $A=-\triangle_g$ the Laplace operator defined on $
D(A)=\{y\in X\ \vert\ Ay\in X\ \textrm{and}\ By=0 \}$ with one of the above boundary conditions whenever $\partial\Omega\neq\emptyset$. 
Note that $A$ is a selfadjoint positive operator.
For all $(y^0,y^1)\in D(A^{1/2})\times X$, there exists a unique solution $y$ of the wave equation  \eqref{waveEqobs} in the space $C^0(0,T;D(A^{1/2}))\cap C^1(0,T;X)$ such that
$y(0,\cdot)=y^0(\cdot)$ and $\partial_t y(0,\cdot)=y^1(\cdot)$.

Let $\omega$ be an arbitrary measurable subset of $\Omega$ of positive measure. Throughout the paper, the notation $\chi_\omega$ stands for the characteristic function of $\omega$.
The equation \eqref{waveEqobs} is said to be \textit{observable} on $\omega$ in time $T$ if there exists $C_T^{(W)}(\chi_\omega)>0$ such that
\begin{equation}\label{ineqobsw}
C_T^{(W)}(\chi_\omega) \Vert (y^0,y^1)\Vert_{D(A^{1/2})\times X}^2
\leq \int_0^T\int_\omega \vert \partial_t y(t,x)\vert^2 \,dV_g\, dt,
\end{equation}
for all $(y^0,y^1)\in D(A^{1/2})\times X$. This is the so-called \textit{observability inequality}, relevant in inverse problems or in control theory because of its dual equivalence with the property of controllability (see \cite{lions}).
It is well known that within the class of $\mathcal{C}^\infty$ domains $\Omega$, this observability property holds, roughly,  if the pair $(\omega,T)$ satisfies the so-called \textit{Geometric Control Condition (GCC)} in $\Omega$ (see \cite{BardosLebeauRauch,BurqGerard}), according to which every geodesic ray  in $\Omega$ and reflected on its boundary according to the laws of geometrical optics intersects the observation set $\omega$ within time $T$.
In particular, if at least one ray does not reach $\overline{\omega}$ within time $T$ then the observability inequality fails because of the existence of gaussian beam solutions concentrated along the ray and, therefore, away from the observation set.

A similar observability problem can also be formulated for the Schr\"odinger equation \eqref{schroEqobs} : For every $y^0\in D(A)$, there exists a unique solution $y$ of \eqref{schroEqobs} in the space $C^0(0,T;D(A))$ such that $y(0,\cdot)=y^0(\cdot)$. 
The equation \eqref{schroEqobs} is said to be observable on $\omega$ in time $T$ if there exists $C_T^{(S)}(\chi_\omega)>0$ such that
\begin{equation}\label{ineqobss}
C_T^{(S)}(\chi_\omega) \Vert y^0\Vert_{D(A)}^2
\leq \int_0^T\int_\omega \vert \partial_t y(t,x)\vert^2 \,dV_g\, dt,
\end{equation}
for every $y^0\in D(A)$.
It is well known that if there exists $T^*$ such that the pair $(\omega,T^*)$ satisfies the \textit{Geometric Control Condition} then the observability inequality \eqref{ineqobss} holds for every $T>0$ (see \cite{lebeau}). Indeed the Schr\"odinger equation can be viewed as a wave equation with an infinite speed of propagation. We refer to \cite{laurent} for a thorough discussion of the problem of obtaining necessary and sufficient conditions ensuring the observability inequality, which is a widely open problem.

In the sequel, $C_T^{(W)}(\chi_\omega)$ and $C_T^{(S)}(\chi_\omega)$ denote the largest possible nonnegative constants for which the inequalities \eqref{ineqobsw} and \eqref{ineqobss} hold, that is,
\begin{equation}\label{CT}
C_T^{(W)}(\chi_\omega)=\inf\left\{ \frac{  \int_0^T\int_\omega  \vert \partial_{t} y(t,x) \vert ^2\,dV_g\, dt }{\Vert (y^0,y^1)\Vert^2_{D(A^{1/2})\times X}} \ \big\vert\  (y^0,y^1)\in D(A^{1/2})\times X\setminus\{(0,0)\} \right\} ,
\end{equation}
and
\begin{equation}\label{KT}
C_T^{(S)}(\chi_\omega)=\inf\left\{ \frac{  \int_0^T\int_\omega  \vert \partial_{t}y(t,x) \vert ^2\,dV_g \, dt }{\Vert y^0\Vert^2_{D(A)}} \ \big\vert\  y^0\in D(A) \setminus\{0\} \right\}.
\end{equation}
They are the so-called \textit{observability constants}.

\begin{remark}\label{rk1}
These properties can be formulated in different spaces. For instance, the observability inequality \eqref{ineqobsw} is equivalent to 
\begin{equation}\label{obsrk1}
C_T^{(W)}(\chi_\omega) \Vert (y^0,y^1)\Vert_{X\times (D(A^{1/2}))'}^2
\leq \int_0^T\int_\omega \vert y(t,x)\vert^2 \,dV_g\, dt,
\end{equation}
for all $(y^0,y^1)\in X\times (D(A^{1/2}))'$, with the same observability constants. Here the dual is considered with respect to the pivot space $X$.
Similarly, the observability inequality \eqref{ineqobss} is equivalent to
\begin{equation}\label{obsrk1_Sch}
C_T^{(S)}(\chi_\omega) \Vert y^0\Vert_{X}^2\leq \int_0^T\int_\omega \vert y(t,x)\vert^2 \,dV_g\, dt,
\end{equation}
for every $y^0\in X$.
%
\end{remark}

Let $(\phi_j)_{j\in\N^*}$ be an orthonormal Hilbertian basis of $X$ consisting of eigenfunctions of $A$ on $\Omega$, associated with the positive\footnote{Note that, in the Neumann case or in the case $\partial\Omega=\emptyset$, one has $X=L^2_{0}(\Omega)$. Otherwise if we would consider $X=L^2(\Omega)$ in those cases, then we would have $\lambda_1=0$ (simple eigenvalue) and $\phi_1=1/\sqrt{V_g(\Omega)}$. The fact that in those cases we define $X=L^2_{0}(\Omega)$ permits to keep a uniform presentation for all boundary conditions considered at the beginning.\label{footnote_Neumann}} eigenvalues $(\lambda^2_j)_{j\in\N^*}$. In particular, $\int_\Omega\phi_j(x)\phi_k(x)\, dV_g$ is equal to $0$ whenever $j\neq k$, and $1$ whenever $j=k$.

\begin{remark}
Let us provide a spectral characterization of the spaces $D(A)$ and $D(A^{1/2})$.
There holds
$$
D(A) = \{ y\in X \ \vert\ \sum_{j=1}^{+\infty}\lambda_{j}^4\langle y,\phi_{j}\rangle_{L^2}^2<+\infty \},
$$
and
$$
D(A^{1/2}) = \{ y\in X \ \vert\ \sum_{j=1}^{+\infty}\lambda_{j}^2\langle y,\phi_{j}\rangle_{L^2}^2<+\infty \}.
$$
In the case of Dirichlet boundary conditions, one has $D(A)=H^2(\Omega,\C)\cap H^1_0(\Omega,\C)$ and $D(A^{1/2})=H^1_0(\Omega,\C)$. For Neumann boundary conditions, one has $D(A)=\{y\in H^2(\Omega,\C) \ \vert \ \frac{\partial y}{\partial n}_{\vert\partial\Omega}=0\ \textrm{and} \ \int_{\Omega}y(x)\, dV_{g}=0\}$ and $D(A^{1/2})=\{y\in H^1(\Omega,\C) \ \vert \ \int_{\Omega}y(x)\, dV_{g}=0\}$. In the mixed Dirichlet-Neumann case (with $\Gamma_{0}\neq \emptyset$), one has $D(A)=\{y\in H^2(\Omega,\C) \ \vert \ y_{\vert\Gamma_{0}}=\frac{\partial y}{\partial n}_{\vert\Gamma_{1}}=0\}$ and $D(A^{1/2})=H^1_{\Gamma_{0}}(\Omega,\C) =\{y\in H^1(\Omega,\C) \ \vert \ y_{\vert\Gamma_{0}}=0\}$ (see e.g. \cite{lt1989}).
\end{remark}
In this article we investigate the two following optimal observability problems.
Let $L\in(0,1)$ be fixed.

\begin{quote}
\noindent{\bf First problem (optimal design for fixed initial data).}
\begin{itemize}
\item \textbf{Wave equation \eqref{waveEqobs}: }\textit{given $(y^0,y^1)\in D(A^{1/2})\times X$, we investigate the problem of maximizing the functional
\begin{equation}\label{quantity1obsW}
G_T(\chi_\omega)=
\int_0^T\int_\omega  \vert \partial_{t}y(t,x) \vert ^2 \, dV_g \, dt ,
\end{equation}
over all possible measurable subsets $\omega$ of $\Omega$ of measure $V_g(\omega)=L V_g( \Omega ) $,
where $y\in C^0(0,T;D(A^{1/2}))\cap C^1(0,T;X)$ is the solution of \eqref{waveEqobs} such that $y(0,\cdot)=y^0(\cdot)$ and $\frac{\partial y}{\partial t}(0,\cdot)=y^1(\cdot)$.}
\item \textbf{Schr\"odinger equation \eqref{schroEqobs}: }\textit{given $y^0\in D(A)$, we investigate the problem of maximizing the functional $G_T$ defined by \eqref{quantity1obsW}
over all possible measurable subsets $\omega$ of $\Omega$ of measure $V_g(\omega)=L V_g(\Omega) $,
where $y\in C^0(0,T;D(A))$ is the solution of \eqref{schroEqobs} such that $y(0,\cdot)=y^0(\cdot)$.}
\end{itemize}
\end{quote}

In the analysis of this first problem, the observability inequalities are not required since we are dealing with  fixed initial data. Accordingly, the optimal set $\omega$, whenever it exists, depends of course on the initial data under consideration. As will be shown, this problem is mathematically challenging and reveals interesting properties. However it is not relevant enough in view of practical applications where the location of the observation or sensors is expected to be uniform with respect to the data and solutions under consideration. 

Consequently, we introduce the following second problem, of a spectral nature, in which, to some extent, all possible solutions are taken into consideration in the optimality criterion.

\begin{quote}
\noindent{\bf Second problem (uniform optimal design)}
\textit{We investigate the problem of maximizing the spectral functional
\begin{equation}\label{defJ}
J(\chi_\omega)=\inf_{j\in\N^*}\int_\omega \phi_j(x)^2 \, dV_g,
\end{equation}
over all possible subsets $\omega$ of $\Omega$ of measure $ V_g(\omega) =L V_g(\Omega) $.
}
\end{quote}

A relevant and natural criterion would certainly consist in maximizing the observability constant over all possible subsets $\omega$ of $\Omega$ of measure $ V_g(\omega) =L V_g(\Omega) $ for a given time $T>0$. Settled as such this problem is however very difficult to handle. Indeed, using an Hilbertian expansion  of the solutions of \eqref{waveEqobs} or \eqref{schroEqobs} in the basis of the eigenfunctions of the Laplacian operator, this leads to inequalities in which the presence of crossed terms makes it difficult to analyze the existence and possible nature of the optimal sets. Furthermore, this criterion depends on the time interval $[0, T]$ while the spectral one above is independent of $T$ and is of diagonal nature, not involving any crossed term. 

The difficulty related with the cross terms already appears in one-dimensional problems (see \cite{PTZObs1}). Actually, this question is very much related with classical problems in non harmonic Fourier analysis, such as the one of determining the best constants in Ingham's inequalities (see \cite{Ingham,JaffardTucsnakZuazua}).

In Section \ref{secMotiv} we describe how the spectral criterion \eqref{defJ} defined above can be derived by various averaging processes applied to the original problem of optimizing the observability constant.
The first one is to perform a time averaging process, leading to interpret the criterion $J(\chi_\omega)$ defined by \eqref{defJ} as a time asymptotic observability constant as $T$ tends to $+\infty$. The second one consists of randomizing the initial data of the wave or Schr\"odinger equation under consideration, which leads to interpret $J(\chi_\omega)$ as a randomized observability constant, corresponding to a randomized observability inequality (see Section \ref{secMotiv} for details).

These notions of time asymptotic or randomized observability inequalities are new and happen to be better fitted to provide a relevant answer to the problem of optimal observability. We provide in Section \ref{secMotiv} precise relations between these new observability constants and their classical deterministic versions.

Note that, when the spectrum of $A$ is not simple, this spectral second problem depends a priori on the choice of the orthonormal basis $(\phi_j)_{j\in\N^*}$ of eigenfunctions. When the spectrum is not simple it is natural to consider an intrinsic variant of the second problem by considering the infimum over all possible normalized eigenfunctions (see Section \ref{sec6.4}).

\subsection{Brief state of the art}
The literature on optimal observation or sensor location problems is abundant in engineering applications (see e.g. \cite{Kumar,Morris,Sigmund,Ucinski,vandeWal} and references therein), but very few mathematical theoretical contributions do exist.
In engineering applications, the aim is to optimize the number, the place and the type of sensors  in order to improve the estimation of the state of the system. Fields of applications are very numerous and concern for example active structural acoustics, piezoelectric actuators, vibration control in mechanical structures, damage detection and chemical reactions, just to name a few of them. In most of these applications however the method consists in approximating appropriately the problem by selecting a finite number of possible optimal candidates and of recasting the problem as a finite dimensional combinatorial optimization problem.
Among these approaches, the closest one to ours consists of considering truncations of Fourier expansion representations.
Adopting such a Fourier point of view, the authors of \cite{henrot_hebrardSCL,henrot_hebrardSICON} studied optimal stabilization issues of the one-dimensional wave equation and, up to our knowledge, these are the first articles in which one can find rigorous mathematical arguments and proofs to characterize the optimal set whenever it exists, for the problem of determining the best possible shape and position of the damping subdomain of a given measure.
In \cite{BellidoDonoso} the authors investigate the problem modeled in \cite{Sigmund} of finding the best possible distributions of two materials (with different elastic Young modulus and different density) in a rod in order to minimize the vibration energy in the structure. For this optimal design problem in wave propagation, the authors of \cite{BellidoDonoso} prove existence results and provide convexification and optimality conditions.
The authors of \cite{Allaire} also propose a convexification formulation of eigenfrequency optimization problems applied to optimal design.
In \cite{FahrooIto} the authors discuss several possible criteria for optimizing the damping of abstract wave equations in Hilbert spaces, and derive optimality conditions for a certain criterion related to a Lyapunov equation.
In \cite{PTZObs1} we investigated the second problem presented previously in the one-dimensional case. We also quote the article \cite{PTZ_HUM} where we study the related problem of finding the optimal location of the support of the control for the one-dimensional wave equation.

\subsection{Short description of the main results of this article}
In this article we provide a complete mathematical analysis of the two optimal observability problems settled in Section \ref{secobservabilityintro}. The article is structured as follows.

\medskip

Section \ref{sec2} is devoted to spectral considerations and to state and prove results interpreting the second problem in terms of a time averaged or a randomized observability inequality (see Corollary \ref{corCTT} and Theorem \ref{corCTT}). In particular it is shown how the time averaging or the randomization with respect to initial data permit to rule out crossed terms and lead to the spectral criterion \eqref{defJ} considered in the second problem.

\medskip

In Section \ref{solvingpb1obs}, we solve the first problem, that is the optimal design problem for fixed initial data.
The main result of this section is Theorem \ref{thmobsopt}, which provides a sufficient condition ensuring existence and uniqueness of a solution of the first problem (see the more precise statement in Section \ref{sec3.1}).

\begin{thmnonumbering}
If the initial data under consideration belong to a suitable class of analytic functions, then the first problem has a unique solution $\omega$, which has a finite number of connected components.
\end{thmnonumbering}

Here, the optimal set $\omega$ is unique up to some subset of zero measure.
Proposition \ref{propCantor} (Section \ref{sec3.2}) shows that the above sufficient condition is, in some sense, sharp, since there exist initial data of class $C^\infty$ such that the first problem has a unique solution $\omega$, which is a fractal set and thus has an infinite number of connected components. An explicit example is built in Appendix \ref{AppendixCantor}.

\medskip

In Section \ref{solvingpb2obs}, we focus on the second problem (uniform optimal design), which is of a spectral nature and, thus, independent on the initial data. As proved in Section \ref{sec2} this problem corresponds to computing the maximal possible value of the time asymptotic or of the randomized observability constant.
We first provide in Section \ref{solvingpb2obssec1} a convexified version of the problem, by considering the convex closure of the set
$$
\mathcal{U}_L = \{ \chi_\omega\ \vert\ \omega\ \textrm{is a measurable subset of}\ \Omega \ \textrm{of measure}\ V_g(\omega)=L V_g(\Omega) \}
$$
for the $L^\infty$ weak star topology, that is
$$
\overline{\mathcal{U}}_L = \{ a\in L^\infty(\Omega,[0,1])\ \vert\ \int_\Omega a(x)\,dV_g=LV_g(\Omega) \}.
$$
The convexified second problem then consists of maximizing the functional
$$J(a)=\inf_{j\in \N^*}\int_{\Omega}a(x) \phi_j(x)^2\, dV_g$$
over $\overline{\mathcal{U}}_L$.
Our main results are Theorems \ref{thmnogap} and \ref{thmnogap2} (stated in Section \ref{secconvexified2}), whose contents are roughly the following.

\begin{thmnonumbering}
\begin{enumerate}
\item Assume that there exists a subsequence of the sequence of probability measures $\mu_{j}=\phi_j^2\, dV_{g}$ converging vaguely to the uniform measure $\frac{1}{ V_g(\Omega) }\, dV_{g}$ (Weak Quantum Ergodicity assumption), and that the sequence of eigenfunctions ${\phi_j}$ is uniformly bounded in $L^\infty(\Omega)$. Then
\begin{equation*}
\sup_{\chi_\omega\in \mathcal{U}_L} \inf_{j\in \N^*}\int_{\omega} \phi_j(x)^2\, dV_g  = \sup_{a\in \overline{\mathcal{U}}_L} \inf_{j\in \N^*}\int_{\Omega}a(x) \phi_j(x)^2\, dV_g = L,
\end{equation*}
for every $L\in (0,1)$.
In other words, there is no gap between the second problem and its convexified version.
\item Assume that the whole sequence of probability measures $\mu_{j}=\phi_j^2\, dV_{g}$ converges vaguely to the uniform measure $\frac{1}{ V_g(\Omega) }\, dV_{g}$ (Quantum Unique Ergodicity assumption), and that the sequence of eigenfunctions ${\phi_j}$ is uniformly bounded in $L^{2p}(\Omega)$, for some $p\in(1,+\infty]$. Then the supremum of $J$ over the subset of $\mathcal{U}_L$ of all characteristic functions of Jordan measurable subsets is as well equal to $L$.
\end{enumerate}
\end{thmnonumbering}

The assumptions of the above result are sufficient but not necessary to derive such a no-gap statement, as it is shown when $\Omega$ is a two-dimensional disk with Dirichlet boundary conditions (see Proposition \ref{propnogapnoQUE}), in spite of the fact that the eigenfunctions do not equidistribute as the eigenfrequencies increase, as illustrated by the well known whispering galleries effect (see discussion in Section \ref{secconvexified2}).

In Section \ref{sec_comments}, we comment on these quantum ergodicity assumptions, that are well known in mathematical physics to be related with concentration phenomena of eigenfunctions (scarring phenomena). We briefly survey the main results of this theory and show their intimate relations with the optimization problems under consideration in the present article.

In Section \ref{sec4.3} we provide some results on the existence of an optimal set achieving the supremum in the above problem, and formulate some open problems. 

Note that Theorems \ref{thmnogap} and \ref{thmnogap2} cannot be inferred from usual $\Gamma$-convergence results. Theorem \ref{thmnogap} is proved in Section \ref{secproofthmnogap}. The proof of Theorem \ref{thmnogap2} in Section \ref{secproofthmnogap2} is of a completely different nature. It is a constructive proof of a maximizing sequence of optimal sets which permits to establish that it is possible to increase the values of $J$ by considering subsets having an increasing number of connected components.

In Section \ref{sec6.4}, we define and study an intrinsic version of our second problem, which does not depend on the choice of an orthonormal basis of eigenfunctions. 
We consider the problem of maximizing the quantity $\inf_{\phi\in \mathcal{E}}\int_\omega \phi(x)^2 \, dV_g$ over $\mathcal{U}_L$, where $\mathcal{E}$ denotes the set of all normalized eigenfunctions of $A$.
For this problem we have a result similar to the one above (Theorems \ref{thmnogap_intrinsic} and \ref{thmnogap_intrinsic2}), and moreover in this intrinsic problem we are able to provide an explicit example where a gap occurs between the problem and its convexified formulation (Proposition \ref{remgapintrinsic}), by considering the unit sphere in $\R^3$ or the unit half-sphere with Dirichlet boundary conditions, and certain quantum limits of a Dirac type.

These results constitute the main contributions of this article. They show precise connections between optimal observability issues and quantum ergodicity properties of $\Omega$.
Such a relation was suggested in the early work \cite{CFNS_1991} concerning the  exponential decay properties of dissipative wave equations.

\medskip

Section \ref{sectrunc2} is devoted to the study of a finite-dimensional spectral approximation of the second problem, namely the problem of maximizing the functional
$$J_N(\chi_\omega)=\min_{1\leq j\leq N}\int_\omega \phi_j(x)^2\,dV_g$$
over $\mathcal{U}_L$.
In Theorem \ref{thm3} we derive a $\Gamma$-convergence property of $J_N$ towards $J$ for the weak star topology of $L^\infty$. In particular, the sets optimizing $J_N$ constitute a maximizing sequence for the convexified version of the maximization problem for $J$, and this, without geometric or ergodicity assumptions on $\Omega$. Of course, then, under the assumptions of the above theorem, these sets constitute a maximizing sequence for the original optimization problem as well, without convexification. We also prove the existence and uniqueness of an optimal set $\omega^N$, which is shown to have a finite number of connected components for every integer $N$.
We also present several numerical simulations showing the shapes and complexity of these sets.

\medskip

In Section \ref{sec_furthercomments}, we provide further comments. First, in Section \ref{sec6.1}, for the second problem \eqref{reducedsecondpb} we analyze possible ways of restricting the classes of domains under consideration to ensure compactness properties. But then, of course, the maximal value of $J$ diminishes. In Section \ref{sec6.2} we extend our main results (for the second problem) to a natural variant of observability inequality for Neumann boundary conditions or in the boundaryless case. Section \ref{sec6.3} is devoted to the analysis of a variant of observability inequality for Dirichlet, mixed Dirichlet-Neumann and Robin boundary conditions, involving a $H^1$ norm. We show that the time averaging or the randomization lead to a slightly different spectral criterion (Theorem \ref{thmconstantRobin}). In contrast to the previous results, in this new situation the uniform optimal design problem has an optimal solution whenever the volume fraction is large enough,  the optimal set being determined by taking only into account a finite number of low frequency modes (see Theorem \ref{thmnogapRobin} for a precise statement). Numerical simulations illustrate this result.
Finally, in Section \ref{sec6.5} we show that the problem of maximizing the observability constant, studied along the paper, is by a classical duality argument equivalent to the optimal design of the control problem of determining the optimal location of internal controllers, for the wave and Schr\"odinger equations.


\section{Preliminaries: spectral considerations}\label{sec2}
\subsection{Spectral expansion of the functional $G_T(\chi_\omega)$}\label{pb1obsnewexpr}
We recall that we have fixed a Hilbertian basis $(\phi_j)_{j\in \N^*}$ of $X$ consisting of eigenfunctions of $A$, associated with the real eigenvalues $(\lambda_j^2)_{j\in\N^*}$. 
Using a series expansion of the solutions of the wave or Schr\"odinger equation in this Hilbertian basis, our objective is to write the functional $G_T$ defined by \eqref{quantity1obsW} in a more suitable way for our mathematical analysis. 

\begin{paragraph}{Wave equation \eqref{waveEqobs}.}
For all initial data $(y^0,y^1)\in D(A^{1/2})\times X$, the solution $y\in C^0(0,T;D(A^{1/2}))\cap C^1(0,T;X)$ of \eqref{waveEqobs} such that
$y(0,\cdot)=y^0(\cdot)$ and $\partial_t y(0,\cdot)=y^1(\cdot)$ can be expanded as
\begin{equation}\label{yDecomp}
y(t,x)=\sum_{j=1}^{+\infty}\left(a_je^{i\lambda_jt}+b_je^{-i\lambda_jt}\right)\phi_j (x),
\end{equation}
where the sequences $(\lambda_{j }a_j)_{j\in\N^*}$ and $(\lambda_{j}b_j)_{j\in\N^*}$ belong to $\ell^2(\C)$ and are determined in terms of the initial data $(y^0,y^1)$ by
\begin{equation}\label{defajbj}
\begin{split}
a_j &= \frac{1}{2}\left(\int_\Omega y^0(x) \phi_j (x)\, dV_g-\frac{i}{\lambda_j} \int_\Omega y^1(x) \phi_j (x)\, dV_g\right),\\
 b_j &= \frac{1}{2}\left(\int_\Omega y^0(x) \phi_j (x)\, dV_g+\frac{i}{\lambda_j} \int_\Omega y^1(x) \phi_j (x)\, dV_g\right).
\end{split}
\end{equation}
for every $j\in\N^*$. Moreover, 
\begin{equation}\label{notice7}
\Vert (y^0,y^1)\Vert_{D(A^{1/2})\times X}^2 = 2\sum_{j=1}^{+\infty} \lambda_{j}^2( \vert a_j \vert ^2+ \vert b_j \vert ^2).
\end{equation}
Plugging \eqref{yDecomp} into \eqref{quantity1obsW} leads to
\begin{eqnarray}\label{GT1}
G_T(\chi_\omega) 
& = &\int_0^T  \int_\omega\left \vert \sum_{j=1}^{+\infty}\lambda_{j}\left(a_je^{i\lambda_jt}-b_je^{-i\lambda_jt}\right)\phi_j(x)  \right \vert ^2\, dV_g\, dt \nonumber \\
& = & \sum_{j,k=1}^{+\infty}\lambda_{j}\lambda_{k}\alpha_{jk} \int_\omega \phi_i(x) \phi_j(x) \, dV_g ,
\end{eqnarray}
where
\begin{equation}\label{defalphaij}
\alpha_{jk} = \int_0^T (a_je^{i\lambda_jt}-b_je^{-i\lambda_jt})(\bar a_ke^{-i\lambda_kt}-\bar b_ke^{i\lambda_kt}) \, dt .
\end{equation}
The coefficients $\alpha_{jk}$, $(j,k)\in(\N^*)^2$, depend only on the initial data $(y^0,y^1)$, and their precise expression is given by
\begin{equation}\label{alphaij}
\begin{split}
\alpha_{jk} &= \frac{2a_j\bar a_k}{\lambda_j-\lambda_k}\sin\left((\lambda_j-\lambda_k)\frac{T}{2}\right)e^{i(\lambda_j-\lambda_k)\frac{T}{2}}-\frac{2a_j\bar b_k}{\lambda_j+\lambda_k}\sin\left((\lambda_j+\lambda_k)\frac{T}{2}\right)e^{i(\lambda_j+\lambda_k)\frac{T}{2}} \\
&\quad - \frac{2b_j\bar a_k}{\lambda_j+\lambda_k}\sin\left((\lambda_j+\lambda_k)\frac{T}{2}\right)e^{-i(\lambda_j+\lambda_k)\frac{T}{2}}+\frac{2b_j\bar b_k}{\lambda_j-\lambda_k}\sin\left((\lambda_j-\lambda_k)\frac{T}{2}\right)e^{-i(\lambda_j-\lambda_k)\frac{T}{2}} 
\end{split}
\end{equation}
whenever $\lambda_j\neq \lambda_k$, and
\begin{equation}\label{alphajj}
\alpha_{jk} = T(a_j\bar a_k+b_j\bar b_k)-\frac{\sin (\lambda_jT)}{\lambda_j}(a_j\bar b_ke^{i\lambda_jT}+b_j\bar a_ke^{-i\lambda_jT})
\end{equation}
when $\lambda_j=\lambda_k$.
\end{paragraph}

\begin{remark}\label{remTmult2pi}
In dimension one, consider $\Omega=[0,\pi]$ with Dirichlet boundary conditions. Then $\phi_j(x)=\sqrt{\frac{2}{\pi}}\sin(jx)$ and $\lambda_j=j$ for every $j\in\N^*$.
In this one-dimensional case, it can be noticed that when the time $T$ is a multiple of $2\pi$  all nondiagonal terms vanish. Indeed, if $T=2p\pi$ with $p\in\N^*$, then $\alpha_{ij}=0$ whenever $i\neq j$, and
\begin{equation}\label{alphajj2pi}
\alpha_{jj} = p\pi ( \vert a_j \vert ^2+ \vert b_j \vert ^2),
\end{equation}
for all $(i,j)\in(\N^*)^2$, and therefore
\begin{equation}\label{GTmult2pi}
G_{2p\pi}(\chi_\omega) = \sum_{j=1}^{+\infty}\lambda_{j}^2\alpha_{jj}\int_\omega\sin^2(jx)\, dx .
\end{equation}
Hence in that case the functional $G_{2p\pi}$ does not involve any crossed terms. The second problem for this one-dimensional case was studied in detail in \cite{PTZObs1}.
\end{remark}

\begin{paragraph}{Schr\"odinger equation \eqref{schroEqobs}.}
For every $y^0\in D(A)$, the solution $y \in C^0(0,T;D(A))$ of \eqref{schroEqobs} such that $y(0,\cdot)=y^0(\cdot)$ can be expanded as
\begin{equation}\label{psiDecomp}
y(t,x)=\sum_{j=1}^{+\infty}c_je^{i\lambda_j^2t}\phi_j (x),
\end{equation}
where the sequence $(\lambda_{j}^2c_j)_{j\in\N^*}$ belongs to $\ell^2(\C)$ and is determined in terms of $y^0$ by
\begin{equation}\label{defcj}
c_j = \int_\Omega y^0(x) \phi_j (x)\, dV_g
\end{equation}
for every $j\in\N^*$. Moreover,
\begin{equation}\label{notice7psi}
\Vert y^0\Vert_{D(A)}^2 = \sum_{j=1}^{+\infty}  \lambda_{j}^4\vert c_j \vert ^2.
\end{equation}
Plugging \eqref{psiDecomp} into \eqref{quantity1obsW} leads to
\begin{equation}\label{GTpsi}
G_T(\chi_\omega) 
= \int_0^T  \int_\omega\left \vert  \sum_{j=1}^{+\infty}\lambda_{j}^2c_je^{i\lambda_j^2t}\phi_j(x)  \right \vert ^2\, dV_g\, dt 
= \sum_{j,k=1}^{+\infty}\lambda_{j}^2\lambda_{k}^2\alpha_{jk} \int_\omega \phi_j(x) \phi_k(x) \, dV_g  ,
\end{equation}
with
\begin{equation}\label{defalphaijS}
\alpha_{jk} = c_j\bar c_k \int_0^Te^{i(\lambda_j^2-\lambda_k^2)t}\, dt = 
 \frac{2c_j\bar c_k}{\lambda_j^2-\lambda_k^2}\sin\left((\lambda_j^2-\lambda_k^2)\frac{T}{2}\right)e^{i(\lambda_j^2-\lambda_k^2)\frac{T}{2}},
\end{equation}
whenever $j\neq k$, and $\alpha_{jj}=\vert c_j\vert^2 T$ whenever $j=k$.
\end{paragraph}

\subsection{Two motivations for studying the second problem}\label{secMotiv}
In this section, as announced at the end of Section \ref{secobservabilityintro}, we provide two motivations for studying the second problem, consisting of maximizing the functional $J$ defined by \eqref{defJ} over the set $\mathcal{U}_L$.

\medskip

\noindent
\textbf{First motivation: averaging in time / time asymptotic observability constant.}\\
First of all, we claim that, for all $(y^0,y^1)\in D(A^{1/2})\times X$, the quantity
$$
\frac{1}{T} \int_0^T\int_\omega  \vert \partial_{t}y(t,x) \vert ^2\, dV_g \, dt ,
$$
where $y\in C^0(0,T;D(A^{1/2}))\cap C^1(0,T;X)$ is the solution of the wave equation \eqref{waveEqobs} such that $y(0,\cdot)=y^0(\cdot)$ and $\partial_t y(0,\cdot)=y^1(\cdot)$, has a limit as $T$ tends to $+\infty$. We refer to lemmas \ref{lemmaGTT} and \ref{lemm1lim} for a proof of this fact.
This leads to define the concept of \textit{time asymptotic observability constant}
\begin{equation}\label{CTinfty}
C_\infty^{(W)}(\chi_\omega)=\inf\left\{ \lim_{T\rightarrow+\infty} \frac{1}{T}  \frac{  \int_0^T\int_\omega  \vert \partial_{t}y(t,x) \vert ^2\, dV_g \, dt }{\Vert (y^0,y^1)\Vert^2_{D(A^{1/2})\times X}} \ \big\vert\  (y^0,y^1)\in D(A^{1/2})\times X \setminus\{(0,0)\} \right\} .
\end{equation}
This constant appears as the largest possible nonnegative constant for which the time asymptotic observability inequality
\begin{equation}\label{ineqobswinfty}
C_\infty^{(W)}(\chi_\omega) \Vert (y^0,y^1)\Vert_{D(A^{1/2})\times X}^2
\leq \lim_{T\rightarrow+\infty} \frac{1}{T} \int_0^T\int_\omega \vert \partial_{t}y(t,x)^2\vert \, dV_g \,dt,
\end{equation}
holds for all $y^0(\cdot)\in D(A^{1/2})$ and $y^1(\cdot)\in X$.

Similarly, for the Schr\"odinger equation, we define
\begin{equation}\label{CSinfty}
C_\infty^{(S)}(\chi_\omega)=\inf\left\{  \lim_{T\rightarrow+\infty} \frac{1}{T} \frac{  \int_0^T\int_\omega  \vert \partial_{t}y(t,x) \vert ^2\, dV_g \, dt }{\Vert y^0\Vert^2_{D(A)}} \ \big\vert\  y^0\in D(A) \setminus\{0\} \right\}.
\end{equation}
This constant is the largest possible nonnegative constant for which the time asymptotic observability inequality
\begin{equation}\label{ineqobssinfty}
C_\infty^{(S)}(\chi_\omega) \Vert y^0\Vert_{D(A)}^2
\leq \lim_{T\rightarrow+\infty} \frac{1}{T} \int_0^T\int_\omega \vert \partial_{t}y(t,x)^2\vert \, dV_g \,dt,
\end{equation}
holds for every $y^0(\cdot)\in D(A)$.

We have the following results.

\begin{theorem}\label{theoCTT}
For every measurable subset $\omega$ of $\Omega$, there holds
$$
2\, C_\infty^{(W)}(\chi_\omega)  = C_\infty^{(S)}(\chi_\omega) 
= \inf \left\{\frac{\int_\omega \sum_{\lambda\in U}\left|\sum_{k\in I(\lambda)}c_k\phi_k(x)\right|^2\, dV_g}{\sum_{k=1}^{+\infty}|c_k|^2}\ \vert\ (c_j)_{j\in \N^*}\in\ell^2(\C)\setminus\{0\} \right\},
$$
where $U$ is the set of all distinct eigenvalues $\lambda_k$ and $I(\lambda)=\{j\in \N^*\ \vert\ \lambda_j=\lambda\}$.
\end{theorem}

\begin{corollary}\label{corCTT}
There holds
$$
2C_{\infty}^{(W)}(\chi_{\omega})=C_{\infty}^{(S)}(\chi_{\omega})\leq J(\chi_{\omega}),
$$
for every measurable subset $\omega$ of $\Omega$.

If the domain $\Omega$ is such that every eigenvalue of $A$ is simple, then
$$
2\, C_\infty^{(W)}(\chi_\omega)  = C_\infty^{(S)}(\chi_\omega) 
= \inf_{j\in\N^*}\int_\omega \phi_j(x)^2 \, dV_g = J(\chi_\omega) ,
$$
for every measurable subset $\omega$ of $\Omega$.
\end{corollary}

The proof of these results are done in Section \ref{prooftheoCTT}.
Note that, as is well known, the assumption of the simplicity of the spectrum of the Dirichlet-Laplacian is generic with respect to the domain $\Omega$ (see e.g. \cite{micheletti,ule,hillairet-judge}). The spectrum of the Neumann-Laplacian is also known to consist of simple eigenvalues for many choices of $\Omega$. For instance, it is proved in \cite{hillairet-judge} that this property holds for almost every polygon of $\R^2$ having $N$ vertices.


\begin{remark}\label{remBZ}
It follows obviously from the definitions of the observability constants that
$$
\limsup_{T\to +\infty} \frac{C_T^{(W)}(\chi_\omega)}{T}\leq C_\infty^{(W)}(\chi_\omega)
\quad\textrm{and}\quad 
\limsup_{T\to +\infty} \frac{C_T^{(S)}(\chi_\omega)}{T}\leq C_\infty^{(S)}(\chi_\omega),
$$
for every measurable subset $\omega$ of $\Omega$.
However, the equalities do not hold in general.
Indeed, consider a set $\Omega$ with a smooth boundary, and a pair $(\omega,T)$ not satisfying the Geometric Control Condition. Then there must hold $C_T^{(W)}(\chi_\omega)=0$. Besides, $J(\chi_\omega)$ may be positive.\\
An example of such a situation for the wave equation is provided by considering $\Omega=[0,\pi]^2$ with Dirichlet boundary conditions and $L=1/2$. It is indeed proved further (see Lemma \ref{lem4} and Remark \ref{rem21}) that the domain $\omega=\{(x,y)\in \Omega\ \vert \ x\leq \pi /2\}$
maximizes $J$ over $\mathcal{U}_L$, and that $J(\chi_\omega)=1/2$. Clearly, such a domain does not satisfy the Geometric Control Condition, and one has $C_T^{(W)}(\chi_\omega)=0$, whereas $C_\infty^{(W)}(\chi_\omega)=1/4$.
\\
Another class of examples for the wave equation is provided by the Bunimovich stadium (shaped at the top right of Figure \ref{ergoShapes} further) with Dirichlet boundary conditions. Setting $\Omega=R\cup W$, where $R$ is the rectangular part and $W$ the circular wings, it is proved in \cite{BurqZworski,BurqZworski2} that, for any open neighborhood $\omega$ of the closure of $W$ (or even, any neighborhood $\omega$ of the vertical intervals between $R$ and $W$) in $\Omega$, there exists $c>0$ such that $\int_\omega \phi_j(x)^2\,dx\geq c$ for every $j\in\N^*$. It follows that $J(\chi_\omega)>0$, whereas $C_T^{(W)}(\chi_\omega)=0$ since $\omega$ does not satisfy the Geometric Control Condition. It can be noted that the result still holds if one replaces the wings $W$ by any other manifold glued along $R$, so that $\Omega$ is a partially rectangular domain.

We are not aware of such kinds of examples for the Schr\"odinger equation, although there exist some configurations for which $C_{T}^{(S)}(\chi_{\omega})=0$. For instance, for $\Omega=S^2$, the unit Euclidean sphere of $\R^3$, it is well known (see for instance \cite{JakobsonZelditch} and Remark \ref{rem_existencegap} further) that, if $(\phi_j)_{j\in\N^*}$ is the usual orthonormal basis of spherical harmonics, then a subsequence of $\phi_j^2$ converges to the Dirac measure along the equator. Therefore, if a subset $\omega$ of $\Omega$ does not contain any neighborhood of this equator, then $C_{T}^{(S)}(\chi_{\omega})=C_{\infty}^{(S)}(\chi_{\omega})=J(\chi_{\omega})=0$ (see Corollary \ref{corCTT}). Note that the same situation occurs in the unit disk of the Euclidean plane, choosing any subset $\omega$ compactly included in the disk, since there exists a subsequence of the squares of the usual Dirichlet-Laplacian eigenfunctions concentrating on the boundary of the disk (see e.g. \cite{Lagnese}, see also Section \ref{sec_proof_propnogapnoQUE} further).
\end{remark}

\medskip

\noindent\textbf{Second motivation: averaging with respect to initial data / randomized observability constant.}\\
The observability constants in \eqref{CT} and \eqref{KT} are defined as an infimum over \textit{all possible} (deterministic) initial data. We are going to modify slightly this definition by randomizing the initial data in some precise sense, and considering an averaged version of the observability inequality with a new (randomized) observability constant.
To make this point precise, we consider spectral expansions of the solutions of the wave and Schr\"odinger equations, and we get
$$
C_T^{(W)}(\chi_\omega)=\frac{1}{2}\inf_{\substack{(\lambda_{j}a_j),(\lambda_{j}b_j)\in\ell^2(\C)\\ \sum_{j=1}^{+\infty}\lambda_{j}^2( \vert a_j \vert ^2+ \vert b_j \vert ^2)=1}}\int_0^T  \int_\omega\left \vert \sum_{j=1}^{+\infty}\lambda_{j}\left(a_je^{i\lambda_jt}-b_je^{-i\lambda_jt}\right)\phi_j(x)  \right \vert ^2\, dV_g\, dt,
$$
and
$$
C_T^{(S)}(\chi_\omega)=\inf_{\substack{(\lambda_{j}^2c_j)\in\ell^2(\C)\\ \sum_{j=1}^{+\infty} \lambda_{j}^4\vert c_j \vert ^2=1}}\int_0^T\int_\omega \left \vert \sum_{j=1}^{+\infty}\lambda_{j}^2c_je^{i\lambda_j^2t}\phi_j(x)\right \vert ^2\, dV_g\, dt.
$$
The coefficients $a_j$, $b_j$ and $c_j$ in the expressions above are the Fourier coefficients of the initial data, defined by \eqref{defajbj} and \eqref{defcj} respectively.

Following the works of N. Burq and N. Tzvetkov on nonlinear partial differential equations with random initial data (see \cite{Burq,BurqTzvetkov1,BurqTzvetkov2,BurqTzvetkov3}) using early ideas of Paley and Zygmund (see \cite{PaleyZygmund}), we randomize these coefficients by multiplying each of them by some well chosen random law. This random selection of all possible initial data for the wave equation \eqref{waveEqobs} consists of replacing $C_T^{(W)}(\chi_\omega)$ by the randomized version
\begin{equation}\label{CTrand}
C_{T,\textrm{rand}}^{(W)}(\chi_\omega)=\frac{1}{2}\inf_{\substack{(\lambda_{j}a_j),(\lambda_{j}b_j)\in\ell^2(\C)\\ \sum_{j=1}^{+\infty}\lambda_{j}^2( \vert a_j \vert ^2+ \vert b_j \vert ^2)=1}}  \mathbb{E}\left( \int_0^T  \int_\omega\left\vert\sum_{j=1}^{+\infty}\lambda_{j}\left(\beta_{1,j}^\nu a_je^{i\lambda_jt}-\beta_{2,j}^\nu b_je^{-i\lambda_jt}\right)\phi_j(x) \right\vert^2\, dV_g \, dt \right),
\end{equation}
where $(\beta_{1,j}^\nu)_{j\in\N^*}$ and $(\beta_{2,j}^\nu)_{j\in\N^*}$ are two sequences of independent Bernoulli random variables on a probability space $(\mathcal{X},\mathcal{A},\mathbb{P})$, satisfying
$$
\mathbb{P}(\beta_{1,j}^\nu=\pm 1)=\mathbb{P}(\beta_{2,j}^\nu=\pm 1)=\frac{1}{2}\hspace{0.5cm} \textrm{ and }\hspace{0.5cm}\mathbb{E}(\beta_{1,j}^\nu\beta_{2,k}^\nu)=0
$$
for every $j$ and $k$ in $\N^*$ and every $\nu\in \mathcal{X}$. Here, the notation $\mathbb{E}$ stands for the expectation over the space $\mathcal{X}$ with respect to the probability measure $\mathbb{P}$.
In other words, instead of considering the deterministic observability inequality \eqref{ineqobsw} for the wave equation \eqref{waveEqobs}, we consider the \textit{randomized observability inequality}
\begin{equation}\label{ineqobswRand}
C_{T,\textrm{rand}}^{(W)}(\chi_\omega) \Vert (y^0,y^1)\Vert_{D(A^{1/2})\times X}^2
\leq \mathbb{E}\left(\int_0^T\int_\omega \vert \partial_{t} y_\nu(t,x)\vert^2 \, dV_g \,dt\right),
\end{equation}
for all $y^0(\cdot)\in D(A^{1/2})$ and $y^1(\cdot)\in X$, where $y_\nu$ denotes the solution of the wave equation with the random initial data $y^0_\nu(\cdot)$ and $y^1_\nu(\cdot)$ determined by their Fourier coefficients $a_j^\nu=\beta_{1,j}^\nu a_j$ and $b_j^\nu=\beta_{2,j}^\nu b_j$ (see \eqref{defajbj} for the explicit relation between the Fourier coefficients and the initial data), that is, 
\begin{equation}\label{defynu}
y_\nu(t,x)=\sum_{j=1}^{+\infty}\left(\beta_{1,j}^\nu a_je^{i\lambda_jt}+\beta_{2,j}^\nu b_je^{-i\lambda_jt}\right)\phi_j(x).
\end{equation}
This new constant $C_{T,\textrm{rand}}^{(W)}(\chi_\omega)$ is called \textit{randomized observability constant}.

Similarly, making a random selection of all possible initial data for the Schr\"odinger equation \eqref{schroEqobs} we replace $C_T^{(S)}(\chi_\omega)$ by
\begin{equation}\label{CSrand}
C_{T,\textrm{rand}}^{(S)}(\chi_\omega)=\inf_{\substack{(\lambda_{j}^2c_j)\in\ell^2(\C)\\ \sum_{j=1}^{+\infty} \lambda_{j}^4\vert c_j \vert ^2=1}}\mathbb{E}\left(  \int_0^T\int_\omega \left\vert\sum_{j=1}^{+\infty}\lambda_{j}^2\beta_j^\nu c_je^{i\lambda_j^2t}\phi_j(x)\right\vert^2\, dV_g\, dt \right),
\end{equation}
where $(\beta_j^\nu)_{j\in\N^*}$ denotes a sequence of independent Bernoulli random variables  on a probability space $(\mathcal{X},\mathcal{A},\mathbb{P})$.
This corresponds to considering the randomized observability inequality\begin{equation}\label{ineqobssRand}
C_T^{(S)}(\chi_\omega) \Vert y^0\Vert_{D(A)}^2\leq \mathbb{E}\left(\int_0^T\int_\omega  \vert \partial_{t}y_\nu(t,x) \vert ^2 \, dV_g \,dt\right),
\end{equation}
for every $y^0(\cdot)\in D(A)$, where $y_\nu$ denotes the solution of the Schr\"odinger equation with the random initial data $y_\nu^0(\cdot)$ determined by its Fourier coefficients $c_j^\nu=\beta^\nu_jc_j$ (see \eqref{defcj} for the explicit dependence between the Fourier coefficients and the initial data), that is, 
$$
y_\nu(t,x)=\sum_{j=1}^{+\infty}\beta_j^\nu c_je^{i\lambda_j^2t}\phi_j(x).
$$

The following theorem, whose proof is done in Section \ref{proofpropHazardCst}, provides one more motivation of studying the second problem \eqref{defJ}.

\begin{theorem}\label{propHazardCst}
There holds
\begin{equation*}
2\,C_{T,\textrm{rand}}^{(W)}(\chi_\omega) = C_{T,\textrm{rand}}^{(S)}(\chi_\omega)
= T\inf_{j\in\N^*}\int_\omega\phi_j(x)^2 \, dV_g =  T J(\chi_\omega),
\end{equation*}
for every measurable subset $\omega$ of $\Omega$.
\end{theorem}

\begin{remark}
It can be easily checked that Theorem \ref{propHazardCst} still holds true when considering, in the above randomization procedure, more general real random variables that are independent, have mean equal to $0$, variance $1$, and have a super exponential decay. We refer to \cite{Burq,BurqTzvetkov1,BurqTzvetkov2} for more details on these randomization issues.
Bernoulli and Gaussian random variables satisfy such appropriate assumptions. As proved in \cite{BurqTzvetkov3}, for all initial data $(y^0,y^1)\in D(A^{1/2})\times X$, the Bernoulli randomization keeps constant the $D(A^{1/2})\times X$ norm, whereas the Gaussian randomization generates a dense subset of $D(A^{1/2})\times X$ through the mapping $R_{(y^0,y^1)}:\nu\in X\mapsto(y^0_\nu,y^1_\nu)$ provided that all Fourier coefficients of $(y^0,y^1)$ are nonzero and that the measure $\theta$ charges all open sets of $\R$. The measure $\mu_{(y^0,y^1)}$ defined as the image of $\mathcal{P}$ by $R_{(y^0,y^1)}$ strongly depends both on the choice of the random variables and on the choice of the initial data $(y^0,y^1)$. Properties of these measures are established in \cite{BurqTzvetkov3}.
\end{remark}

\begin{remark}
It is easy to see that  
$C_{T,\textrm{rand}}^{(W)}(\chi_\omega)\geq C_{T}^{(W)}(\chi_\omega)$
and
$C_{T,\textrm{rand}}^{(S)}(\chi_\omega)\geq C_{T}^{(S)}(\chi_\omega)$, for every measurable subset $\omega$ of $\Omega$, and every $T>0$.
\end{remark}

\begin{remark}\label{remBZ2}
As mentioned previously, the problem of maximizing the \textit{deterministic} (classical) observability constants $C_T^{(W)}(\chi_\omega)$ and $C_T^{(S)}(\chi_\omega)$ defined by \eqref{CT} and \eqref{KT} respectively, over all possible measurable subsets $\omega$ of $\Omega$ of measure $V_g(\omega)=L V_g(\Omega) $, is open and is probably very difficult. It can however be noticed that, for practical issues, it is actually more natural to consider the problem of maximizing the \textit{randomized} observability constants defined by \eqref{CTrand} and \eqref{CSrand} respectively. Indeed, when considering for instance the practical problem of locating sensors in an optimal way, the optimality should be thought in terms of an average with respect to a large number of experiments.
From this point of view, the deterministic observability constants are expected to be pessimistic with respect to their randomized versions. Indeed, in general it is expected that $C_{T,\textrm{rand}}^{(W)}(\chi_\omega)>C_T^{(W)}(\chi_\omega)$ and $C_{T,\textrm{rand}}^{(S)}(\chi_\omega)>C_T^{(S)}(\chi_\omega)$.\\
In dimension one, with $\Omega=[0,\pi]$ and Dirichlet boundary conditions, it follows from \cite[Proposition 2]{PTZObs1} (where this one-dimensional case is studied in detail) that these strict inequalities hold if and only if $T$ is not an integer multiple of $\pi$ (note that if $T$ is a multiple of $2\pi$ then the equalities follow immediately from Parseval's Theorem). Note that, in the one-dimensional case, the  GCC is satisfied for every $T\geq 2\pi$, and the fact that the deterministic and the randomized observability constants do not coincide is due to crossed Fourier modes in the deterministic case.

In dimension greater than one, there is a further  class of obvious examples where the strict inequality holds. This is particularly the case when one is able to assert that $C_T^{(W)}(\chi_\omega)=0$ whereas $J(\chi_\omega)>0$. Such examples have been given and discussed in Remark \ref{remBZ}. Note however that, in this multi-dimensional case, this is due to the fact that the GCC can fail but the minimal spectral trace over $\omega$ is positive. This fact was also observed in \cite{lebeau2} when characterizing the decay rates for multi-dimensional dissipative wave equations.
\end{remark}

\subsection{Proofs of Theorem \ref{theoCTT} and Corollary \ref{corCTT}}\label{prooftheoCTT}
We prove Theorem \ref{theoCTT} and Corollary \ref{corCTT} only for $C_T^{(W)}(\chi_\omega)$ (wave equation). The proof for $C_T^{(S)}(\chi_\omega)$ (Schr\"odinger equation) follows the same lines. For the convenience of the reader, we first prove Theorem \ref{theoCTT} in the particular case where all the eigenvalues of $\triangle_{g}$ are simple (it corresponds exactly to the proof of Corollary \ref{corCTT}) and we then comment the generalization to the case of multiple eigenvalues.

From \eqref{yDecomp}, we have $y(t,x)=\sum_{j=1}^{+\infty}y_j(t,x)$ with 
\begin{equation}\label{yDecomp2}
y_j(t,x)=(a_je^{i\lambda_jt}+b_je^{-i\lambda_jt})\phi_j(x).
\end{equation}
Without loss of generality, we consider initial data $(y^0,y^1)\in D(A^{1/2})\times X$ such that $\Vert (y^0,y^1)\Vert^2_{D(A^{1/2})\times X}=2$, in other words such that $\sum_{j\in\N^*} \lambda_{j}^2( \vert a_j \vert ^2+ \vert b_j \vert ^2)=1$ (using \eqref{notice7}).

Setting
$$
\Sigma_T(a,b)= \frac{1}{T}\frac{ G_T(\chi_\omega) }{\Vert (y^0,y^1)\Vert^2_{D(A^{1/2})\times X}}=\frac{1}{2T} G_T(\chi_\omega),
$$
we write for an arbitrary $N\in \N^*$,
\begin{eqnarray}
\Sigma_T(a,b)  &
 =  \displaystyle \frac{1}{T} \int_{0}^T \int_\omega &
 \left(\left \vert \sum_{j=1}^N y_j(t,x)\right \vert ^2+\left \vert \sum_{k=N+1}^{+\infty} y_k(t,x)\right \vert ^2 
 \right. \nonumber\\
 & & \left. \quad
 +2\Re e \left(\sum_{j=1}^N y_j(t,x)\sum_{k=N+1}^{+\infty} \bar y_k(t,x)\right) \right)\, dV_g\, dt.\label{RHSSigma}
 \end{eqnarray}
Using the assumption that the spectrum of $A$ consists of simple eigenvalues, we have the following result. 

\begin{lemma}\label{lemmaGTT}
With the notations above,
$$
\lim_{T\to +\infty}\frac{1}{T} \int_{0}^T \int_\omega \left \vert \sum_{j=1}^{N} y_j(t,x) \right \vert ^2 dV_g\, dt =\sum_{j=1}^N\lambda_{j}^2( \vert a_j \vert ^2+ \vert b_j \vert ^2)\int_\omega  \phi_j(x)^2\, dV_g.
$$
\end{lemma}

\begin{proof}
Since the sum is finite we can invert the infimum (which is a minimum) and the limit. 
Now, we write
\begin{eqnarray*}
\frac{1}{T} \int_{0}^T \int_\omega \left \vert \sum_{j=1}^{N} y_j(t,x) \right \vert ^2 dV_g\, dt & = & \frac{1}{T} \sum_{j=1}^N\lambda_{j}^2\alpha_{jj}\int_\omega\phi_j(x)^2\, dV_g
\\ & & 
+\frac{1}{T}\sum_{j=1}^N\sum_{\substack{k=1\\ k\neq j}}^N\lambda_{j}\lambda_{k}\alpha_{jk}\int_\omega \phi_j(x)\phi_k(x)\, dV_g,
\end{eqnarray*}
where $\alpha_{jk}$ is defined by \eqref{defalphaij}.
Using \eqref{alphaij} and \eqref{alphajj}, we get
$$
\lim_{T\to +\infty}\frac{\alpha_{jj}}{T}= \vert a_j \vert ^2+ \vert b_j \vert ^2,
$$
for every $j\in\N^*$ and, using that the spectrum of $A$ consists of simple eigenvalues, 
\begin{equation}\label{cjkestimate}
\vert\alpha_{jk}\vert\leq \frac{4\max_{1\leq j,k\leq N} (\lambda_j,\lambda_k)}{ \vert \lambda_j^2-\lambda_k^2 \vert },
\end{equation}
whenever $j\neq k$. The conclusion follows easily.
\end{proof}
Let us now estimate the remaining terms
$$
R =  \frac{1}{T} \int_{0}^T \int_\omega \left \vert \sum_{j=N+1}^{+\infty} y_j(t,x) \right \vert ^2 \, dV_g\, dt
$$
and
$$
\delta = \frac{1}{T}\Re e\left( \int_{0}^T \int_\omega \sum_{j=1}^N y_j(t,x)\sum_{k=N+1}^{+\infty} \bar y_k(t,x) \, dV_g\, dt\right)
$$
of the right-hand side of \eqref{RHSSigma}.

\begin{paragraph}{Estimate of $R$.}
Using the fact that the $\phi_j$'s form a hilbertian basis, we get
\begin{eqnarray*}
R & \leq & \frac{1}{T} \int_{0}^T \int_\Omega \left \vert \sum_{j=N+1}^{+\infty} y_j(t,x) \right \vert ^2 \, dV_g\, dt\\
 & = &  \frac{1}{T} \sum_{j=N+1}^{+\infty}\int_0^T \lambda_{j}^2\vert a_je^{i\lambda_jt}-b_je^{-i\lambda_jt} \vert ^2\, dt\\
 & = &  \frac{1}{T} \sum_{j=N+1}^{+\infty}\lambda_{j}^2\left(T( \vert a_j \vert ^2+ \vert b_j \vert ^2)-\frac{1}{\lambda_j}\Re e \left(a_j\bar b_j\frac{e^{2i\lambda_jT}-1}{i}\right)\right)
\end{eqnarray*}
and finally
\begin{equation}\label{estimR}
R\leq \left(1+\frac{1}{\lambda_NT}\right)\sum_{j=N+1}^{+\infty}\lambda_{j}^2( \vert a_j \vert ^2+ \vert b_j \vert ^2).
\end{equation}
\end{paragraph}
\begin{paragraph}{Estimate of $\delta$.}
Using \eqref{alphaij} and the fact that $\lambda_j\neq \lambda_k$ for every $j\in \{1,\cdots,N\}$ and every $k\geq N+1$, we have
$$
 \vert \delta \vert  \leq \frac{2}{T}(S_1^N+S_2^N+S_3^N+S_4^N) ,
$$
with
\begin{eqnarray*}
S_1^N & = & \left \vert \sum_{j=1}^N\sum_{k=N+1}^{+\infty}\frac{\lambda_{j}\lambda_{k}}{\lambda_j-\lambda_k}a_j\bar a_ke^{i(\lambda_j-\lambda_k)\frac{T}{2}}\sin\left((\lambda_j-\lambda_k)\frac{T}{2}\right)\int_\omega \phi_j(x)\phi_k(x)\, dV_g\right \vert , \\
S_2^N & = & \left \vert \sum_{j=1}^N\sum_{k=N+1}^{+\infty}\frac{\lambda_{j}\lambda_{k}}{\lambda_j+\lambda_k}a_j\bar b_ke^{i(\lambda_j+\lambda_k)\frac{T}{2}}\sin\left((\lambda_j+\lambda_k)\frac{T}{2}\right)\int_\omega \phi_j(x)\phi_k(x)\, dV_g\right \vert , \\
\end{eqnarray*}
\begin{eqnarray*}
S_3^N & = & \left \vert \sum_{j=1}^N\sum_{k=N+1}^{+\infty}\frac{\lambda_{j}\lambda_{k}}{\lambda_j+\lambda_k}b_j\bar a_ke^{-i(\lambda_j+\lambda_k)\frac{T}{2}}\sin\left((\lambda_j+\lambda_k)\frac{T}{2}\right)\int_\omega \phi_j(x)\phi_k(x)\, dV_g\right \vert , \\
S_4^N & = & \left \vert \sum_{j=1}^N\sum_{k=N+1}^{+\infty}\frac{\lambda_{j}\lambda_{k}}{\lambda_j-\lambda_k}b_j\bar b_ke^{-i(\lambda_j-\lambda_k)\frac{T}{2}}\sin\left((\lambda_j-\lambda_k)\frac{T}{2}\right)\int_\omega \phi_j(x)\phi_k(x)\, dV_g\right \vert .
\end{eqnarray*}
Let us estimate $S_1^N$. We write
$$
S_1^N = \left \vert \sum_{j=1}^N\lambda_{j}a_j \int_\omega \phi_j(x) \sum_{k=N+1}^{+\infty}\frac{\lambda_{k}\bar a_k}{\lambda_j-\lambda_k}e^{i(\lambda_j-\lambda_k)\frac{T}{2}}\sin\left((\lambda_j-\lambda_k)\frac{T}{2}\right)\phi_k(x)\, dV_g\right \vert ,
$$
and, using the Cauchy-Schwarz inequality and the fact that the integral of a nonnegative function over $\omega$ is lower than the integral of the same function over $\Omega$, one gets
\begin{eqnarray*}
S_1^N & \leq &  \sum_{j=1}^N \lambda_{j}\vert a_j \vert  \left(\int_\Omega \left \vert \sum_{k=N+1}^{+\infty}\frac{\lambda_{k}\bar a_k}{\lambda_j-\lambda_k}e^{i(\lambda_j-\lambda_k)\frac{T}{2}}\sin\left((\lambda_j-\lambda_k)\frac{T}{2}\right)\phi_k(x)\right \vert ^2\, dV_g\right)^{1/2}\\
 & = &  \sum_{j=1}^N \lambda_{j}\vert a_j \vert  \left(\sum_{k=N+1}^{+\infty}\frac{ \lambda_{k}^2\vert a_k \vert ^2}{(\lambda_j-\lambda_k)^2}\sin\left((\lambda_j-\lambda_k)\frac{T}{2}\right)^2\right)^{1/2}.
\end{eqnarray*}
The last equality is established by expanding the square of the sum inside the integral, and by using the fact that the $\phi_k$'s are orthonormal in $L^2(\Omega)$.
Since the spectrum of $A$ consists of simple eigenvalues (assumed to form an increasing sequence), we infer that $\lambda_k-\lambda_j\geq \lambda_{N+1}-\lambda_N$ for all $j\in \{1,\cdots,N\}$ and $k\geq N+1$, and since $\sum_{j=1}^{+\infty}\lambda_{j}^2 \vert a_j \vert ^2\leq 1$, it follows that
$$
S_1^N\leq \frac{1}{\lambda_{N+1}-\lambda_N}\sum_{j=1}^N\lambda_{j} \vert a_j \vert \left(\sum_{k=N+1}^{+\infty} \lambda_{k}^2\vert a_k \vert ^2\right)^{1/2}\leq \frac{N}{\lambda_{N+1}-\lambda_N}.
$$
The same arguments lead to the estimates
$$
S_2^N\leq \frac{N}{\lambda_N}, \qquad S_3^N\leq \frac{N}{\lambda_N}, \qquad S_4^N\leq \frac{N}{\lambda_{N+1}-\lambda_N},
$$
and therefore,
\begin{equation}\label{liminf2}
 \vert \delta \vert  
\leq \frac{4N}{T}\left(\frac{1}{\lambda_N}+\frac{1}{\lambda_{N+1}-\lambda_N}\right).
\end{equation}\ \\
\end{paragraph}

Now, combining Lemma \ref{lemmaGTT} with the estimates \eqref{estimR} and \eqref{liminf2} yields that for every $\varepsilon>0$, there exist $N_\varepsilon\in\N^*$ and $T(\varepsilon,N_\varepsilon) >0$ such that, if $N\geq N_\varepsilon$ and $T\geq T(\varepsilon,N_\varepsilon)$, then
$$
\left|\Sigma_T(a,b)-\sum_{j=1}^N\lambda_{j}^2( \vert a_j \vert ^2+ \vert b_j \vert ^2)\int_\omega  \phi_j(x)^2\, dV_g\right|\leq \varepsilon.
$$
As an immediate consequence, and using the obvious fact that, for every $\eta>0$, there exists $N_\eta\in\mathbb{N}^*$ such that, if $N\geq N_\eta$ then
$$
\left|\sum_{j=1}^{+\infty}\lambda_{j}^2( \vert a_j \vert ^2+ \vert b_j \vert ^2)\int_\omega  \phi_j(x)^2\,dV_g-\sum_{j=1}^N\lambda_{j}^2( \vert a_j \vert ^2+ \vert b_j \vert ^2)\int_\omega  \phi_j(x)^2\,dV_g\right|\leq \eta,
$$
one deduces that
$$
\lim_{T\to +\infty}\Sigma_T(a,b)=\sum_{j=1}^{+\infty}\lambda_{j}^2( \vert a_j \vert ^2+ \vert b_j \vert ^2)\int_\omega  \phi_j(x)^2\, dV_g.
$$
At this step, we have proved the following lemma, which improves the statement of Lemma \ref{lemmaGTT}.

\begin{lemma}\label{lemm1lim}
Denoting by $a_j$ and $b_j$ the Fourier coefficients of $(y^0,y^1)$ defined by \eqref{defajbj}, there holds
$$
\lim_{T\rightarrow+\infty} \frac{1}{T}  \int_0^T\int_\omega  \vert y(t,x) \vert ^2\, dV_g \, dt = 
\sum_{j=1}^{+\infty}\lambda_{j}^2(\vert a_j\vert^2+\vert b_j\vert^2)\int_\omega \phi_j(x)^2\,dV_g .
$$
\end{lemma}

Corollary \ref{corCTT} follows, noting that
$$
\inf_{\substack{(\lambda_{j}a_j), (\lambda_{j}b_j)\in \ell^2(\C)\\ \sum_{j=1}^{+\infty}\lambda_{j}^2 (|a_j|^2+|b_j|^2)=1}}\sum_{j=1}^{+\infty}\lambda_{j}^2( \vert a_j \vert ^2+ \vert b_j \vert ^2)\int_\omega  \phi_j(x)^2\, dV_g
= \inf_{j\in\N^*}\int_\omega \phi_j(x)^2\,dV_g.
$$

To finish the proof, we now explain how the arguments above can be generalized to the case of multiple eigenvalues. In particular, the statement of Lemma 1 is adapted in the following way.
\begin{lemma}
Using the previous notations, one has
$$
\lim_{T\to +\infty}\frac{1}{T} \int_{0}^T \int_\omega \left \vert \sum_{j=1}^{N} y_j(t,x) \right \vert ^2 dV_g\, dt =\sum_{\substack{\lambda\in U\\ \lambda\leq \lambda_N}}\int_\omega \left(\left|\sum_{k\in I(\lambda)} \lambda_{k}a_k\phi_k(x)\right|^2+\left|\sum_{k\in I(\lambda)} \lambda_{k}b_k\phi_k(x)\right|^2\right) \, dV_g.
$$
\end{lemma}

\begin{proof}
Following the proof of Lemma \ref{lemmaGTT}, simple computations show that
\begin{eqnarray*}
\frac{1}{T} \int_{0}^T \int_\omega \left \vert \sum_{j=1}^{N} y_j(t,x) \right \vert ^2 dV_g\, dt  & = & \frac{1}{T}\sum_{\lambda\in U}\sum_{(j,k)\in I(\lambda)^2}\lambda_{j}\lambda_{k}\alpha_{jk}\int_\omega \phi_j(x)\phi_k(x)\, dV_g\\
 & & +\frac{1}{T}\sum_{\substack{(\lambda,\mu)\in U^2\\ \lambda\neq \mu}}\sum_{\substack{j\in I(\lambda)\\ k\in I(\mu)}}\lambda_{j}\lambda_{k}\alpha_{jk}\int_\omega \phi_j(x)\phi_k(x)\, dV_g,
\end{eqnarray*}
where 
\begin{equation*}
\lim_{T\to +\infty}\frac{\alpha_{jk}}{T} = 
\left\{ \begin{array}{ll}
a_j\bar a_k+b_j\bar b_k & \textrm{if}\ (j,k)\in I(\lambda)^2,\\
0 & \textrm{if}\ j\in I(\lambda),\ k\in I(\mu),\ \textrm{with}\ (\lambda,\mu)\in U^2\ \textrm{and}\ \lambda\neq \mu.
\end{array}\right.
\end{equation*}
The conclusion of the lemma follows.
\end{proof}

To derive Theorem \ref{theoCTT}, it suffices to note that the previous estimates on $R$ and $\delta$ are still valid and that
\begin{equation*}
\begin{split}
& \inf_{\substack{(\lambda_{j}a_j), (\lambda_{j}b_j)\in \ell^2(\C)\\ \sum_{j=1}^{+\infty} \lambda_{j}^2(|a_j|^2+|b_j|^2)=1}}\sum_{\substack{\lambda\in U\\ \lambda\leq \lambda_N}}\int_\omega \left(\left|\sum_{k\in I(\lambda)} \lambda_{k}a_k\phi_k(x)\right|^2+\left|\sum_{k\in I(\lambda)} \lambda_{k}b_k\phi_k(x)\right|^2\right) \, dV_g \\
& = \inf_{\substack{(c_k)_{j\in \N^*}\in\ell^2(\C)\\ \sum_{k=1}^{+\infty}|c_k|^2}} \int_\omega \sum_{\lambda\in U}\left|\sum_{k\in I(\lambda)}c_k\phi_k(x)\right|^2\, dV_g.
 \end{split}
\end{equation*}

\subsection{Proof of Theorem \ref{propHazardCst}}\label{proofpropHazardCst}
From Fubini's theorem, using the fact that the random laws are independent, of zero mean and of variance $1$, we have
\begin{eqnarray*}
C_{T,\textrm{rand}}^{(S)}(\chi_\omega) & = & \inf_{\substack{(\lambda_{j}^2c_j)\in\ell^2(\C)\\ \sum_{j=1}^{+\infty} \lambda_{j}^4\vert c_j \vert ^2=1}}\int_0^T\int_\omega \mathbb{E}\left(\left \vert \sum_{j=1}^{+\infty}\beta_j^\nu \lambda_{j}^2c_je^{i\lambda_{j}^2t}\phi_j(x)\right \vert ^2\right)  dV_g\, dt\\
 & = &  \inf_{\substack{(c_j)\in\ell^2(\C)\\ \sum_{j=1}^{+\infty} \vert c_j \vert ^2=1}}\int_0^T\int_\omega \sum_{j,k=1}^{+\infty}\mathbb{E}(\beta_j^\nu\beta_k^\nu) c_j\bar c_k e^{i(\lambda_{j}^2-\lambda_{k}^2)t}\phi_j(x)\phi_k(x)\, dV_g\, dt\\
 & = &  T\inf_{\substack{(c_j)\in\ell^2(\C)\\ \sum_{j=1}^{+\infty} \vert c_j \vert ^2=1}}\int_\omega \sum_{j=1}^{+\infty} \vert c_j \vert ^2\phi_j(x)^2\, dV_g\\
 & = &  T\inf_{j\in\N^*}\int_\omega\phi_j(x)^2\, dV_g .
\end{eqnarray*}
The proof for $C_{T,\textrm{rand}}^{(W)}(\chi_\omega)$ is similar.


\section{First problem: optimal design for fixed initial data}\label{solvingpb1obs}
This section is devoted to solving the first problem, that is, the problem of best observation for fixed initial data.
Throughout the section, we fix initial data $(y^0,y^1)\in D(A^{1/2})\times X$ (resp., $y^0\in D(A)$) for the wave equation \eqref{waveEqobs} (resp., for the Schr\"odinger equation \eqref{schroEqobs}), and we consider their associated coefficients $\alpha_{ij}$, $(i,j)\in(\N^*)^2$, defined by \eqref{defalphaij} (resp., by \eqref{defalphaijS}).
The next considerations are valuable for both wave and Schr\"odinger equations.
For every $x\in\Omega$, we define
\begin{equation}\label{defphi}
\varphi(x) = \int_0^T \vert \partial_{t} y(t,x)\vert^2 dt = \sum_{i,j=1}^{+\infty}\lambda_{i}\lambda_{j}\alpha_{ij} \phi_i(x) \phi_j(x),
\end{equation}
where $y$ is defined by \eqref{yDecomp} or \eqref{psiDecomp} and is such that $y\in C^0(0,T;D(A^{1/2}))$, and where the coefficients $\alpha_{ij}$ are defined by \eqref{defalphaij} or \eqref{defalphaijS}. Note that the function $\varphi$ is integrable on $\Omega$.
Then, from \eqref{GT1} or \eqref{GTpsi}, there holds
\begin{equation}\label{GT2}
G_T(\chi_\omega) = \int_\omega \varphi(x) \, dV_g,
\end{equation}
for every measurable subset $\omega$ of $\Omega$.

\subsection{Main result}\label{sec3.1}
\begin{theorem}\label{thmobsopt}
There exists at least one measurable subset $\omega$ of $\Omega$, solution of the first problem, characterized as follows. There exists a real number $\lambda$ such that every optimal set $\omega$ is contained in the level set $\{\varphi\geq\lambda\}$, where the function $\varphi$ defined by \eqref{defphi} is integrable on $\Omega$.

Moreover, if $M$ is an analytic Riemannian manifold, if $\Omega$ has a nontrivial boundary of class $C^\infty$ and if there exists $R>0$ such that
\begin{equation}\label{condanalwave}
\sum_{j=0}^{+\infty} \frac{R^j}{j!} \left( \Vert A^{j/2} y^0\Vert_{L^2}^2 + \Vert A^{(j-1)/2} y^1\Vert_{L^2}^2 \right)^{1/2} < +\infty ,
\end{equation}
in the case of the wave equation, and
\begin{equation}\label{condanalSch}
\sum_{j=0}^{+\infty} \frac{R^j}{j!}  \Vert A^{j/2} y^0\Vert_{L^2}  < +\infty ,
\end{equation}
in the case of the Schr\"odinger equation, then  the first problem has a unique\footnote{Similarly to the definition of elements of $L^p$-spaces, the subset $\omega$ is unique within the class of all measurable subsets of $\Omega$ quotiented by the set of all measurable subsets of $\Omega$ of zero measure.} solution $\chi_\omega$, where $\omega$ is a measurable subset of $\Omega$ of measure $L V_g(\Omega) $, satisfying moreover the following properties:
\begin{itemize}
\item $\omega$ is semi-analytic\footnote{A subset $\omega$ of a real analytic finite dimensional manifold $M$ is said to be semi-analytic if it can be written in terms of equalities and inequalities of analytic functions, that is, for every $x\in\omega$, there exists a neighborhood $U$ of $x$ in $M$ and $2pq$ analytic functions $g_{ij}$, $h_{ij}$ (with $1\leq i\leq p$ and $1\leq j\leq q$) such that
$$
\omega\cap U = \bigcup_{i=1}^p \{y\in U\ \vert\ g_{ij}(y)=0\ \textrm{and}\ h_{ij}(y)>0,\ j=1,\ldots,q\}.
$$
We recall that such semi-analytic (and more generally, subanalytic) subsets enjoy nice properties, for instance they are stratifiable in the sense of Whitney (see \cite{Hardt,Hironaka}).
\label{footnoteSemiAna}}, and has a finite number of connected components;
\item if $M=\R^n$, if $\Omega$ is symmetric with respect to an hyperplane, if $y^{0}\circ\sigma =y^{0}$ and $y^{1}\circ\sigma=y^{1}$ where $\sigma$ denotes the symmetry operator with respect to this hyperplane, then $\omega$ enjoys the same symmetry property;
\item for Dirichlet boundary conditions, there exists $\eta>0$ such that $d(\omega,\partial\Omega)>\eta$, where $d$ denotes the Riemannian distance on $M$.
\end{itemize}
\end{theorem}

The first statement of this theorem covers the case where $\partial\Omega=\emptyset$. In this case, $\Omega$ is a compact connected analytic Riemannian manifold.

\begin{proof}[Proof of Theorem \ref{thmobsopt}]
The existence and the characterization in function of the level sets of $\varphi$ of a set $\omega$ maximizing \eqref{GT2} is obvious since the function $\varphi$ is integrable on $\Omega$. 
Let us prove the second part of the theorem, for the wave equation.
First of all we claim that, under the additional assumption \eqref{condanalwave}, the corresponding solution $y$ of the wave equation is analytic over $\R^+\times\Omega$. Indeed, we first note that the quantity 
$$
\Vert A^{j/2} y(t,\cdot)\Vert_{L^2}^2 + \Vert A^{(j-1)/2} \partial_t y(t,\cdot)\Vert_{L^2}^2
$$
is constant with respect to $t$. Then, since $\Omega$ has a smooth boundary, it follows from \eqref{condanalwave} and from the Sobolev imbedding theorems that there exists $C>0$ such that
$$
\Vert y^{(k)}(t,\cdot)\Vert_\infty \leq C \frac{(2n+k)!}{R^{2n+k}},
$$
for every $t\geq 0$ and every integer $k$. The claim follows. As a consequence, the function $\varphi$ defined by \eqref{defphi} is analytic on $\Omega$. Hence $\varphi$ cannot be constant on a subset of positive measure (otherwise by analyticity it would be constant on $\Omega$ and hence equal to $0$ due to the boundary conditions). This ensures the uniqueness of the optimal set $\omega$.

The first additional property follows from the analyticity properties.
The symmetry property (if $M=\R^n$) follows from the fact that $\varphi\circ \sigma (x)=\varphi(x)$ for every $x\in\Omega$. If $\omega$ were not symmetric with respect to this hyperplane, the uniqueness of the solution of the first problem would fail, which is a contradiction.
For Dirichlet boundary conditions, since $\varphi(x) = \int_0^T \vert \partial_{t}y(t,x)\vert^2 dt=0$ for every $x\in\partial\Omega$, it follows that $\varphi$ reaches its global minimum on the boundary of $\Omega$. 
\end{proof}

\begin{remark}\label{rem5}
The solution of the first problem depends on the initial data under consideration. More specifically, it depends on their associated coefficients $\alpha_{ij}$, defined by \eqref{defalphaij} in the case of the wave equation, and by \eqref{defalphaijS} in the case of the Schr\"odinger equation. Note that there exist an infinite number of initial data $(z^0,z^1)\in D(A^{1/2})\times X$ (resp. $z^0\in D(A)$) having the same coefficients $\alpha_{ij}$ than $(y^0,y^1)$ (resp., $y^0$), and all of them lead to the same solution of the first problem. Similar considerations have been discussed in the one-dimensional case in \cite{PTZObs1}.
\end{remark}

\begin{remark}
We have seen that the optimal solution may not be unique whenever the function $\varphi$ is constant on some subset of $\Omega$ of positive measure.
More precisely, assume that $\varphi$ is constant, equal to $c$, on some subset $I$ of $\Omega$ of positive measure $\vert I\vert$. If $ \vert \{\varphi\geq c \} \vert <  L V_g(\Omega)  < \vert \{\varphi>c \} \vert$ then there exists an infinite number of measurable subsets $\omega$ of $\Omega$ maximizing \eqref{GT2}, all of them containing the subset $\{\varphi>c \}$. The part of $\omega$ lying in $\{\varphi=c \}$ can indeed be chosen arbitrarily.

Note that there is no simple characterization of all initial data 
for which this non-uniqueness phenomenon occurs, however to get convinced that this may indeed happen it is convenient to consider the one-dimensional case where $T$ is moreover an integer multiple of $2\pi$ (see Remark \ref{remTmult2pi}), with Dirichlet boundary conditions. Indeed in that case the functional $G_T$ does not involve any crossed terms and therefore the corresponding function $\varphi$ reduces to $\varphi(x) = \sum_{j=1}^{+\infty} \lambda_{j}^2\alpha_{jj} \sin^2(jx)$, with $\alpha_{jj}$ given by \eqref{alphajj2pi}. Writing $\sin^2(jx)=\frac{1}{2}-\frac{1}{2}\cos(2jx)$ permits to write $\varphi$ as a Fourier series whose sine Fourier coefficients vanish and cosine coefficients are nonpositive and summable (because of \eqref{alphajj2pi}). Hence, to provide an explicit example where the non-uniqueness phenomenon occurs, it suffices to consider a nonpositive triangle function defined on $[\frac{\pi}{2}-\alpha,\frac{\pi}{2}+\alpha]$, for some $\alpha>0$, equal to $0$ outside. Its Fourier coefficients are the values on integers of the Fourier transform of the nonpositive triangle function, hence are negative and summable. The rest of the construction is obvious. We refer to \cite{PTZObs1} for details on this one-dimensional case.\\
Note that it is easy to generalize such a characterization in a $n$-dimensional hypercube for the Schr\"odinger equation, since in this case the solution remains periodic with period $2\pi$ and $G_T$ does not involve any crossed terms.
\end{remark}

\subsection{Further comments on the complexity of the optimal set}\label{sec3.2}
It is interesting to raise the question of the complexity of the optimal sets solutions of the first problem.
In Theorem \ref{thmobsopt} we prove that, if the initial data belong to some analyticity spaces,
then the (unique) optimal set $\omega$ is the union of a finite number of connected components. 
Hence, analyticity implies finiteness and it is interesting to wonder whether this property still holds true for less regular initial data.

In what follows we show that, in some particular cases (those mentioned above where the periodicity of solutions of the wave and Schr\"odinger equations can be exploited), there exist $C^\infty$ initial data for which the optimal set $\omega$ has a fractal structure and, more precisely, is of Cantor type.

\begin{proposition}\label{propCantor}
For the one-dimensional wave equation on $[0,\pi]$ (resp. for the $n$-dimensional Schr\"odinger equation on the hypercube $[0,\pi]^n$) with Dirichlet boundary conditions, there exist $C^\infty$ initial data $(y^0,y^1)$ (resp. $C^\infty$ initial data $y^0$) for which the first problem has a unique solution $\omega$ with fractal structure and thus, in particular, it has an infinite number of connected components.
\end{proposition}

The proof of this proposition is quite technical and relies on a careful Fourier analysis construction. It is done in Appendix \ref{AppendixCantor}.

\subsection{Several numerical simulations}
We provide hereafter a numerical illustration of the results presented in this section. 

According to Theorem \ref{thmobsopt}, the optimal domain is characterized as a level set of the  function $\varphi$. Some numerical simulations are provided on Figure \ref{figpb1}, with $\Omega=[0,\pi]^{2}$, $L=0.6$, $T=3$, $y^{1}=0$ and
$$
y^{0}(x)=\sum_{n,k=1}^{N_{0}}a_{n,k}\sin (nx_{1})\sin (kx_{2}),
$$
where $N_{0}\in\N^{*}$ and $(a_{n,k})_{n,k\in \N^{*}}$ are real numbers. The level set is numerically computed using a simple dichotomy procedure.

\begin{figure}[h!]
\begin{center}
\includegraphics[width=7.3cm]{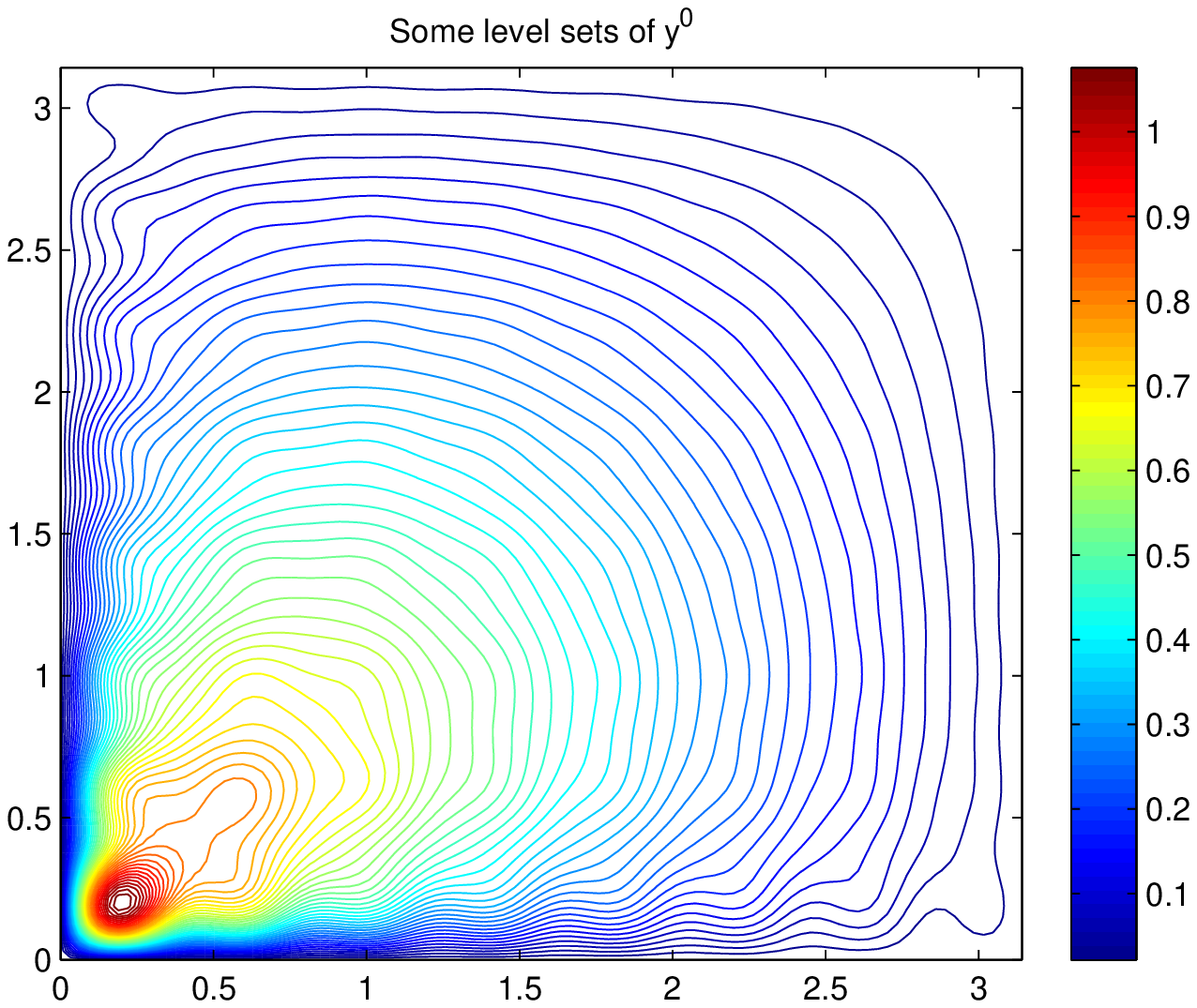}
\includegraphics[width=7.3cm]{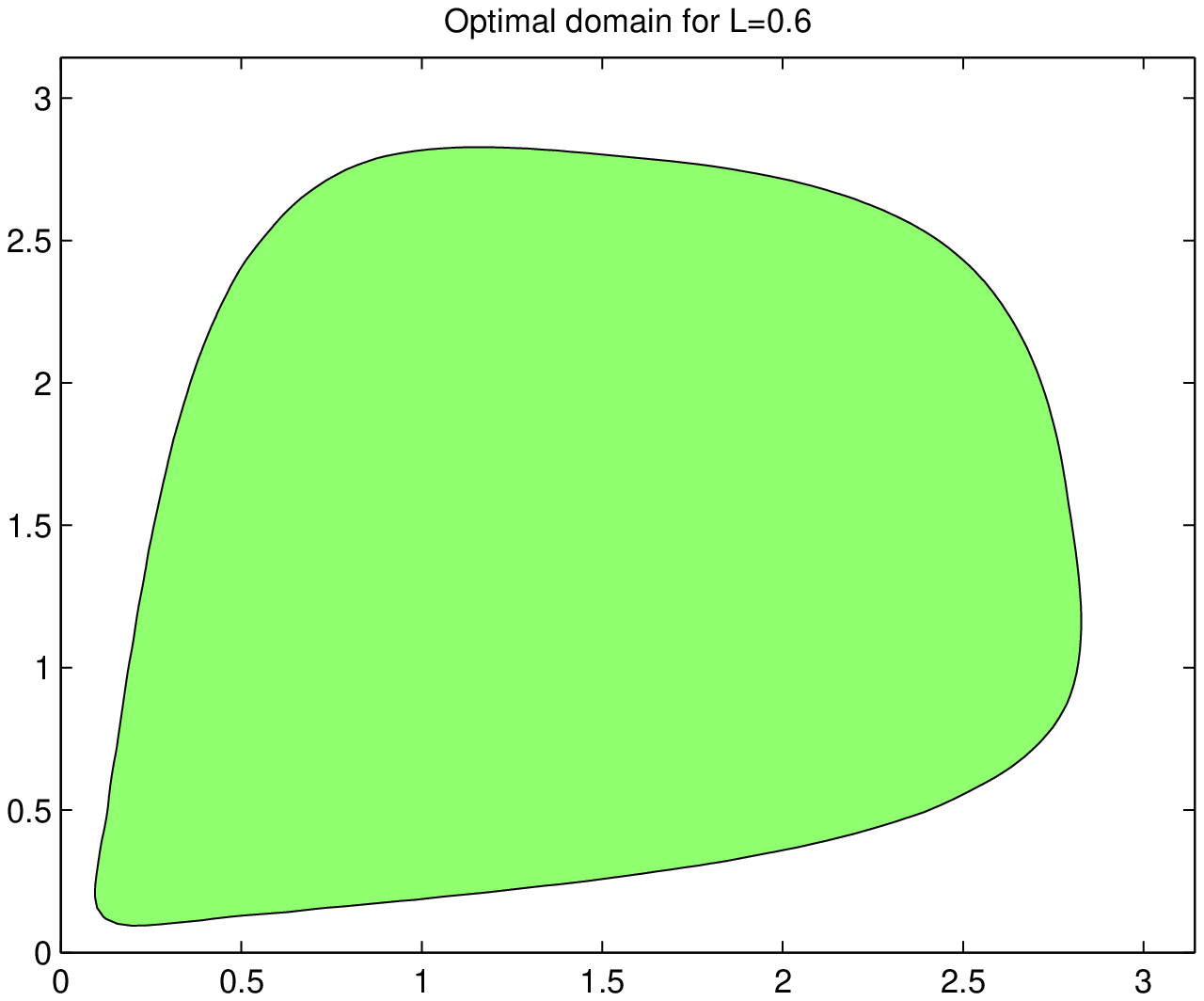}
\includegraphics[width=7.3cm]{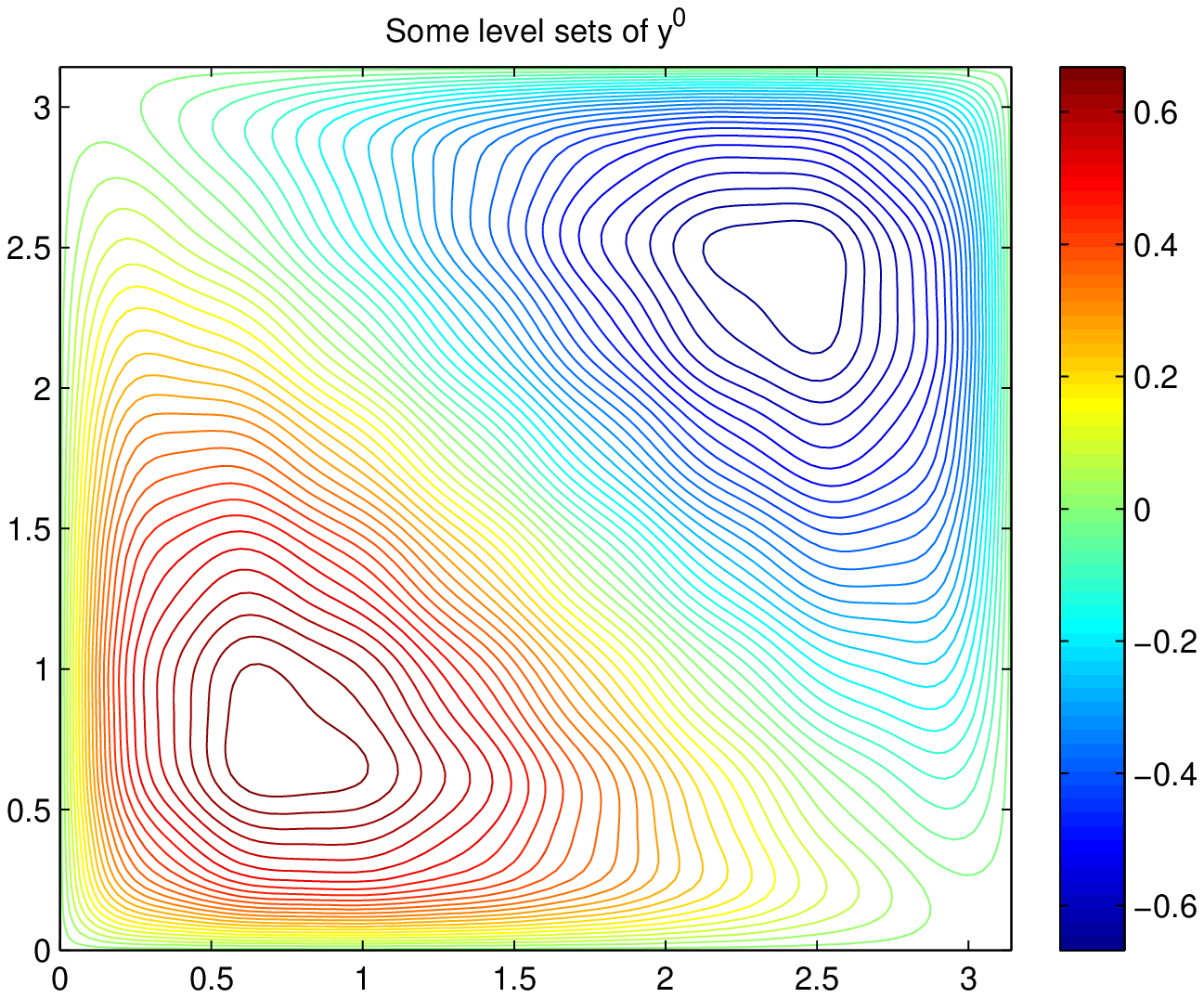}
\includegraphics[width=7.3cm]{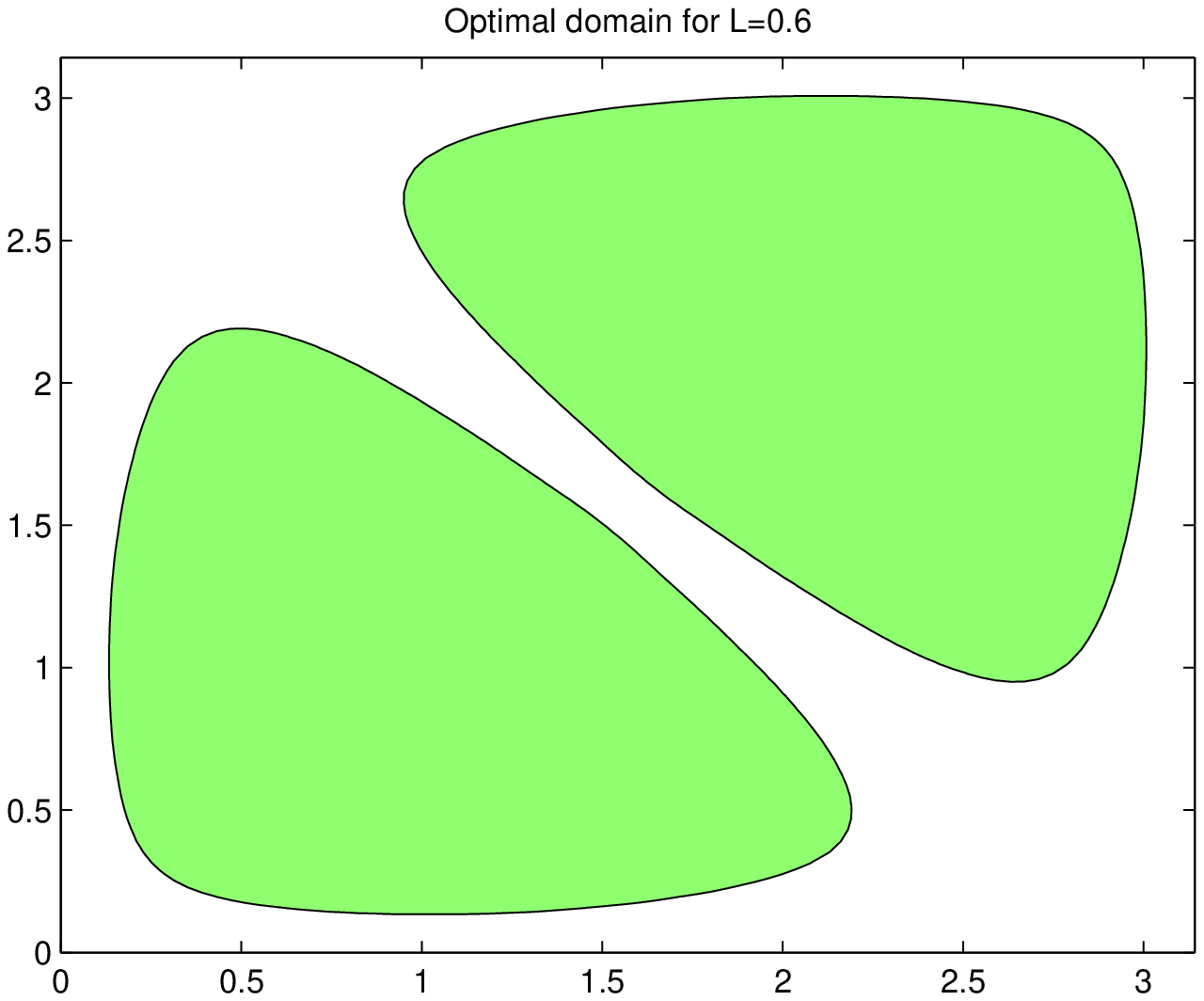}
\caption{On this figure, $\Omega=[0,\pi]^{2}$ with Dirichlet boundary conditions, $L=0.6$, $T=3$ and $y^{1}=0$. 
At the top: $N_{0}=15$ and $a_{n,k}=\frac{1}{n^{2}+k^{2}}$.
At the bottom: $N_{0}=15$ and $a_{n,k}=\frac{1-(-1)^{n+k}}{n^{2}k^{2}}$.
On the left: some level sets of $y^{0}$.
On the right: representation of the optimal domain for the corresponding choice of $y^{0}$.}\label{figpb1}
\end{center}
\end{figure}


\section{Second problem: uniform optimal design}\label{solvingpb2obs}

\subsection{Preliminary remarks}\label{solvingpb2obssec1}
We define the set
\begin{equation}\label{defUL}
\mathcal{U}_L = \{ \chi_\omega\ \vert\ \omega\ \textrm{is a measurable subset of}\ \Omega \ \textrm{of measure}\ V_g(\omega)=L V_g( \Omega ) \}.
\end{equation}
Recall that the second problem \eqref{defJ} is written as
\begin{equation}\label{reducedsecondpb}
\sup_{\chi_\omega\in\mathcal{U}_L}  J(\chi_\omega) ,
\end{equation}
with
$$
J(\chi_\omega)=\inf_{j\in\N^*} \int_\omega \phi_j(x)^2 \, dV_g .
$$
The criterion $J(\chi_\omega)$ can be seen as a spectral energy concentration criterion. For every $j\in\N^*$, the integral $\int_\omega\phi_j(x)^2\,dV_g$ is the energy of the $j^\mathrm{th}$ eigenfunction restricted to $\omega$, and the problem is to maximize the infimum over $j$ of these energies, over all subsets $\omega$ of measure $V_g(\omega)=L V_g( \Omega ) $.

Since the set $\mathcal{U}_L$ does not have compactness properties ensuring the existence of a solution of \eqref{reducedsecondpb}, we consider the convex closure of $\mathcal{U}_L$ for the weak star topology of $L^\infty$, 
\begin{equation}\label{defULbar}
\overline{\mathcal{U}}_L = \left\{ a\in L^\infty(\Omega,[0,1])\ \big\vert\ \int_{\Omega} a(x)\, dV_g=L V_g(\Omega)  \right\}.
\end{equation}
This convexification procedure is standard in shape optimization problems where an optimum may fail to exist  because of hard constraints (see e.g. \cite{BucurButtazzo}).

Replacing $\chi_\omega\in\mathcal{U}_L$ with $a\in\overline{\mathcal{U}}_L$, we define a convexified formulation of the second problem \eqref{reducedsecondpb} by
\begin{equation}\label{quantity2obsAsympconvexified}
\sup_{a\in \overline{\mathcal{U}}_L} J(a) ,
\end{equation}
where
\begin{equation}\label{defJrelax}
J(a)=\inf_{j\in \N^*}\int_{\Omega}a(x) \phi_j(x)^2\, dV_g.
\end{equation}
Since $J(a)$ is defined as the infimum of linear continuous functionals for the weak star topology of $L^\infty$, it is upper semi continuous for this topology. This yields to the following result.

\begin{lemma}\label{lemmaExistConv}
The problem \eqref{quantity2obsAsympconvexified} has at least one solution.
\end{lemma}

Obviously, there holds
\begin{equation}\label{ineqGap}
\sup_{\chi_\omega\in \mathcal{U}_L} \inf_{j\in \N^*}\int_{\Omega}\chi_\omega(x) \phi_j(x)^2\, dV_g\ \leq\ \sup_{a\in \overline{\mathcal{U}}_L} \inf_{j\in \N^*}\int_{\Omega}a(x) \phi_j(x)^2\, dV_g.
\end{equation}
Note that, since the constant function $a(\cdot)=L$ belongs to $\overline{\mathcal{U}}_L$, it follows that $\displaystyle \sup_{a\in \overline{\mathcal{U}}_L} J(a) \geq L$.
In the next section, under an additional ergodicity assumption, we compute the optimal value \eqref{quantity2obsAsympconvexified} of this convexified problem and investigate the question of knowing whether the above inequality is strict or not. In other words we investigate whether there is a gap or not between the problem \eqref{reducedsecondpb} and its convexified version \eqref{quantity2obsAsympconvexified}.

\begin{remark}\label{remtopo}{\it Comments on the choice of the topology.}\\
In our study we consider measurable subsets $\omega$ of $\Omega$, and we endow the set $L^\infty(\Omega,\{0,1\})$ of all characteristic functions of measurable subsets with the weak-star topology. Other topologies are used in shape optimization problems, such as the Hausdorff topology. Note however that, although the Hausdorff topology shares nice compactness properties, it cannot be used in our study because of the measure constraint on $\omega$. Indeed, the Hausdorff convergence does not preserve measure, and the class of admissible domains is not closed for this topology.
Topologies associated with convergence in the sense of characteristic functions or in the sense of compact sets (see for instance \cite[Chapter 2]{henrot-pierre}) do not guarantee easily the compactness of minimizing sequences of domains, unless one restricts the class of admissible domains, imposing for example some kind of uniform regularity.
\end{remark}

\begin{remark}\label{rempossiblegap}
We stress that the question of the possible existence of a gap between the original problem and its convexified version is not obvious and cannot be handled with usual $\Gamma$-convergence tools, in particular because the function $J$ defined by \eqref{defJrelax} is is not lower semi-continuous for the weak star topology of $L^\infty$ (it is however upper semi-continuous for that topology, as an infimum of linear functions).
To illustrate this fact, consider the one-dimensional case of Remark \ref{remTmult2pi}. In this specific situation, since $\phi_j(x)=\sqrt{\frac{2}{\pi}}\sin (jx)$ for every $j\in\N^*$, one has
$$
J(a)=\frac{2}{\pi}\inf_{j\in \N^*}\int_0^\pi a(x)\sin^2(jx)\, dx,
$$
for every $a\in \overline{\mathcal{U}}_L$. Since the functions $x\mapsto\sin^2(jx)$ converge weakly to $1/2$, it clearly follows that $J(a)\leq L$ for every $a\in \overline{\mathcal{U}}_L$. Therefore,
$$
\sup_{a\in \overline{\mathcal{U}}_L} J(a)=L,
$$
and the supremum is reached with the constant function $a(\cdot)=L$.
Consider the sequence of subsets $\omega_N$ of $[0,\pi]$ of measure $L\pi$ defined by
$$
\omega_N=\bigcup_{k=1}^N\left[\frac{k\pi}{N+1}-\frac{L\pi}{2N},\frac{k\pi}{N+1}+\frac{L\pi}{2N}\right],
$$
for every $N\in\N^*$. Clearly, the sequence of functions $\chi_{\omega_N}$ converges to the constant function $a(\cdot)=L$ for the weak star topology of $L^\infty$, but nevertheless, an easy computation shows that
\begin{equation*}
\int_{\omega_N}\sin ^2(jx)\, dx 
=
\left\{\begin{array}{ll}
\frac{L\pi}{2} -\frac{N}{2j}\sin\left(\frac{jL\pi}{N}\right) & \textrm{if}\ \ (N+1)\ \vert\ j,\\
\frac{L\pi}{2}+ \frac{1}{2j}\sin\left(\frac{jL\pi}{N}\right) & \textrm{otherwise},
 \end{array}\right.
\end{equation*}
and hence,
$$
\limsup_{N\to +\infty}\frac{2}{\pi}\inf_{j\in\N^*}\int_{\omega_N}\sin ^2(jx)\,dx<L.
$$
This simple example illustrates the difficulty in understanding the limiting behavior of the functional because of the lack of the lower semicontinuity, what makes possible the occurrence of a gap in the convexification procedure.
In Section \ref{secconvexified2}, we will prove that there is no such a gap under an additional geometric spectral assumption. 
\end{remark}

\subsection{Main results}\label{secconvexified2}
In what follows, we make the following assumptions on the basis $(\phi_j^2)_{j\in \N^*}$ of eigenfunctions under consideration.
\begin{quote}
\textit{
\textbf{Weak Quantum Ergodicity on the basis (WQE) property.}
There exists a subsequence of the sequence of probability measures $\mu_{j}=\phi_j^2\, dV_{g}$ converging vaguely to the uniform measure $\frac{1}{ V_g(\Omega) }\, dV_{g}$.
}
\end{quote}
\begin{quote}
\textit{
\textbf{Uniform $L^\infty$-boundedness property.}
There exists $A>0$ such that 
\begin{equation}\label{UBP}
\Vert \phi_{j}\Vert_{L^\infty(\Omega)}\leq A,
\end{equation}
for every $j\in \N^*$.
}
\end{quote}
Note that the two assumptions above imply what we call the \textit{$L^\infty$-Weak Quantum Ergodicity on the base ($L^\infty$-WQE) property}\footnote{The wording used here is motivated and explained further in a series of remarks.}, that is, there exists a subsequence of $(\phi_{j}^2)_{j\in \N^*}$ converging to $\frac{1}{V_{g}(\Omega)}$ for the weak star topology of $L^\infty(\Omega)$.

Obviously, this property implies that
\begin{equation}\label{relvalueL}
\sup_{a\in \overline{\mathcal{U}}_L} \inf_{j\in \N^*}\int_{\Omega}a(x) \phi_j(x)^2\, dV_g = L,
\end{equation}
and moreover the supremum is reached with the constant function $a=L$ on $\Omega$.

\begin{remark}
In general the convexified problem \eqref{quantity2obsAsympconvexified} does not admit a unique solution. Indeed, under symmetry assumptions on $\Omega$ there exists an infinite number of solutions. For example, in dimension one, with $\Omega=[0,\pi]$, all solutions of \eqref{quantity2obsAsympconvexified} are given by all functions of $\overline{\mathcal{U}}_L$ whose Fourier expansion series is of the form
$a(x)= L + \sum_{j=1}^{+\infty} ( a_j \cos(2jx) + b_j \sin(2jx))$
with coefficients $a_j\leq 0$.
\end{remark}

It follows from \eqref{ineqGap} and \eqref{relvalueL} that
\begin{equation*}
\sup_{\chi_\omega\in \mathcal{U}_L} \inf_{j\in \N^*}\int_{\omega}\phi_j(x)^2\, dV_g\ \leq\ L.
\end{equation*}
The next result states that this inequality is actually an equality.

\begin{theorem}\label{thmnogap}
If the WQE and uniform $L^\infty$-boundedness properties hold, then
\begin{equation}\label{eqNoGap}
\sup_{\chi_\omega\in \mathcal{U}_L} \inf_{j\in \N^*}\int_{\omega} \phi_j(x)^2\, dV_g  = L,
\end{equation}
for every $L\in (0,1)$.
In other words, under these assumptions there is no gap between the original problem \eqref{reducedsecondpb} and the convexified one.
\end{theorem}

It follows from this result, from Corollary \ref{corCTT} and Theorem \ref{propHazardCst}, that the maximal value of the randomized observability constants $2\,C_{T,\textrm{rand}}^{(W)}(\chi_\omega) = C_{T,\textrm{rand}}^{(S)}(\chi_\omega)$ over the set $\mathcal{U}_L$ is equal to $TL$, and that, if the spectrum of $A$ is simple, the maximal value of the time asymptotic observability constants $2C_{\infty}^{(W)}(\chi_{\omega})=C_{\infty}^{(S)}(\chi_{\omega})$ over the set $\mathcal{U}_L$ is equal to $L$.

The question of knowing whether the supremum in \eqref{eqNoGap} is reached (existence of an optimal set) is investigated in Section \ref{sec4.3}.

Theorem \ref{thmnogap} is established within the class of measurable subsets. We next state a similar (but distinct) result within the class of measurable subsets whose boundary is of  measure zero.
We define the set
\begin{equation}\label{defULb}
\mathcal{U}_L^b = \{ \chi_\omega\in \mathcal{U}_L \ \vert\ V_{g}(\partial\omega)=0 \}.
\end{equation}
This is the set of all characteristic functions of Jordan measurable subsets of $\Omega$ of measure $LV_{g}(\Omega)$.
We make the following assumptions.

\begin{quote}
\textit{
\textbf{Quantum Unique Ergodicity on the base (QUE) property.}
The whole sequence of probability measures $\mu_{j}=\phi_j^2\, dV_{g}$ converges vaguely to the uniform measure $\frac{1}{ V_g(\Omega) }\, dV_{g}$.
}
\end{quote}
\begin{quote}
\textit{
\textbf{Uniform $L^p$-boundedness property.}
There exist $p\in(1,+\infty]$ and $A>0$ such that 
\begin{equation}\label{UBPp}
\Vert \phi_{j}\Vert_{L^{2p}(\Omega)}\leq A,
\end{equation}
for every $j\in \N^*$.
}
\end{quote}

\begin{theorem}\label{thmnogap2}
Assume that $\partial \Omega$ is Lipschitz whenever it is nonempty.
If the QUE and uniform $L^p$-boundedness properties hold, then
\begin{equation}\label{eqNoGap1}
\sup_{\chi_\omega\in \mathcal{U}_L^b} \inf_{j\in \N^*}\int_{\omega} \phi_j(x)^2\, dV_g  = L,
\end{equation}
for every $L\in (0,1)$.
\end{theorem}

Theorems \ref{thmnogap} and \ref{thmnogap2} are proved in Sections \ref{secproofthmnogap} and \ref{secproofthmnogap2} respectively.

\begin{remark}\label{rem_others}
It follows from the proof of Theorem \ref{thmnogap2} that this statement holds true as well whenever the set $\mathcal{U}_L^b$ is replaced with the set of all measurable subsets $\omega$ of $\Omega$, of measure $V_g(\omega)=LV_g(\Omega)$, that are moreover either open with a Lipschitz boundary, or open with a bounded perimeter.
\end{remark}

\begin{remark}\label{remarksquare}
The assumptions made in Theorems \ref{thmnogap} or \ref{thmnogap2} are sufficient conditions implying \eqref{eqNoGap} or \eqref{eqNoGap1}, but they are however not sharp, as proved in the next proposition.
\begin{proposition}\label{propnogapnoQUE}
Consider the Dirichlet-Laplacian on the domain $\Omega$ defined as the unit disk of the Euclidean two-dimensional space.
Then, for every $p\in(1,+\infty]$ and for any basis of eigenfunctions, the uniform $L^p$-boundedness property is not satisfied, and QUE does not hold as well.
However, the equalities \eqref{eqNoGap} and \eqref{eqNoGap1} hold true.
\end{proposition}
To establish this result, in the proof of this proposition (done in Section \ref{sec_proof_propnogapnoQUE}) we use the explicit expression of certain semi-classical measures in the disk (weak limits of the probability measures $\phi_j^2\, dV_g$). Among these quantum limits, one can find the Dirac measure along the boundary which causes the well known phenomenon of whispering galleries. Having in mind this phenomenon, it might be expected that there exists an optimal set, concentrating around the boundary. The calculations show that it is not the case, and \eqref{eqNoGap} and \eqref{eqNoGap1} are proved to hold.
\end{remark}

The next section is devoted to gather some comments on the ergodicity assumptions made in these theorems.

\subsection{Comments on ergodicity assumptions}\label{sec_comments}
This section is organized as a series of remarks.

\begin{remark}\label{rem8}
The assumptions of Theorem \ref{thmnogap} hold true in dimension one. Indeed, it has already been mentioned that the eigenfunctions of the Dirichlet-Laplacian operator on $\Omega=[0,\pi]$ are given by $\phi_j(x)=\sqrt{\frac{2}{\pi}}\sin(jx)$, for every $j\in\N^*$. Therefore clearly the whole sequence (not only a subsequence) $(\phi_j^2)_{j\in \N^*}$ converges weakly to $\frac{1}{\pi}$ for the weak star topology of $L^\infty(0,\pi)$. The same property clearly holds for all other boundary conditions considered in this article.
\end{remark}

\begin{remark}\label{remWQE}
In dimension greater than one the situation is more intricate, but we have the following facts.

Any hypercube (tensorised version of the previous one-dimensional case) or flat torus satisfies the assumptions. Indeed, the whole sequence of eigenfunctions is uniformly bounded and  converges to a constant for the weak star topology of $L^\infty$.


\medskip

Generally speaking, these assumptions are related to ergodicity properties of $\Omega$. Before providing precise results, we recall the following well known definition.
\begin{quote}
\textit{
\textbf{Quantum Ergodicity on the base (QE) property.}
There exists a subsequence of the sequence of probability measures $\mu_j=\phi_j^2\, dV_g$ of density one converging vaguely to the uniform measure $\frac{1}{V_g(\Omega)}dV_g$.
}
\end{quote}
Here, \textit{density one} means that there exists $\mathcal{I}\subset \N^*$ such that
$$
\lim_{N\to+\infty}\frac{\# \{j\in\mathcal{I}\ \vert\  j\leq N\}}{N}=1.
$$
Obviously, QE implies WQE\footnote{Note that, up to our knowledge, the notion of WQE is new, whereas the notions of QE and QUE are classical in mathematical physics.}.
The well known Shnirelman Theorem asserts that the property QE is satisfied on every compact ergodic Riemannian manifolds having no boundary (see \cite{CdV2,Sh1,Sh2,Zel1}). 
For domains having a boundary, it is proved in \cite{gerard} that, if the domain $\Omega$ is a convex ergodic billiard with $W^{2,\infty}$ boundary, then the property QE is satisfied. The result of \cite{gerard} has been extended to arbitrary ergodic manifolds with piecewise smooth boundaries in \cite{ZZ} and \cite{HZ}.

Note that this result lets however open the possibility of having an exceptional subsequence of measures $\mu_j$ converging vaguely to something else. We will come back on this interesting issue later.

Actually these results relating the ergodicity of $\Omega$ (seen as a billiard where the geodesic flow moves at unit speed and bounces at the boundary according to the Geometric Optics laws)   to the QE property are even stronger, for two reasons. Firstly, they are valid for any Hilbertian basis of eigenfunctions of $A$, whereas here we make this kind of assumption only for the specific basis $(\phi_j)_{j\in\N^*}$ that has been fixed at the beginning of the article. Secondly, they establish that a stronger microlocal version of the QE property holds for pseudodifferential operators, in the unit cotangent bundle $S^*\Omega$ of $\Omega$: more precisely it is proved that, under ergodicity assumptions on the manifold, a density one subsequence of the linear functionals $\rho_j(P)=\langle P \phi_j,\phi_j\rangle$, defined on the space of zero-th order pseudo-differential operators $P$, converges vaguely to the uniform Liouville measure. It is a much stronger conclusion since it says that the eigenfunctions become uniformly distributed on the phase space $S^*\Omega$ and not just on the configuration space $\Omega$. Here however we do not need (de)concentration results in the full phase space, but only in the configuration space. This is why, following \cite{Zel3}, we use the wording ``on the base".

Note that the vague convergence of the measures $\mu_j$ is weaker than the convergence of the functions $\phi_j^2$ for the weak star topology of $L^\infty(\Omega)$. Since $\Omega$ is bounded, the property of vague convergence is equivalent to saying that, for a subsequence of density one, $\int_\omega\phi_j(x)^2\, dV_g$ converges to $V_g(\omega)/V_g(\Omega)$ for every measurable subset $\omega$ of $\Omega$ such that $V_g(\partial\omega)=0$ (Portmanteau theorem). In contrast, the property of convergence for the weak star topology of $L^\infty(\Omega)$ is equivalent to saying that, for a subsequence of density one, $\int_\omega\phi_j(x)^2\, dV_g$ converges to $V_g(\omega)/V_g(\Omega)$ for every measurable subset $\omega$ of $\Omega$. Under the assumption that all eigenfunctions are uniformly bounded in $L^\infty(\Omega)$, both notions are equivalent. This is the case for instance in flat tori. But, for instance, if $\Omega$ is a ball or a sphere of any dimension, then the eigenfunctions of the Laplacian are not uniformly bounded.
This is well known to be a delicate issue (see \cite{Zel3}). We refer to \cite{Zel4} where it is conjectured that flat tori are the sole compact manifolds without boundary where the whole family of eigenfunctions is uniformly bounded in $L^\infty$.

Note that the notion of $L^\infty$-QE property, meaning that the above QE property holds for the weak star topology of $L^\infty$, is defined and mentioned in \cite{Zel3} as a delicate open problem. As said above we stress that, under the assumption that all eigenfunctions are uniformly bounded in $L^\infty(\Omega)$, $QE$ and $L^\infty$-QE are equivalent.

To the best of our knowledge, nothing seems to be known on the uniform $L^p$-boundedness property \eqref{UBP}. This property holds for flat tori but does not hold for balls or spheres.
\end{remark}

\begin{remark}\label{remQUE}
Let us comment on the QUE property, which is an important issue in quantum and mathematical physics. Note indeed that the quantity $\int_\omega\phi_j^2(x)\, dV_g$ is interpreted as the probability of finding the quantum state of energy $\lambda_j^2$ in $\omega$. 
We stress again on the fact that, here, we consider a version of QUE in the configuration space only, not in the full phase space. Moreover, we consider the QUE property for the basis $(\phi_j)_{j\in\N^*}$ under consideration, but not necessarily for any such basis of eigenfunctions.

First of all, QUE obviously holds true in the one-dimensional case of Remark \ref{remTmult2pi} (see also Remark \ref{rempossiblegap}) but it does however not hold true for multi-dimensional hypercubes.

In the general multi-dimensional case, many interesting open questions and issues occur.
As in Remark \ref{remWQE}, consider for every $j\in\N^*$ the probability measure $\mu_j=\phi_j^2\, dV_g$, representing in quantum mechanics the probability of being in the state $\phi_j$ (or, probability density of finding a particle of energy $\lambda_j^2$ at $x$). An interesting question is to know whether the supports of these measures tend to equidistribute or can concentrate as $j\rightarrow+\infty$.
As already mentioned, under ergodicity assumptions, the property QE holds true, that is, a subsequence of density one of $(\mu_j)_{j\in\N^*}$ converges vaguely to the uniform measure on $\Omega$. 
But this result lets open the possibility of having an exceptional subsequence converging to some other measure. Typically it may happen that a subsequence of density zero converges to an invariant measure like for instance a measure carried by closed geodesics. These so-called (strong) \textit{scars} result from such energy concentration phenomena, that are allowed in the context of Shnirelman Theorem.
This fascinating question of knowing whether the quantum states of such an ergodic system can concentrate or not on some instable closed orbits or on some invariant tori generated by such geodesics is still widely open in mathematics and physics.
We refer to \cite{BonechiDeBievre,FaureNonnenmacherDeBievre}
for results showing a scarring phenomenon on the periodic orbits of the dynamics of the quantum Arnold's cat map. Note however that, as already mentioned, here we are concerned with concentration results in the configuration space only.

The QUE property on the base, stating that the whole sequence of measures $\mu_j=\phi_j^2\, dV_g$ converges vaguely to the uniform measure, postulates that there is no such concentration phenomenon (see \cite{RudnickSarnak1994}).

\begin{figure}[h]
\begin{center}
\includegraphics[width=6cm,height=4.2cm]{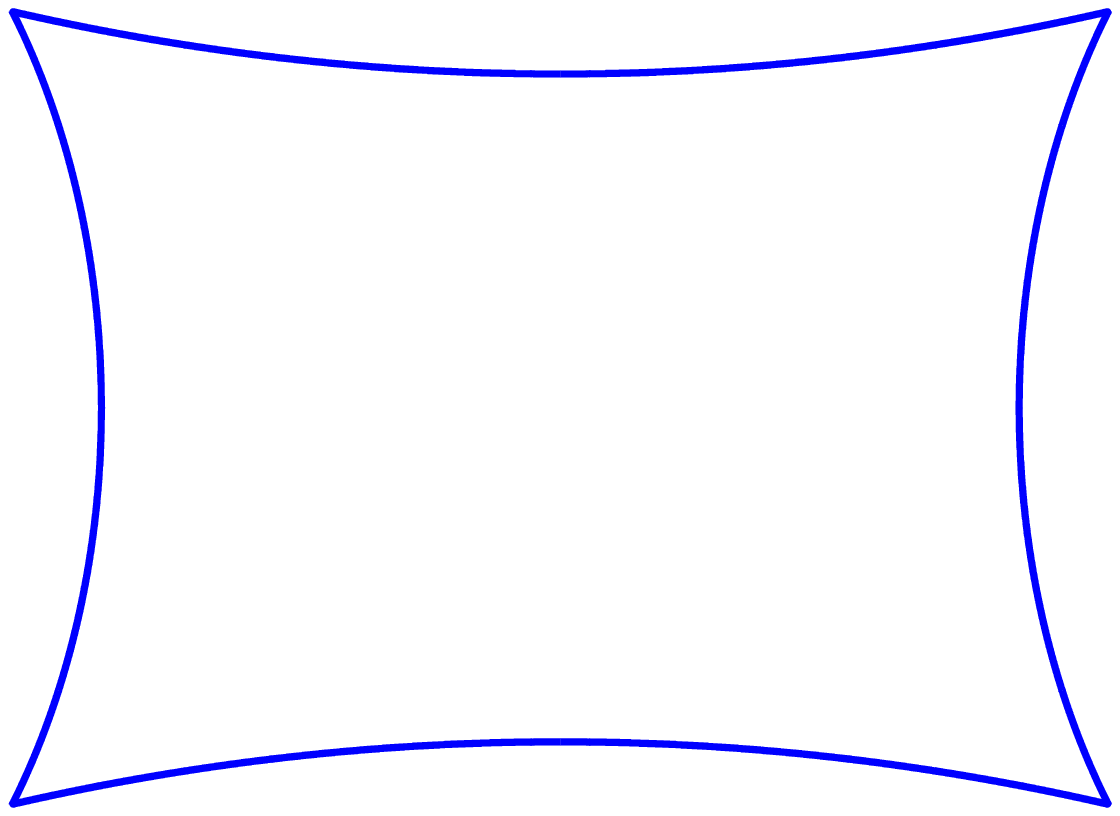}
\includegraphics[width=6cm,height=4.2cm]{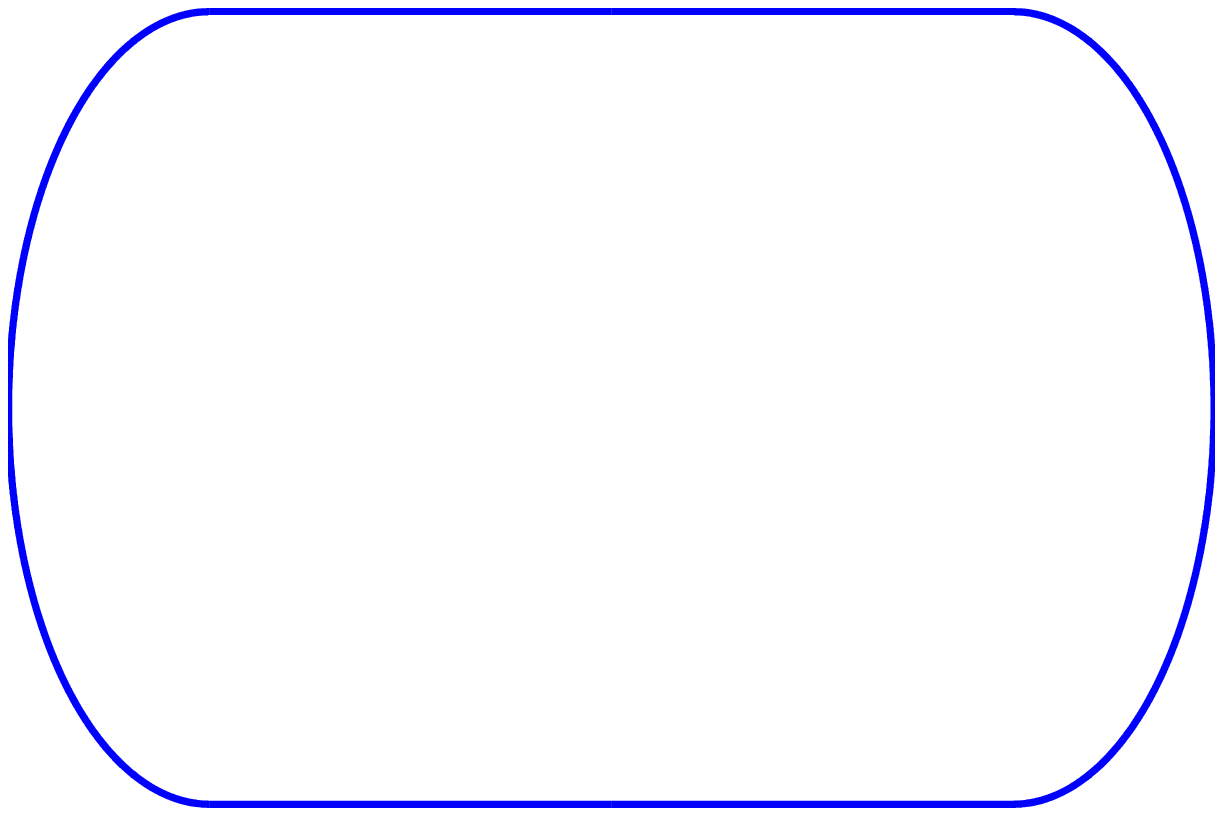}
\includegraphics[width=6cm,height=4.2cm]{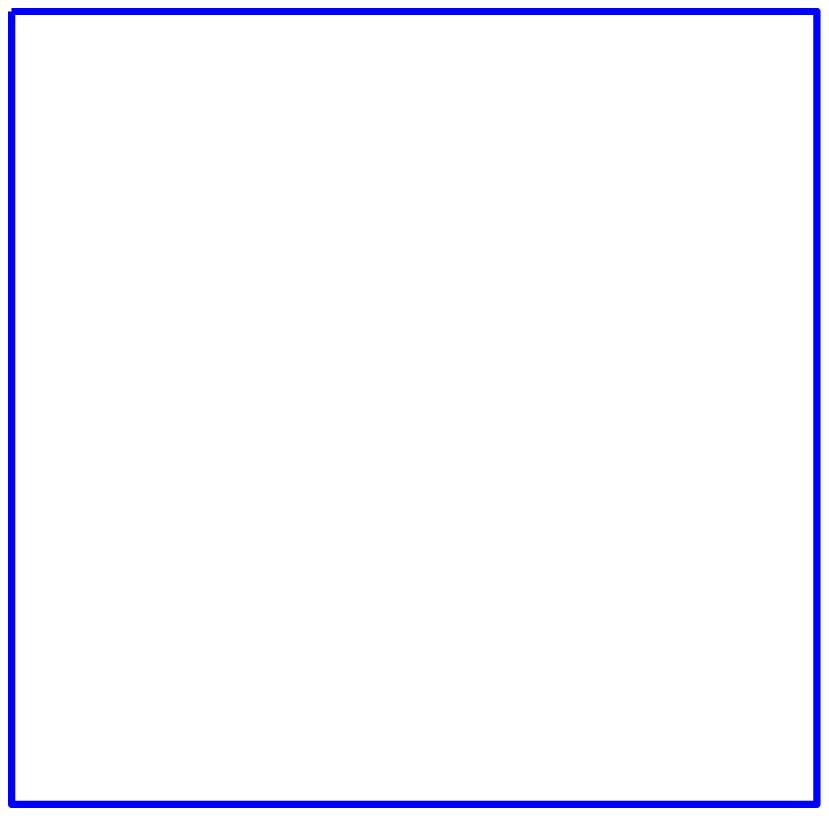}
\includegraphics[width=6cm,height=4.2cm]{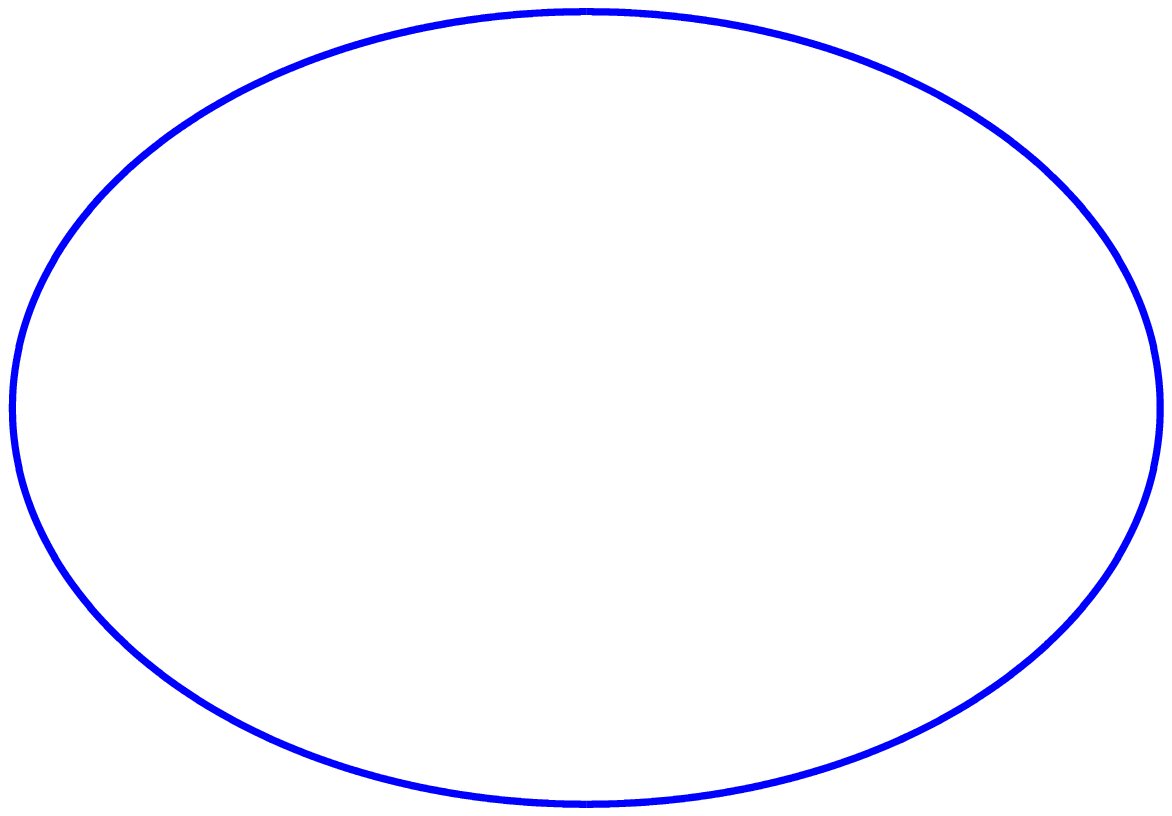}
\caption{Top left: a Barnett billiard, conjectured in \cite{barnett} to satisfy QUE. Top right: a Bunimovich stadium satisfying QE but not QUE for almost every value of the straight edge $t$. Bottom left and right: two shapes with piecewise regular boundary that are not ergodic and thus not quantum ergodic.}\label{ergoShapes}
\begin{picture}(1,1)
\put(90,280){\vector(1,0){48}}
\put(90,280){\vector(-1,0){48}}
\put(90,285){$t$}
\end{picture}
\end{center}
\end{figure}

We recall that what is called the billiard in $\Omega$ is the dynamical system posed on the unit cotangent bundle of $\overline{\Omega}$, representing (almost all) trajectories in $\Omega$ along geodesics with unit speed and reflecting on the boundary of $\Omega$ according to the usual reflection rules.
The ergodicity of $\Omega$ is however just a necessary assumption for QUE to hold. Note that strictly convex billiards whose boundary is $C^6$ are not ergodic in the phase space (see \cite{Lazutkin1973}), and there are sequences of positive density of eigenfunctions which concentrate on caustics.
It has been shown in \cite{KerckhoffMasurSmillie1986} that rational polygonal billiards are not ergodic in the phase space, while polygonal billiards are generically ergodic. In \cite{MarklofRudcnick} the authors prove that quantum ergodicity on the base holds in any rational polygon\footnote{A rational polygon is a planar polygon whose interior is connected and simply connected and whose vertex angles are rational multiples of $\pi$.}.
It has been proved recently in \cite{hassell} that there exist some convex sets satisfying QE but not QUE (independently on the basis of eigenfunctions under consideration). More precisely, in this reference the author studies the particular case of a stadium $S_t$ with straight edge $t$ (see Figure \ref{ergoShapes}, top right). He shows that for almost every $t$ the stadium is not quantum unique ergodic although it is quantum ergodic. He exhibits some particular quantum limits giving a positive mass on the set of bouncing ball trajectories.
Up to now obtaining sufficient conditions on the domain such that QUE holds is a widely open difficult question. It was conjectured in \cite{RudnickSarnak1994} that every compact negatively curved manifold satisfies QUE (see \cite{Sarnak} for a recent survey).
A numerical method has been developed in \cite{barnett} in order to compute the $700 000$ first modes for a (planar) domain analogue of variable negative curvature, and the results confirm numerically the QUE conjecture for such general systems (see Figure \ref{ergoShapes}, top left). 
Note that the QUE property has been proved to hold on arithmetic manifolds in \cite{Lindenstrauss}.
Finally, note that, using a concept of entropy, the authors of \cite{Anantharaman,AnantharamanNonnenmacher} show that on a compact manifold of negative curvature the eigenfunctions cannot concentrate entirely on closed geodesics and at least half of their energy remains chaotic (see also \cite{BurqZworski} for related issues). Up to now, except the case of arithmetic manifolds, there does not exist any example of multi-dimensional domain in which QUE holds, and this is still currently one of the deepest issues in mathematical physics.

\begin{figure}[h]
\begin{center}
\includegraphics[width=4cm,angle=-90]{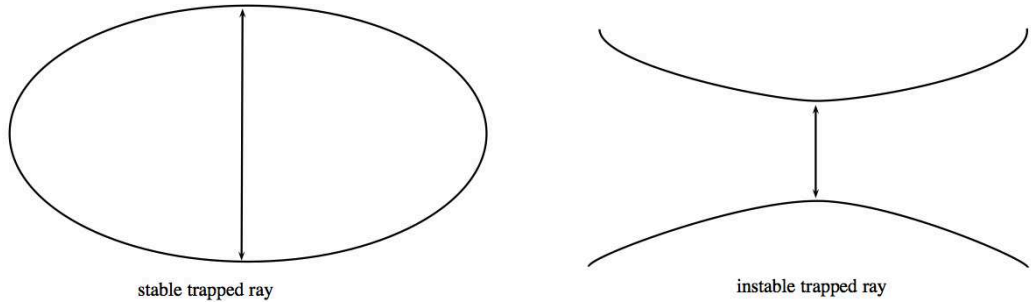}
\caption{Stable and instable trapped rays}\label{frays}
\end{center}
\end{figure}

Finally, having Figure \ref{frays}  in mind , it is not surprising that the stable trapped ray of the left-side figure causes a concentration of eigenfunctions, due to the convexity of the domain $\Omega$. On the right-side figure, we have shaped a neighorhood of a domain $\Omega$ in which negative curvature is suggested by the hyperbolic boundary, and there is a unique trapped ray, which is instable. Due to this instability feature, it might be expected that the energies of eigenfunctions spread away and that QUE holds true. This intuition is however not true.
In \cite{CdVParisse}, the authors build a compact surface of $\R^3$ endowed with a metric of negative curvature, by truncating (for instance) an hyperboloid symmetrically with respect to its center, and considering Dirichlet boundary conditions on both truncated sides. Then, they show the existence of sequences of eigenfunctions of the Dirichlet-Laplacian that concentrate around the equator that is a closed instable geodesic.
This surprising example shows that the QUE property, as well as the QUE conjecture, is definitely a global one, and cannot be inferred from local considerations.

It can however be noticed that, as a consequence of \cite{Anantharaman}, the arc-length measure along a closed geodesic on a negatively curved manifold cannot be a quantum limit (see also \cite{Sarnak}).
\end{remark}

\begin{remark}\label{rem_motivate}
The results Theorems \ref{thmnogap} and \ref{thmnogap2} are similar but distinct. The QUE property assumed in Theorem \ref{thmnogap2} is a very strong one and as said above up to now examples of domains in dimension more than one satisfying QUE are not known.
The proofs of these results, provided in Sections \ref{secproofthmnogap} and \ref{secproofthmnogap2}, are of a completely different nature.
In particular, our proof of Theorem \ref{thmnogap} is short but does not permit to get an insight on the possible theoretical construction of a maximizing sequence of subsets. 
In contrast, our proof of Theorem \ref{thmnogap2} is constructive and provides a theoretical way of building a maximizing sequence of subsets, by implementing a kind of homogenization procedure. Moreover, this proof highlights the following interesting feature:
\begin{quote}
\textit{It is possible to increase the values of $J$ by considering subsets having an increasing number of connected components.}
\end{quote}
Note that another way of building maximizing sequences is provided in Section \ref{sectrunc2}, by considering an appropriate spectral approximation of the problem, suitable for numerical simulations.
\end{remark}

\begin{remark}\label{rem_existencegap}
The question of knowing whether there exists an example where there is a gap between the convexified problem \eqref{quantity2obsAsympconvexified} and the original one \eqref{reducedsecondpb}, is an open problem.
We think that, if such an example exists, then the underlying geodesic flow ought to be completely integrable and have strong concentration properties. 

Note that, according to \cite{JakobsonZelditch}, the set of quantum limits on the unit sphere $S^2$ of $\R^3$ is equal to the whole convex set of invariant probability measures for the geodesic flow that are time-reversal invariant (that is, invariant under the anti-symplectic involution $(x,\xi) \mapsto (x,-\xi)$ on $T^*M$). In particular, the Dirac measure along any great circle $\gamma$ on $S^2$ (defined as an equator, up to a rotation) is the projection of a semi-classical measure. However, as already mentioned in our framework we have fixed a given basis $(\phi_j)_{j\in\N^*}$ of eigenvectors, and we consider only the weak limits of the measures $\phi_j(x)^2 dV_g$, whereas the result of \cite{JakobsonZelditch} holds when one considers the limits over all possible bases. With a fixed given basis, we are not aware of any example having concentration properties strong enough to derive a gap statement. We refer to Section \ref{sec6.4} and in particular to Proposition \ref{remgapintrinsic} for an example of a gap for an intrinsic variant of the second problem.

Note that the usual basis on $S^2$, consisting of spherical harmonics, does not satisfy the QE property.
But, due to the high multiplicity of eigenvalues, there is an infinite dimensional manifold of orthonormal bases of eigenfunctions. Indeed, using the spherical harmonics, the space of orthonormal bases is identified with the infinite product $\prod_{k=0}^{+\infty} U(2k+1)$ of unitary groups, and thus inherits of the corresponding probability Haar measure.
It is then proved in \cite{Zel2} that, on the standard sphere, almost every orthonormal basis of eigenfunctions of the Laplacian satisfies the QE property.
\end{remark}

\begin{remark}
Our results here show that shape optimization problems are intimately related with the ergodicity properties of $\Omega$. Notice that, in the early article \cite{CFNS_1991}, the authors suggested such connections. They analyzed the exponential decay of solutions of damped wave equations. 
Their results reflected that the quantum effects of bouncing balls or whispering galleries play an important role in the success of failure of the exponential decay property. At the end of the article, the authors conjectured that such considerations could be useful in the placement and design of actuators or sensors. Our results of this section provide precise results showing these connections and new perspectives on those intuitions. In our view they are the main contribution of our article, in the sense that they have pointed out the close relations existing between shape optimization and ergodicity, and provide new open problems and directions to domain optimization analysis.
\end{remark}


\subsection{On the existence of an optimal set}\label{sec4.3}
In this section we comment on the problem of knowing whether the supremum in \eqref{eqNoGap} is reached or not, in the framework of Theorem \ref{thmnogap}. This problem remains essentially open except in several particular cases.

For the one-dimensional case already mentioned in Remarks \ref{remTmult2pi}, \ref{rempossiblegap} and \ref{rem8}, we have the following result.

\begin{lemma}\label{lemprelim}
Assume that $\Omega=[0,\pi]$.
Let $L\in (0,1)$. The supremum of $J$ over $\mathcal{U}_L$ (which is equal to $L$) is reached if and only if $L=1/2$. In that case, it is reached for all measurable subsets $\omega\subset[0,\pi]$ of measure $\pi/2$ such that $\omega$ and its symmetric image $\omega'=\pi-\omega$ are disjoint and complementary in $[0,\pi]$.
\end{lemma}

\begin{proof}
Although the proof of that result can be found in \cite{henrot_hebrardSCL} and in \cite{PTZObs1}, we recall it here  shortly since similar arguments will be used  in the proof of the forthcoming Lemma \ref{lem4}.

A subset $\omega\subset[0,\pi]$ of Lebesgue measure $L\pi$ is solution of \eqref{eqNoGap} if and only if $ \int_{\omega}\sin^2(jx)\, dx \geq L\pi/2$ for every $j\in\N^*$, that is,
$ \int_{\omega}\cos(2jx)\, dx\leq 0$. Therefore the Fourier series expansion of $\chi_\omega$ on $[0,\pi]$ must be of the form
$$
L+\sum_{j=1}^{+\infty}(a_j\cos (2jx)+b_j\sin (2jx)),
$$
with coefficients $a_j\leq 0$. Let $\omega'=\pi-\omega$ be the symmetric set of $\omega$ with respect to $\pi /2$. The Fourier series expansion of $\chi_{\omega'}$ is
$$
L+\sum_{j=1}^{+\infty}(a_j\cos (2jx)-b_j\sin (2jx)).
$$
Set $g(x)=L-\frac{1}{2}(\chi_\omega(x)+\chi_{\omega'}(x))$, for almost every $x\in[0,\pi]$. The Fourier series expansion of $g$ is $-\sum_{j=1}^{+\infty}a_j\cos (2jx)$, with $a_j\leq 0$ for every $j\in\N^*$. Assume that $L\neq 1/2$. Then the sets $\omega$ and $\omega'$ are not disjoint and complementary, and hence $g$ is discontinuous. It then follows that $\sum_{j=1}^\infty a_j=-\infty$. Besides, the sum $\sum_{j=1}^\infty a_j$ is also the limit of $\sum_{k=1}^{+\infty}a_k\widehat{\Delta}_n(k)$ as $n\to +\infty$, where $\widehat{\Delta}_n$ is the Fourier transform of the positive function $\Delta_n$ whose graph is the triangle joining the points $(-\frac{1}{n},0)$, $(0,2n)$ and $(\frac{1}{n},0)$ (note that $\Delta_n$ is an approximation of the Dirac measure, with integral equal to $1$). This raises a contradiction with the fact that
$$
\int_0^\pi g(t)\Delta_n(t)dt=\sum_{k=1}^{+\infty}a_k\widehat{\Delta}_n(k),
$$
derived from Plancherel's Theorem.
\end{proof}

For the two-dimensional square $\Omega=[0,\pi]^2$ studied in Proposition \ref{propnogapnoQUE} we are not able to provide a complete answer to the question of the existence. We are however able to characterize the existence of optimal sets that are a Cartesian product.

\begin{lemma}\label{lem4}
Assume that $\Omega=[0,\pi]^2$. Let $L\in (0,1)$. The supremum of $J$ over the class of all possible subsets $\omega=\omega_1\times\omega_2$ of Lebesgue measure $L\pi^2$, where $\omega_1$ and $\omega_2$ are measurable subsets of $[0,\pi]$, is reached if and only if $L\in \{1/4,1/2,3/4\}$. In that case, it is reached for all such sets $\omega$ satisfying
$$
\frac{1}{4}(\chi_{\omega}(x,y)+\chi_{\omega}(\pi-x,y)+\chi_{\omega}(x,\pi-y)+\chi_{\omega}(\pi-x,\pi-y)) = L,
$$
for almost all $(x,y)\in[0,\pi^2]$.
\end{lemma}

\begin{proof}
A subset $\omega\subset[0,\pi]^2$ of Lebesgue measure $L\pi^2$ is solution of \eqref{eqNoGap} if and only if
$$ \frac{4}{\pi^2}\int_{\omega}\sin^2(jx)\sin^2(ky)\, dx\, dy \geq L, $$
for all $(j,k)\in (\N^*)^2$, that is,
\begin{equation}\label{eqSquare}
\int_{\omega}\cos(2jx)\cos(2ky)\, dx\, dy \geq \int_{\omega}\cos(2jx)\, dx\, dy+\int_{\omega}\cos(2ky)\, dx\, dy.
\end{equation}
Set $\ell_{x}= \int_0^\pi \chi_{\omega}(x,y)\, dy$ for almost every $x\in [0,\pi]$, and $\ell_{y}= \int_0^\pi \chi_{\omega}(x,y)\, dx$ for almost every $y\in [0,\pi]$.
Letting either $j$ or $k$ tend to $+\infty$ and using Fubini's theorem in \eqref{eqSquare} leads to
$$
 \int_0^\pi \ell_x\cos(2jx)\, dx\leq 0\textrm{ and }\int_0^\pi \ell_y\cos(2ky)\, dy\leq 0,
$$
for every $j\in\N^*$ and every $k\in\N^*$.

Now, if $\omega=\omega_1\times\omega_2$, where $\omega_1$ and $\omega_2$ are measurable subsets of $[0,\pi]$, then the functions $x\mapsto\ell_x$ and $y\mapsto\ell_y$ must be discontinuous. Using similar arguments as in the proof of Lemma \ref{lemprelim}, it follows that the functions $x\mapsto\ell_x+\ell_{\pi-x}$ and $y\mapsto\ell_y+\ell_{\pi-y}$ must be constant on $[0,\pi]$, and hence,
$$
\int_0^\pi \ell_x\cos(2jx)\, dx= 0\textrm{ and }\int_0^\pi \ell_y\cos(2ky)\, dy= 0,
$$
for every $j\in\N^*$ and every $k\in\N^*$. Using \eqref{eqSquare}, it follows that
$$
\int_{\omega}\cos(2jx)\cos(2ky)\, dx\, dy \geq 0,
$$
for all $(j,k)\in(\N^*)^2$.
The function $F$ defined by
$$
F(x,y)=\frac{1}{4}(\chi_{\omega}(x,y)+\chi_{\omega}(\pi-x,y)+\chi_{\omega}(x,\pi-y)+\chi_{\omega}(\pi-x,\pi-y)),
$$
for almost all $(x,y)\in[0,\pi]^2$, can only take the values $0$, $1/4$, $1/2$, $3/4$ and $1$, and its Fourier series is of the form
$$
L+\frac{4}{\pi^2}\sum_{j,k=1}^{+\infty}\left(\int_\omega \cos(2ju)\cos(2kv)\, du\, dv\right)\cos (2jx)\cos (2ky),
$$
and all Fourier coefficients are nonnegative. 
Using once again similar arguments as in the proof of Lemma \ref{lemprelim} (Fourier transform and Plancherel's Theorem), it follows that $F$ must necessarily be continuous on $[0,\pi]^2$ and thus constant. The conclusion follows.
\end{proof}

\begin{remark}
All results of this section can obviously be generalized to multi-dimensional domains $\Omega$ written as $N$ cartesian products of one-dimensional sets. 
\end{remark}

\begin{remark}\label{rem21}
According to Lemma \ref{lem4}, if $L=1/2$ then there exists an infinite number of optimal sets.
Four of them are drawn on Figure \ref{figcarre}.
It is interesting to note that the optimal sets drawn on the left-side of the figure do not satisfy the Geometric Control Condition mentioned in Section \ref{secobservabilityintro}, and that in this configuration the (classical, deterministic) observability constants $C_T^{(W)}(\chi_\omega)$ and $C_T^{(S)}(\chi_\omega)$ are equal to $0$, whereas, according to the previous results, there holds
$$
2\,C_{T,\textrm{rand}}^{(W)}(\chi_\omega) = C_{T,\textrm{rand}}^{(S)}(\chi_\omega)
=  T L.
$$
This fact is in accordance with Remarks \ref{remBZ} and \ref{remBZ2}. 
\begin{figure}[h]
\begin{center}
\includegraphics[width=11cm]{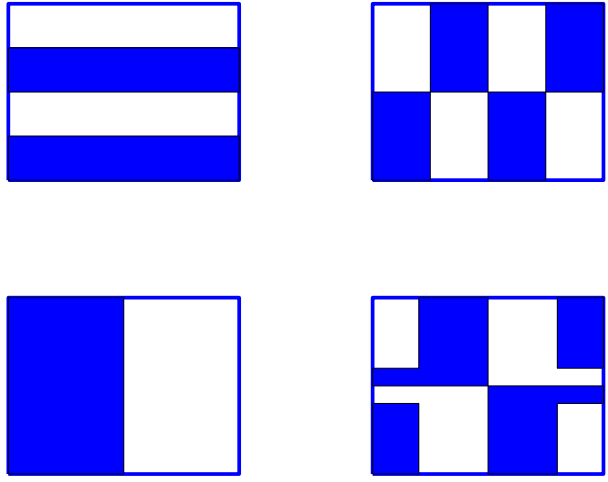}
\caption{In the case $\Omega=[0,\pi]^{2}$, $L=1/2$, representation of four domains (in blue) maximizing $J$ in $\mathcal{U}_L$.}\label{figcarre}
\end{center}
\end{figure}
\end{remark}

\begin{remark}
Similar considerations hold for the two-dimensional unit disk. Actually it easily follows from Lemma \ref{lemprelim} and from the proof of Proposition \ref{propnogapnoQUE} that, for $L=1/2$, the supremum of $J$ over $\mathcal{U}_L$ is reached for every subset $\omega$ of the form
$$\omega = \{ (r,\theta)\in[0,1]\times[0,2\pi] \ \vert\ \theta\in\omega_\theta \},$$
where $\omega_\theta$ is any subset of $[0,2\pi]$ such that $\omega_\theta$ and its symmetric image $\omega'_\theta=2\pi-\omega_\theta$ are disjoint and complementary in $[0,2\pi]$. But we do not know whether or not there are other maximizing subsets.
\end{remark}

\begin{remark}
In view of the results above one could expect that when $\Omega$ is the unit $N$-dimensional hypercube, there exists a finite number of values of $L\in (0,1)$ such that the supremum in \eqref{eqNoGap} is reached. The same result can probably be expected for generic domains $\Omega$. But these issues are open.
\end{remark}


\subsection{Proof of Theorem \ref{thmnogap}}\label{secproofthmnogap}
Since we already have the inequality
$$
\sup_{\chi_\omega\in \mathcal{U}_L} \inf_{j\in \N^*}\int_{\Omega}\chi_\omega(x) \phi_j(x)^2\, dV_g\ \leq\ L,
$$
it suffices to prove that, for every $\varepsilon>0$, there exists $\chi_\omega\in\mathcal{U}_L$ such that
$$
\left\vert \int_\omega \phi_j(x)^2\, dV_g - L  \right\vert \leq\varepsilon,
$$
for every $j\in\N^*$.
To prove this fact, we consider the function $f$ defined by $f(x) = (\phi_j(x)^2)_{j\in\N^*}$, for every $x\in\Omega$.
Using the fact that the eigenfunctions are uniformly bounded in $L^\infty(\Omega)$, it is clear that $f(x)\in\ell^\infty$, for every $x\in\Omega$.
Then, clearly, $f\in L^1(\Omega,\ell^\infty)$ (using the Bochner integral), and $\int_\Omega f\, dV_g$ is the constant sequence of $\ell^\infty$ equal to $1$.
For every $\varepsilon>0$, there exists a partition of $\Omega=\cup_{k=1}^n\Omega_k$, with $\Omega_k$ measurable, such that $\int_\Omega\Vert (f-f_n)\Vert_{\ell^\infty} dV_g\leq\varepsilon/(L+1)$, with $f_n=\sum_{k=1}^n\alpha_k\chi_{\Omega_k}$. For every $k\in\{1,\ldots,n\}$, let $\omega_k$ be a measurable subset of $\Omega_k$ such that $V_g(\omega_k)=LV_g(\Omega_k)$. We set $\omega=\cup_{k=1}^n\omega_k$. Note that, by construction, one has $\chi_\omega\in\mathcal{U}_L$, and
$$
\int_\Omega (\chi_\omega - L) f_n\, dV_g = \sum_{k=1}^n\alpha_k \int_\Omega (\chi_\omega - L)\chi_{\Omega_k}\, dV_g = \sum_{k=1}^n\alpha_k ( V_g(\omega_k)-LV_g(\Omega_k))=0.
$$
Therefore, there holds
\begin{equation*}
\begin{split}
\left\Vert \int_\omega f\, dV_g - L \int_\Omega f\, dV_g  \right\Vert_{\ell^\infty} 
&= \left\Vert \int_\Omega (\chi_\omega - L) f\, dV_g  \right\Vert_{\ell^\infty} \\
&\leq \left\Vert \int_\Omega (\chi_\omega - L) f_n\, dV_g  \right\Vert_{\ell^\infty} +\left\Vert \int_\Omega (\chi_\omega - L) (f-f_n)\, dV_g  \right\Vert_{\ell^\infty}  \\
&\leq \varepsilon
\end{split}
\end{equation*}
and the conclusion follows.

\subsection{Proof of Theorem \ref{thmnogap2}}\label{secproofthmnogap2}
In what follows, for every measurable subset $\omega$ of $\Omega$, we set
$$
I_j(\omega) = \int_\omega\phi_j(x)^2\, dV_g ,
$$
for every $j\in\N^*$. By definition, there holds
$$
J(\omega) = \inf_{j\in\N^*}I_j(\omega) .
$$
Note that it follows from QUE and from the Portmanteau theorem (see Remark \ref{remWQE}) that, for every measurable subset $\omega$ of $\Omega$ such that $V_g(\omega)=LV_g(\Omega)$ and $V_g(\partial\omega)=0$, one has $I_j(\omega)\rightarrow L$ as $j\rightarrow+\infty$, and hence $J(\omega)\leq L$. 

Let $\omega_0$ be an open connected subset of $\Omega$ of measure $LV_g(\Omega)$ having a Lipschitz boundary. In the sequel we assume that $J(\omega_0)<L$, otherwise there is nothing to prove. Using QUE, there exists an integer $j_0$ such that
\begin{equation}\label{highfreq}
I_j(\omega_0) \geq L-\frac{1}{4}(L-J(\omega_0)) ,
\end{equation}
for every $j> j_0$. 
Our proof below consists of implementing a kind of homogenization procedure by constructing a sequence of open subsets $\omega_k$ (starting from $\omega_0$) such that $V_g(\omega_k)=LV_g(\omega_k)$ and $\displaystyle \lim_{k\to +\infty} J(\omega_k)= L$.
Denote by $\overline{\omega_0}$ the closure of $\omega_0$, and by $\omega_0^c$ the complement of $\omega_0$ in $\Omega$.
Since $\Omega$  and $\omega_0$ have a Lipschitz boundary, it follows that $\omega_0$ and $\Omega\backslash \omega_0$ satisfy a $\delta$-cone property\footnote{\label{footnoteEpscone}We recall that an open subset $\Omega$ of $\R^n$ verifies a $\delta$-cone property if, for every $x\in\partial\Omega$, there exists a normalized vector $\xi_x$ such that $C(y,\xi_x,\delta)\subset \Omega$ for every $y\in \overline{\Omega}\cap B(x,\delta)$, where $C(y,\xi_x,\delta)=\{z\in\R^n\ \vert\ \langle z-y,\xi\rangle \geq \cos\delta \Vert z-y\Vert\text{ and }0<\Vert z-y\Vert <\delta\}$. For manifolds, the definition is done accordingly in some charts, for $\delta>0$ small enough.}, for some $\delta>0$ (see \cite[Theorem 2.4.7]{henrot-pierre}).
Consider partitions of $\overline{\omega_0}$ and $\omega_0^c$,
\begin{equation}\label{labsubdivisions}
\overline{\omega}_0=\bigcup_{i=1}^KF_i\quad\textrm{and}\quad\omega_0^c=\bigcup_{i=1}^{\tilde{K}} \widetilde F_i,
\end{equation}
to be chosen later. As a consequence of the $\delta$-cone property, there exists $c_\delta>0$ and a choice of partition $(F_i)_{1\leq i \leq K}$ (resp. $(\widetilde F_i)_{1\leq i \leq \tilde{K}}$) such that, for $V_g(F_i)$ small enough,
\begin{equation}\label{regularMesh}
\forall i\in \{1,\cdots,K\} \ \left(\textrm{resp. }\forall i\in \{1,\cdots,\tilde{K}\} \right), \ \frac{\eta_i}{\operatorname{diam} F_i }\geq c_\delta \ \left(\textrm{resp. }\frac{\widetilde \eta_i}{\operatorname{diam}\widetilde F_i}\geq c_\delta \right),
\end{equation}
where $\eta_i$ (resp., $\widetilde\eta_i$) is the inradius\footnote{In other words, the largest radius of Riemannian balls contained in $F_i$.} of $F_i$ (resp., $\widetilde F_i$), and $\operatorname{diam}F_i$ (resp., $\operatorname{diam}\widetilde F_i$) the Riemannian diameter of $F_i$ (resp., of $\widetilde F_i$). 

It is then clear that, for every $i\in\{1,\ldots,K\}$ (resp., for every $i\in\{1,\ldots,\tilde{K}\}$), there exists $\xi_i\in F_i$ (resp., $\tilde{\xi}_i\in\widetilde F_i$) such that
$B(\xi_i,\eta_i/2)\subset F_i\subset B(\xi_i,\eta_i/c_\delta)$ (resp., $B(\tilde{\xi}_i,\widetilde\eta_i/2)\subset \widetilde F_i\subset B(\tilde{\xi}_i,\widetilde\eta_i/c_\delta)$), where the notation $B(\xi,\eta)$ stands for the open Riemannian ball centered at $\xi$ with radius $\eta$. These features characterize a substantial family of sets (also called nicely shrinking sets), as is well known in measure theory. 
By continuity, the points $\xi_i$ and $\tilde{\xi}_i$ are Lebesgue points of the functions $\phi_j^2$, for every $j\leq j_0$. This implies that, for every $j\leq j_0$, there holds
$$
\int_{F_i}\phi_j(x)^2\, dV_g = V_g(F_i)\phi_j(\xi_i)^2+\mathrm{o}(V_g(F_i))\quad\textrm{as}\ \eta_i\rightarrow 0,
$$
for every $i\in\{1,\ldots,K\}$, and 
$$
\int_{\widetilde F_i}\phi_j(x)^2\, dV_g = V_g(\widetilde F_i)\phi_j(\xi_i)^2+\mathrm{o}(V_g(\widetilde F_i))\quad\textrm{as}\ \widetilde \eta_i\rightarrow 0,
$$
for every $i\in\{1,\ldots,\tilde{K}\}$. Setting $\eta=\displaystyle \max\left(\max_{1\leq i \leq K}\operatorname{diam}F_i,\max_{1\leq i \leq \tilde{K}}\operatorname{diam}\widetilde F_i\right)$, it follows that
\begin{equation}\label{quadrature}
\begin{split}
I_j(\omega_0)&=\int_{\omega_0}\phi_j(x)^2\, dV_g = \sum_{i=1}^K V_g(F_i)  \phi_j(\xi_i)^2 +\mathrm{o}(\eta^d)\quad\textrm{as}\ \eta\rightarrow 0,\\
I_j(\omega_0^c)&=\int_{\omega_0^c}\phi_j(x)^2\, dV_g= \sum_{i=1}^{\tilde{K}} V_g(\widetilde F_i)  \phi_j(\tilde{\xi}_i)^2+\mathrm{o}(\eta^d)\quad\textrm{as}\ \eta\rightarrow 0,
\end{split}
\end{equation}
for every $j\leq j_0$.
Note that, since $\omega_0^c$ is the complement of $\omega_0$ in $\Omega$, there holds
\begin{equation}\label{egopp}
I_j(\omega_0)+I_j(\omega_0^c)=\int_{\omega_0}\phi_j(x)^2\, dV_g+\int_{\omega_0^c}\phi_j(x)^2\, dV_g = 1,
\end{equation}
for every $j$. Note also that
$$
\sum_{i=1}^K V_g(F_i) =L V_g(\Omega) \quad\textrm{and}\quad\sum_{i=1}^{\tilde{K}} V_g(\widetilde F_i) =(1-L) V_g(\Omega) .
$$
Set $h_i=(1-L)  V_g(F_i) $ and $\ell_i=L V_g(\widetilde F_i) $. 
Then, we infer from \eqref{quadrature} and \eqref{egopp} that
\begin{equation}\label{eggrandO}
\begin{split}
(1-L)\, I_j(\omega_0) &= \sum_{i=1}^Kh_i\phi_j(\xi_i)^2 +\mathrm{o}(\eta^d)\quad\textrm{as}\ \eta\rightarrow 0,\\
L\, I_j(\omega_0) &= L-\sum_{i=1}^{\tilde{K}}\ell_i\phi_j(\tilde{\xi}_i)^2 +\mathrm{o}(\eta^d)\quad\textrm{as}\ \eta\rightarrow 0,
\end{split}
\end{equation}
for every $j\leq j_0$. For $\varepsilon>0$ to be chosen later, define the perturbation $\omega^\varepsilon$ of $\omega_0$ by
$$
\omega^\varepsilon =\left(\omega_0\backslash \bigcup_{i=1}^K \overline{B(\xi_i,\varepsilon_i)}\right)\quad \bigcup\quad \bigcup_{i=1}^{\tilde{K}} B(\tilde{\xi}_i,\widetilde \varepsilon_i),
$$
where $\varepsilon_i=\varepsilon h_i^{1/n}/V_g( B(\xi_i,1))^{1/n}$ and $\widetilde \varepsilon_i=\varepsilon \ell_i^{1/n}/V_g( B(\tilde\xi_i,1))^{1/n}$.
Note that it is possible to define such a perturbation, provided that
$$
0<\varepsilon< \min\left(\min_{1\leq i \leq K}\frac{\eta_i V_g( B(\xi_i,1))^{1/n}}{h_i^{1/n}},\min_{1\leq i \leq \tilde{K}}\frac{\widetilde \eta_i V_g( B(\tilde\xi_i,1))^{1/n}}{\ell_i^{1/n}}\right).
$$
It follows from the well known isodiametric inequality\footnote{The isodiametric inequality states that, for every compact $K$ of the Euclidean space $\R^n$, there holds $\vert K\vert\leq \vert B(0,\operatorname{diam}(K)/2)\vert$.} and from a compactness argument that there exists a constant $V_n>0$ (only depending on $\Omega$) such that $V_g(F_i) \leq V_n({\operatorname{diam}F_i})^n$ for every $i\in \{1,\cdots,K\}$, and $V_g( \tilde F_i) \leq V_n({\operatorname{diam}\tilde F_i})^n$ for every $i\in \{1,\cdots,\tilde K\}$, independently on the partitions considered. Again, by compactness of $\Omega$, there exists $v_n>0$ (only depending on $\Omega$) such that $V_g( B(x,1))\geq v_n$ for every $x\in\Omega$.
Set $\varepsilon_0=\min(1,c_\delta v_n/V_n^{1/n})$. Using \eqref{regularMesh}, we get
$$
\frac{\eta_iV_g( B(\xi_i,1))^{1/n}}{h_i^{1/n}}\geq \frac{v_n}{(1-L)^{1/n}V_n^{1/n}}\frac{\eta_i}{\operatorname{diam}F_i}\geq \varepsilon_0 ,
$$
for every $i\in \{1,\cdots,K\}$, and similarly,
$$
\frac{\widetilde\eta_iV_g( B(\tilde\xi_i,1))^{1/n}}{\ell_i^{1/n}}\geq \varepsilon_0 ,
$$
for every $i\in \{1,\cdots,\tilde{K}\}$. It follows that the previous perturbation is well defined for every $\varepsilon\in (0,\varepsilon_0)$.
Note that, by construction, 
\begin{eqnarray*}
 V_g( \omega^\varepsilon )  & = &  V_g(\omega_0) - \sum_{i=1}^K\varepsilon_i^nV_g( B(\xi_i,1))+\sum_{i=1}^{\tilde{K}}\widetilde \varepsilon_i^nV_g( B(\tilde\xi_i,1)) \\
 & = &   V_g(\omega_0) -\varepsilon^n \sum_{i=1}^Kh_i+\varepsilon^n\sum_{i=1}^{\tilde{K}}\ell_i \\
 & = &  V_g(\omega_0)-\varepsilon^n (1-L)\sum_{i=1}^KV_g(F_i)+\varepsilon^n L \sum_{i=1}^{\tilde{K}}V_g(\tilde F_i) \\
 & = &   V_g(\omega_0) -\varepsilon^n (1-L) LV_g(\Omega) +\varepsilon^n L (1-L)V_g(\Omega) \\
 &=&  V_g(\omega_0) =L V_g(\Omega) .
\end{eqnarray*}
Moreover, one has
\begin{equation*}
I_j(\omega^\varepsilon)   =  \int _{\omega^\varepsilon}\phi_j(x)^2\, dV_g 
= I_j(\omega_0)-\sum_{i=1}^K\int_{B(\xi_i,\varepsilon_i)}\phi_j(x)^2\, dV_g+\sum_{i=1}^{\tilde{K}}\int_{B(\tilde{\xi}_i,\widetilde \varepsilon_i)}\phi_j(x)^2\, dV_g 
\end{equation*}
and using again the fact that the $\xi_i$ and $\tilde{\xi}_i$ are Lebesgue points of the functions $\phi_j^2$, for every $j\leq j_0$, we infer that
\begin{eqnarray*}
I_j(\omega^\varepsilon)  & = &  I_j(\omega_0)- \sum_{i=1}^K\varepsilon_i^nV_g( B(\xi_i,1)) \phi_j(\xi_i)^2 + \sum_{i=1}^{\tilde{K}}\widetilde\varepsilon_i^nV_g( B(\tilde\xi_i,1)) \phi_j(\tilde{\xi}_i)^2 
+\mathrm{o}(\eta^d)\quad\textrm{as}\ \eta\rightarrow 0\\
&=&  I_j(\omega_0) - \varepsilon^n \left( \sum_{i=1}^K h_i \phi_j(\xi_i)^2 - \sum_{i=1}^{\tilde{K}} \ell_i \phi_j(\tilde{\xi}_i)^2 \right) 
+\mathrm{o}(\eta^d)\quad\textrm{as}\ \eta\rightarrow 0 ,   
\end{eqnarray*}
and hence, using \eqref{eggrandO},
\begin{equation*}
I_j(\omega^\varepsilon) =  I_j(\omega_0)+\varepsilon^n \left(  L  -I_j(\omega_0)\right)+\varepsilon^{n} \mathrm{o}(\eta^d)\quad\textrm{as}\ \eta\rightarrow 0,
\end{equation*}
for every $j\leq j_0$ and every $\varepsilon\in(0,\varepsilon_0)$.
Since $\varepsilon_0^n\leq 1$, it then follows that
\begin{equation}\label{lowfreq}
I_j(\omega^\varepsilon) \geq J(\omega_0)+\varepsilon^n (L-J(\omega_0))+\varepsilon^{n} \mathrm{o}(\eta^d)\quad\textrm{as}\ \eta\rightarrow 0 ,
\end{equation}
for every $j\leq j_0$ and every $\varepsilon\in(0,\varepsilon_0)$, where the functional $J$ is defined by \eqref{defJ}.

We now choose the subdivisions \eqref{labsubdivisions} fine enough (that is, $\eta>0$ small enough) so that, for every $j\leq j_0$, the remainder term $\underset{\eta\rightarrow 0}{\mathrm{o}}(\eta^d)$ in \eqref{lowfreq} is bounded by $\frac{1}{2} (L-J(\omega_0))$. It follows from \eqref{lowfreq} that
\begin{equation}\label{in29eps_1}
I_j(\omega^\varepsilon) \geq J(\omega_0)+\frac{\varepsilon^n}{2}  (L-J(\omega_0)),
\end{equation}
for every $j\leq j_0$ and every $\varepsilon\in(0,\varepsilon_0)$.

\medskip

Let us first show that the set $\omega^\varepsilon$ still satisfies an inequality of the type \eqref{highfreq} for $\varepsilon$ small enough. 
Using \eqref{UBPp} and H\"older's inequality, we have
\begin{equation*}
\begin{split}
\vert I_j(\omega^\varepsilon)-I_j(\omega_0)\vert & =  \left\vert\int_\Omega \left(\chi_{\omega^\varepsilon}(x)-\chi_{\omega_0}(x)\right)\phi_j(x)^2\,dV_g\right\vert \\
& \leq A^2 \left(\int_{\Omega}\vert\chi_{\omega^{\varepsilon}}(x)-\chi_{\omega_{0}}(x)\vert^q \, dV_{g}\right)^{1/q},
\end{split}
\end{equation*}
for every integer $j$ and every $\varepsilon\in(0,\varepsilon_0)$, where $q$ is defined by $\frac{1}{p}+\frac{1}{q}=1$.
Moreover,
\begin{equation*}
\begin{split}
\int_{\Omega}\vert\chi_{\omega^{\varepsilon}}(x)-\chi_{\omega_{0}}(x)\vert^q \, dV_{g} 
&= \int_{\Omega}\vert\chi_{\omega^{\varepsilon}}(x)-\chi_{\omega_{0}}(x)\vert \, dV_{g} \\
&= \varepsilon^n \left(\sum_{i=1}^Kh_i+\sum_{i=1}^{\tilde{K}}\ell_i\right) \\
&=2\varepsilon^n L(1-L) V_g(\Omega) ,
\end{split}
\end{equation*}
and hence
$$
\vert I_j(\omega^\varepsilon)-I_j(\omega_0)\vert \leq
 \left(2 A^{2q} \varepsilon^n L(1-L) V_g(\Omega)  \right)^{1/q}.
$$
Therefore, setting
$$ \varepsilon_1 = \min\left(\varepsilon_0, \left( \frac{(L-J(\omega_0))^q}{2^{2q+1}A^{2q}L(1-L)V_g(\Omega)}\right)^\frac{1}{n}\right),$$
it follows from \eqref{highfreq} that
\begin{equation}\label{highfreqeps}
I_j(\omega^\varepsilon) \geq L - \frac{1}{2}(L-J(\omega_0))
\end{equation}
for every $j\geq j_0$ and every $\varepsilon\in(0,\varepsilon_1)$.

Now, using the fact that $J(\omega_0)+\frac{\varepsilon^n}{2}  (L-J(\omega_0)) \leq L-\frac{1}{2}(L-J(\omega_0))$ for every $\varepsilon\in(0,\varepsilon_0)$, we infer from \eqref{in29eps_1} and \eqref{highfreqeps} that
\begin{equation}
J(\omega^\varepsilon)\geq J(\omega_0)+\frac{\varepsilon^n}{2}  (L-J(\omega_0)),
\end{equation}
for every $\varepsilon\in(0,\varepsilon_1)$. In particular, this inequality holds for $\varepsilon$ such that $\varepsilon^n = C_1 \min(C_2,L-J(\omega_0) )$,
where the positive constants $C_1$ and $C_2$ are defined by
$$
C_1 = \frac{1}{8AL(1-L)V_g(\Omega)},\qquad C_2 = \frac{1}{2^nC_1} .
$$
For this specific value of $\varepsilon$, we set $\omega_1=\omega^\varepsilon$, and hence we have obtained
\begin{equation}\label{debutiter}
J(\omega_1)\geq J(\omega_0)+\frac{C_1}{2} \min(C_2, L-J(\omega_0))\, (L-J(\omega_0)).
\end{equation}
Note that the constants involved in this inequality depend only on $L$, $A$ and $\Omega$. Note also that by construction $\omega_1$ satisfies a $\delta$-cone property.

If $J(\omega_1)\geq L$ then we are done. Otherwise, we apply all the previous arguments to this new set $\omega_1$: using QUE, there exists an integer still denoted $j_0$ such that \eqref{highfreq} holds with $\omega_0$ replaced with $\omega_1$. This provides a lower bound for highfrequencies. The lower frequencies $j\leq j_0$ are then handled as previously, and we end up with \eqref{in29eps_1} with $\omega_0$ replaced with $\omega_1$. Finally, this leads to the existence of $\omega_2$ such that \eqref{debutiter} holds with $\omega_1$ replaced with $\omega_2$ and $\omega_0$ replaced with $\omega_1$.

By iteration, we construct a sequence of subsets $\omega_k$ of $\Omega$ (satisfying a $\delta$-cone property) of measure $V_g(\omega_k)=LV_g(\Omega)$, as long as $J(\omega_k)<L$, satisfying
$$
J(\omega_{k+1})\geq J(\omega_k)+ \frac{C_1}{2} \min(C_2, L-J(\omega_k))\, (L-J(\omega_k)).
$$
If $J(\omega_k)<L$ for every integer $k$, then clearly the sequence $J(\omega_k)$ is increasing, bounded above by $L$, and converges to $L$.
This finishes the proof. 

\begin{remark}
It can be noted that, in the above construction, the subsets $\omega_k$ are open, Lipschitz and of bounded perimeter. Hence, if the second problem is considered on the class of measurable subsets $\omega$ of $\Omega$, of measure $V_g(\omega)=LV_g(\Omega)$, that are moreover either open with a Lipschitz boundary, or open with a bounded perimeter, then the conclusion holds as well that the supremum is equal to $L$. This proves the contents of Remark \ref{rem_others}.
\end{remark}

\subsection{Proof of Proposition \ref{propnogapnoQUE}}\label{sec_proof_propnogapnoQUE}
Assume that $\Omega$ is the unit (Euclidean) disk of $\R^2$, $\Omega=\{ x\in\R^2\ \vert\ \Vert x\Vert\leq 1\}$. It is well known that the normalized eigenfunctions of the Dirichlet-Laplacian are a triply indexed sequence given by
\begin{equation*}
\phi_{jkm}(r,\theta) = 
\left\{ \begin{array}{ll}
R_{0k}(r) & \ \textrm{if}\ j=0,\\
R_{jk}(r)Y_{jm}(\theta) & \ \textrm{if}\ j\geq 1,
\end{array} \right.
\end{equation*}
for $j\in\N$, $k\in\N^*$ and $m=1,2$, where $(r,\theta)$ are the usual polar coordinates.
The functions $Y_{jm}(\theta)$ are defined by $Y_{j1}(\theta)=\frac{1}{\sqrt{\pi}}\cos(j\theta)$ and $Y_{j2}(\theta)=\frac{1}{\sqrt{\pi}}\sin(j\theta)$, and the functions $R_{jk}$ are defined by
$$
R_{jk}(r) = \sqrt{2}\,\frac{J_j(z_{jk}r)}{\vert J'_{j}(z_{jk}) \vert},
$$
where $J_j$ is the Bessel function of the first kind of order $j$, and $z_{jk}>0$ is the $k^\textrm{th}$-zero of $J_{j}$.
The eigenvalues of the Dirichlet-Laplacian are given by the double sequence of $-z_{jk}^2$ and are of multiplicity $1$ if $j=0$, and $2$ if $j\geq 1$.
Many properties are known on these functions and, in particular (see \cite{Lagnese}):
\begin{itemize}
\item for every $j\in\N$, the sequence of probability measures $r\mapsto  R_{jk}(r)^2rdr$ converges vaguely to $1$ as $k$ tends to $+\infty$,
\item for every $k\in\N^*$, the sequence of probability measures $r\mapsto  R_{jk}(r)^2rdr$ converges vaguely to the Dirac at $r=1$ as $j$ tends to $+\infty$.
\end{itemize}
These convergence properties permit to identify certain quantum limits, the second property accounting for the well known phenomenon of whispering galleries. Less known is the convergence of the above sequence of measures when the ratio $j/k$ is kept constant. Simple computations (due to \cite{Burqpersonal}) show that, when taking the limit of $R_{jk}(r)^2rdr$ with a fixed ratio $j/k$, and making this ratio vary, we obtain the family of probability measures
$$ \mu_s = f_s(r)\, dr=\frac{ 1} { \sqrt{ 1 - s^2}} \frac{ r } { \sqrt{ r^2 - s^2}} \chi_{(s,1)(r)} \, dr,$$
parametrized by $s\in[0,1)$. We can even extend to $s=1$ by defining $\mu_1$ as the Dirac at $r=1$.
It easily follows that
\begin{equation*}
\sup_{a\in\overline{\mathcal{U}}_L}J(a)
=
\sup_{a\in \overline{\mathcal{U}}_L} \inf_{\stackrel{j\in \N, k\in\N^*}{m\in\{1,2\}}} \int_0^{2\pi}\int_0^1a(r,\theta) \phi_{jkm}(r,\theta)^2\, r drd\theta 
\leq 
\sup_{a\in \overline{\mathcal{U}}_L} K(a) ,
\end{equation*}
where
$$
K(a) = \inf_{s\in[0,1]} \int_0^1 \int_0^{2\pi} a(r,\theta)\, d\theta\, f_s(r)\, dr.
$$

\begin{lemma}\label{lem28aout}
There holds $\displaystyle \sup_{a\in \overline{\mathcal{U}}_L} K(a) = L$, and the supremum is reached with the constant function $a=L$ on $\Omega$.
\end{lemma}

\begin{proof}[Proof of Lemma \ref{lem28aout}]
First, note that $K(a=L)=L$ and that the infimum in the definition of $K$ is then reached for every $s\in[0,1]$. Since $K$ is concave (as infimum of linear functions), it suffices to prove that $\langle DK(a=L),h\rangle\leq 0$ (directional derivative), for every function $h$ defined on $\Omega$ such that $\int_\Omega h(x)\, dx=0$. Using Danskin's Theorem (see \cite{Danskin,bernhard}), we have
$$
\langle DK(a=L),h\rangle = \inf_{s\in[0,1]} \int_0^1 \int_0^{2\pi} h(r,\theta)\, d\theta\, f_s(r)\, dr.
$$
By contradiction, let us assume that there exists a function $h$ on $\Omega$ such that $\int_\Omega h(x)\, dx=0$ and such that
$$
\int_0^1 \int_0^{2\pi} h(r,\theta)\, d\theta\, f_s(r)\, dr >0
$$
for every $s\in[0,1]$. Then, it follows that
$$
\int_s^1 \int_0^{2\pi} h(r,\theta)\, d\theta\, \frac{ r } { \sqrt{ r^2 - s^2}} \, dr >0
$$
for every $s\in[0,1]$, and integrating in $s$ over $[0,1]$, we get
\begin{equation*}
\begin{split}
0<\int_0^1 \int_s^1 \int_0^{2\pi} h(r,\theta)\, d\theta\,\frac{ r } { \sqrt{ r^2 - s^2}} \, dr ds
&= \int_0^1 \int_0^r \frac{ r } { \sqrt{ r^2 - s^2}}\, ds \int_0^{2\pi} h(r,\theta)\, d\theta \, dr \\
&= \frac{\pi}{2} \int_0^1 r \int_0^{2\pi} h(r,\theta)\, d\theta \, dr \\
&= \frac{\pi}{2} \int_\Omega h(x)\, dx=0,
\end{split}
\end{equation*}
which is a contradiction. The lemma is proved.
\end{proof}
It follows from this lemma that $\sup_{a\in\overline{\mathcal{U}}_L}J(a)=L$ (note that $a=L$ realizes the maximum), and hence, $\sup_{\chi_\omega\in\mathcal{U}_L}J(\chi_\omega)\leq L$.
To prove the no-gap statement, we use particular (radial) subsets $\omega$, of the form
$$\omega = \{ (r,\theta)\in[0,1]\times[0,2\pi] \ \vert\ \theta\in\omega_\theta \},$$
where $\vert\omega_\theta\vert=2L\pi$, as drawn on Figure \ref{figminseqdisk}.

\begin{figure}[h!]
\begin{center}
\includegraphics[width=16cm]{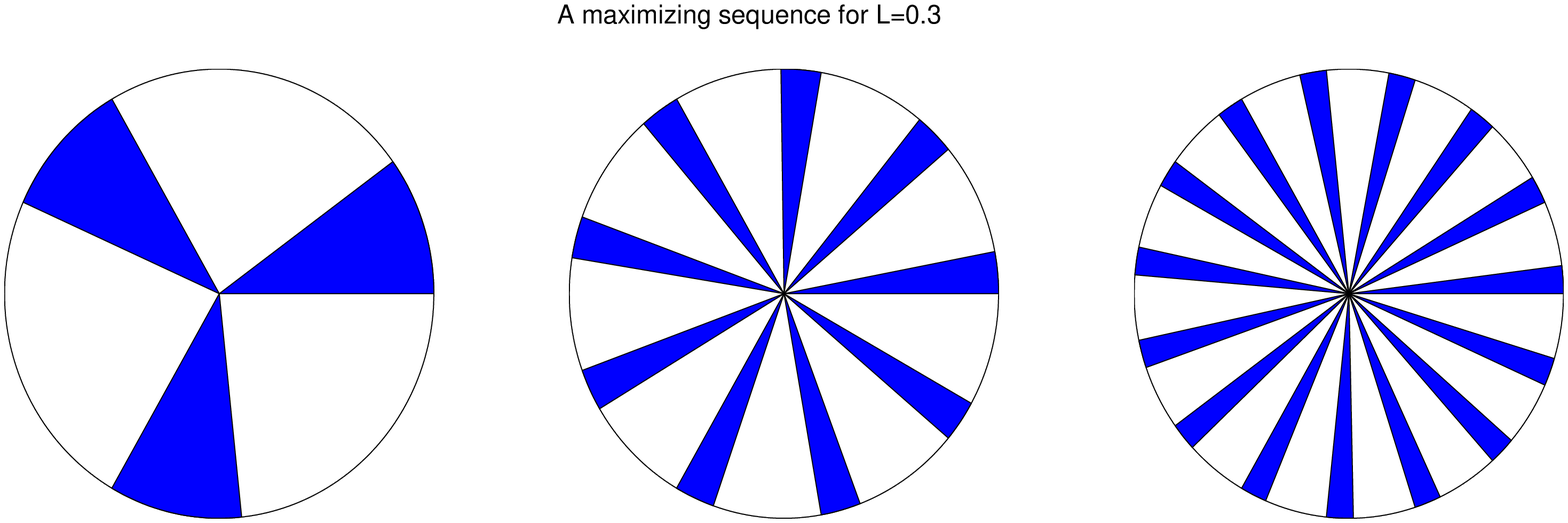}
\caption{Particular radial subsets}\label{figminseqdisk}
\end{center}
\end{figure}

For such a subset $\omega$, one has
$$
\int_\omega \phi_{jkm}(x)^2\, dx = \int_0^1 R_{jk}(r)^2 r\, dr \int_{\omega_\theta}  Y_{jm}(\theta)^2 \, d\theta =  \int_{\omega_\theta}  Y_{jm}(\theta)^2 \, d\theta,
$$
for all $j\in\N^*$, $k\in\N^*$ and $m=1,2$. For $j=0$, there holds
$$
\int_\omega \phi_{0km}(x)^2\, dx =\int_0^1 R_{jk}(r)^2 r\, dr \int_{\omega_\theta}   \, d\theta = \vert\omega_\theta\vert.
$$
Besides, since
$$
L\pi=\vert\Omega\vert= \int_0^1 r\, dr\int_{\omega_\theta}d\theta = \frac{1}{2}\vert\omega_\theta\vert,
$$
it follows that $\vert\omega_\theta\vert=2L\pi$.
By applying the no-gap result in dimension one (clearly, it can be applied as well with the cosine functions), one has
$$
\sup_{\stackrel{\omega_\theta\subset[0,2\pi]}{\vert\omega_\theta\vert=2L\pi}}  \inf_{j\in\N^*} \int_{\omega_\theta} \sin^2(j\theta) \, d\theta = \sup_{\stackrel{\omega_\theta\subset[0,2\pi]}{\vert\omega_\theta\vert=2L\pi}}  \inf_{j\in\N^*} \int_{\omega_\theta} \cos^2(j\theta) \, d\theta = L\pi .
$$
Therefore, we deduce that
$$
\sup_{\chi_\omega\in\mathcal{U}_L} \inf_{\stackrel{j\in \N, k\in\N^*}{m\in\{1,2\}}} \int_\omega \phi_{jkm}(x)^2\, dx =L,
$$
and the conclusion follows.

\subsection{An intrinsic spectral variant of the second problem}\label{sec6.4}
The second problem \eqref{defJ}, defined in Section \ref{secobservabilityintro}, depends a priori on the orthonormal Hilbertian basis $(\phi_j)_{j\in\N^*}$ of $L^2(\Omega)$ under consideration, at least whenever the spectrum of $A$ is not simple.
In this section we assume that the eigenvalues $(\lambda^2_j)_{j\in\N^*}$ of $A$ are multiple, so that the choice of the basis $(\phi_j)_{j\in\N^*}$ enters into play.

We have already seen in Theorem \ref{theoCTT} (see Section \ref{secMotiv}) that, in the case of multiple eigenvalues, the spectral expression for the time-asymptotic observability constant is more intricate and it does not seem that our analysis can be adapted in an easy way to that case.

Besides, recall that the criterion $J$ defined by \eqref{defJ} has been motivated in Section \ref{secMotiv} by means of randomizing initial data, and has been interpreted as a randomized observability constant (see Theorem \ref{propHazardCst}), but then this criterion depends a priori on the preliminary choice of the basis $(\phi_j)_{j\in\N^*}$ of eigenfunctions.

In order to get rid of this dependence, and to deal with a more intrinsic criterion, it makes sense to consider the infimum of the criteria $J$ defined by \eqref{defJ} over all possible choices of orthonormal bases of eigenfunctions.
This leads us to consider the following intrinsic variant of our second problem.

\begin{quote}
\noindent{\bf Intrinsic uniform optimal design problem.}
\textit{We investigate the problem of maximizing the functional
\begin{equation}\label{defJint}
J_\textrm{int}(\chi_\omega)=\inf_{\phi\in \mathcal{E}}\int_\omega \phi(x)^2 \, dV_g,
\end{equation}
over all possible subsets $\omega$ of $\Omega$ of measure $ V_g(\omega) =L V_g(\Omega) $, where $\mathcal{E}$ denotes the set of all normalized eigenfunctions of $A$.
}
\end{quote}

Here, the word intrinsic means that this problem does not depend on the choice of the basis of eigenfunctions of $A$.

As in Theorem \ref{propHazardCst}, the quantity $\frac{T}{2}J_\textrm{int}(\chi_\omega)$ (resp., $TJ_\textrm{int}(\chi_\omega)$) can be interpreted as a constant for which the randomized observability inequality \eqref{ineqobswRand} for the wave equation (resp.,  \eqref{ineqobssRand} for the Schr\"odinger equation) holds, but this constant is less than or equal to $C_{T,\textrm{rand}}^{(W)}(\chi_\omega)$ (resp., $C_{T,\textrm{rand}}^{(S)}(\chi_\omega)$).
Besides, there obviously holds
$C_{T}^{(W)}(\chi_\omega) \leq \frac{T}{2}J_\textrm{int}(\chi_\omega)$ and
$C_{T}^{(S)}(\chi_\omega) \leq T J_\textrm{int}(\chi_\omega)$. Indeed this inequality follows form the deterministic observability inequality applied to the particular solution $y(t,x) = \mathrm{e}^{i\lambda t}\phi(x)$, for every eigenfunction $\phi$.
In brief, there holds
$$
C_{T}^{(W)}(\chi_\omega) \leq \frac{T}{2}J_\textrm{int}(\chi_\omega) \leq C_{T,\textrm{rand}}^{(W)}(\chi_\omega),\quad\textrm{and}\quad
C_{T}^{(S)}(\chi_\omega) \leq T J_\textrm{int}(\chi_\omega) \leq C_{T,\textrm{rand}}^{(S)}(\chi_\omega).
$$
As in Section \ref{solvingpb2obssec1}, the convexified version of the above problem consists of maximizing the functional
$$
J_\textrm{int}(a)=\inf_{\phi\in \mathcal{E}}\int_\Omega a(x) \phi(x)^2 \, dV_g,
$$
over the set $\overline{\mathcal{U}}_L$. This problem obviously has at least one solution, and
$$
\sup_{\chi_\omega\in \mathcal{U}_L} \inf_{\phi\in \mathcal{E}}\int_{\Omega}\chi_\omega(x) \phi(x)^2\, dV_g\ \leq\ \sup_{a\in \overline{\mathcal{U}}_L} \inf_{\phi\in \mathcal{E}}\int_{\Omega}a(x) \phi(x)^2\, dV_g.
$$

\begin{theorem}\label{thmnogap_intrinsic}
Assume that the uniform measure $\frac{1}{ V_g(\Omega) }\, dV_{g}$ is a closure point of the family of probability measures $\mu_{\phi}=\phi^2\, dV_g$, $\phi\in \mathcal{E}$, for the vague  topology, and that the whole family of eigenfunctions in $ \mathcal{E}$ is uniformly bounded in $L^\infty(\Omega)$.
Then
\begin{equation}\label{eqNoGap_intrinsic}
\sup_{\chi_\omega\in \mathcal{U}_L} \inf_{\phi\in \mathcal{E}}\int_\omega \phi(x)^2\, dV_g  = \sup_{a\in \overline{\mathcal{U}}_L} \inf_{\phi\in \mathcal{E}}\int_{\Omega}a(x) \phi(x)^2\, dV_g = L,
\end{equation}
for every $L\in (0,1)$.
In other words, 
there is no gap between the intrinsic uniform  optimal design problem and its convexified version.
\end{theorem}

\begin{proof}
The proof follows the same lines as in Section \ref{secproofthmnogap}, by considering the function $f$ defined by
$f(x) = (\phi(x)^2)_{\phi\in\mathcal{E}}$. Then $f\in L^1(\Omega,X)$ with $X=L^\infty(\mathcal{E},\R)$ which is a Banach manifold that can be seen as an infinite product of spheres of dimension equal to the respective multiplicities of the eigenvalues.
\end{proof}

Similarly, the intrinsic counterpart of Theorem \ref{thmnogap2} is the following. 

\begin{theorem}\label{thmnogap_intrinsic2}
Assume that the uniform measure $\frac{1}{ V_g(\Omega) }\, dV_{g}$ is the unique closure point of the family of probability measures $\mu_{\phi}=\phi^2\, dV_g$, $\phi\in \mathcal{E}$, for the vague topology, and that the whole family of eigenfunctions in $ \mathcal{E}$ is uniformly bounded in $L^{2p}(\Omega)$, for some $p\in(1,+\infty]$.
Then
\begin{equation}\label{eqNoGap_intrinsic2}
\sup_{\chi_\omega\in \mathcal{U}_L^b} \inf_{\phi\in \mathcal{E}}\int_\omega \phi(x)^2\, dV_g = L,
\end{equation}
for every $L\in (0,1)$.
\end{theorem}

\begin{proof}
The proof follows the same lines as in Section \ref{secproofthmnogap}, replacing the integer index $j$ with the continuous index $\lambda$ (standing for the eigenvalues of $A$). The only thing that has to be noticed is the derivation of the estimate corresponding to \eqref{in29eps_1}. In Section \ref{secproofthmnogap}, to obtain \eqref{in29eps_1} from \eqref{lowfreq}, we used the fact that only a finite number of terms have to be considered. Now the number of terms is infinite, but however one has to consider all possible normalized eigenfunctions associated with an eigenvalue $\vert\lambda\vert\leq\vert\lambda_0\vert$. Since this set is compact for every $\lambda_0$, there is no difficulty to extend our previous proof.
\end{proof}

With respect to Remark \ref{rem_existencegap}, it is interesting to note that, here, we are able to provide examples where there is a gap between the intrinsic second problem \eqref{defJint} and its convexified version.

\begin{proposition}\label{remgapintrinsic}
In any of the two following examples:
\begin{itemize}
\item $\Omega=S^2$, the unit sphere in $\R^3$, endowed with the usual flat metric;
\item $\Omega$ is the unit half-sphere in $\R^3$, endowed with the usual flat metric, and Dirichlet conditions are imposed on the great circle which is the boundary of $\Omega$;
\end{itemize}
if $L$ is close enough to $1$ then $\sup_{\chi_\omega\in\mathcal{U}_L} J(\chi_\omega)<L$, and hence there is a gap between the problem \eqref{defJint} and its convexified version.
\end{proposition}

\begin{proof}
Assume first that $\Omega=S^2$, the unit sphere in $\R^3$.
As mentioned in Remark \ref{rem_existencegap}, in \cite{JakobsonZelditch} it is proved  that the set of semi-classical measures on $S^2$ coincides with the convex set of invariant probability measures for the geodesic flow that are time-reversal invariant. In particular, the Dirac measure $\mu_\gamma$ of any great circle $\gamma$ on $S^2$ (defined as an equator, up to a rotation) is the projection of a semi-classical measure. The measure $\mu_\gamma$ is the arc-length measure defined by
$$
\mu_\gamma(\omega)=\frac{1}{2\pi}\int_{\gamma\cap\omega}ds = \frac{1}{2\pi}\vert\gamma\cap\omega\vert,
$$
for every measurable subset $\omega$ of $S^2$. 
Besides, since the uniform measure is a quantum limit as well, $S^2$ satisfies WQE and hence $\sup_{a\in\overline{\mathcal{U}}_L} J(a)=L$ (and the supremum is reached with the constant function $a=L$).
Denoting by $\sigma$ the Lebesgue measure of $S^2$, $\mathcal{U}_L$ is the set of all measurable subsets $\omega$ of $S^2$ of measure $\sigma(L)=4\pi L$. 
For every $\omega\in\mathcal{U}_L$, one has
$$
4\pi L = \int_0^{2\pi} \int_0^\pi \chi_\omega(\varphi,\theta)\sin\varphi\,d\varphi d\theta
\geq \sin\varepsilon\int_0^{2\pi} \int_\varepsilon^{\pi-\varepsilon} \chi_\omega(\varphi,\theta)\, d\varphi d\theta \geq \sin\varepsilon \int_0^{2\pi}( \vert\gamma_\theta\cap\omega\vert-2\varepsilon) d\theta,
$$
for every $\varepsilon\in[0,\pi/2]$, where $\gamma_\theta$ denotes the great circle joining the north pole to the south pole at longitude $\theta$ (where a north pole is fixed arbitrarily). By contradiction, assume that $\mu_{\gamma_\theta}(\omega)>3L/4$ for every $\theta\in[0,2\pi]$. Then we infer that
$4\pi L>2\pi\sin\varepsilon(3\pi L/2-2\varepsilon)$, which raises a contradiction when choosing e.g. $\varepsilon=\pi/4$ and $L$ close to $1$.
It then follows that
$$
J(\chi_\omega)=\inf_{j\in\N^*}\int_\omega\phi_j^2\leq \inf_{\theta\in[0,2\pi]} \mu_{\gamma_\theta}(\omega) \leq \frac{3L}{4},
$$
for every $\omega\in\mathcal{U}_L$, whence the gap.

Assume now that $\Omega$ is the unit half-sphere of $\R^3$. As recalled above, for every great circle $C$ of $S^2$ there exists a sequence of squares of eigenfunctions $\phi_j$ whose support concentrates along $C$. Let $\mathcal{S}$ denote the orthogonal symmetry with respect to the hyperplane passing through the origin, cutting $S^2$ into two half-spheres, one of which being $\Omega$. Then, $\psi_j=(\phi_j-\phi_j\circ\mathcal{S})/\sqrt{2}$ is an eigenfunction of the Dirichlet-Laplacian on $\Omega$. Let us prove\footnote{This idea emerged from discussions with Luc Hillairet.} that the support of $\psi_j^2$ concentrates on the union of two symmetric half-circles of $\Omega$, as drawn on Figure \ref{fighalfsphere}.
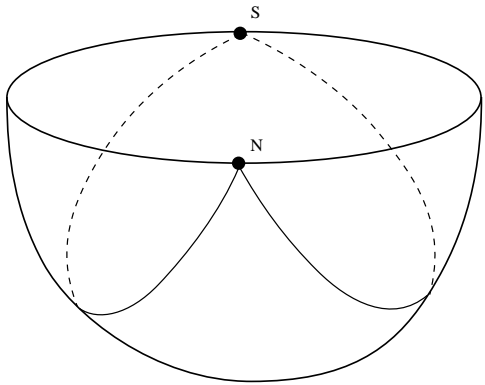
\begin{figure}[H]
\begin{center}
\scalebox{0.55} 
{
\begin{pspicture}(0,-4.568994)(11.52,4.588994)
\psellipse[linewidth=0.04,dimen=outer](5.75,2.3210058)(5.75,1.61)
\psbezier[linewidth=0.04](0.019493513,2.3310058)(0.019493513,0.8423509)(0.22383018,-0.44645917)(0.9,-1.6889942)(1.5761698,-2.931529)(3.6,-4.548994)(5.9242086,-4.548994)(8.248418,-4.548994)(9.34,-3.768994)(10.14,-2.548994)(10.94,-1.3289942)(11.5,0.13100585)(11.489801,2.3310058)
\psdots[dotsize=0.322](5.62,0.73100585)
\psdots[dotsize=0.322](5.66,3.8710058)
\usefont{T1}{ptm}{m}{n}
\rput(6.058662,1.1610059){\large N}
\usefont{T1}{ptm}{m}{n}
\rput(6.036787,4.381006){\large S}
\psbezier[linewidth=0.03](5.62,0.67100585)(6.16,-0.28899413)(6.74,-1.0889941)(7.52,-1.8689941)(8.3,-2.6489942)(9.34,-3.1689942)(10.18,-2.4689941)
\psbezier[linewidth=0.03](5.62,0.61100584)(5.18,-0.42899415)(4.3897667,-1.4649626)(3.68,-2.2089942)(2.9702332,-2.9530256)(2.22,-3.088994)(1.78,-2.788994)
\psbezier[linewidth=0.03,linestyle=dashed,dash=0.16cm 0.16cm](1.72,-2.7289941)(1.22,-1.4489942)(1.5,-0.40899414)(2.24,0.8310059)(2.98,2.0710058)(4.1,3.0910058)(5.6,3.8110058)
\psbezier[linewidth=0.03,linestyle=dashed,dash=0.16cm 0.16cm](10.24,-2.4489942)(10.72,-0.6089941)(9.74,0.55100584)(9.0,1.4710059)(8.26,2.3910058)(7.0,3.3110058)(5.76,3.8510058)
\end{pspicture} 
}
\end{center}
\caption{The half-sphere.}\label{fighalfsphere}
\end{figure}
Indeed, since $\psi_j^2=\frac{1}{2}(\phi_j^2+\vert\phi_j\circ\mathcal{S}\vert^2+\phi_j\cdot\phi_j\circ\mathcal{S})$, it suffices to prove that for every $a\in L^\infty(\Omega)$, $\int_\Omega a(x)\phi_j(x)\phi_j\circ\mathcal{S}(x)\, d\sigma(x)$ tends to $0$ as $j$ tends to $+\infty$. But this fact is obvious since the measure of the intersection of the corresponding supports tends to $0$.
The following (interesting in itself) fact follows: the Dirac measure along every union of symmetric half-circles on $\Omega$ is the projection of a semi-classical measure. Note however that, in this construction, the half-circles passing through the lowest point of the half-sphere cannot be considered.

Then, the same calculation as before can be led. Indeed, let us fix a point $N$ of the boundary of $\Omega$, and let $S$ be the diametrically symmetric point, as on Figure \ref{fighalfsphere}. If we think of $N$ and $S$ as a north pole and south pole, then any curve consisting of the union of two symmetric half-circles emerging from $N$ and $S$ can be viewed, with evident symmetries, as a great circle $\gamma_\theta$ of $S^2$ as considered previously.
Then, the same argument can be applied and leads to the desired conclusion.
\end{proof}

\section{Spectral approximation of the uniform optimal design problem}\label{sectrunc2}
In this section, we consider a spectral truncation of the functional $J$ defined by \eqref{defJ}, and we define
\begin{equation}\label{defJN}
J_N(\chi_\omega)=\min_{1\leq j\leq N}\int_\omega \phi_j(x)^2\, dV_g,
\end{equation}
for every $N\in\N^*$ and every measurable subset $\omega$ of $\Omega$, and we consider the spectral approximation of the second problem (uniform optimal design problem)
\begin{equation}\label{PbHH2}
\sup_{\chi_\omega\in\mathcal{U}_L}J_N(\chi_\omega).
\end{equation}
As before, the functional $J_N$ is naturally extended to $\overline{\mathcal{U}}_L$ by
$$ J_N(a)=\min_{1\leq j\leq N}\int_\Omega a(x)\phi_j(x)^2\, dV_g,$$
for every $a\in\overline{\mathcal{U}}_L$.

\subsection{Existence, uniqueness and $\Gamma$-convergence properties}
One has the following result.
\begin{theorem}\label{thm3}
\begin{enumerate}
\item For every measurable subset $\omega$ of $\Omega$, the sequence $(J_N(\chi_\omega))_{N\in \N^*}$ is nonincreasing and converges to $J(\chi_\omega)$.
\item There holds
$$
\lim_{N\to +\infty}\ \max_{a\in\overline{\mathcal{U}}_L}J_N(a)=\max_{a\in\overline{\mathcal{U}}_L}J(a).
$$
Moreover, if $(a^N)_{n\in\N^*}$ is a sequence of maximizers of $J_N$ in $\overline{\mathcal{U}}_L$, then up to a subsequence, it converges to a maximizer of $J$ in $\overline{\mathcal{U}}_L$ for the weak star topology of $L^\infty$.
\item 
Assume that $M$ is an analytic Riemannian manifold and that $\Omega$ has a nontrivial boundary. Then, for every $N\in\N^*$, the problem \eqref{PbHH2} has a unique solution $\chi_{\omega^N}$, where $\omega^N\in{\mathcal{U}}_L$. Moreover, $\omega^N$ is semi-analytic (see Footnote \ref{footnoteSemiAna}) and has a finite number of connected components.
\end{enumerate}
\end{theorem}

\begin{proof}
For every measurable subset $\omega$ of $\Omega$, the sequence $(J_N(\chi_\omega))_{N\in \N^*}$ is clearly nonincreasing and thus is convergent. Note that
\begin{equation*}
\begin{split}
J_N(\chi_\omega) & =  \inf\left\{\sum_{j=1}^N \alpha_j \int_\omega \phi_j(x)^2\, dV_g\ \Big\vert\ \alpha_j\geq0,\ \sum_{j=1}^N\alpha_j=1\right\},\\
J(\chi_\omega) & =  \inf\left\{\sum_{j\in \N^*} \alpha_j \int_\omega \phi_j(x)^2\, dV_g\ \Big\vert\ \alpha_j\geq0,\ \sum_{j\in\N^*}\alpha_j=1\right\}.
\end{split}
\end{equation*}
Hence, for every $(\alpha_j)_{j\in \N^*}\in \ell^1(\mathbb{R}^+)$, one has
$$
\sum_{j=1}^N\alpha_j \int_\omega \phi_j(x)^2\, dV_g\geq J_N(\chi_\omega)\sum_{j=1}^N\alpha_j,
$$
for every $N\in\N^*$, and letting $N$ tend to $+\infty$ yields
$$
\sum_{j\in\N^*}\alpha_j \int_\omega \phi_j(x)^2\, dV_g\geq \lim_{N\to+\infty}J_N(\omega)\sum_{j\in \N^*}\alpha_j,
$$
and thus $\displaystyle \lim_{N\to +\infty}J_N(\chi_\omega)\leq J(\chi_\omega)$. This proves the first item since there always holds $J_N(\chi_\omega)\geq J(\chi_\omega)$.

Since $J_N$ is upper semi-continuous (and even continuous) for the $L^\infty$ weak star topology and since $\overline{\mathcal{U}}_L$ is compact for this topology, it follows that $J_N$ has at least one maximizer $a^N\in\overline{\mathcal{U}}_L$. Let $\bar a\in \overline{\mathcal{U}}_L$ be a closure point of the sequence $(a^N)_{n\in\N^*}$ in the $L^\infty$ weak star topology. One has, for every $p\leq N$,
$$
\sup_{a\in \overline{\mathcal{U}}_L}J(a)\leq \sup_{a\in\overline{U}_L}J_N(a)=J_N(a^N)\leq J_p(a^N),
$$
and letting $N$ tend to $+\infty$ yields
$$
\sup_{a\in \overline{\mathcal{U}}_L}J(a)\leq \lim_{N\to +\infty}J_N(a^N)\leq \lim_{N\to +\infty}J_p(a^N)=J_p(\bar a),
$$
for every $p\in\N^*$. Since $J_p(\bar a)$ tends to $J(\bar a)\leq \sup_{a\in \overline{\mathcal{U}}_L}J(a)$ as $p$ tends to $+\infty$, it follows that $\bar a$ is a maximizer of $J$ in $\overline{\mathcal{U}_L}$. The second item is proved.

To prove the third item, let us now prove that $J_N$ has a unique maximizer $a^N\in\overline{\mathcal{U}}_L$ of $J_N$, which is moreover a characteristic function. We define the simplex set
$$
\mathcal{A}_N=\{\alpha=(\alpha_j)_{1\leq j\leq N}\ \vert\ \alpha_j\geq0,\ \sum_{j=1}^N\alpha_j=1\}.
$$
Note that
$$
\min_{1\leq j\leq N}\int_{\Omega}a(x) \phi_j(x)^2\, dV_g  = \min_{\alpha\in\mathcal{A}_N}\int_{\Omega}a(x) \sum_{j=1}^N\alpha_j\phi_j(x)^2\, dV_g,
$$
for every $a\in\overline{\mathcal{U}}_L$.
It follows from Sion's minimax theorem (see \cite{Sion}) that there exists $\alpha^N\in \mathcal{A}_N$ such that $(a^N,\alpha^N)$ is a saddle point of the bilinear functional
$$
(a,\alpha)\mapsto \int_{\Omega}a(x) \sum_{j=1}^N\alpha_j\phi_j(x)^2\, dV_g
$$
defined on $\overline{\mathcal{U}}_L\times\mathcal{A}_N$, and
\begin{equation}\label{lasteq17h31}
\begin{split}
& \max_{a\in \overline{\mathcal{U}}_L} \min_{\alpha\in\mathcal{A}_N}\int_{\Omega}a(x) \sum_{j=1}^N\alpha_j\phi_j(x)^2\, dV_g
  =   \min_{\alpha\in\mathcal{A}_N}\max_{a\in \overline{\mathcal{U}}_L}\int_{\Omega}a(x) \sum_{j=1}^N\alpha_j\phi_j(x)^2\, dV_g\\
& =\max_{a\in \overline{\mathcal{U}}_L}\int_{\Omega}a(x) \sum_{j=1}^N\alpha_j^N\phi_j(x)^2\, dV_g
 = \int_{\Omega}a^N(x) \sum_{j=1}^N\alpha_j^N\phi_j(x)^2\, dV_g.
\end{split}
\end{equation}
We claim that the function $x\mapsto \sum_{j=1}^N\alpha_j^N\phi_j(x)^2$ is never constant on any subset of positive measure. 
This fact is proved by contradiction. Indeed otherwise this function would be constant on $\Omega$ (by analyticity).
At this step we have to distinguish between the different boundary conditions under consideration. For Neumann boundary conditions, we infer that $\triangle_g(\sum_{j=1}^N\alpha_j^N\phi_j(x)^2)=0$ on $\partial\Omega$ (by continuity), and therefore $\sum_{j=1}^N\alpha_j^N\lambda_j\phi_j(x)^2=0$ on $\partial\Omega$, whence the contradiction since the coefficients $\alpha_{j}^N$ are nonnegative for every $j\in\{1,\cdots,N\}$ and $\sum_{j=1}^N\alpha_j^N=1$. For the other boundary conditions, we infer that the function $x\mapsto \sum_{j=1}^N\alpha_j^N\phi_j(x)^2$ vanishes on $\bar \Omega$, which is a contradiction.

It follows from this fact and from \eqref{lasteq17h31} that there exists $\lambda^N> 0$ such that
\begin{equation*}
a^N(x) = \left\{ \begin{array}{rl}
1 & \textrm{if}\ \displaystyle\sum_{j=1}^N\alpha_j^N\phi_j(x)^2 \geq \lambda^N , \\
0 & \textrm{otherwise,}
\end{array}\right.
\end{equation*}
for almost every $x\in\Omega$.
Hence there exists $\omega^N\in\mathcal{U}_L$ such that $a^N=\chi_{\omega^N}$.
Since the eigenfunctions $\phi_j$ are analytic in $\Omega$ (by analytic hypoellipticity), it follows that $\omega^N$ is semi-analytic and has a finite number of connected components.
\end{proof}

\begin{remark}
Note that the third item of Theorem \ref{thm3} can be seen as a generalization of \cite[Theorem 3.1]{henrot_hebrardSICON} and \cite[Theorem 3.1]{PriSiga}. We have also provided a shorter proof.
\end{remark}

\begin{remark}
It is proved in \cite{henrot_hebrardSICON,PTZObs1} that, in the one-dimensional case $\Omega=[0,\pi]$ with Dirichlet boundary conditions, the optimal set $\omega_N$ maximizing $J_N$ is the union of $N$ intervals concentrating around equidistant points and that $\omega_N$ is actually the worst possible subset for the problem of maximizing $J_{N+1}$. This is the \textit{spillover phenomenon}.
\end{remark}

\subsection{Numerical simulations: maximizing sequences}\label{sec5.2}
We provide hereafter several numerical simulations based on the modal approximation described previously. 

Assume first that $\Omega=[0,\pi]^2$, the Euclidean two-dimensional square. We consider Dirichlet or Neumann boundary conditions. In the Dirichlet case, the normalized eigenfunctions of $A$ are
$$
\phi_{j,k}(x_{1},x_{2})=\frac{2}{\pi}\sin (jx_{1})\sin (kx_{2}),
$$
for every $(x_{1},x_{2})\in [0,\pi]^{2}$, and in the Neumann case the sine functions are replaced with cosine functions. Let $N\in\N^{*}$. We use an interior point line search filter method to solve the spectral approximation of the second problem
$$
\sup_{\chi_{\omega}\in\mathcal{U}_{L}}J_{N}(\chi_{\omega}),
$$
where
$$
J_{N}(\chi_{\omega})=\min_{1\leq j,k\leq N}\int_{0}^{\pi}\!\int_{0}^{\pi}\chi_{\omega}(x_{1},x_{2})\phi_{j,k}(x_{1},x_{2})^{2}\, dx_{1}\, dx_{2}.
$$
Some results are provided on Figure \ref{figpb2} in the Dirichlet case, and on Figure \ref{figpb2N} in the Neumann case.

\begin{figure}[h!]
\begin{center}
\includegraphics[width=4.9cm]{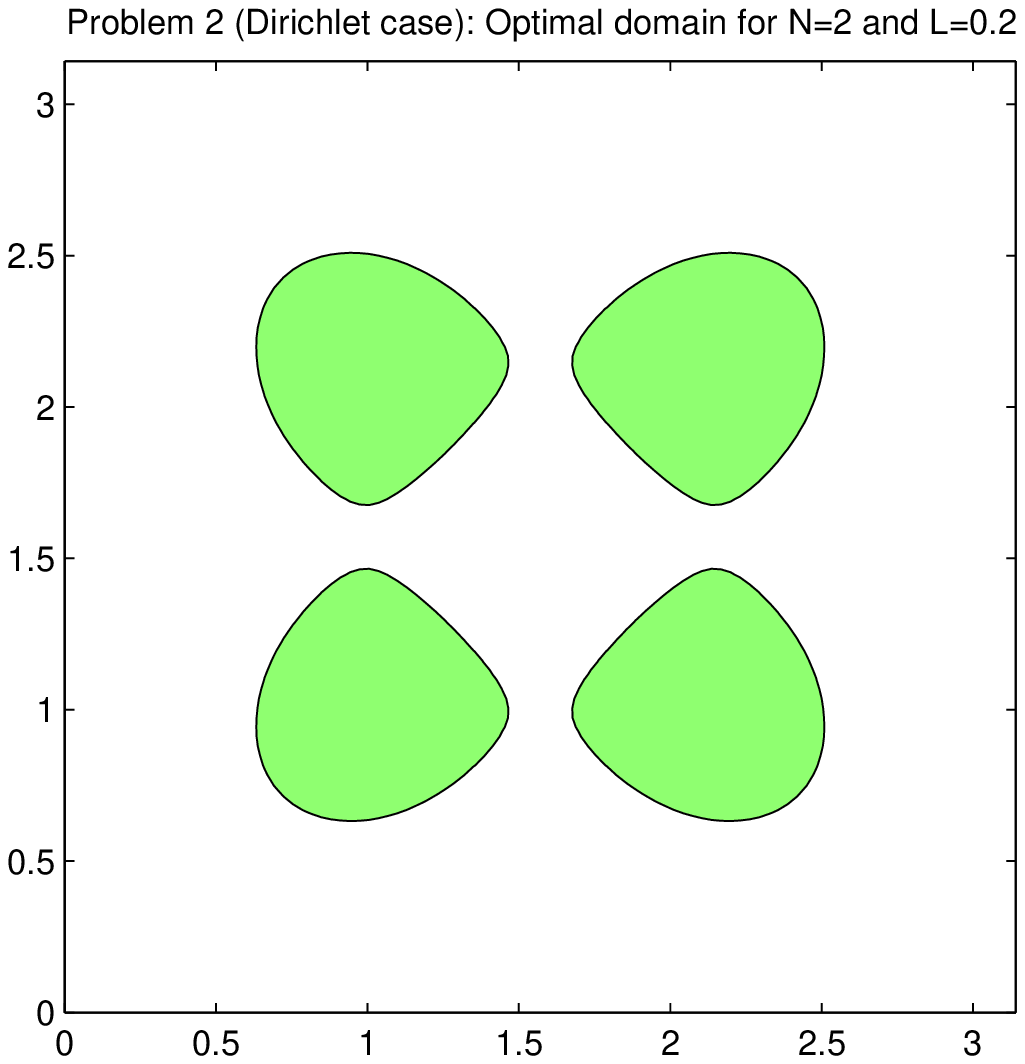}
\includegraphics[width=4.9cm]{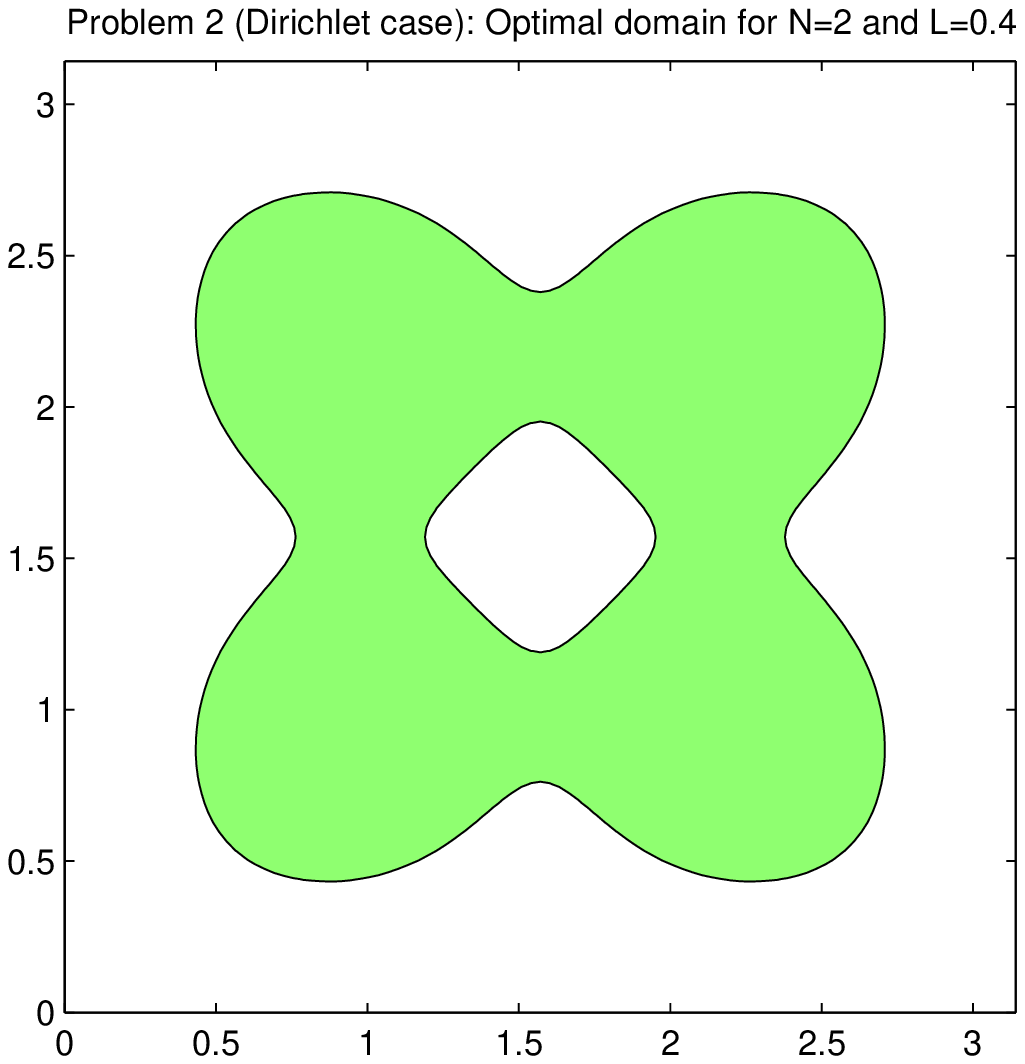}
\includegraphics[width=4.9cm]{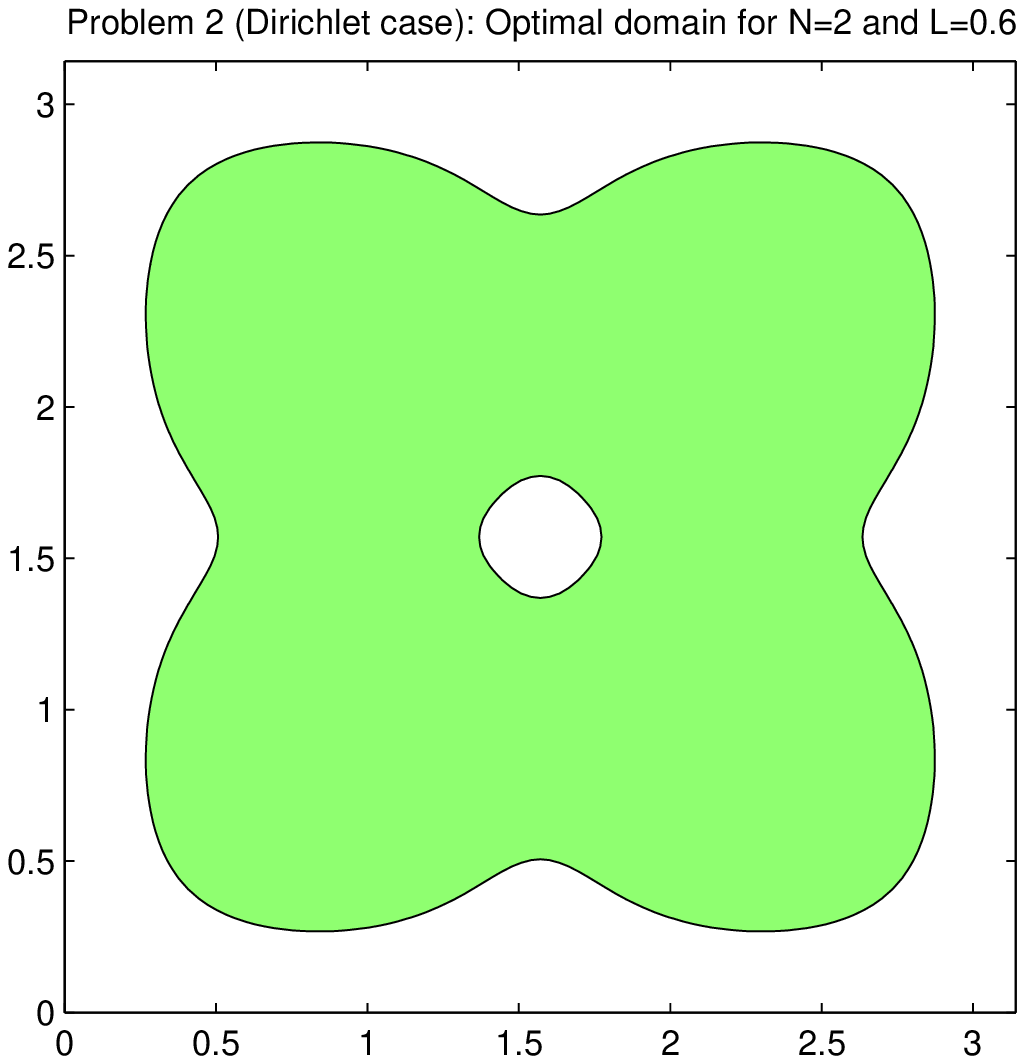}
\includegraphics[width=4.9cm]{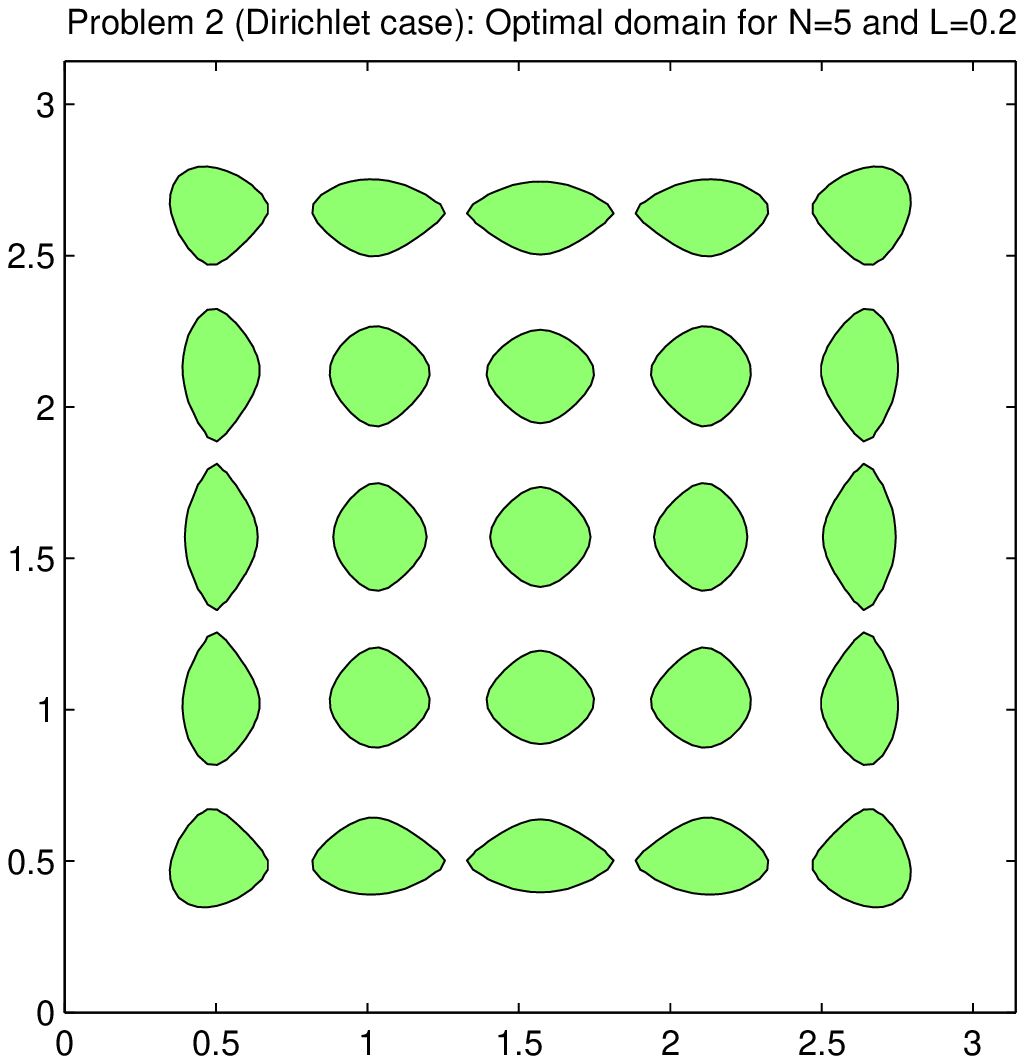}
\includegraphics[width=4.9cm]{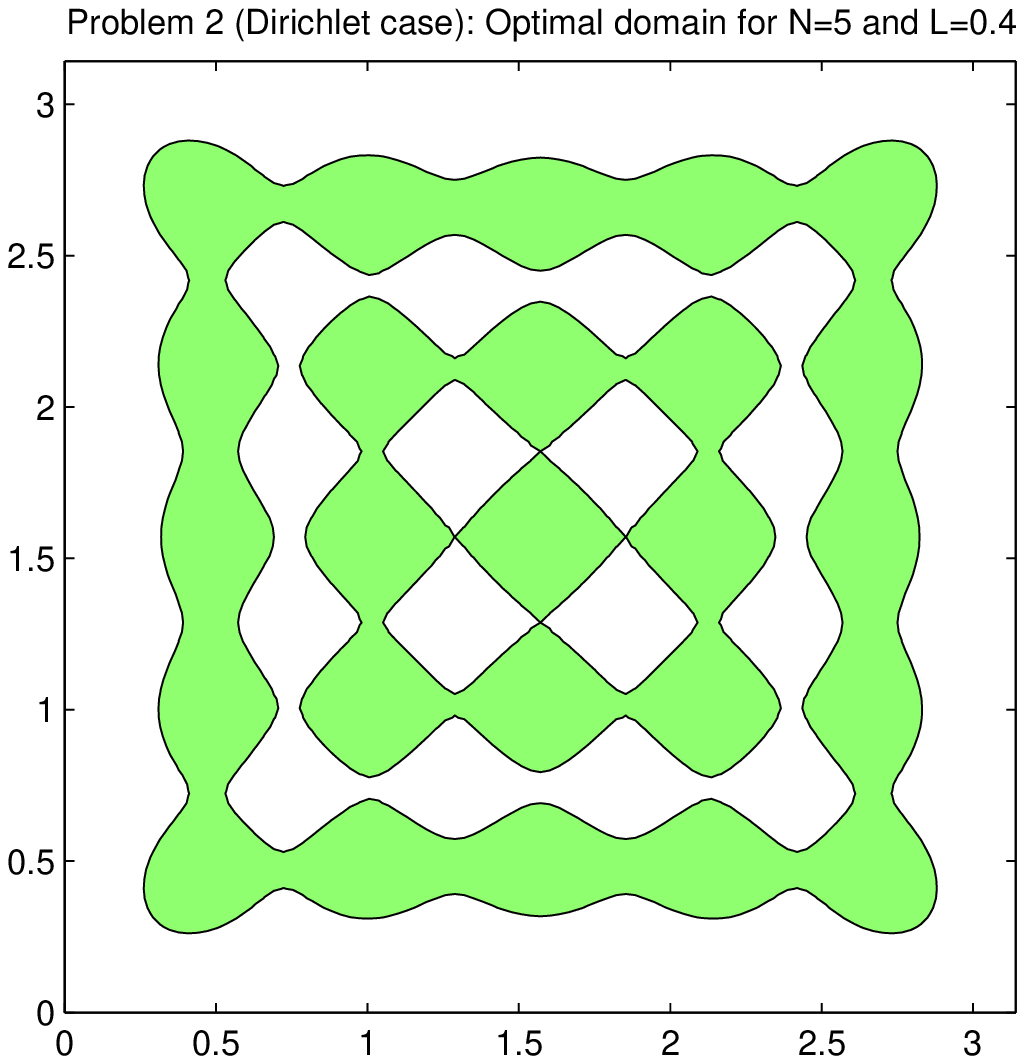}
\includegraphics[width=4.9cm]{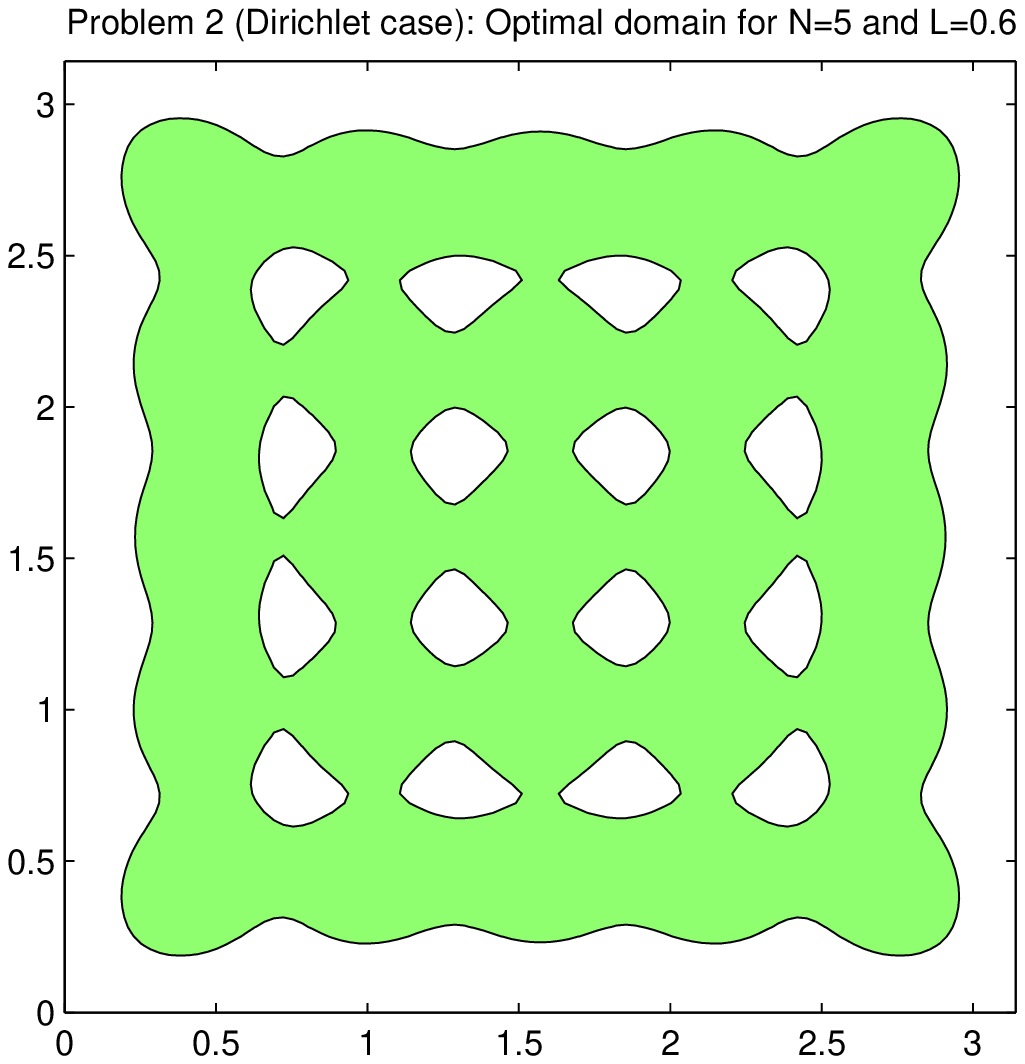}
\includegraphics[width=4.9cm]{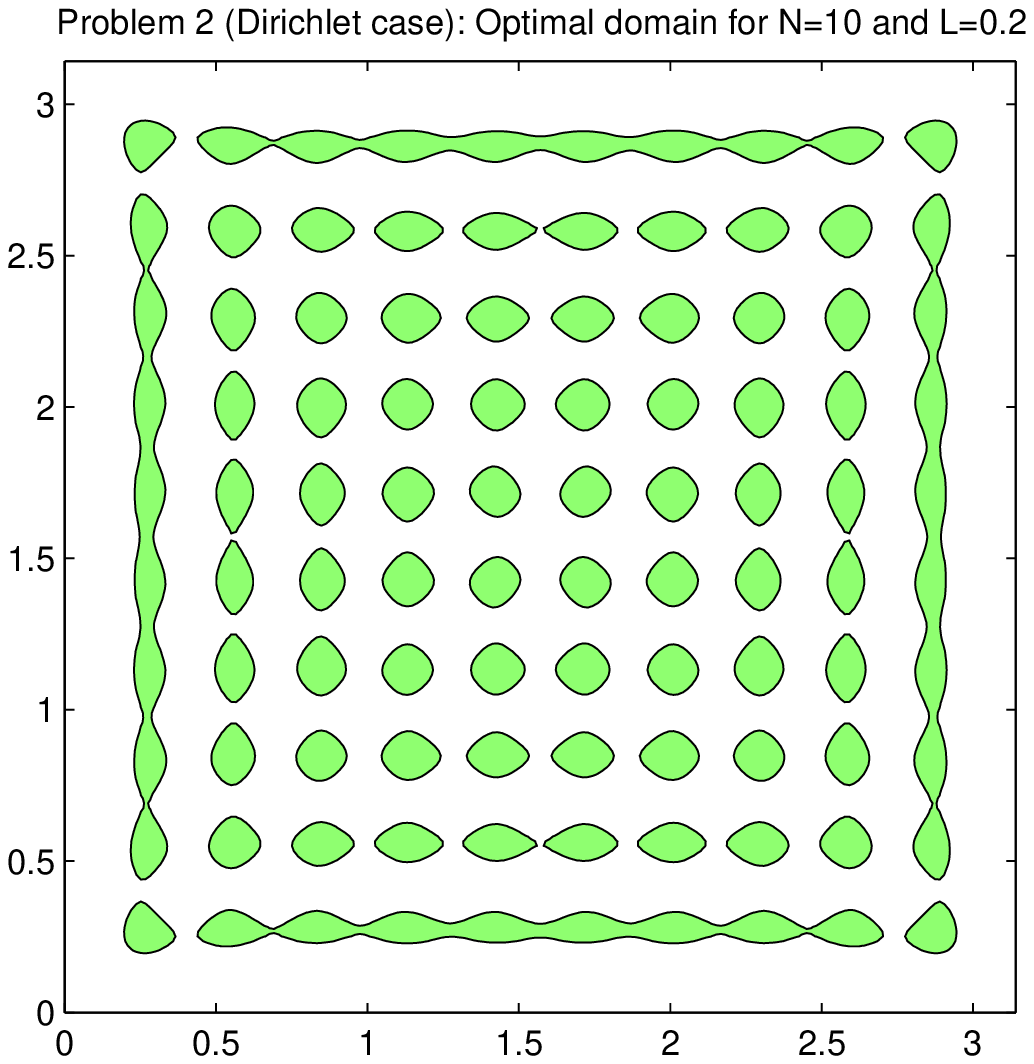}
\includegraphics[width=4.9cm]{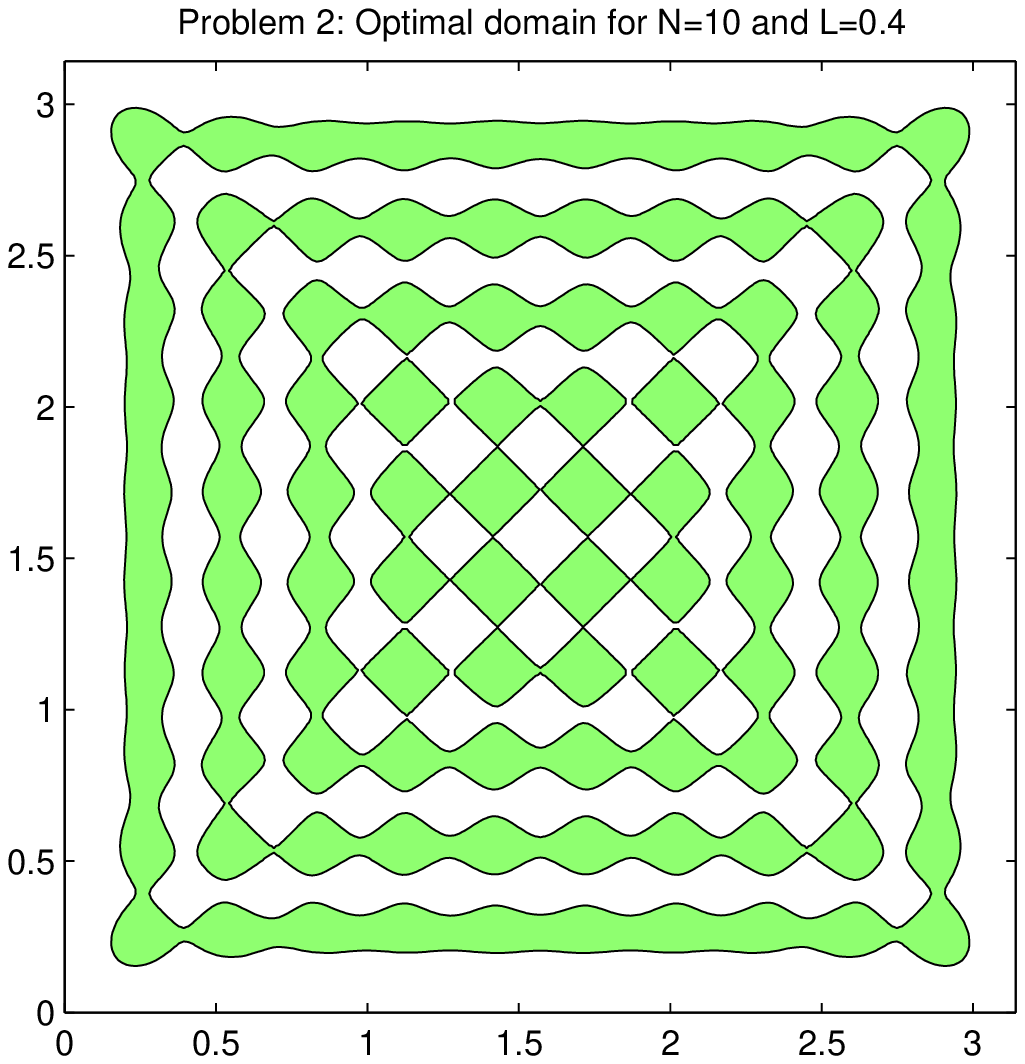}
\includegraphics[width=4.9cm]{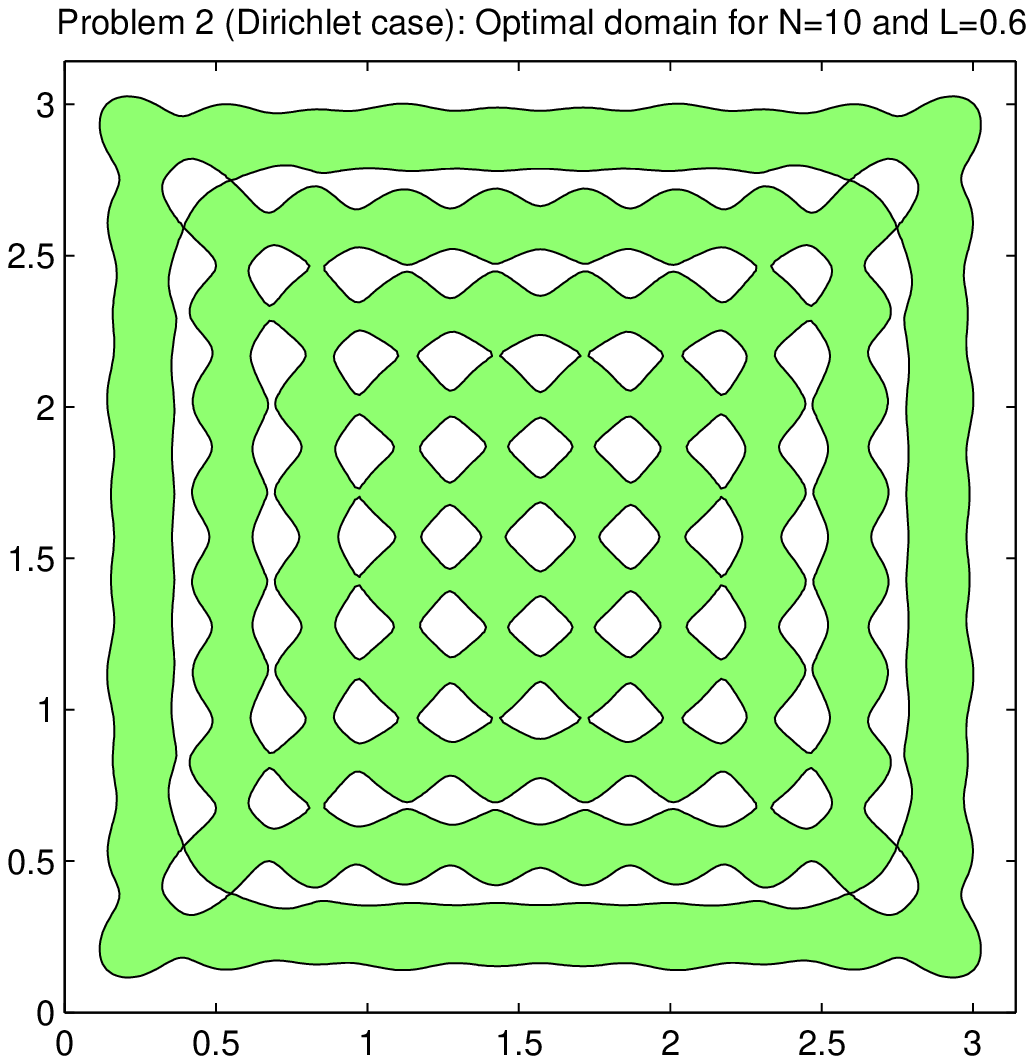}
\includegraphics[width=4.9cm]{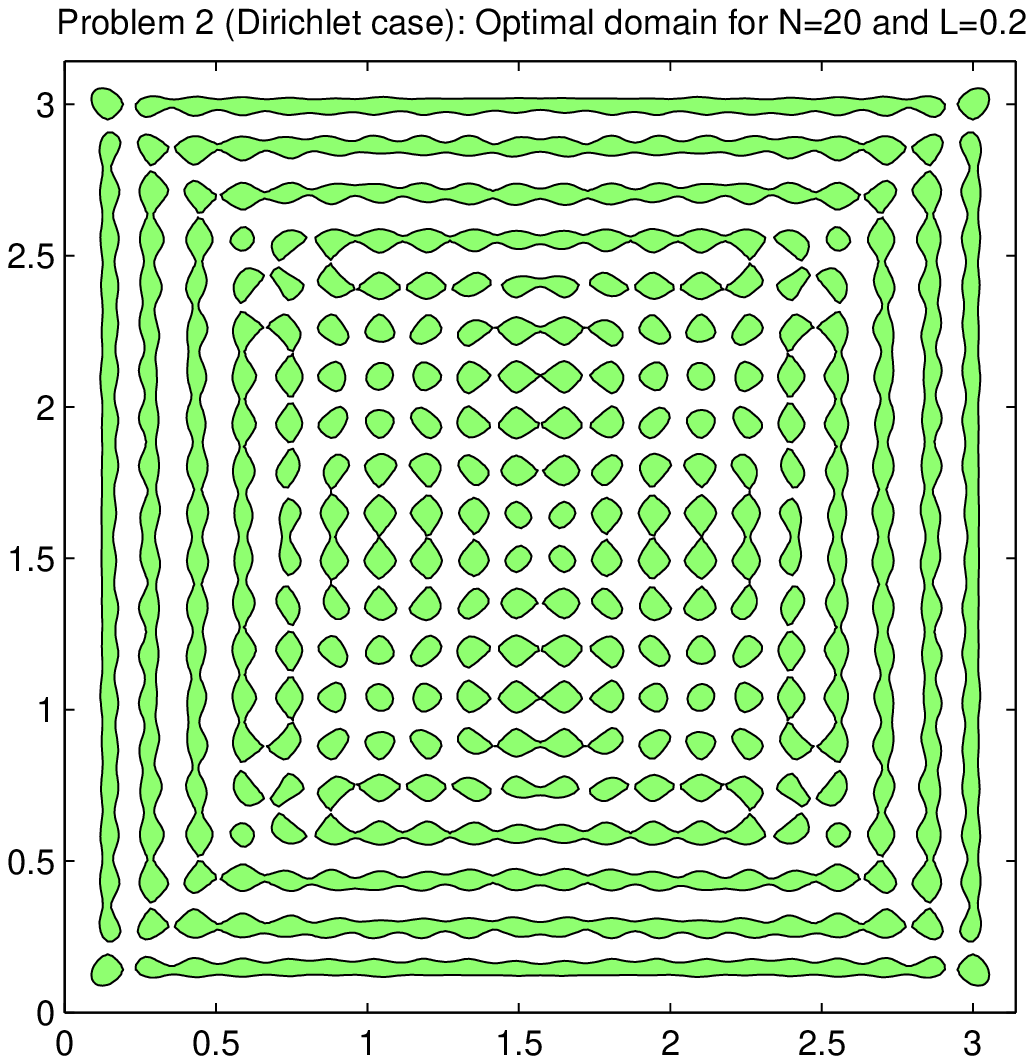}
\includegraphics[width=4.9cm]{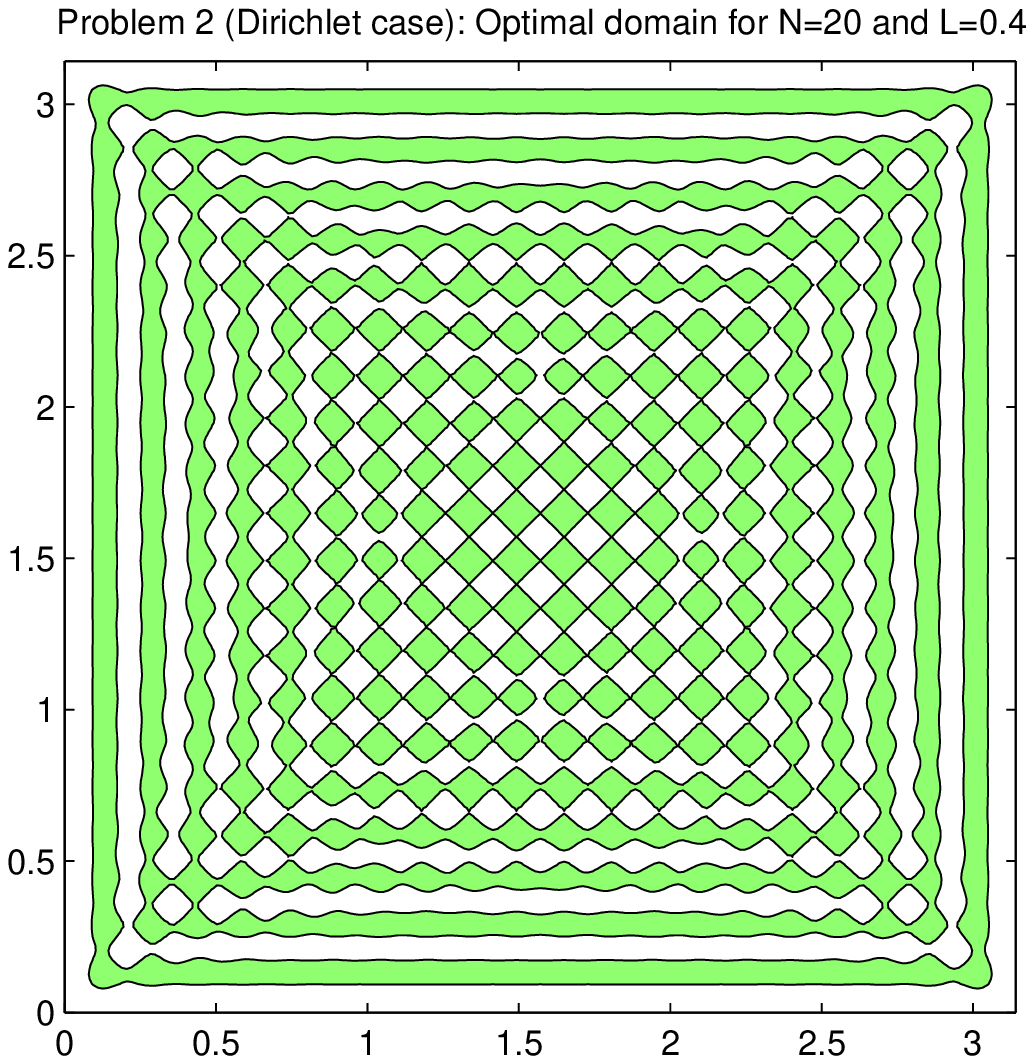}
\includegraphics[width=4.9cm]{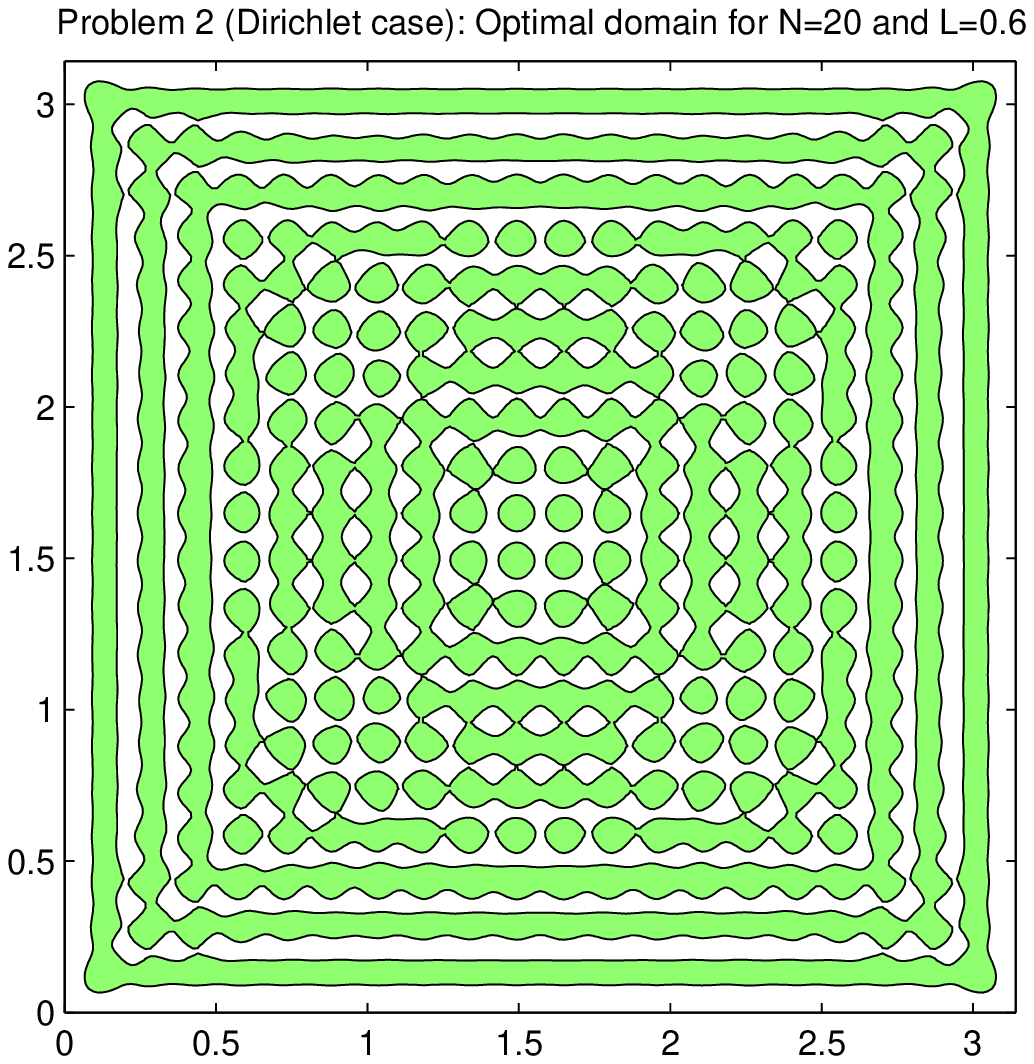}
\caption{On this figure, $\Omega=[0,\pi]^{2}$. Line 1, from left to right: optimal domain (in green) in the Dirichlet case for $N=2$ (4 eigenmodes) and $L\in \{0.2,0.4,0.6\}$. Line 2, from left to right: optimal domain (in green) for $N=5$ (25 eigenmodes) and $L\in \{0.2,0.4,0.6\}$. Line 3, from left to right: optimal domain (in green) for $N=10$ (100 eigenmodes) and $L\in \{0.2,0.4,0.6\}$. Line 4, from left to right: optimal domain (in green) for $N=20$ (400 eigenmodes) and $L\in \{0.2,0.4,0.6\}$}\label{figpb2}
\end{center}
\end{figure}

\begin{figure}[h!]
\begin{center}
\includegraphics[width=4.9cm]{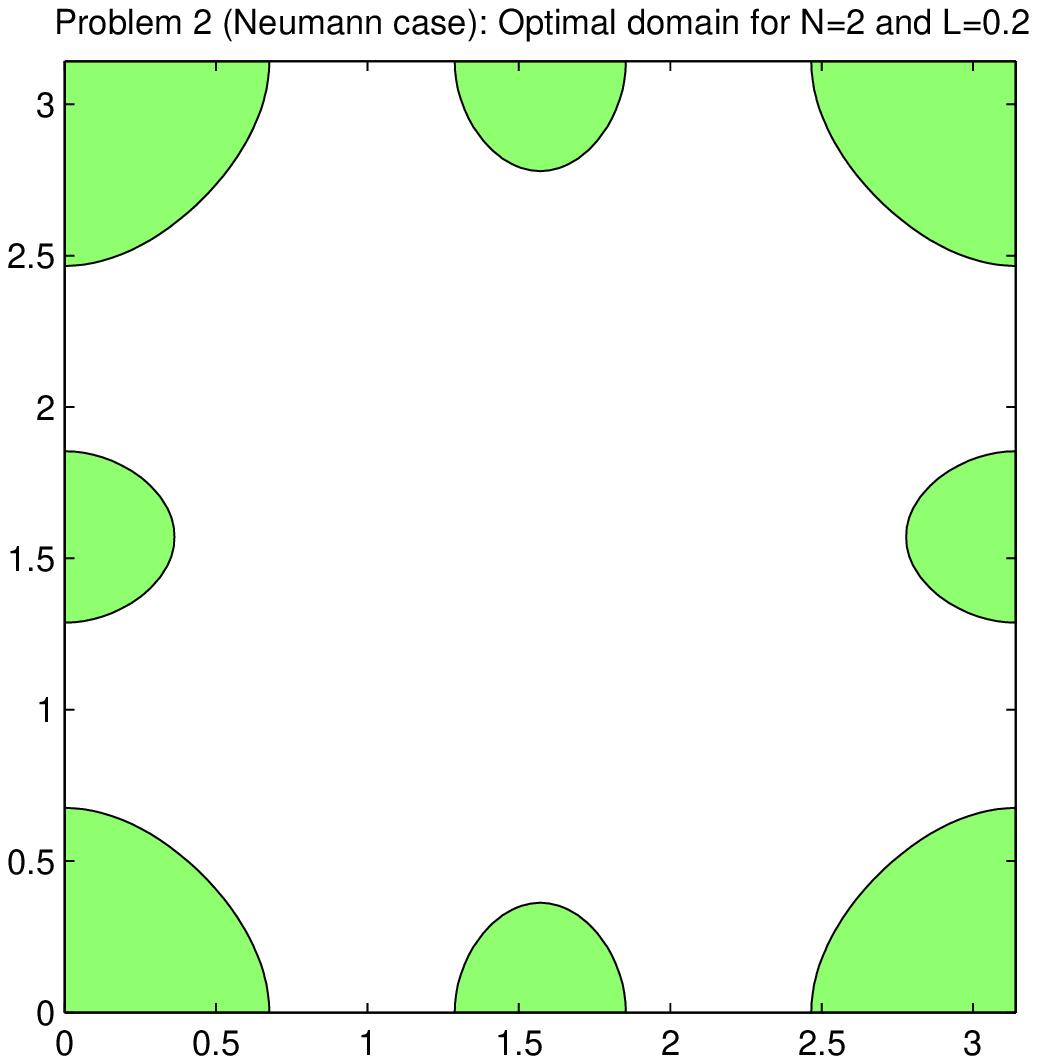}
\includegraphics[width=4.9cm]{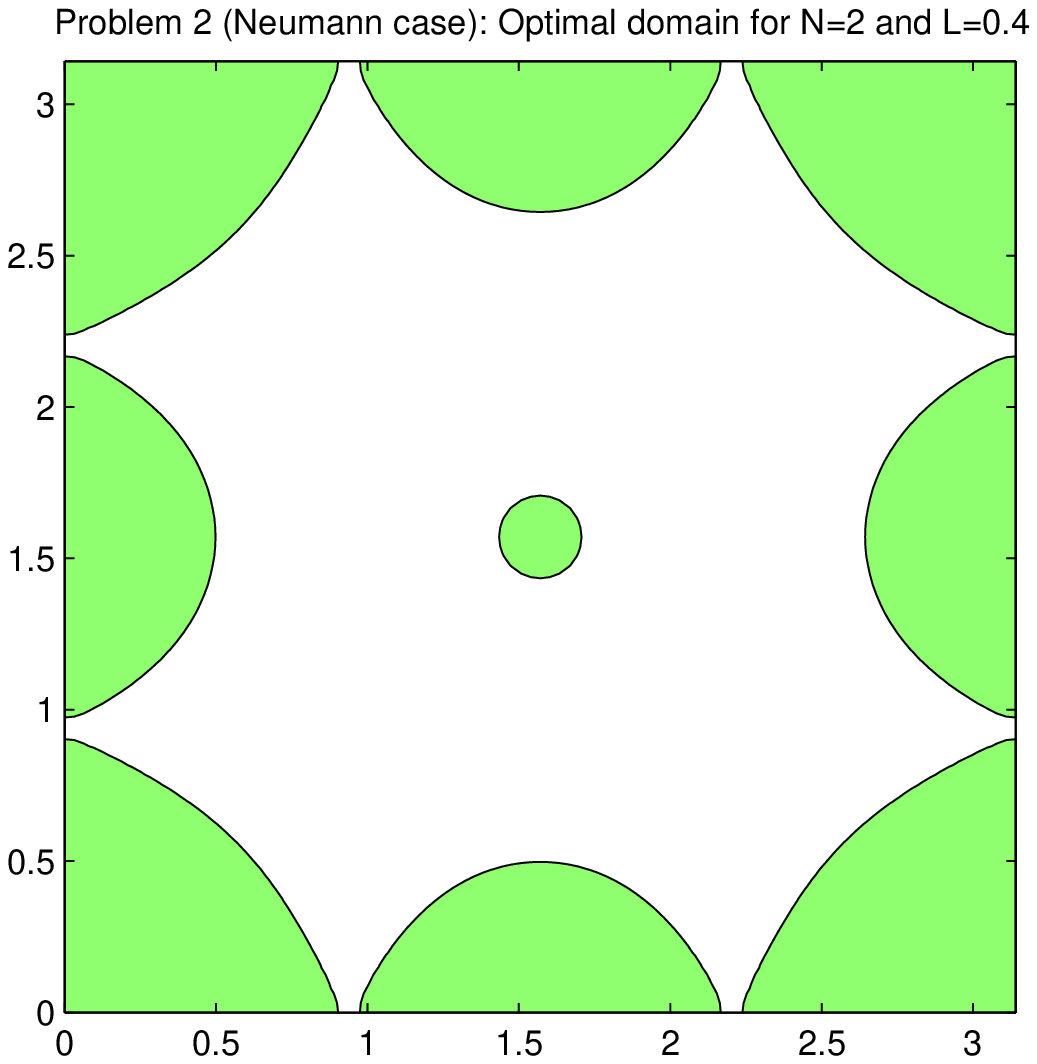}
\includegraphics[width=4.9cm]{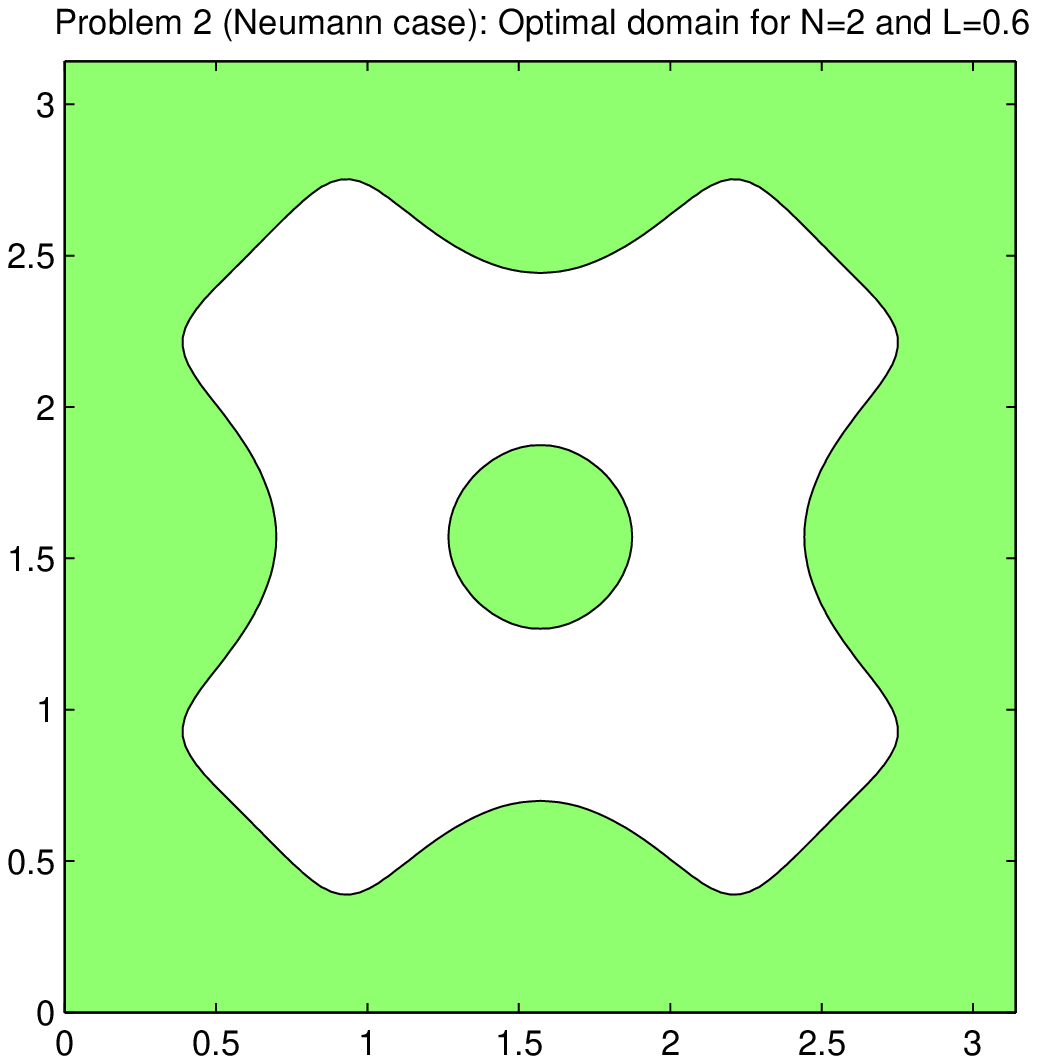}
\includegraphics[width=4.9cm]{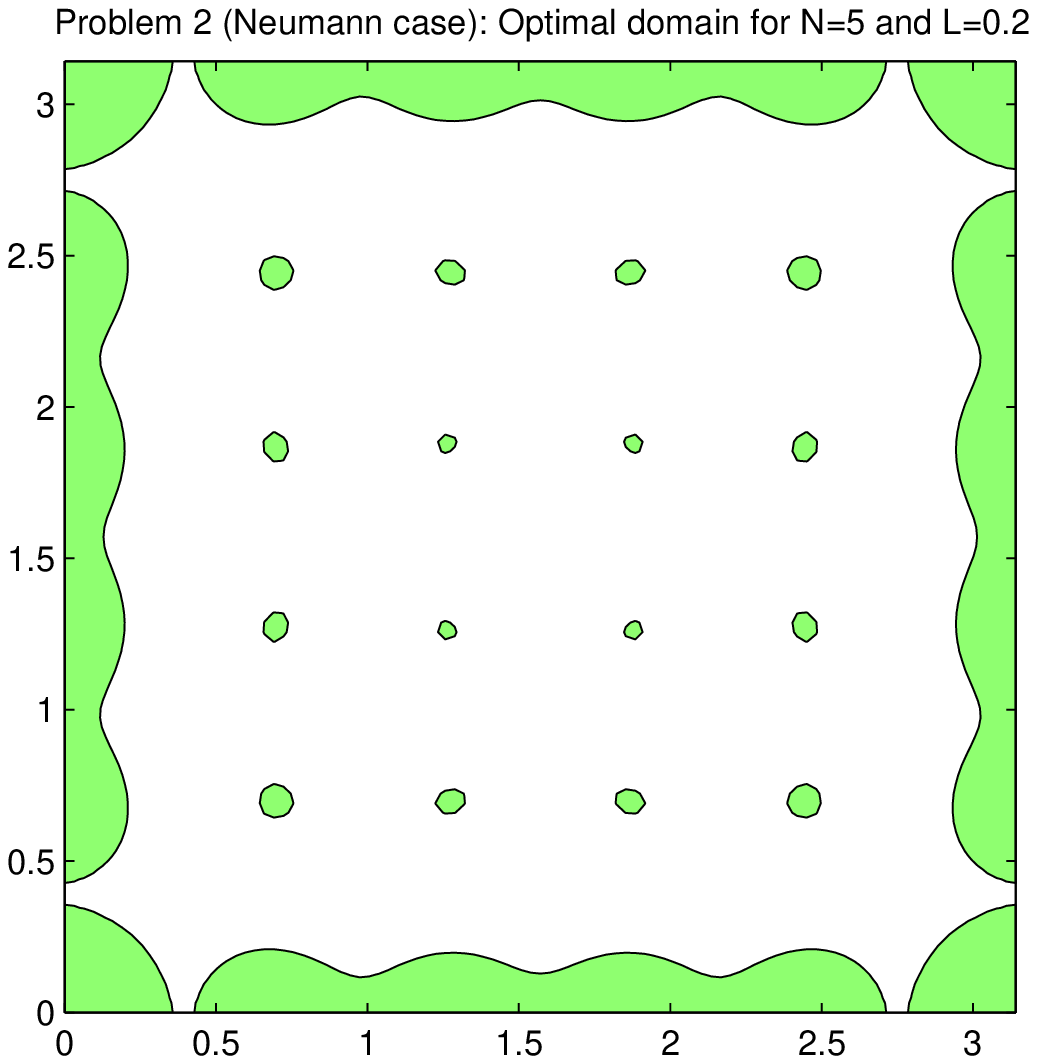}
\includegraphics[width=4.9cm]{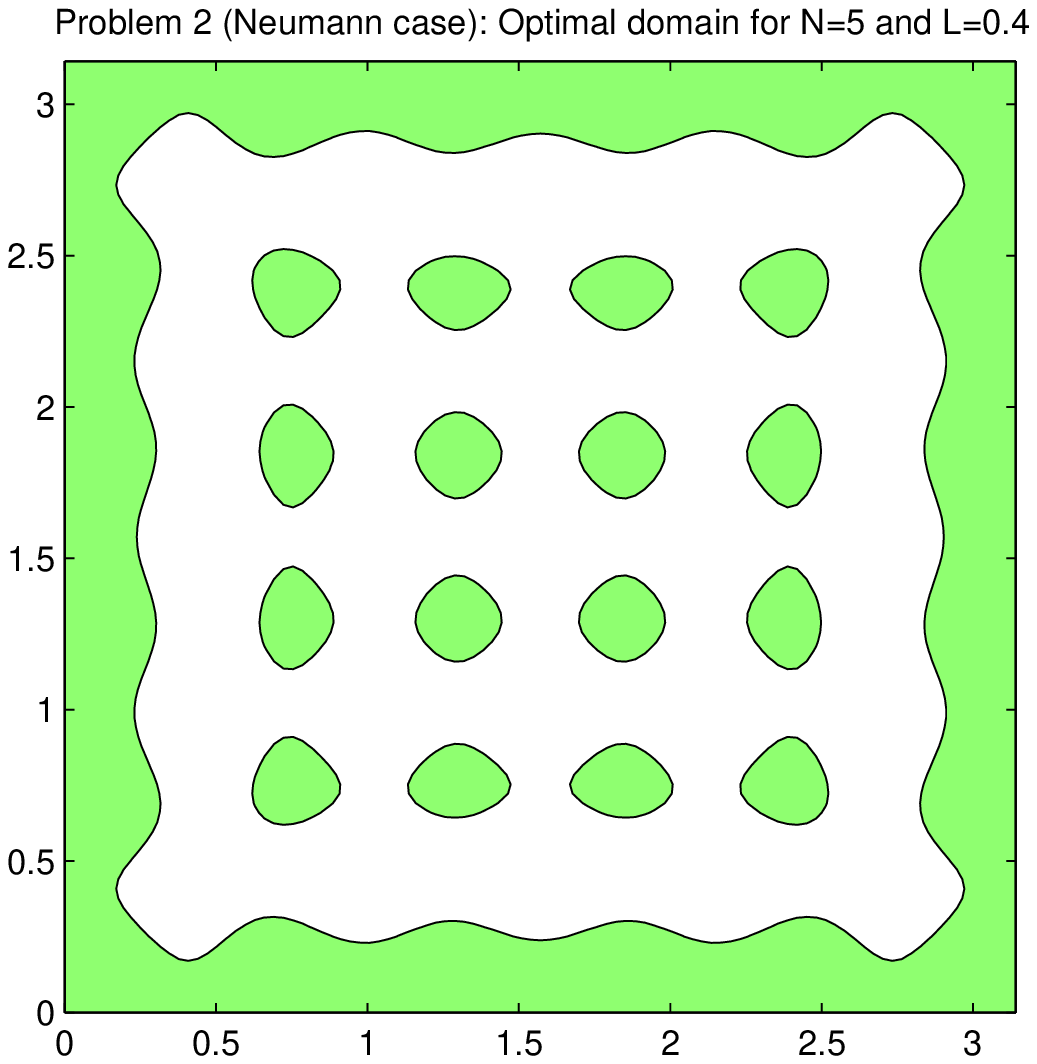}
\includegraphics[width=4.9cm]{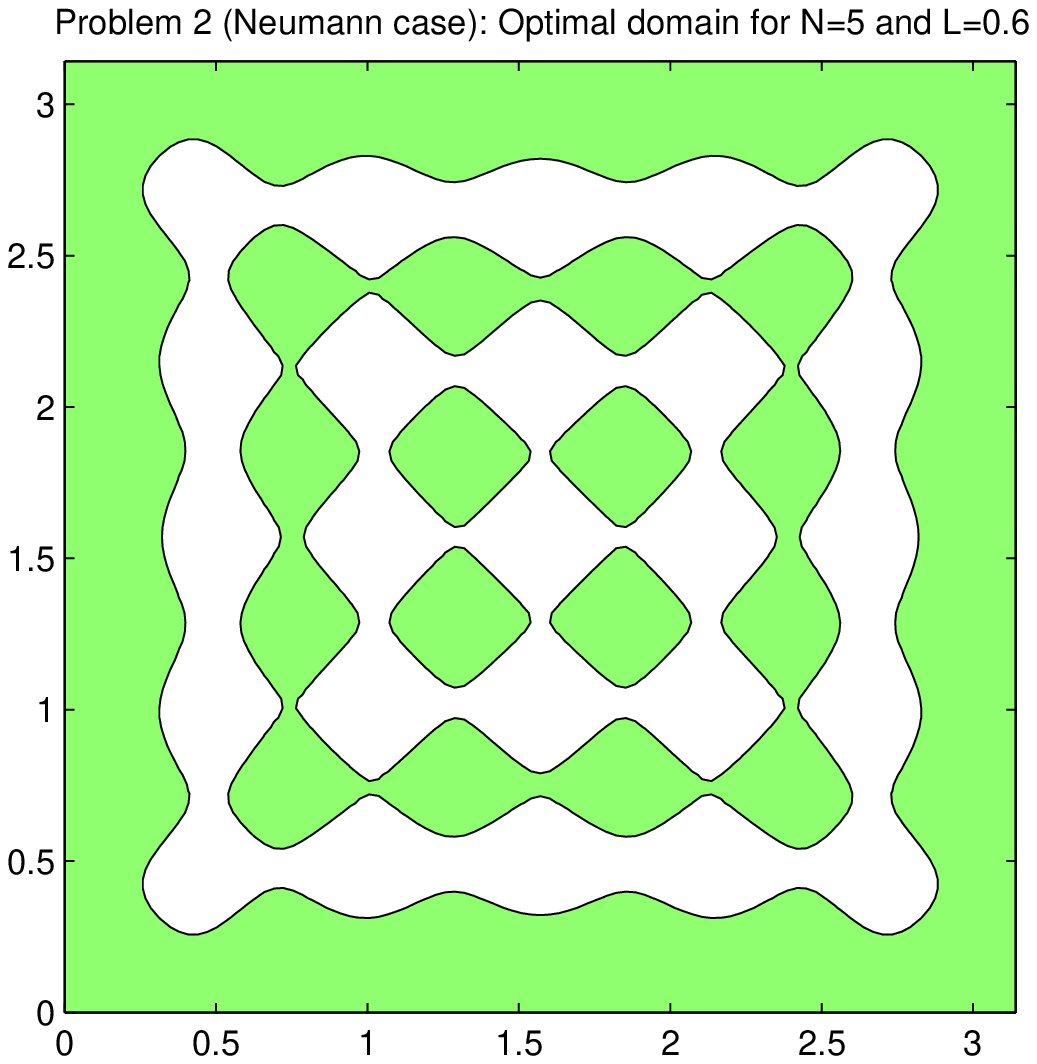}
\includegraphics[width=4.9cm]{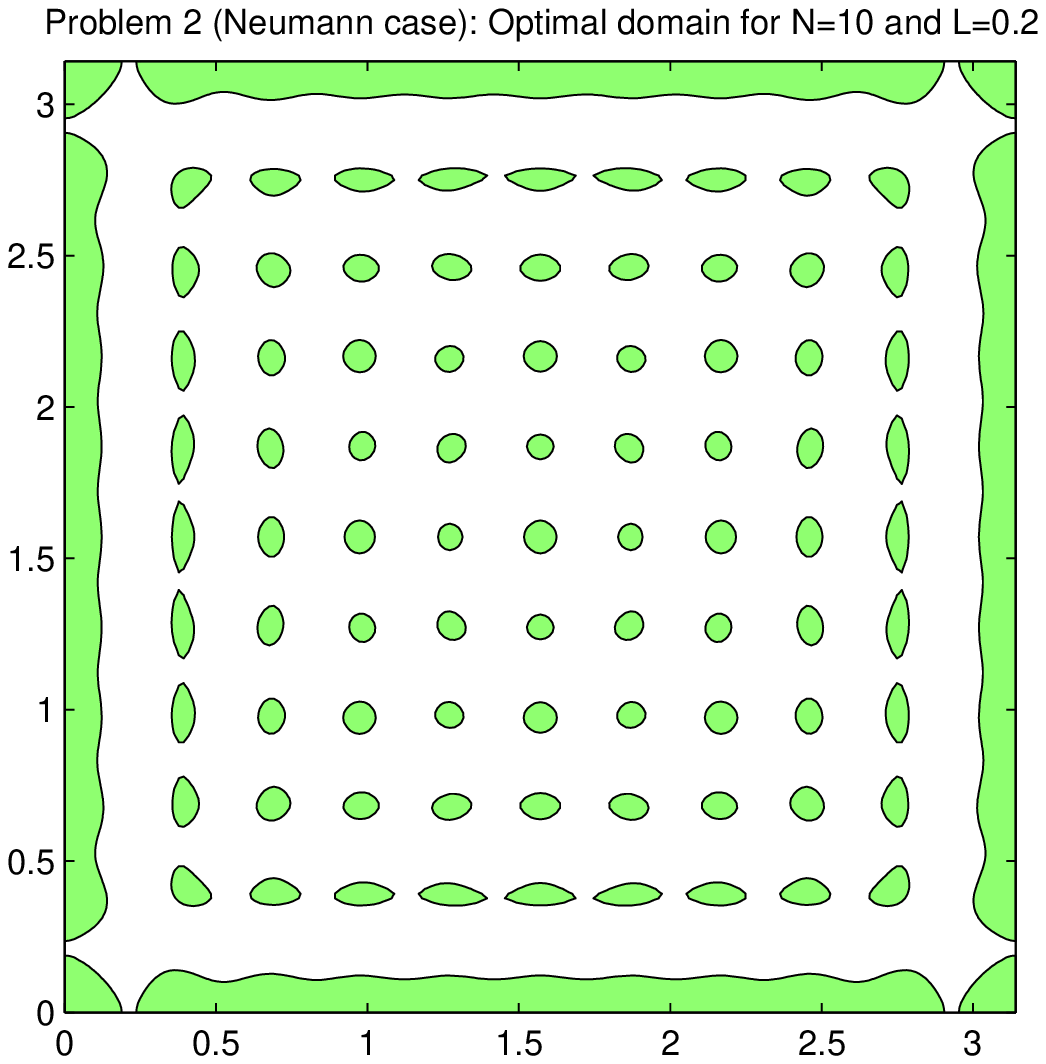}
\includegraphics[width=4.9cm]{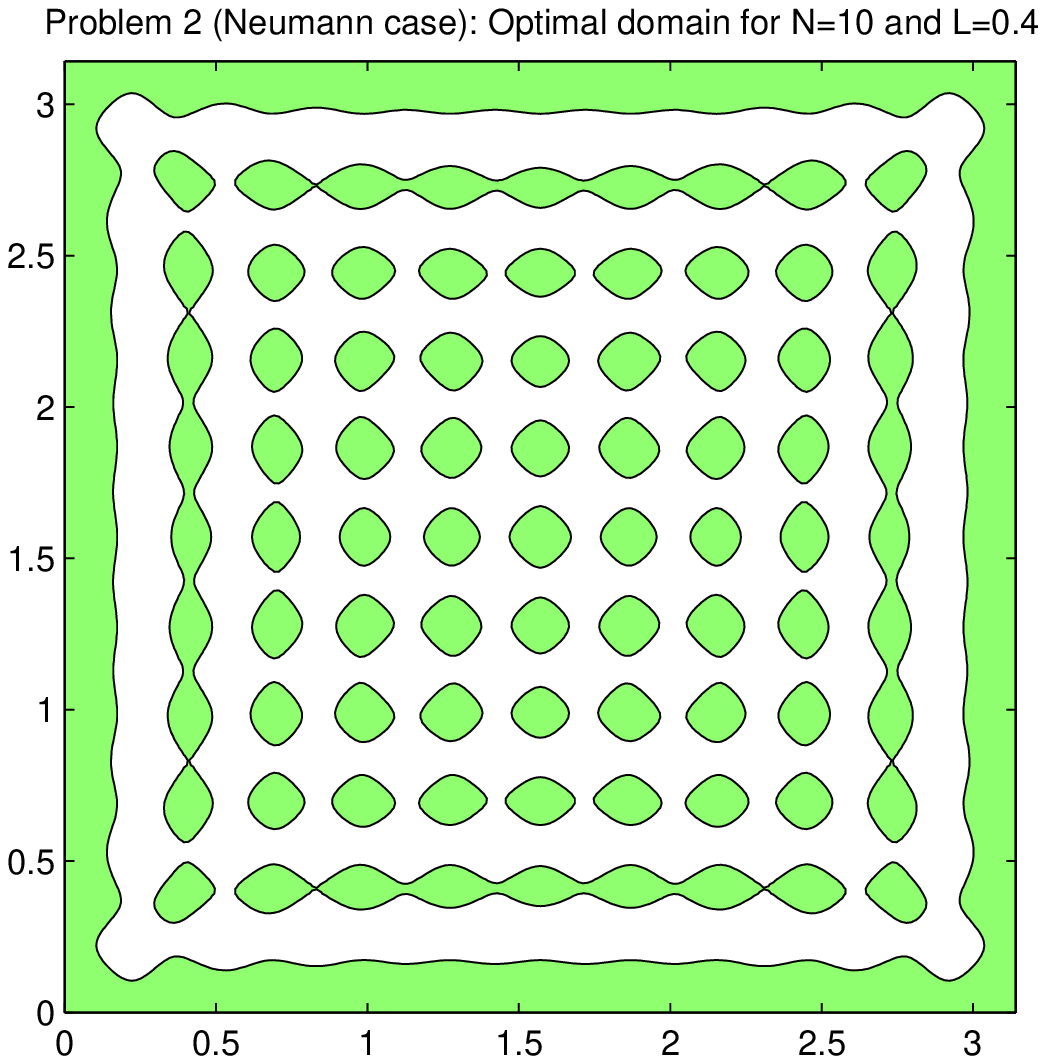}
\includegraphics[width=4.9cm]{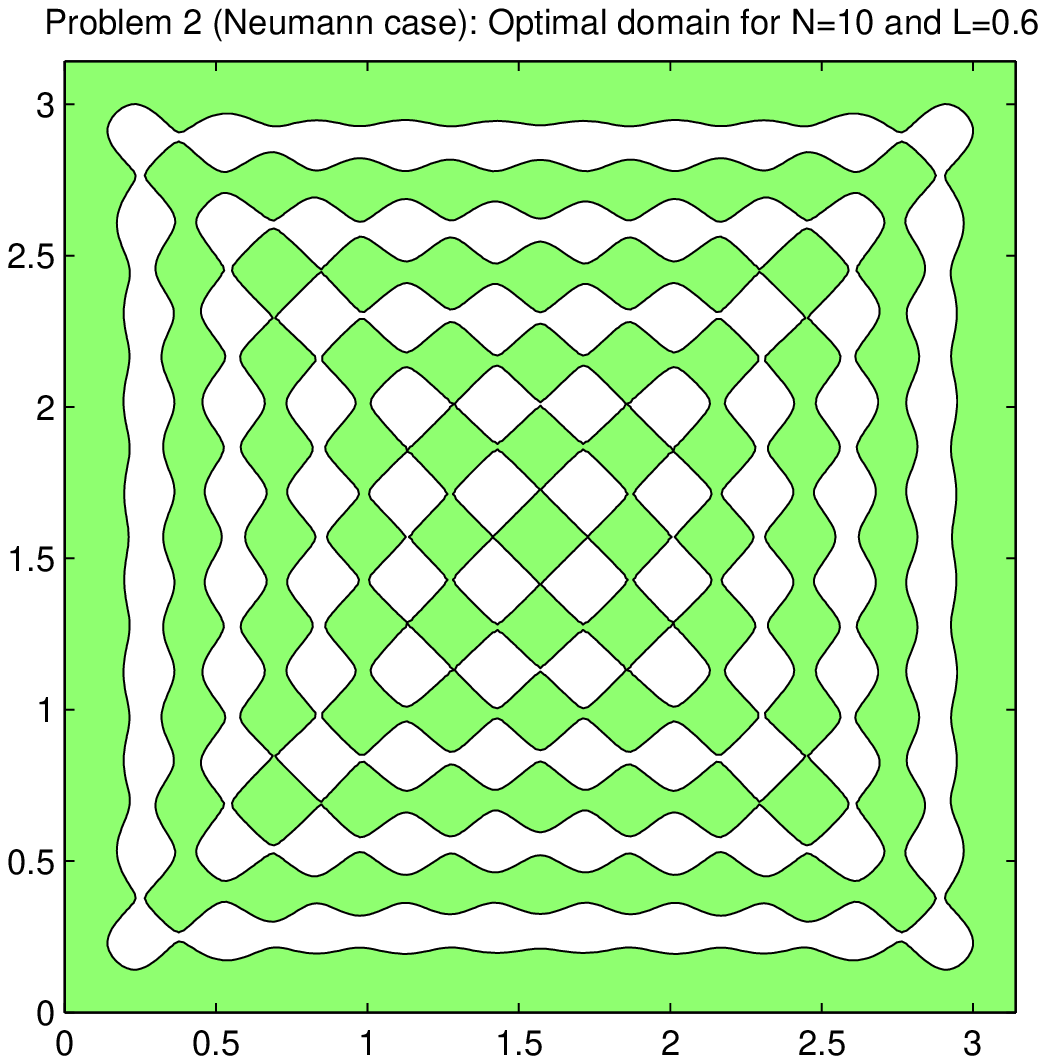}
\includegraphics[width=4.9cm]{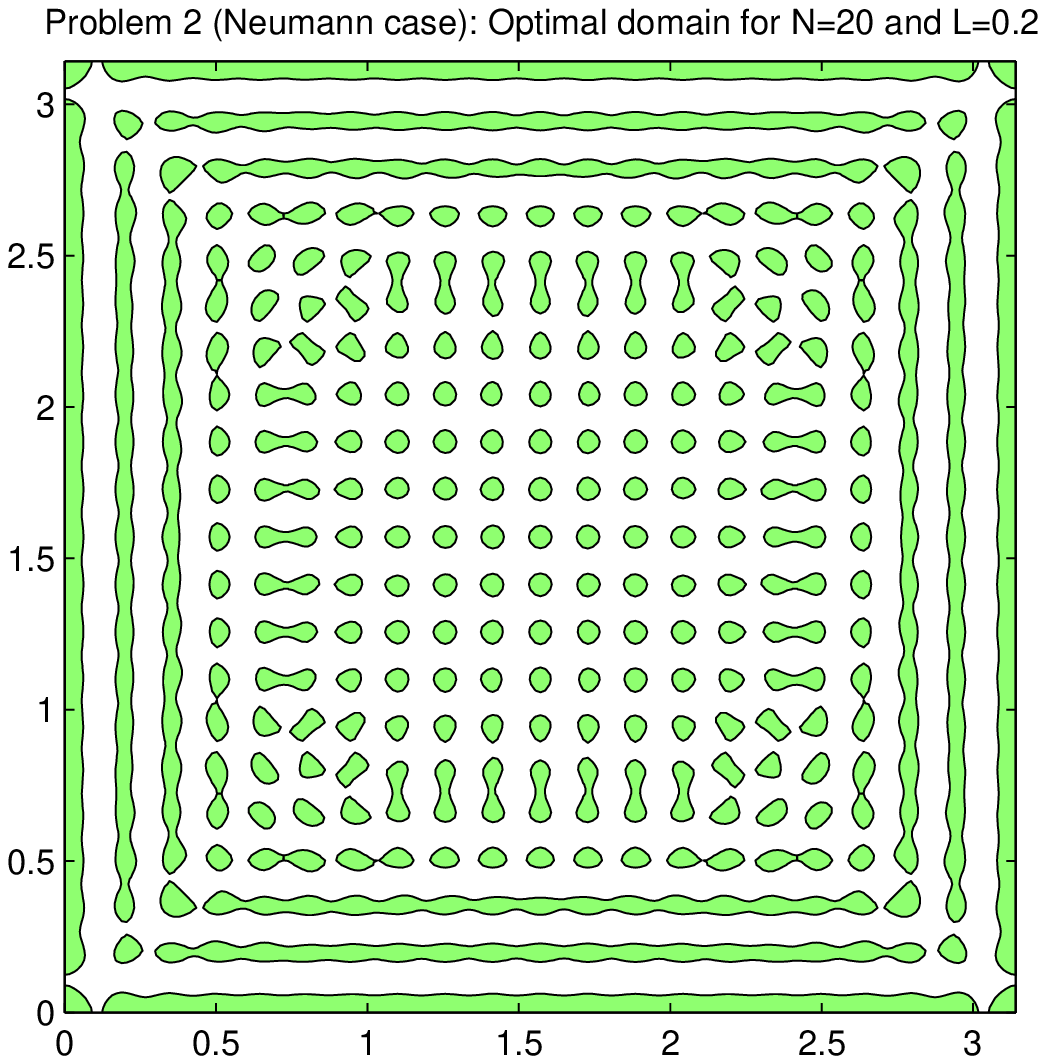}
\includegraphics[width=4.9cm]{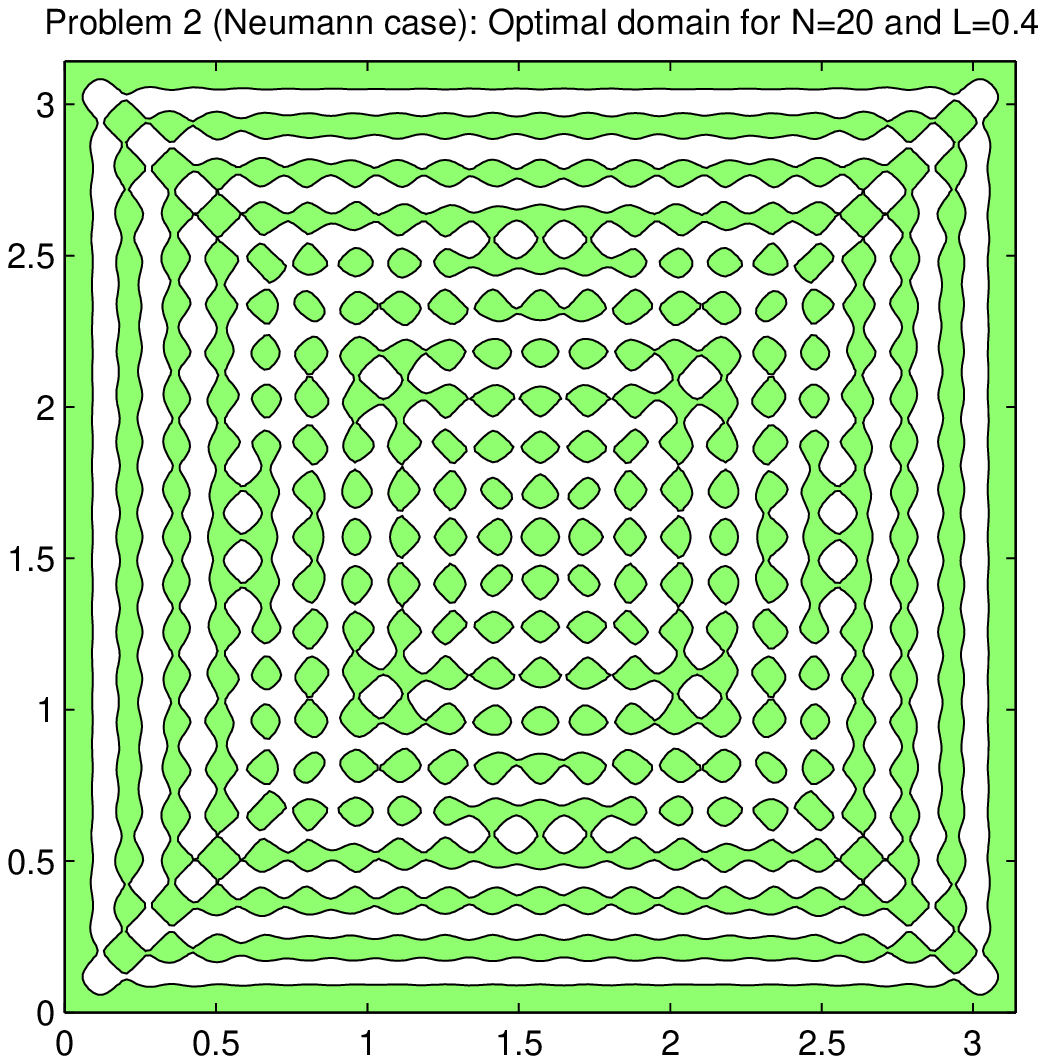}
\includegraphics[width=4.9cm]{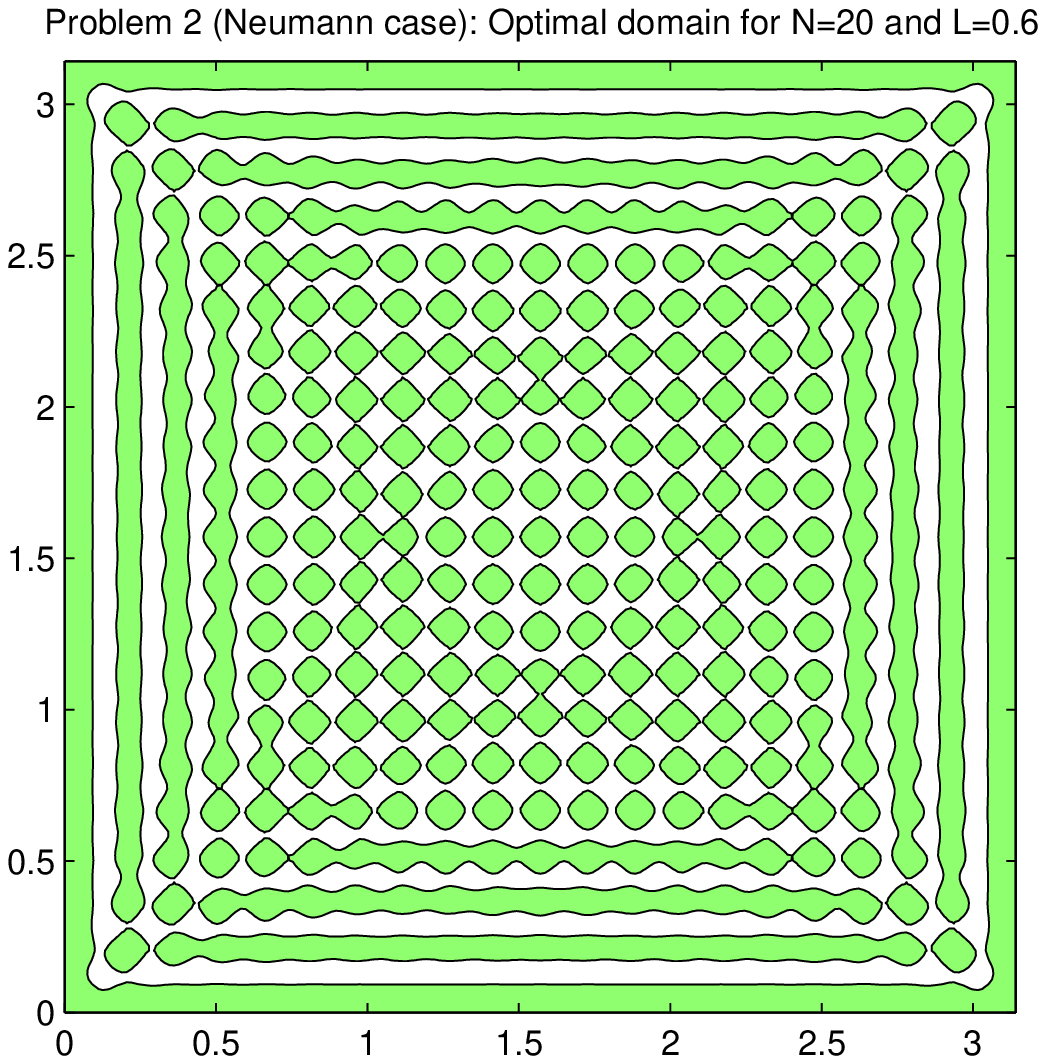}
\caption{On this figure, $\Omega=[0,\pi]^{2}$. Line 1, from left to right: optimal domain (in green) in the Neumann case for $N=2$ (4 eigenmodes) and $L\in \{0.2,0.4,0.6\}$. Line 2, from left to right: optimal domain (in green) for $N=5$ (25 eigenmodes) and $L\in \{0.2,0.4,0.6\}$. Line 3, from left to right: optimal domain (in green) for $N=10$ (100 eigenmodes) and $L\in \{0.2,0.4,0.6\}$. Line 4, from left to right: optimal domain (in green) for $N=20$ (400 eigenmodes) and $L\in \{0.2,0.4,0.6\}$}\label{figpb2N}
\end{center}
\end{figure}

Assume now that $\Omega=\{x\in \R^2\mid \vert x\Vert\leq 1\}$, the unit Euclidian disk of $\R^2$. We consider Dirichlet boundary conditions.
The normalized eigenfunctions of the Dirichlet-Laplacian are a triply indexed sequence given by
\begin{equation*}
\phi_{jkm}(r,\theta) = 
\left\{ \begin{array}{ll}
R_{0k}(r) & \ \textrm{if}\ j=0,\\
R_{jk}(r)Y_{jm}(\theta) & \ \textrm{if}\ j\geq 1,
\end{array} \right.
\end{equation*}
for $j\in\N$, $k\in\N^*$ and $m=1,2$, where $(r,\theta)$ are the usual polar coordinates.
The functions $Y_{jm}(\theta)$ are defined by $Y_{j1}(\theta)=\frac{1}{\sqrt{\pi}}\cos(j\theta)$ and $Y_{j2}(\theta)=\frac{1}{\sqrt{\pi}}\sin(j\theta)$, and the functions $R_{jk}$ are defined by
$$
R_{jk}(r) = \sqrt{2}\,\frac{J_j(z_{jk}r)}{\vert J'_{j}(z_{jk}) \vert},
$$
where $J_j$ is the Bessel function of the first kind of order $j$, and $z_{jk}>0$ is the $k^\textrm{th}$-zero of $J_{j}$.
The eigenvalues of the Dirichlet-Laplacian are given by the double sequence of $-z_{jk}^2$ and are of multiplicity $1$ if $j=0$, and $2$ if $j\geq 1$. In Proposition \ref{propnogapnoQUE}, a no-gap result is stated in this case.
Some simulations are provided on Figure \ref{figpb2disk}.
We observe that optimal domains are radially symmetric. This is actually an immediate consequence of the uniqueness of a maximizer for the modal approximations problem stated in Theorem \ref{thm3} and of the fact that $\Omega$ is itself radially symmetric.

\begin{figure}[h!]
\begin{center}
\includegraphics[width=4.9cm]{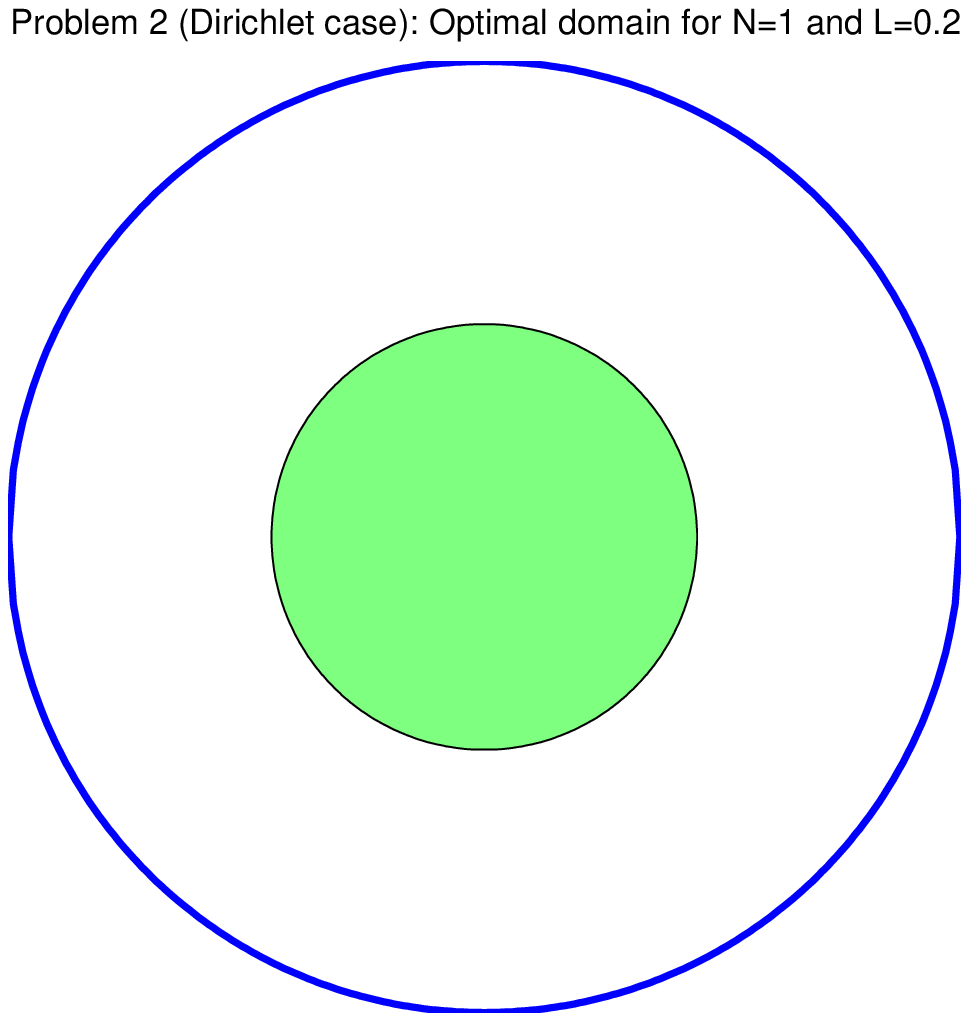}
\includegraphics[width=4.9cm]{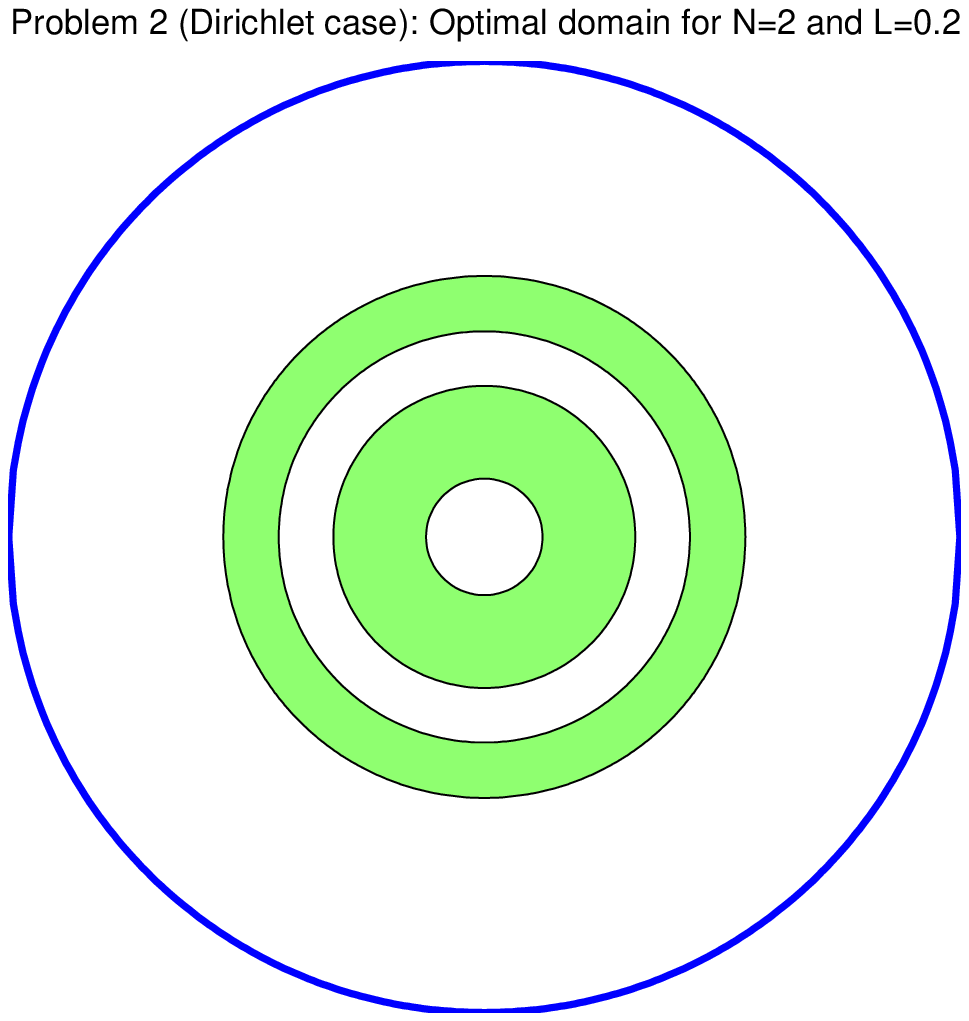}
\includegraphics[width=4.9cm]{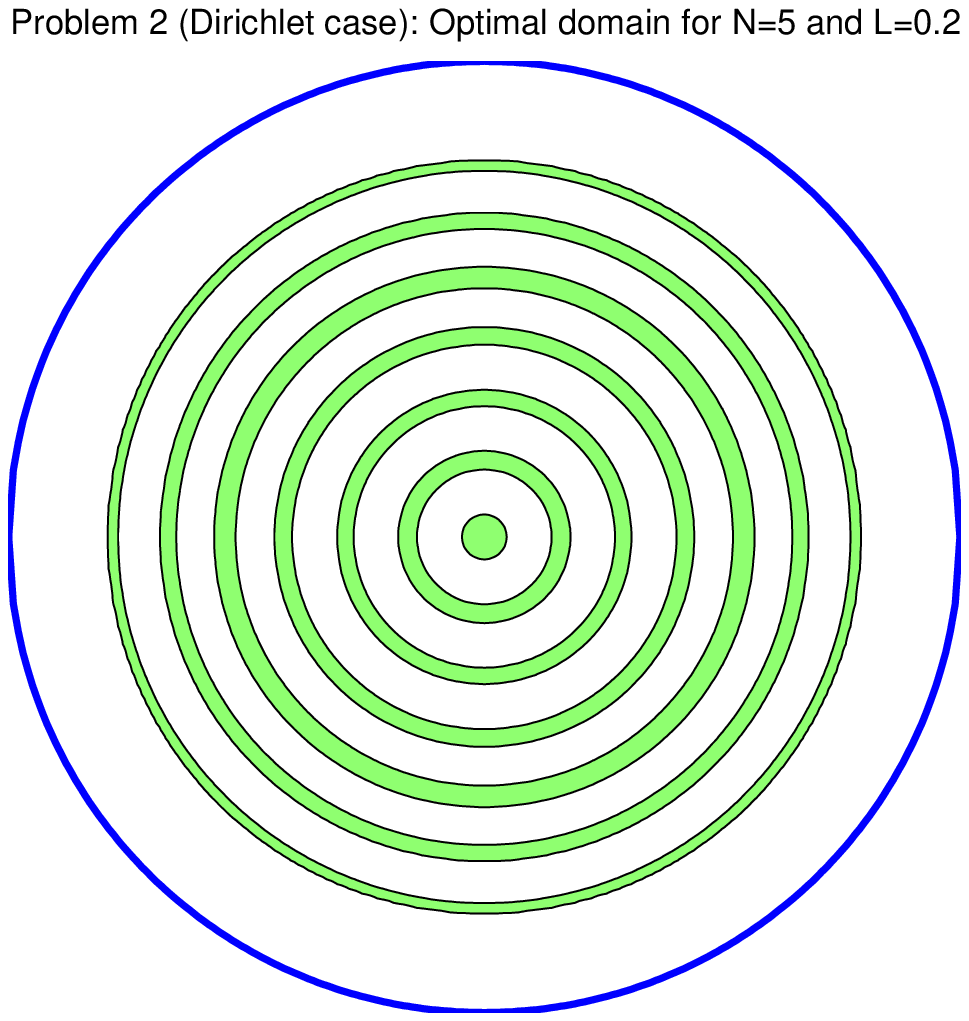}
\includegraphics[width=4.9cm]{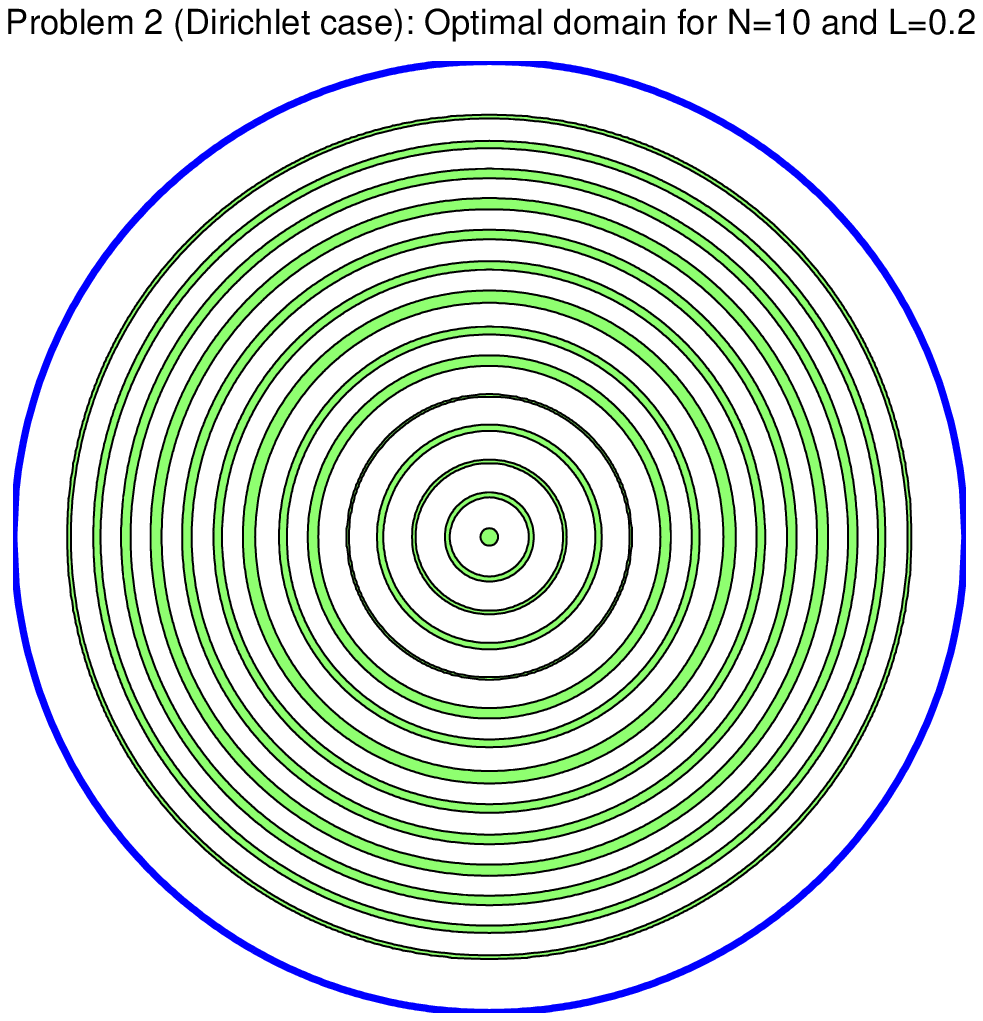}
\includegraphics[width=4.9cm]{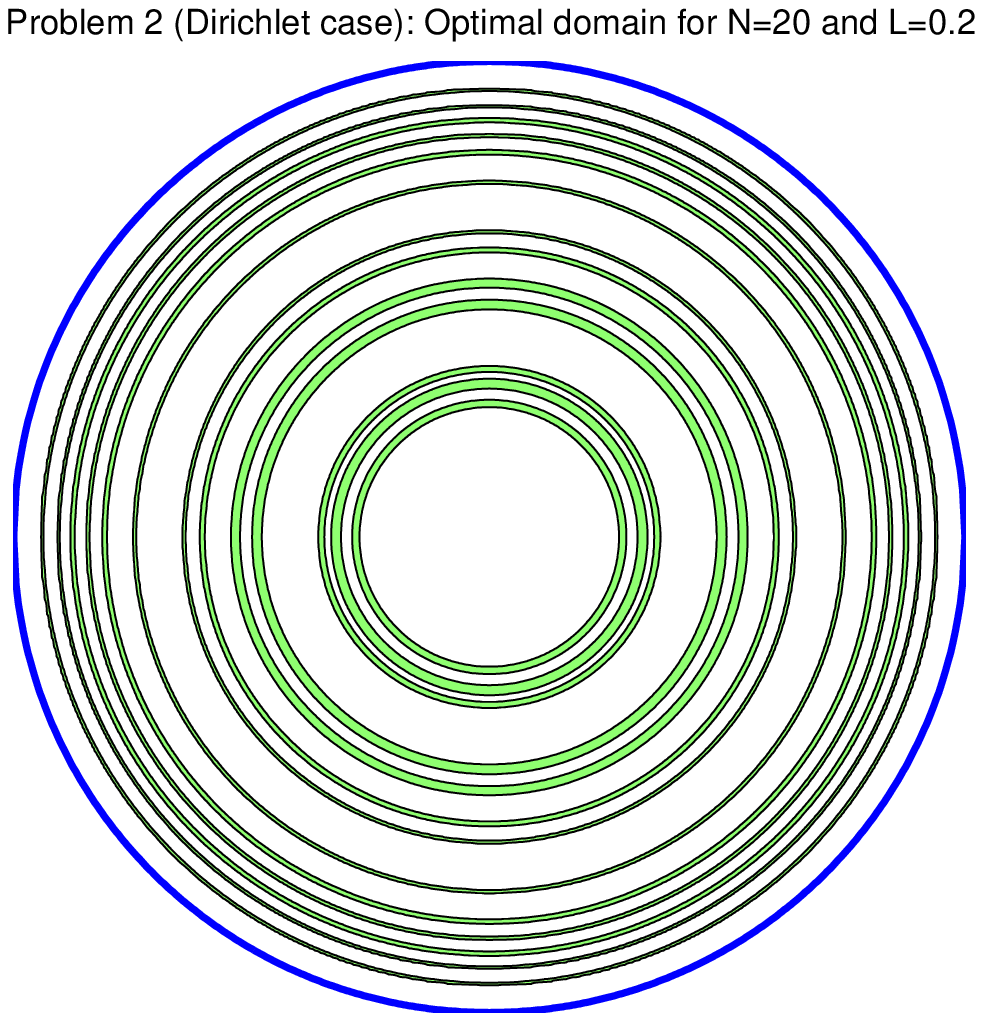}
\caption{On this figure, $\Omega=\{x\in \R^2\mid \vert x\Vert\leq 1\}$ and $L=0.2$. Line 1, from left to right: optimal domain (in green) in the Dirichlet case for $N=1$ (1 eigenmode), $N=2$ (4 eigenmodes) and $N=5$ (25 eigenmodes). Line 2, from left to right: optimal domain (in green) for $N=10$ (100 eigenmodes) and  $N=20$ (400 eigenmodes)}\label{figpb2disk}
\end{center}
\end{figure}


\section{Further comments}\label{sec_furthercomments}
In this section, we first consider in Section \ref{sec6.1} classes of subsets sharing compactness properties, in view of ensuring existence results for the second problem \eqref{reducedsecondpb} (uniform optimal design). In Section \ref{sec6.2}, we show how our results for the second problem can be extended to a natural variant of observability inequality for Neumann boundary conditions or in the boundaryless case. In Section \ref{sec6.3}, we study 
a variant of observability inequality for Dirichlet, mixed Dirichlet-Neumann and Robin boundary conditions, involving a $H^1$ norm, and we show that the criterion $J$ of the second problem has to be slightly modified. We discover that the corresponding second problem has a unique solution (that is, an optimal set), in contrast with the previous results, whenever $L$ is large enough.
Finally, in Section \ref{sec6.5} we show how the problem of maximizing the observability constant is equivalent to an optimal  design of a control problem and, namely, to that of controllability in which solutions are driven to rest in final time by means of a suitable control function.

\subsection{Uniform optimal design on other classes of admissible domains}\label{sec6.1}

According to Lemma \ref{lemprelim}, we know that, in the one-dimensional case, the second problem \eqref{reducedsecondpb} is ill-posed in the sense that it has no solution except for $L=1/2$. In larger dimension, we expect a similar conclusion. One of the reasons is that the set $\mathcal{U}_L$ defined by \eqref{defUL} is not compact for the usual topologies, as discussed in Remark \ref{remtopo}. To overcome this difficulty, a possibility consists of defining a new class of admissible sets, $\mathcal{V}_L \subset \mathcal{U}_L$, enjoying sufficient compactness properties and to replace the problem \eqref{reducedsecondpb} with
\begin{equation}\label{newPb}
\sup_{\chi_\omega\in \mathcal{V}_L}J(\chi_\omega).
\end{equation}
Of course, now, the extremal value is not necessarily the same since the class of admissible domains has been further restricted.

To ensure the existence of a maximizer $\chi_{\omega^*}$ of \eqref{newPb}, it suffices to endow $\mathcal{V}_L$ with a topology, finer than the weak star topology of $L^\infty$, for which $\mathcal{V}_L$ is compact. Of course in this case, one has
$$
J(\chi_{\omega^*})=\max_{\chi_\omega\in \mathcal{V}_L}J(\chi_\omega)\leq \sup_{\chi_\omega\in \mathcal{U}_L}J(\chi_\omega).
$$
This extra compactness property can be guaranteed by, for instance, considering some $\alpha>0$, and then any of the following possibles choices 
\begin{equation}\label{XL1}
\mathcal{V}_L=\{\chi_\omega\in\mathcal{U}_L\ \vert\ P_\Omega(\omega)\leq \alpha\},
\end{equation}
where $P_\Omega (\omega)$ is the relative perimeter of $\omega$ with respect to $\Omega$, 
\begin{equation}\label{XL2}
\mathcal{V}_L=\{\chi_\omega\in\mathcal{U}_L\ \vert\  \Vert \chi_\omega\Vert_{BV(\Omega)}\leq \alpha\},
\end{equation}
where $\Vert \cdot \Vert_{BV(\Omega)}$ is the  $BV(\Omega)$-norm  of all functions of bounded variations on $\Omega$ (see for example \cite{attouch}), or
\begin{equation}\label{XL3}
\mathcal{V}_L=\{\chi_\omega\in\mathcal{U}_L\ \vert\  \omega\textrm{ satisfies the $1/\alpha$-cone property}\},
\end{equation}
(see Section \ref{secproofthmnogap}, footnote \ref{footnoteEpscone}). Naturally, the optimal set then depends on the bound $\alpha$ under consideration, and numerical simulations (not reported here) show that, as $\alpha$ tends to $+\infty$, the family of optimal sets behaves as the maximizing sequence built in Section \ref{sectrunc2}, in particular the number of connected components grows as $\alpha$ is increasing.

\subsection{Further remarks for Neumann boundary conditions or in the boundaryless case}\label{sec6.2}
In the Neumann case, or in the case $\partial\Omega=\emptyset$, there is a problem with the constants, as explained in Footnote \ref{footnote_Neumann}. In this section, let us show that, if instead of considering the observability inequalities \eqref{ineqobsw} and \eqref{ineqobss}, we consider the inequalities
\begin{equation}\label{ineqobswNeu}
{C}_T^{(W)}(\chi_\omega) \Vert (y^0,y^1)\Vert_{H^1(\Omega,\C)\times L^2(\Omega,\C)}^2
\leq \int_0^T\int_\omega \left( \vert \partial_ty(t,x)\vert^2 + \vert  y(t,x)\vert^2\right) \, dV_g \, dt
\end{equation}
in the case of the wave equation, and
\begin{equation}\label{ineqobssNeu}
{C}_T^{(S)}(\chi_\omega) \Vert y^0\Vert_{H^2(\Omega,\C)}^2
\leq \int_0^T\int_\omega \left( \vert \partial_ty(t,x)\vert^2 + \vert  y(t,x)\vert^2\right) \, dV_g \, dt
\end{equation}
in the case of the Schr\"odinger equation (see \cite[Chapter 11]{TucsnakWeiss} for a survey on these problems), then all results remain unchanged.
Accordingly, the functional $G_T$ formerly defined by \eqref{quantity1obsW} is now replaced with
$${G}_T(\chi_\omega)=\int_0^T \int_\omega \left( \vert \partial_ty(t,x)\vert^2 + \vert  y(t,x)\vert^2\right) \, dV_g \, dt.$$

Indeed, consider initial data $(y^0,y^1)\in H^1(\Omega,\C)\times L^2(\Omega,\C)$. The corresponding solution $y$ can still be expanded as \eqref{yDecomp}, except that now $(\phi_j)_{j\in \N^*}$ consists of the eigenfunctions of the Neumann-Laplacian or of the Laplace-Beltrami operator in the boundaryless case, associated with the eigenvalues $(-\lambda_j^2)_{j\in\N^*}$, with $\lambda_1=0$ and $\phi_1$ which is constant, equal to $1/\sqrt{V_g(\Omega)}$.
The relation \eqref{notice7} does not hold any more and is replaced with
\begin{equation}\label{notice7neu}
\Vert (y^0,y^1)\Vert_{H^1(\Omega,\C)\times L^2(\Omega,\C)}^2 = \sum_{j=1}^{+\infty} \left(2\lambda_j^2 \vert a_j \vert ^2+2\lambda_j^2 \vert b_j \vert ^2+|a_j+b_j|^2\right).
\end{equation}
Following Section \ref{secMotiv}, we define the time asymptotic observability constant ${C}_\infty^{(W)}(\chi_\omega)$ as the largest possible nonnegative constant for which the time asymptotic observability inequality
\begin{equation}\label{ineqobswinfty_suppl}
{C}_\infty^{(W)}(\chi_\omega) \Vert (y^0,y^1)\Vert_{H^1(\Omega,\C)\times L^2(\Omega,\C)}^2
\leq \lim_{T\rightarrow+\infty} \frac{1}{T} \int_0^T\int_\omega \left( \vert \partial_ty(t,x)\vert^2 + \vert  y(t,x)\vert^2\right) \, dV_g \,dt
\end{equation}
holds, for all $(y^0,y^1)\in H^1(\Omega,\C)\times L^2(\Omega,\C)$.
Similarly, we define the randomized observability constant ${C}_{T,\textrm{rand}}^{(W)}(\chi_\omega)$ as the largest possible nonnegative constant for which the randomized observability inequality
\begin{equation}\label{ineqobswRandNeu}
{C}_{T,\textrm{rand}}^{(W)}(\chi_\omega) \Vert (y^0,y^1)\Vert_{H^1(\Omega,\C)\times L^2(\Omega,\C)}^2
\leq \mathbb{E}\left(\int_0^T\int_\omega \left( \vert \partial_ty_\nu(t,x)\vert^2 + \vert y_\nu(t,x)\vert^2\right) \, dV_g \,dt\right)
\end{equation}
holds, for all $y^0(\cdot)\in H^1(\Omega,\C)$ and $y^1(\cdot)\in L^2(\Omega,\C)$, where $y_\nu$ is defined as before by \eqref{defynu}.

The time asymptotic and randomized observability constants are defined accordingly for the Schr\"odinger equation.

\begin{theorem}
Let $\omega$ be a measurable subset of $\Omega$.
\begin{enumerate}
\item If the domain $\Omega$ is such that every eigenvalue of the Neumann-Laplacian is simple, then
$$
2\, {C}_\infty^{(W)}(\chi_\omega)  ={C}_\infty^{(S)}(\chi_\omega)  = {J}(\chi_\omega) .
$$
\item There holds
\begin{equation*}
2\, {C}_{T,\textnormal{rand}}^{(W)}(\chi_\omega) = {C}_{T,\textnormal{rand}}^{(S)}(\chi_\omega) = T {J}(\chi_\omega).
\end{equation*}
\end{enumerate}
\end{theorem}


\begin{proof}
Following the same lines as those in the proofs of Theorems \ref{theoCTT} and \ref{propHazardCst} (see Sections \ref{prooftheoCTT} and \ref{proofpropHazardCst}), we obtain
$$
{C}_{T,\textrm{rand}}^{(W)}(\chi_\omega) = T {C}_\infty^{(W)}(\chi_\omega) = T \, \Gamma,
$$
with
$$
\Gamma = \inf_{((a_j),(b_j))\in(\ell^2(\C))^2\setminus\{0\}}\frac{\sum_{j=1}^{+\infty}(1+\lambda_j^2)(|a_j|^2+|b_j|^2)\int_\omega \phi_j(x)^2\, dV_g}{ \sum_{j=1}^{+\infty} \left(2\lambda_j^2( \vert a_j \vert ^2+ \vert b_j \vert ^2)+|a_j+b_j|^2\right)} .
$$
Let us prove that $\Gamma=\frac{1}{2}{J}(\chi_\omega)$. First of all, it is easy to see that, in the definition of $\Gamma$, it suffices to consider the infimum over real sequences $(a_j)$ and $(b_j)$.
Next, setting $a_j=\rho_j\cos \theta_j$ and $b_j=\rho_j\sin \theta_j$, since
$\vert a_j+b_j\vert^2 = \rho_j^2 (1+\sin(2\theta_j))$, to reach the infimum one has to take $\theta_j=\pi/4$ for every $j\in\N^*$.
It finally follows that
$$\Gamma  =   \inf_{\substack{(\rho_j)\in\ell^2(\R)\\ \sum_{j=1}^{+\infty}\rho_j^2=1}}\frac{1}{2}\sum_{j=1}^{+\infty}\rho_j^2\int_\omega \phi_j(x)^2\, dV_g
  =  \frac{1}{2}{J}(\chi_\omega).
$$
\end{proof}

\subsection{Variant of observability inequalities and optimal design results}\label{sec6.3}
In this section we consider only the following boundary conditions: Dirichlet, mixed Dirichlet-Neumann (with $\Gamma_0\neq\emptyset$), Robin.
We replace the observability inequalities \eqref{ineqobsw} and \eqref{ineqobss} with the inequalities
\begin{equation}\label{ineqobswRobin}
\tilde{C}_T^{(W)}(\chi_\omega) \Vert (y^0,y^1)\Vert_{H^1(\Omega,\C)\times L^2(\Omega,\C)}^2
\leq \int_0^T\int_\omega \vert \partial_ty(t,x)\vert^2 \, dV_g \, dt
\end{equation}
in the case of the wave equation, and
\begin{equation}\label{ineqobssRobin}
\tilde{C}_T^{(S)}(\chi_\omega) \Vert y^0\Vert_{H^2(\Omega,\C)}^2
\leq \int_0^T\int_\omega  \vert \partial_ty(t,x)\vert^2  \, dV_g \, dt
\end{equation}
in the case of the Schr\"odinger equation. 
Note that here, even in the Dirichlet case, we consider the full $H^1$ norm defined by $\Vert f\Vert_{H^1(\Omega,\C)}=(\Vert f\Vert_{L^2(\Omega,\C)}^2+\Vert \nabla f\Vert_{L^2(\Omega,\C)}^2)^{1/2}$.
These inequalities hold true under GCC, as already mentioned. Of course these inequalities cannot hold in the boundaryless case or for Neumann boundary conditions, due to constants (as  discussed previously).
Since the norm used at the left-hand side is stronger, it follows that $\tilde{C}_T^{(W)}(\chi_\omega)\leq C_T^{(W)}(\chi_\omega)$ and $\tilde{C}_T^{(S)}(\chi_\omega)\leq C_T^{(S)}(\chi_\omega)$.

Accordingly, the functional $G_T$ formerly defined by \eqref{quantity1obsW} is now replaced with
$$\tilde{G}_T(\chi_\omega)=\int_0^T \int_\omega \vert \partial_ty(t,x)\vert^2  \, dV_g \, dt.$$

In contrast to the previous results, the time averaging procedure or the randomization with respect to initial data does not lead to the functional $J$ defined by \eqref{defJ} but to the slightly different functional
$$
\tilde{J}(\chi_\omega)=\inf_{j\in\N^*}\frac{\lambda_j^2}{1+\lambda_j^2}\int_\omega \phi_j(x)^2 \, dV_g .
$$
We will see that the study of the associated optimization problem differs significantly from the one considered previously.

\medskip

Let us first explain how to derive this expression in the case of the wave equation (the arguments being similar for the Schr\"odinger equation).
Consider initial data $(y^0,y^1)\in H^1(\Omega,\C)\times L^2(\Omega,\C)$. One has
\begin{equation}\label{notice7robin}
\Vert (y^0,y^1)\Vert_{H^1(\Omega,\C)\times L^2(\Omega,\C)}^2 = \sum_{j=1}^{+\infty} \left(2\lambda_j^2 \vert a_j \vert ^2+2\lambda_j^2 \vert b_j \vert ^2+|a_j+b_j|^2\right).
\end{equation}
As previously, the time asymptotic observability constant $\tilde{C}_\infty^{(W)}(\chi_\omega)$ is defined as the largest possible nonnegative constant for which the time asymptotic observability inequality
\begin{equation}\label{ineqobswinfty_supplRobin}
\tilde{C}_\infty^{(W)}(\chi_\omega) \Vert (y^0,y^1)\Vert_{H^1(\Omega,\C)\times L^2(\Omega,\C)}^2
\leq \lim_{T\rightarrow+\infty} \frac{1}{T} \int_0^T\int_\omega \vert \partial_t y(t,x)\vert^2 \, dV_g \,dt
\end{equation}
holds, for all $(y^0,y^1)\in H^1(\Omega,\C)\times L^2(\Omega,\C)$.
Similarly, the randomized observability constant $\tilde{C}_{T,\textrm{rand}}^{(W)}(\chi_\omega)$ is the largest possible nonnegative constant for which the randomized observability inequality
\begin{equation}\label{ineqobswRandRobin}
\tilde{C}_{T,\textrm{rand}}^{(W)}(\chi_\omega) \Vert (y^0,y^1)\Vert_{H^1(\Omega,\C)\times L^2(\Omega,\C)}^2
\leq \mathbb{E}\left(\int_0^T\int_\omega \vert \partial_t y_\nu(t,x)\vert^2 \, dV_g \,dt\right)
\end{equation}
holds, for all $y^0(\cdot)\in H^1(\Omega,\C)$ and $y^1(\cdot)\in L^2(\Omega,\C)$, where $y_\nu$ is defined as before by \eqref{defynu}.
The time asymptotic and randomized observability constants are defined accordingly for the Schr\"odinger equation.

\begin{theorem}\label{thmconstantRobin}
Let $\omega$ be a measurable subset of $\Omega$.
\begin{enumerate}
\item If the domain $\Omega$ is such that every eigenvalue of $A$ is simple, then
$$
2\, \tilde{C}_\infty^{(W)}(\chi_\omega)  =\tilde{C}_\infty^{(S)}(\chi_\omega)  = \tilde{J}(\chi_\omega) .
$$
\item There holds
\begin{equation*}
2\, \tilde{C}_{T,\textrm{rand}}^{(W)}(\chi_\omega) = \tilde{C}_{T,\textrm{rand}}^{(S)}(\chi_\omega) = T \tilde{J}(\chi_\omega).
\end{equation*}
\end{enumerate}
\end{theorem}

\begin{proof}
Following the same lines as those in the proofs of Theorems \ref{theoCTT} and \ref{propHazardCst} (see Sections \ref{prooftheoCTT} and \ref{proofpropHazardCst}), we obtain
$$
\tilde{C}_{T,\textrm{rand}}^{(W)}(\chi_\omega) = T \tilde{C}_\infty^{(W)}(\chi_\omega) = T \, \Gamma,
$$
with
$$
\Gamma = \inf_{((a_j),(b_j))\in(\ell^2(\C))^2\setminus\{0\}}\frac{\sum_{j=1}^{+\infty}\lambda_j^2(|a_j|^2+|b_j|^2)\int_\omega \phi_j(x)^2\, dV_g}{ \sum_{j=1}^{+\infty} \left(2\lambda_j^2( \vert a_j \vert ^2+ \vert b_j \vert ^2)+|a_j+b_j|^2\right)} .
$$
Let us prove that $\Gamma=\frac{1}{2}\tilde{J}(\chi_\omega)$. First of all, it is easy to see that, in the definition of $\Gamma$, it suffices to consider the infimum over real sequences $(a_j)$ and $(b_j)$.
Next, setting $a_j=\rho_j\cos \theta_j$ and $b_j=\rho_j\sin \theta_j$, since
$\vert a_j+b_j\vert^2 = \rho_j^2 (1+\sin(2\theta_j))$, to reach the infimum one has to take $\theta_j=\pi/4$ for every $j\in\N^*$.
It finally follows that
$$\Gamma  =   \inf_{\substack{(\rho_j)\in\ell^2(\R)\\ \sum_{j=1}^{+\infty}\rho_j^2=1}}\frac{1}{2}\sum_{j=1}^{+\infty}\rho_j^2\frac{\lambda_j^2}{1+\lambda_j^2}\int_\omega \phi_j(x)^2\, dV_g
  =  \frac{1}{2}\tilde{J}(\chi_\omega).
$$
\end{proof}

We are thus led to introduce the following version of the second problem (uniform optimal design problem), associated with the above observability inequalities.

\begin{quote}
\noindent{\bf Uniform optimal design problem.}
\textit{We investigate the problem of maximizing the functional
\begin{equation}\label{defJRobin}
\tilde{J}(\chi_\omega)=\inf_{j\in\N^*}\gamma_j\int_\omega \phi_j(x)^2\, dV_g,
\end{equation}
over all possible subsets $\omega$ of $\Omega$ of measure $ V_g(\omega) =L V_g(\Omega) $, where the $\gamma_j$'s are defined by
$$
\gamma_j=\frac{\lambda_j^2}{1+\lambda_j^2}.
$$
}

\end{quote}

Note that the sequence $(\gamma_j)_{j\in\N^*}$ is monotone increasing, and that $0<\gamma_1\leq \gamma_j<1$ for every $j\in\N^*$.

As in Section \ref{solvingpb2obssec1}, the convexified version of this problem is defined accordingly by
\begin{equation}\label{quantity2obsAsympconvexifiedRobin}
\sup_{a\in \overline{\mathcal{U}}_L} \tilde{J}(a) ,
\end{equation}
where
\begin{equation}\label{defJrelaxRobin}
\tilde{J}(a)=\inf_{j\in\N^*}\gamma_j\int_\Omega a(x)\phi_j(x)^2\, dV_g.
\end{equation}
As in Sections \ref{solvingpb2obssec1} and \ref{secconvexified2}, under the assumption that there exists a subsequence of $(\phi_j^2)_{j\in \N^*}$ converging to $\frac{1}{ V_g(\Omega) }$\ in weak star $L^\infty$ topology ($L^\infty$-WQE property), the problem \eqref{quantity2obsAsympconvexifiedRobin} has at least one solution, and
$\sup_{a\in \overline{\mathcal{U}}_L} \tilde{J}(a) = L$, and the supremum is reached with the constant function $a=L$.

We will next establish a no-gap result, similar to Theorem \ref{thmnogap}, but only valuable for nonsmall values of $L$. Actually, we will show that the present situation differs significantly from the previous one, in the sense that, if $\gamma_1<L<1$ then the highfrequency modes do not play any role in the problem \eqref{defJRobin}.
Before coming to that result, let us first define the truncated versions of the problem \eqref{defJRobin}.
For every $N\in\N^*$, we define
\begin{equation}\label{defJNRobin}
\tilde J_N(a) = \inf_{1\leq j\leq N} \gamma_j \int_\Omega a (x)\phi_j(x)^2\, dV_g.
\end{equation}
An immediate adaptation of the proof of Theorem \ref{thm3} yields the following result.

\begin{proposition}\label{prop3Robin}
For every $N\in\N^*$, the problem 
\begin{equation}\label{truncPbRobin}
\sup_{a\in\overline{\mathcal{U}}_L}\tilde J_N(a)
\end{equation}
has a unique solution $a^N$ that is the characteristic function of a set $\omega^N$. Moreover, if $M$ is analytic then $\omega^N$ is semi-analytic and has a finite number of connected components.
\end{proposition}

The main result of this section is the following.

\begin{theorem}\label{thmnogapRobin}
Assume that the QUE and uniform $L^\infty$-boundedness properties hold.
Let $L\in (\gamma_1,1)$. Then there exists $N_0\in\N^*$ such that
\begin{equation}\label{nogapRobin}
\max_{\chi_\omega\in\mathcal{U}_L}\tilde J(\chi_\omega)=\max_{\chi_\omega\in\mathcal{U}_L}\tilde J_N(\chi_\omega)\leq \gamma_1 <L ,
\end{equation}
for every $N\geq N_0$. In particular, the second problem \eqref{defJRobin} has a unique solution $\chi_{\omega^{N_0}}$, and moreover if $M$ is analytic then the set $\omega^{N_0}$ is semi-analytic and has a finite number of connected components. 
\end{theorem}

\begin{proof}
Using the same arguments as in Lemma \ref{lemmaExistConv}, it is clear that the problem \eqref{quantity2obsAsympconvexifiedRobin} has at least one solution, denoted by $a^\infty$. Let us first prove that there exists $N_0\in\N^*$ such that $\tilde J(a^\infty)=\tilde J_{N_0}(a^\infty)$. Let $\varepsilon\in (0,L-\gamma_1)$. It follows from the $L^\infty$-QUE property that there exists $N_0\in\N^*$ such that
\begin{equation}\label{defN0Robin}
\gamma_j\int_\Omega a^\infty(x)\phi_j(x)^2\, dV_g\geq L-\varepsilon,
\end{equation}
for every $j > N_0$.
Therefore,
\begin{equation*}
\begin{split}
\tilde J(a^\infty)  &=  \inf_{j\in\N^*}\gamma_j\int_\Omega a^\infty(x) \phi_j(x)^2\, dV_g \\
&= \min \left( \inf_{1\leq j\leq N_0}\gamma_j\int_\Omega a^\infty(x) \phi_j(x)^2\, dV_g ,\inf_{j>N_0}\gamma_j\int_\Omega a^\infty(x) \phi_j(x)^2\, dV_g\right) \\
& \geq   \min \left(\tilde J_{N_0}(a^\infty), L-\varepsilon\right)=\tilde J_{N_0}(a^\infty),
\end{split}
\end{equation*}
since $L-\varepsilon>\gamma_1$ and $\tilde J_{N_0}(a^\infty)\leq \gamma_1$. It follows that $\tilde J(a^\infty)=\tilde J_{N_0}(a^\infty)$.

Let us now prove that  $\tilde J(a^\infty) = \tilde J_{N_0}(a^{N_0})$, where $a^{N_0}$ is the unique maximizer of $\tilde J_{N_0}$ (see Proposition \ref{prop3Robin}). By definition of a maximizer, one has $\tilde J(a^\infty) = \tilde J_{N_0}(a^\infty) \leq \tilde J_{N_0}(a^{N_0})$.
By contradiction, assume that $\tilde J_{N_0}(a^\infty) < \tilde J_{N_0}(a^{N_0})$. Let us then design an admissible perturbation $a_t\in\overline{\mathcal{U}}_L$ of $a^\infty$ such that $\tilde J(a_t)>\tilde J(a^\infty)$, which raises a contradiction with the optimality of $a^\infty$. For every $t\in[0,1]$, set $
a_t = a^\infty + t ( a^{N_0} - a^\infty)$.
Since $\tilde J_{N_0}$ is concave, one gets
$$
\tilde J_{N_0}(a_t) \geq (1-t)\tilde J_{N_0}(a^\infty) + t \tilde J_{N_0}(a^{N_0}) > \tilde J_{N_0}(a^\infty),
$$
for every $t\in(0,1]$, which means that
\begin{equation}\label{lowmodes_Robin}
\inf_{1\leq j\leq N_0} \gamma_j \int_\Omega a_t(x) \phi_j(x)^2\, dV_g > \inf_{1\leq j\leq N_0} \gamma_j \int_\Omega a^\infty(x) \phi_j(x)^2\, dV_g  \geq \tilde J(a^\infty),
\end{equation}
for every $t\in(0,1]$. 
Besides, since $a^{N_0}-a^\infty\in (-2,2)$ almost everywhere in $\Omega$, it follows from \eqref{defN0Robin} that
$$
\gamma_j \int_\Omega a_t(x) \phi_j^2(x)\, dV_g   =   \gamma_j \int_\Omega a^\infty(x)\phi_j(x)^2\, dV_g+t\gamma_j\int_\Omega (a^{N_0}(x)-a^\infty(x)) \phi_j(x)^2 \, dV_g
\geq   L-\varepsilon-2t ,
$$
for every $j\geq N_0$. Let us choose $t$ such that $0<t<\frac{L-\varepsilon-\gamma_1}{2}$, so that the previous inequality yields   
\begin{equation}\label{highmodes_Robin}
\gamma_j \int_\Omega a_t(x) \phi_j(x)^2\, dV_g > \gamma_1  \geq  \gamma_1 \int_\Omega a^\infty (x)\phi_1(x)^2\, dV_g \geq \tilde J(a^\infty).
\end{equation}
for every $j\geq N_0$. 
Combining the low modes estimate \eqref{lowmodes_Robin} with the high modes estimate \eqref{highmodes_Robin}, we conclude that
$$
\tilde J(a_t) = \inf_{j\in\N^*} \gamma_j \int_\Omega a_t(x) \phi_j(x)^2 \, dV_g > \tilde J(a^\infty),
$$
which contradicts the optimality of $a^\infty$.

Therefore $\tilde J_{N_0}(a^\infty) = \tilde J(a^\infty)= \tilde J_{N_0}(a^{N_0})$, and the result follows.
\end{proof}

\begin{remark}
Under the assumptions of the theorem, there is no gap between the second problem \eqref{defJRobin} and its convexified formulation \eqref{quantity2obsAsympconvexifiedRobin}, as well as before. But, contrarily to the previous results, here there always exists a maximizer in the class of characteristic functions whenever $L$ is large enough, and moreover, this optimal set can be computed from a truncated formulation \eqref{defJNRobin} for a certain value of $N$. In other words, the maximizing sequence $(\chi_{\omega^N})_{N\in\N^*}$ resulting from Proposition \ref{prop3Robin} is stationary. Here, the high modes play no role, whereas in the previous results all modes had the same impact. This is due to the fact that, in the left hand-side of the observability inequalities \eqref{ineqobswRobin} and \eqref{ineqobssRobin}, we use the $H^1$-norm instead of the norm of $D(A^{1/2})$. This causes a spectral asymmetry that finally leads to the above result. 
\end{remark}

\begin{remark}
Here, if $L$ is not too small then there exists an optimal set (sharing nice regularity properties) realizing the largest possible time asymptotic and randomized observability constants. The optimal value of these constants is known to be less than $L$ but its exact value is not known. It is related to solving a finite dimensional numerical optimization problem.
\end{remark}

\begin{remark}
In the case where $L\leq\gamma_1$, we do not know whether or not there is a gap between the second problem \eqref{defJRobin} and its convexified formulation \eqref{quantity2obsAsympconvexifiedRobin}. 
Adapting shrewdly the proof of Theorems \ref{thmnogap} or \ref{thmnogap2} does not seem to allow one to derive a no-gap result. Nevertheless, one can prove using these arguments that
$\sup_{\chi_\omega\in\mathcal{U}_L}\tilde J(\chi_\omega)\geq \gamma_1 L$.
\end{remark}

\begin{remark}
We formulate the following two open questions.
\begin{itemize}
\item Under the assumptions of Theorem \ref{thmnogapRobin}, does the conclusion hold true for every $L\in (0,1)$?
\item Does the statement of Theorem \ref{thmnogapRobin} still hold true under weaker ergodicity assumptions, for instance is it possible to weaken QUE into WQE?
\end{itemize}
\end{remark}

\begin{remark}\label{rem_extension_tensorized}
The QUE assumption made in Theorem \ref{thmnogapRobin} is very strong, as already discussed. It is true in the one-dimensional case but up to now no example of a multi-dimensional domain satisfying QUE is known.

Anyway, we are able to prove that the conclusion of Theorem \ref{thmnogapRobin} holds true in a domain which is a tensorized version of a one-dimensional domain. Indeed, consider either the domain $\Omega=\mathbb{T}^n$ (flat torus), or the domain $\Omega=[0,\pi]^n$, the Euclidean $n$-dimensional square, with  Dirichlet boundary conditions, or mixed Dirichlet-Neumann boundary conditions with either Dirichlet or Neumann condition on every full edge of the hypercube. 
The normalized eigenfunctions of $A$ are then
$$
\phi_{j_1 \ldots j_n}(x_1,\ldots,x_n)=\prod_{k=1}^n\phi_{j_k}(x_k),
$$
for all $(j_1,\ldots,j_n)\in(\N^*)^n$, where the $\phi_j$'s are the normalized eigenfunctions of the corresponding one-dimensional case, i.e., $\phi_{j}(x)=\sqrt{\frac{2}{\pi}}\sin (\pi x)$ or $\phi_{j}(x)=\sqrt{\frac{2}{\pi}}\cos (\pi x)$ for every $x\in [0,\pi]$.
Obviously, $\Omega$ does not satisfy QUE (nor QE), but satisfies WQE, and moreover the eigenfunctions $\phi_{j_1 \ldots j_n}$ are uniformly bounded in $L^\infty(\Omega)$.
Let us prove however that the equality \eqref{nogapRobin} holds. More precisely, one has the following result.

\begin{proposition}\label{propnogapPoidsHypercube}
There exists $L_{0}\in (0,1)$ and $N_{0}\in \N^*$ such that 
\begin{equation}\label{nogapPoidsHypercube}
\max_{\chi_\omega\in\mathcal{U}_L}\tilde J(\chi_\omega)=\max_{\chi_\omega\in\mathcal{U}_L}\tilde J_N(\chi_\omega) ,
\end{equation}
for every $L\in [L_0,1)$ and every $N\geq N_{0}$.
\end{proposition}

\begin{proof}
The proof follows the same lines as the one of Theorem \ref{thmnogapRobin}. Nevertheless, the inequality \eqref{defN0Robin} may not hold whenever QUE is not satisfied and has to be questioned. In the specific case under consideration, \eqref{defN0Robin} is replaced with the following assertion: 
there exists $N_{0}\in\N^*$ such that for $\varepsilon>0$, there exists $L_{0}\in (0,1)$ such that
$$
\gamma_{j_1 \ldots j_n}\int_{\Omega}a^\infty (x)\phi_{j_1 \ldots j_n}(x)^2\, dx \geq L-\varepsilon ,
$$
for every $L\in [L_0,1)$ and for all $(j_1,\ldots,j_n)\in(\N^*)^n$ such that $\min(j_1,\ldots,j_n)\geq N_0$. This assertion indeed follows from the following general lemma.

\begin{lemma}\label{lemmaPTZHUM}
Let $\rho\in L^\infty(\Omega,\R_+)$ be such that $\int_\Omega \rho(x)\, dx>0$. Then
$$
\inf_{(j_1,\ldots,j_n)\in(\N^*)^n}\int_\Omega \rho(x)\phi_{j_1 \ldots j_n}(x)^2\, dx\geq F^{[n]} \left(\int_{\Omega} \rho(x)\, dx\right)>0,
$$
where $F(x)=\frac{1}{\pi}(x-\sin x)$ for every $x\in [0,\pi]$ and $F^{[n]}=F\circ\dots\circ F$ ($n$ times).
\end{lemma}

This lemma itself easily follows from \cite[Lemma 6]{PTZ_HUM} (case $n=1$) and from an induction argument.
%
%
\end{proof}
\end{remark}

We end this section by providing several numerical simulations based on the modal approximation of this problem for $\Omega=[0,\pi]^2$, the Euclidean square, with Dirichlet boundary conditions on $\partial\Omega\cap (\{x_{2}=0\}\cup \{x_{2}=\pi\}$ and Neumann boundary conditions on the rest of the boundary.
Note that we are then in the framework of Remark \ref{rem_extension_tensorized}, and hence the conclusion of Proposition \ref{nogapPoidsHypercube} holds true.
The normalized eigenfunctions of $A$ are then
$$
\phi_{j,k}(x_{1},x_{2})=\frac{2}{\pi}\sin (jx_{1})\cos (kx_{2}),
$$
for all $(x_{1},x_{2})\in [0,\pi]^2$.
The eigenvalue $\lambda_{j,k}$ associated with the eigenfunction $\phi_{j,k}$ is $\lambda_{j,k}=j^2+k^2$.
Let $N\in\N^*$. As in Section \ref{sec5.2}, we use an interior point line search filter method to solve the spectral approximation of the second problem
$\sup_{\chi_{\omega}\in\mathcal{U}_{L}}\tilde J_{N}(\chi_{\omega})$, with
$$\tilde J_{N}(\chi_{\omega})=\min_{1\leq j,k\leq N}\int_{0}^{\pi}\!\int_{0}^{\pi}\chi_{\omega}(x_{1},x_{2})\phi_{j,k}(x_{1},x_{2})^{2}\, dx_{1} dx_{2}.$$
Some numerical simulations are provided on Figures \ref{figpb2DN} and \ref{figpb2DNbis}.
On Figure \ref{figpb2DN}, the optimal domains are represented for $L\in \{0.2,0.4,0.6\}$. In the three first cases, the number of connected components of the optimal set seems to increase with $N$. 
The numerical results provided in the case $L=0.9$ on Figure \ref{figpb2DNbis} illustrate the conclusion of Proposition \ref{propnogapPoidsHypercube}, showing clear evidence of the stationarity feature proved in this proposition.

\begin{figure}[h!]
\begin{center}
\includegraphics[width=4.9cm]{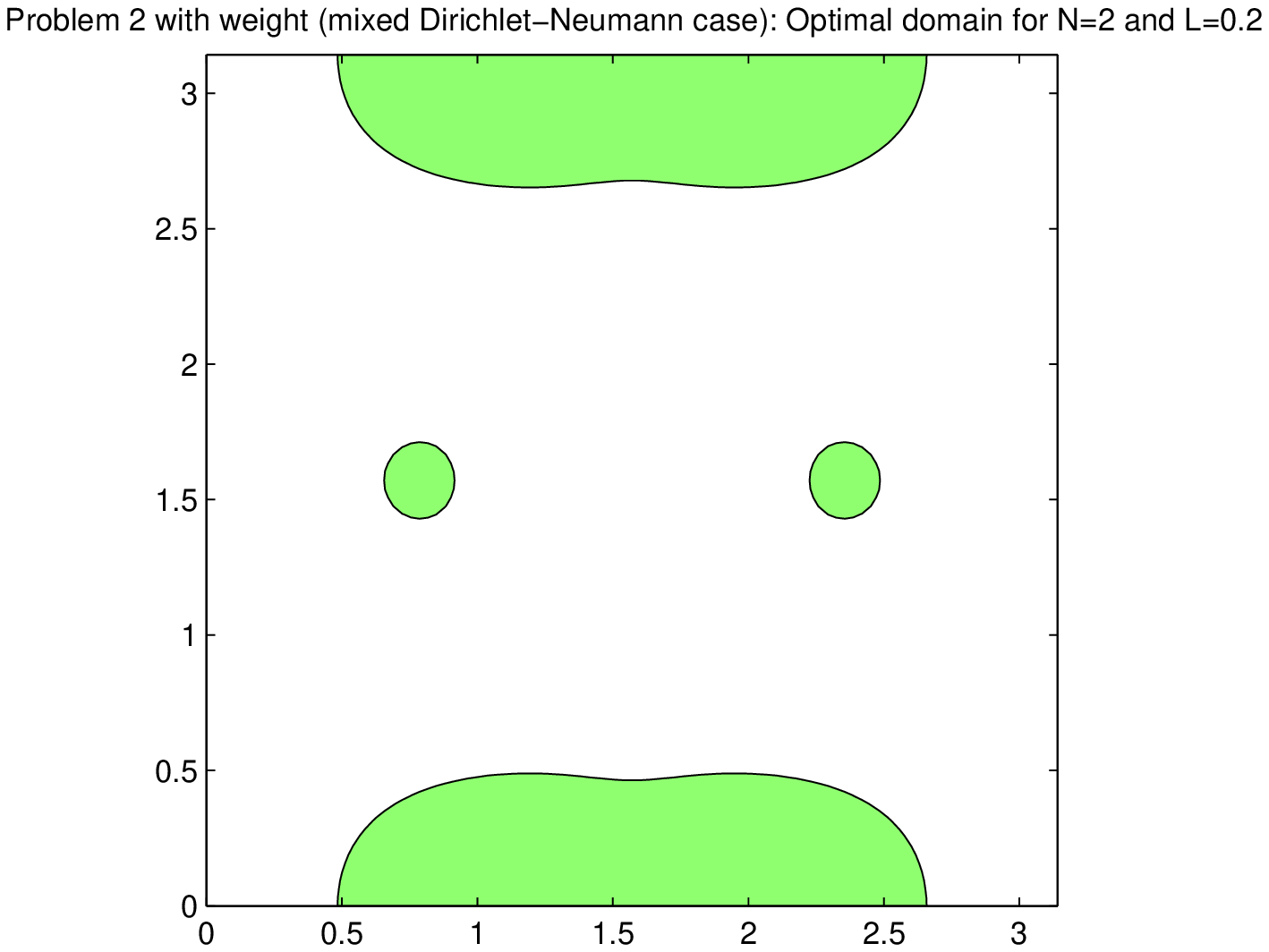}
\includegraphics[width=4.9cm]{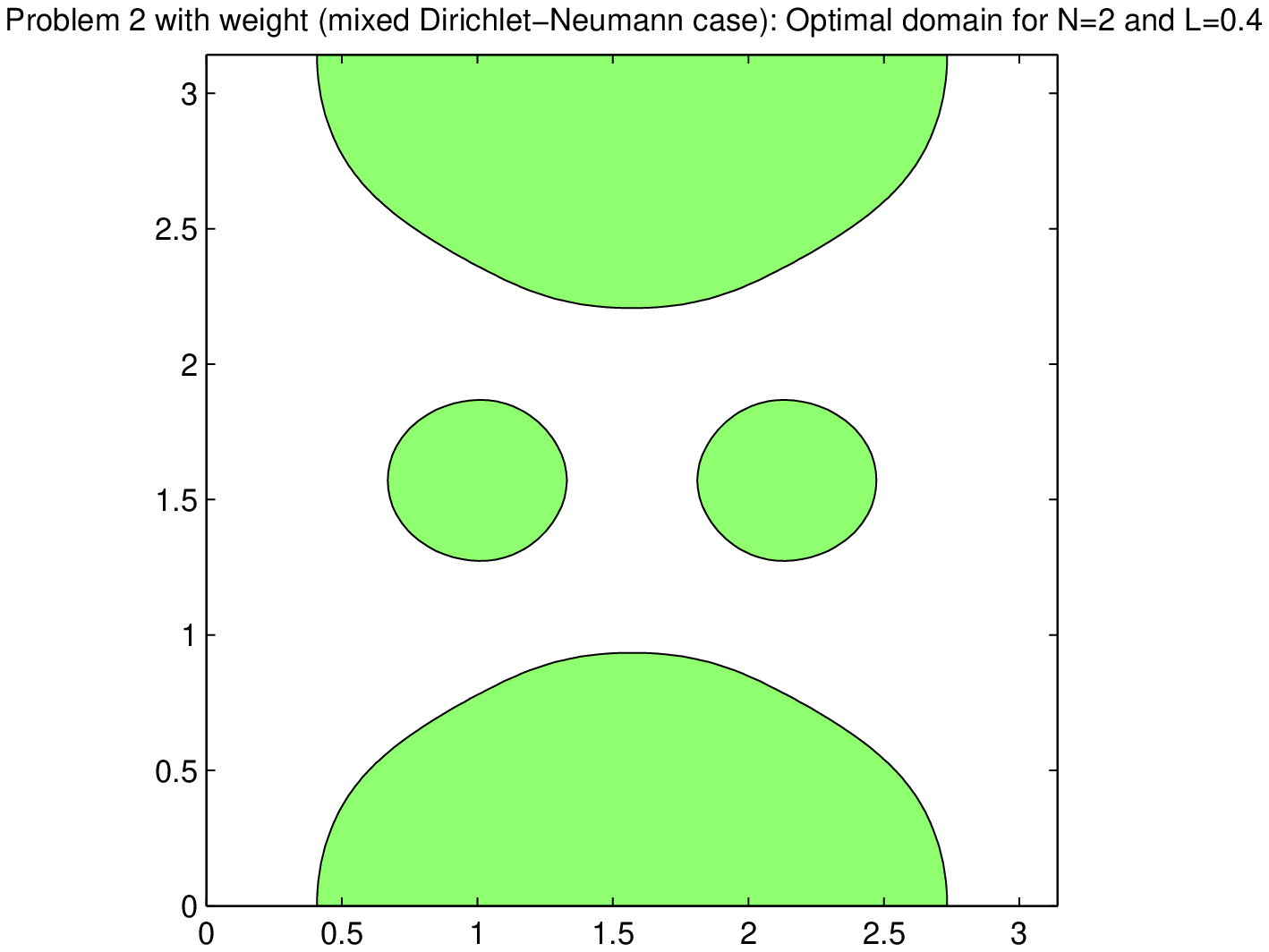}
\includegraphics[width=4.9cm]{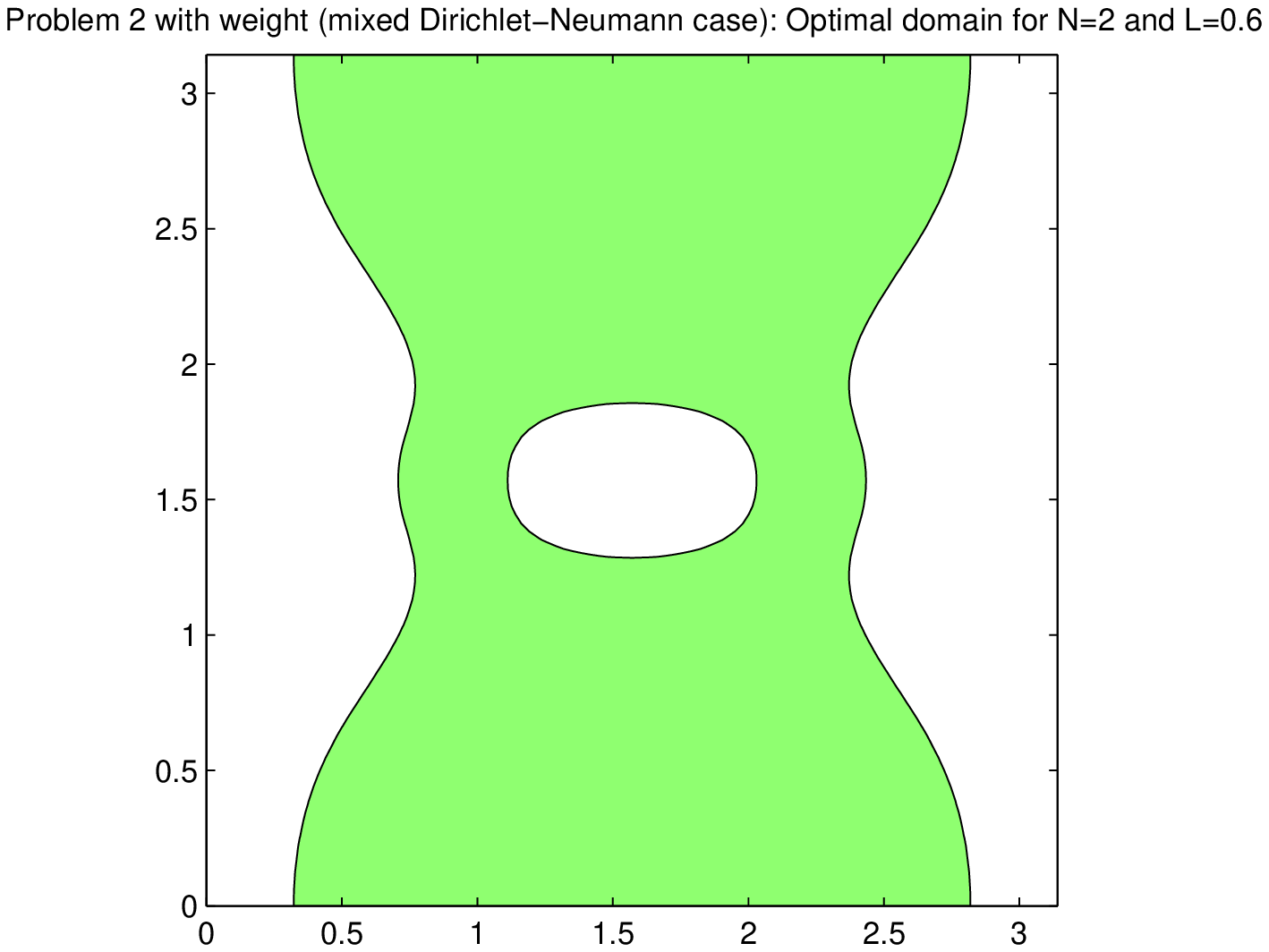}
\includegraphics[width=4.9cm]{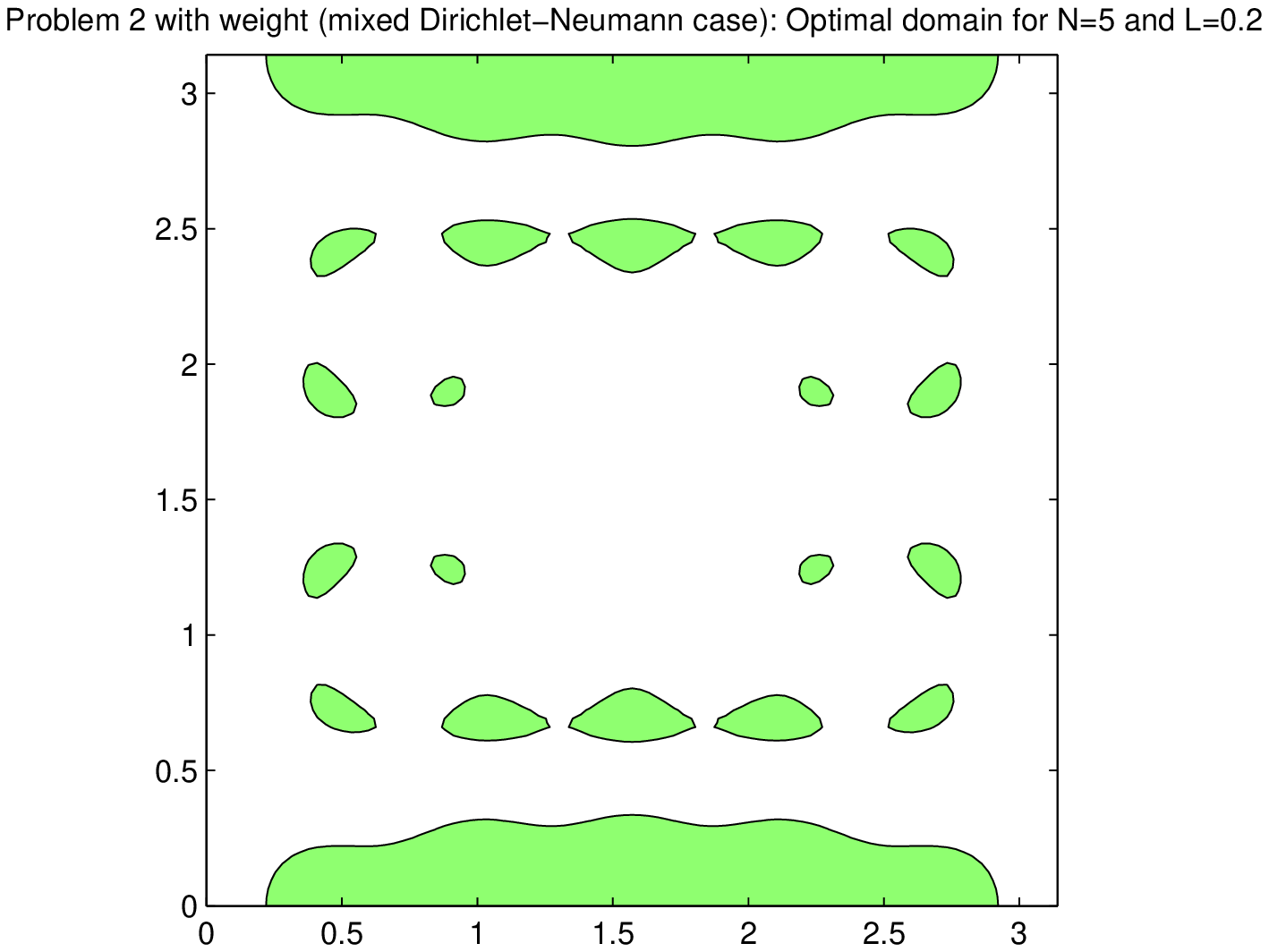}
\includegraphics[width=4.9cm]{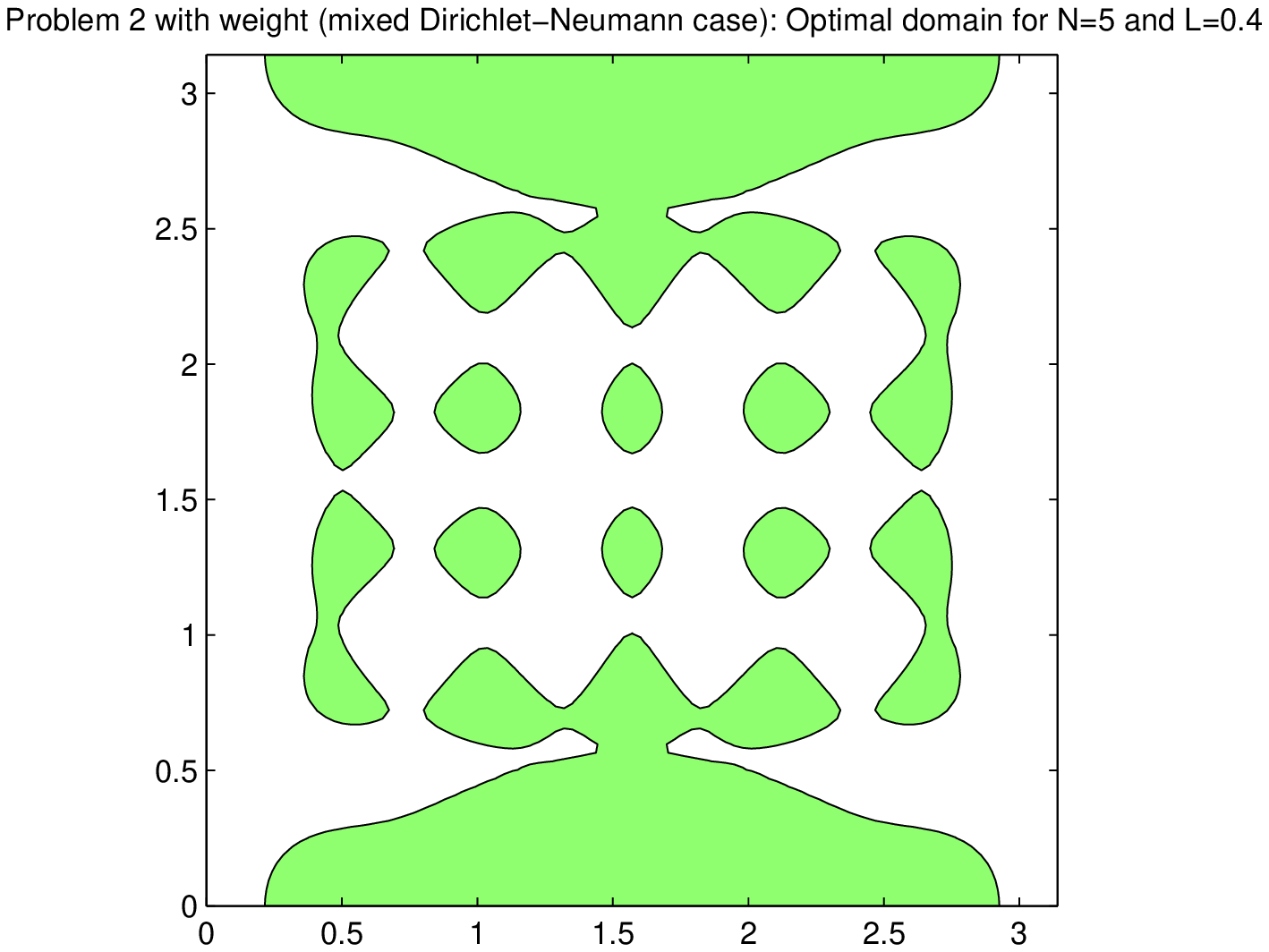}
\includegraphics[width=4.9cm]{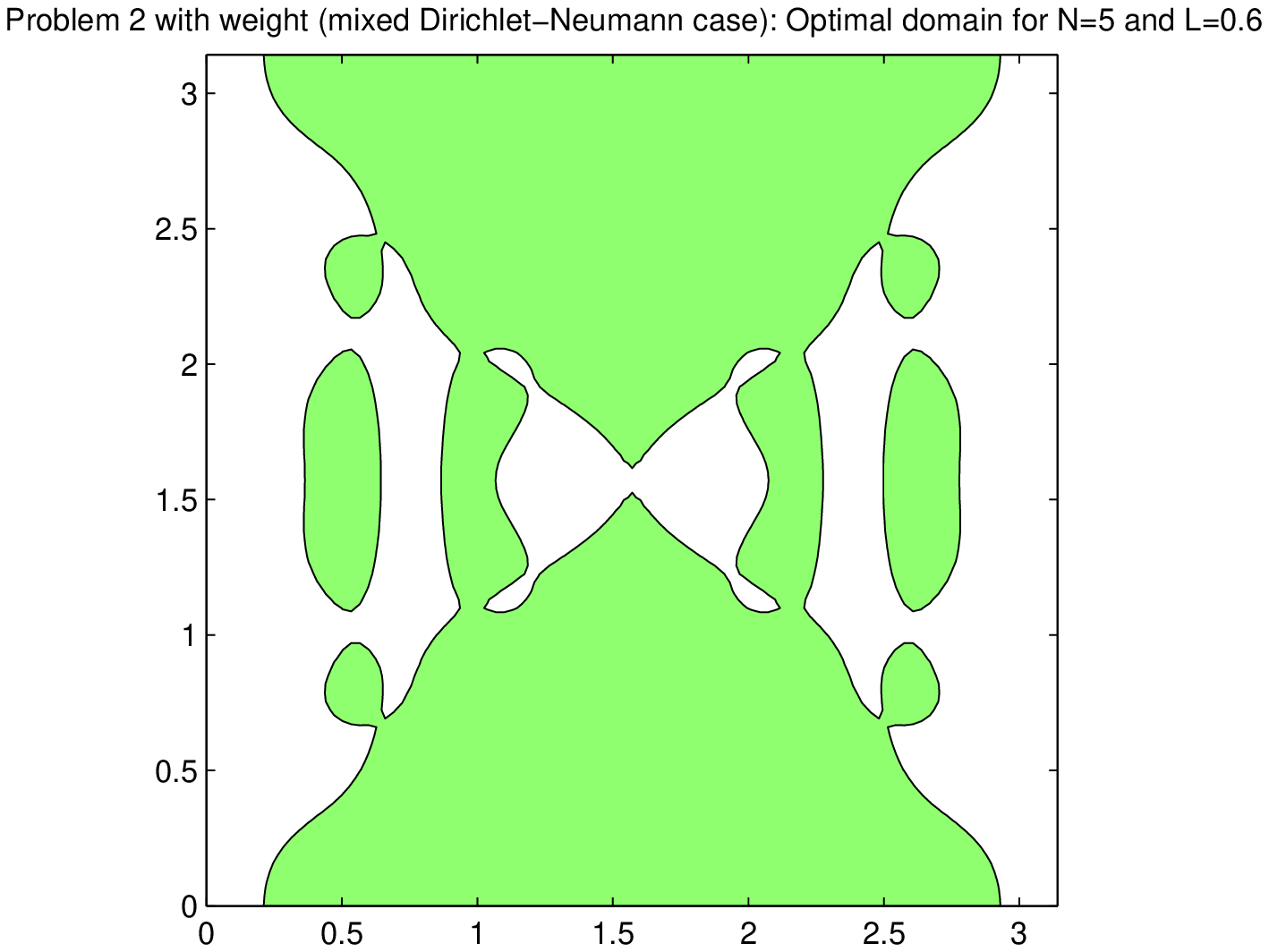}
\includegraphics[width=4.9cm]{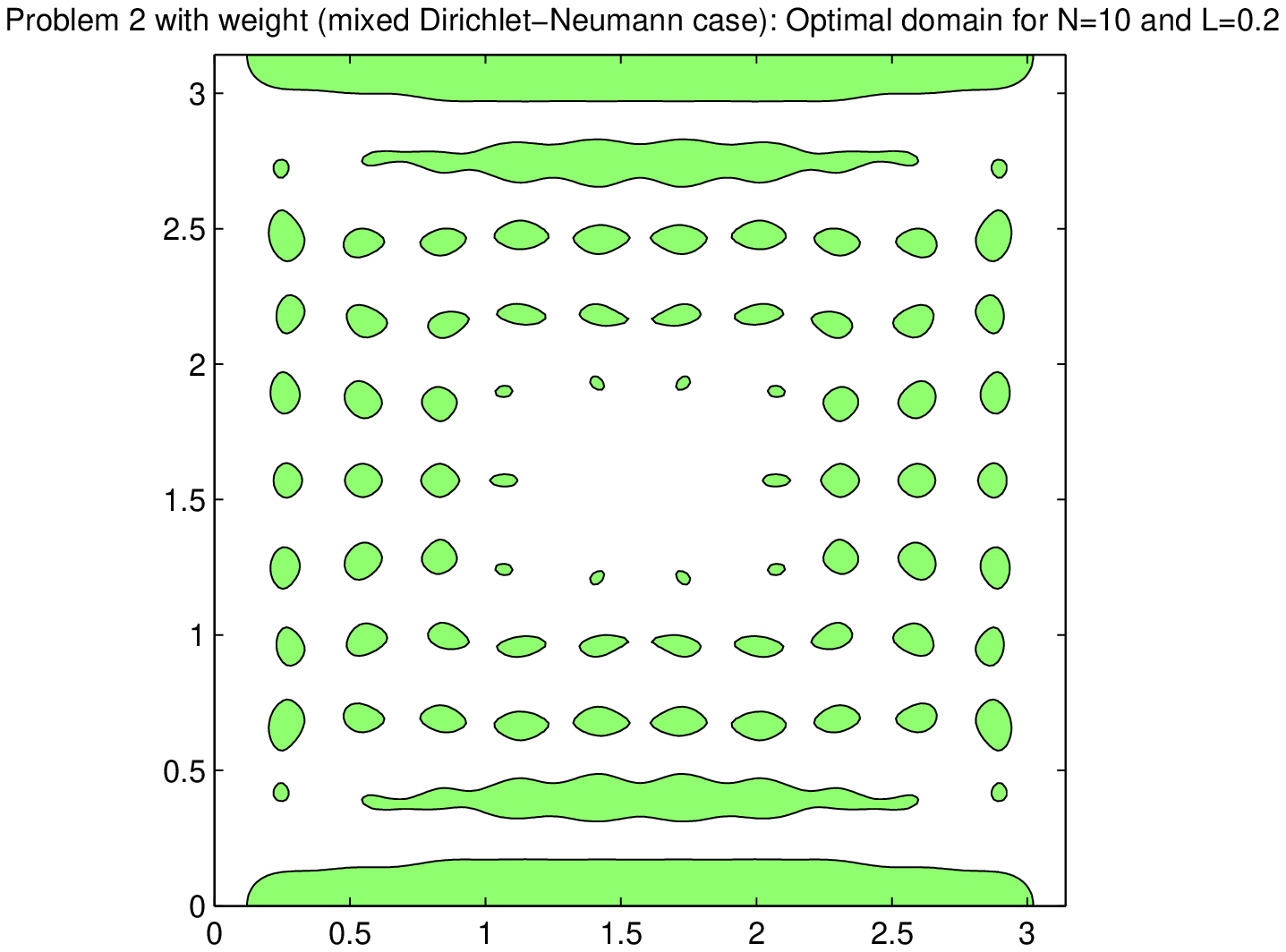}
\includegraphics[width=4.9cm]{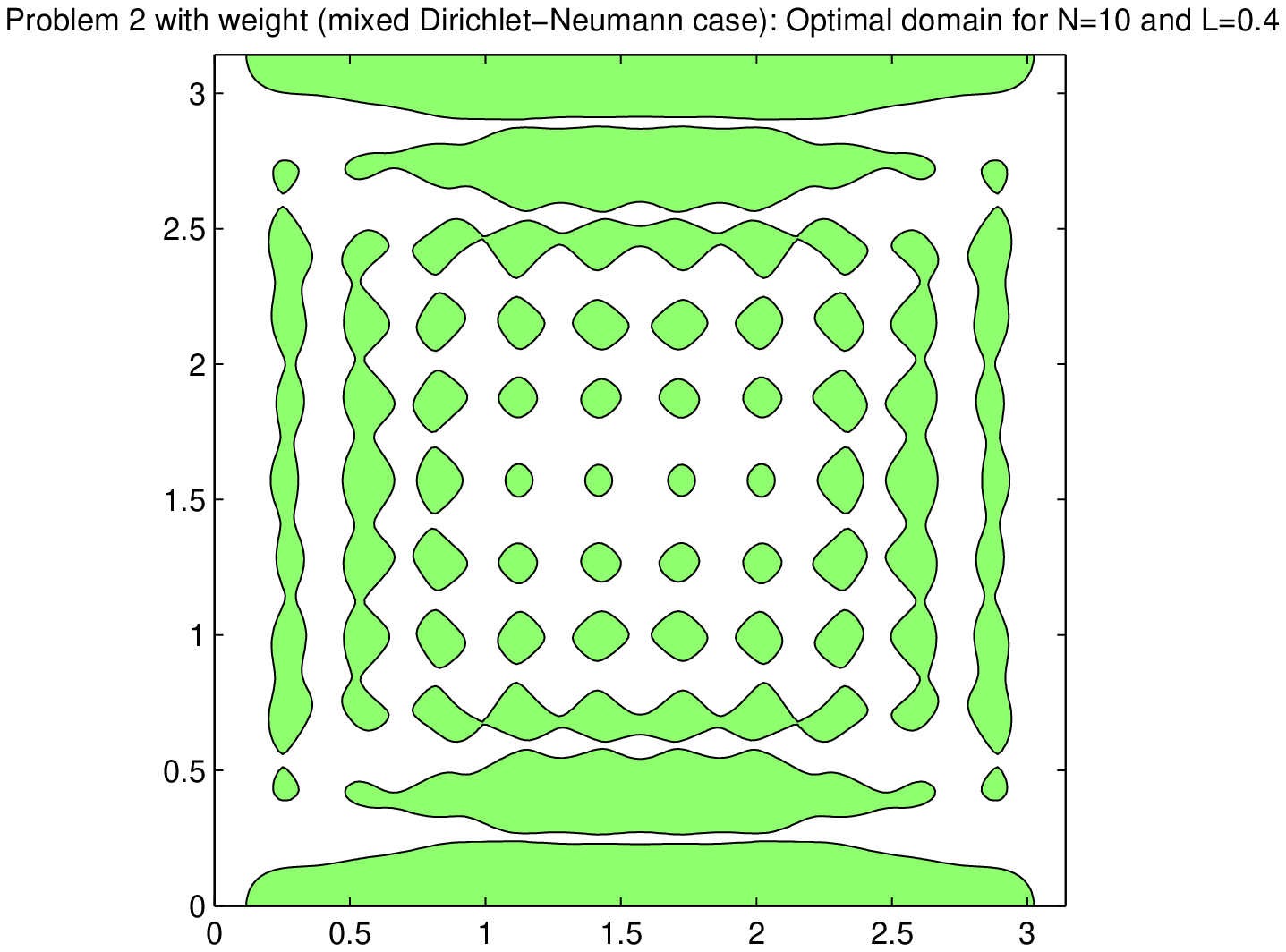}
\includegraphics[width=4.9cm]{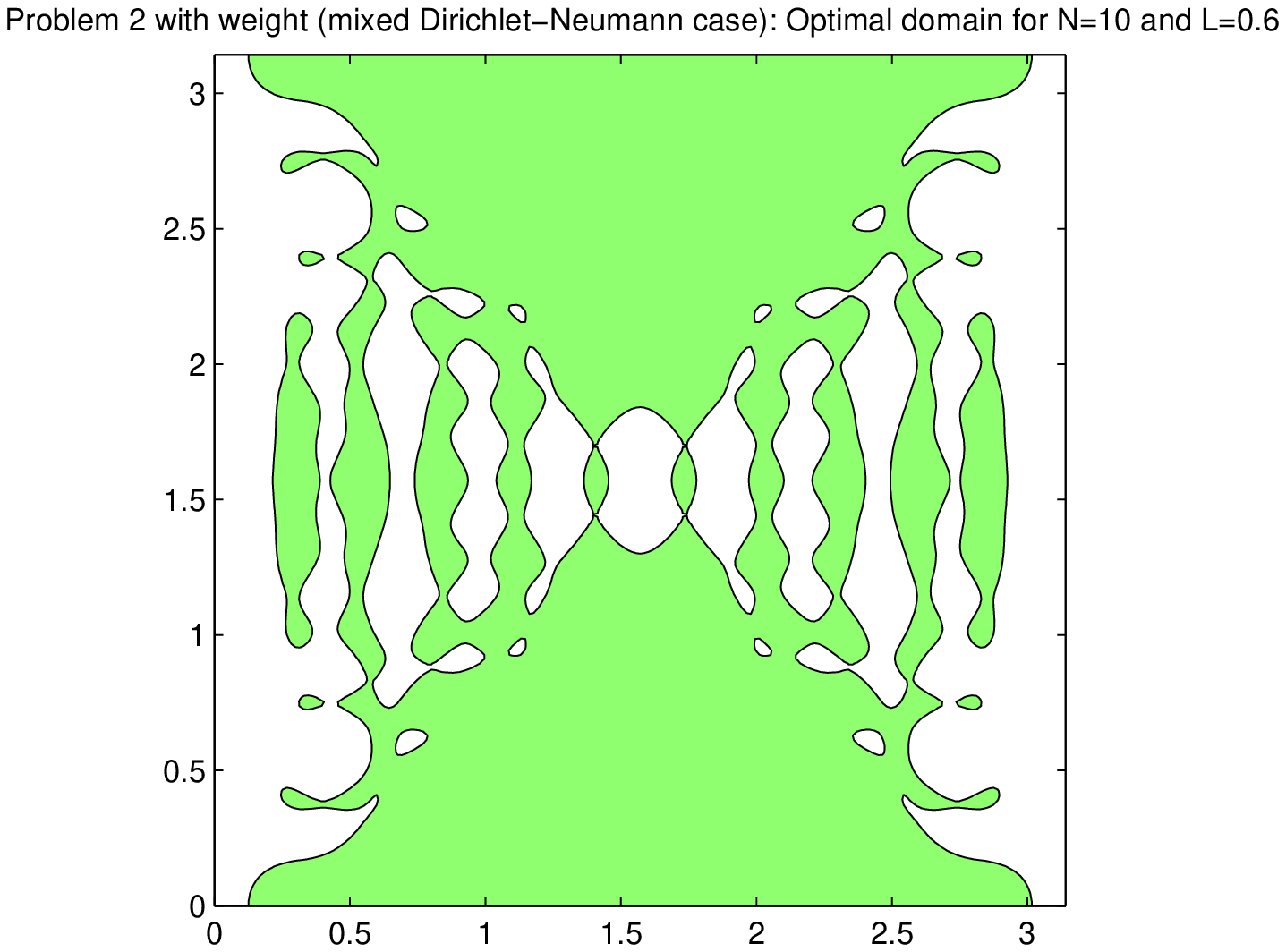}
\includegraphics[width=4.9cm]{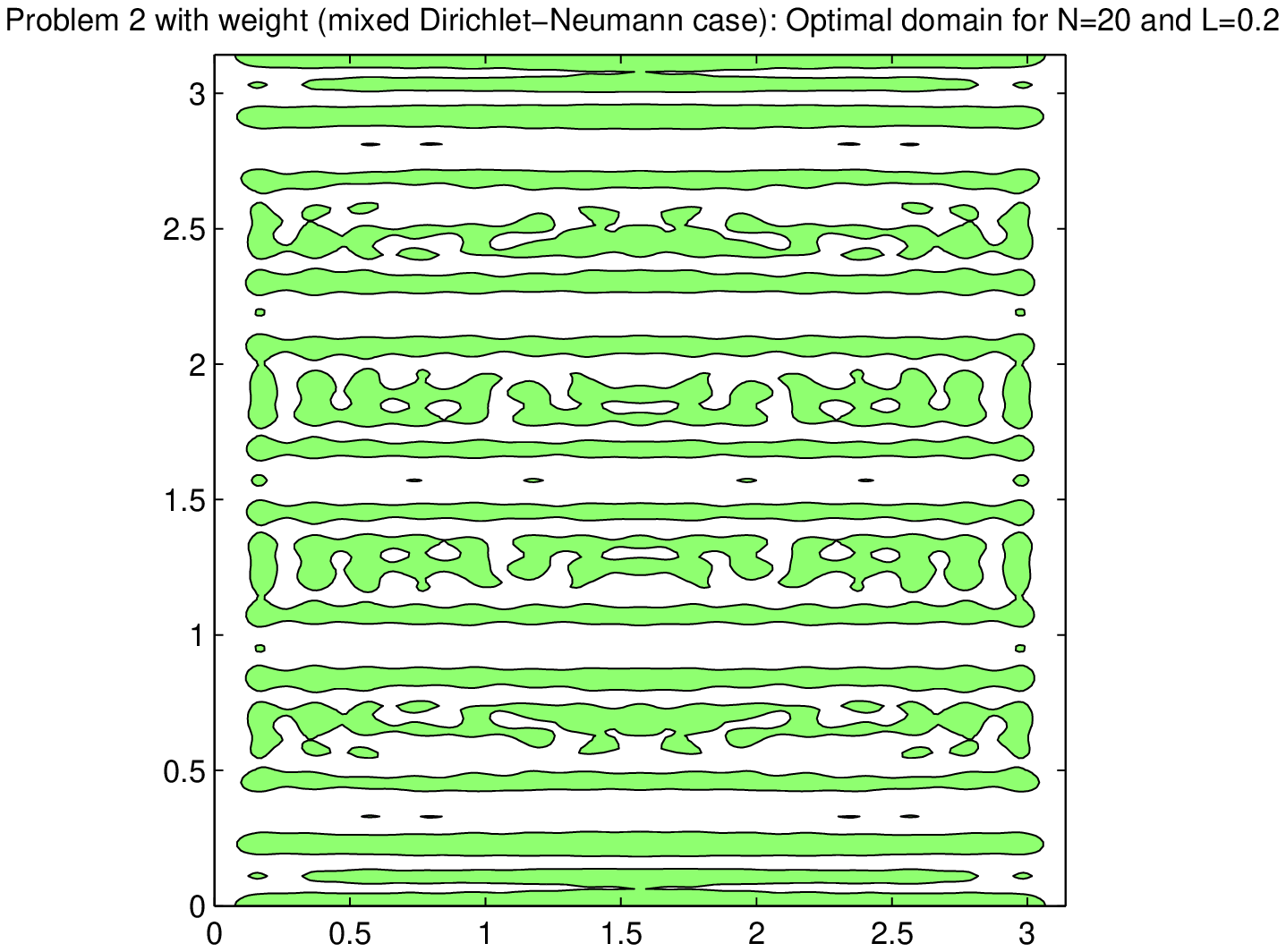}
\includegraphics[width=4.9cm]{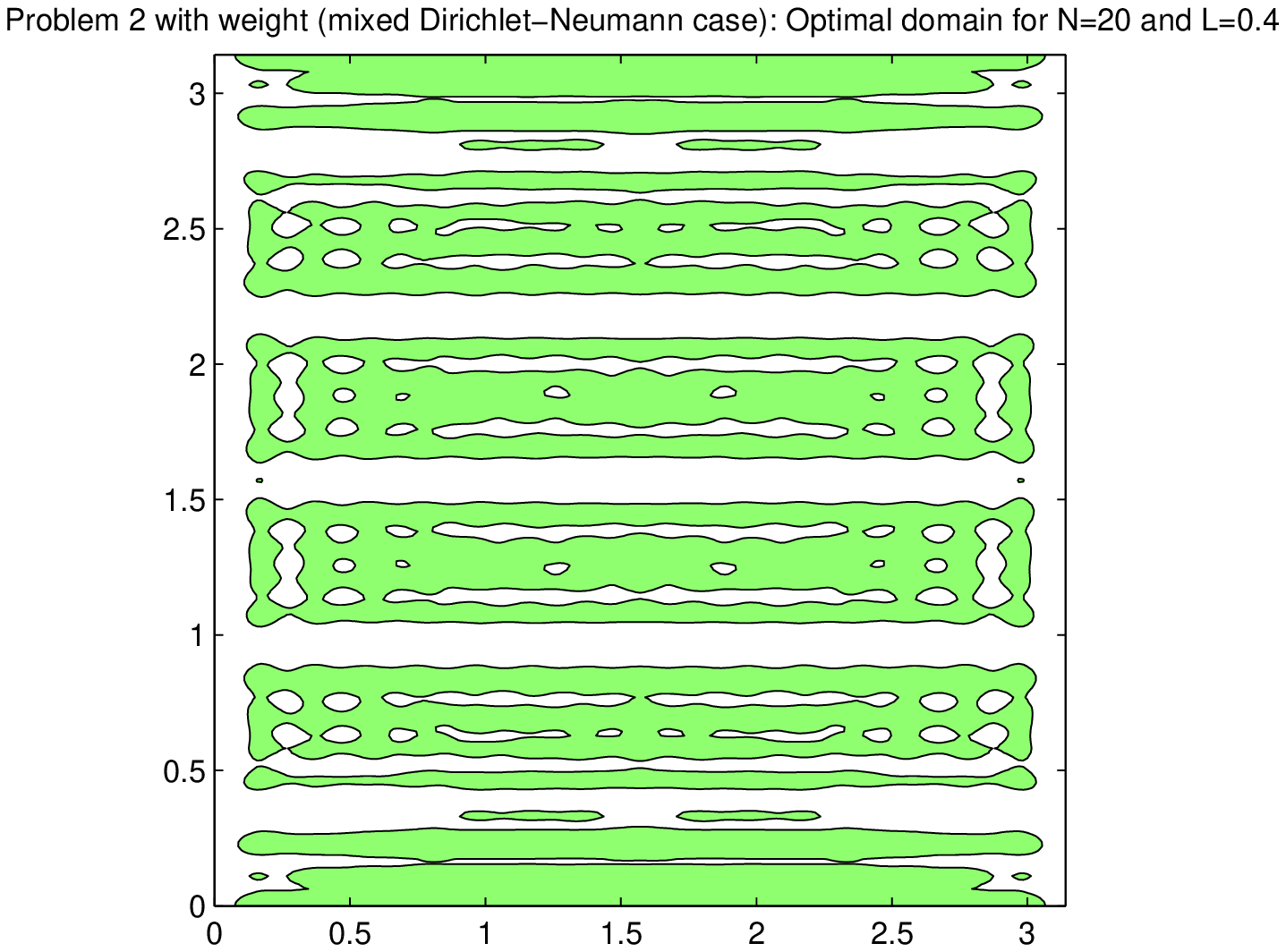}
\includegraphics[width=4.9cm]{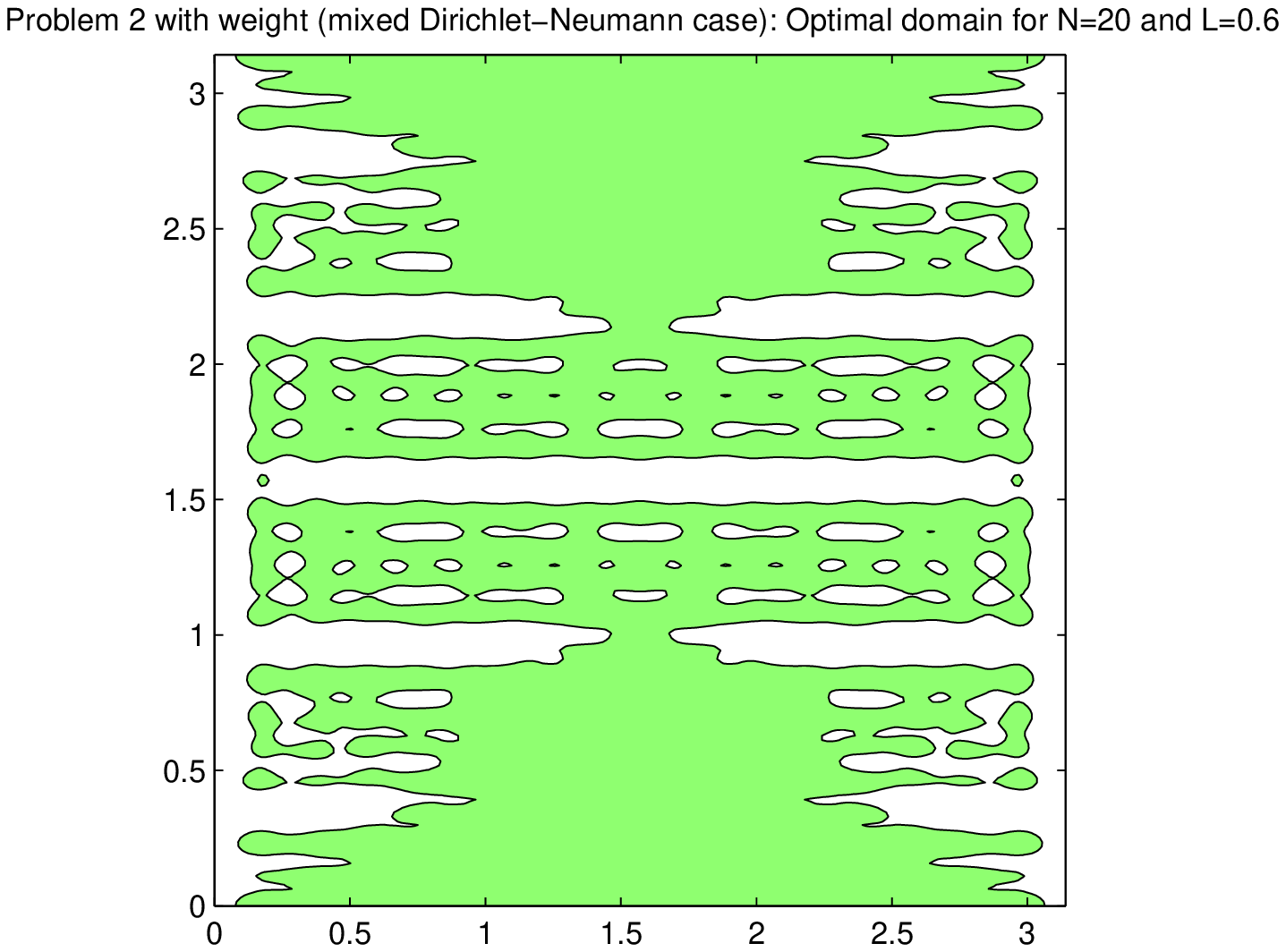}
\caption{On this figure, $\Omega=[0,\pi]^{2}$, with mixed Dirichlet-Neumann boundary conditions. Line 1, from left to right: optimal domain (in green)  for $N=2$ (4 eigenmodes) and $L\in \{0.2,0.4,0.6\}$. Line 2, from left to right: optimal domain (in green) for $N=5$ (25 eigenmodes) and $L\in \{0.2,0.4,0.6\}$. Line 3, from left to right: optimal domain (in green) for $N=10$ (100 eigenmodes) and $L\in \{0.2,0.4,0.6\}$. Line 4, from left to right: optimal domain (in green) for $N=15$ (225 eigenmodes) and $L\in \{0.2,0.4,0.6\}$
}\label{figpb2DN}
\end{center}
\end{figure}

\begin{figure}[h!]
\begin{center}
\includegraphics[width=4.9cm]{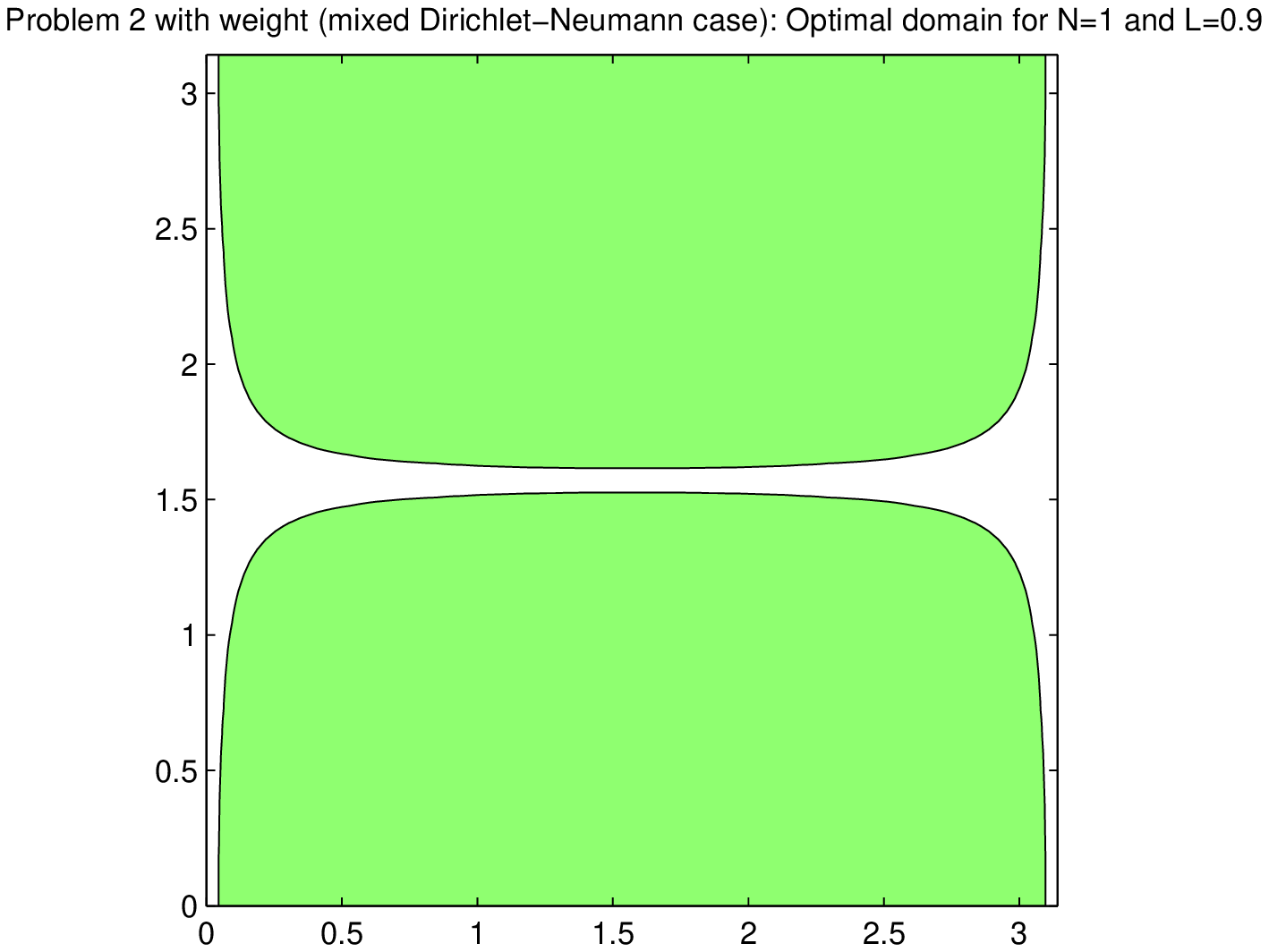}
\includegraphics[width=4.9cm]{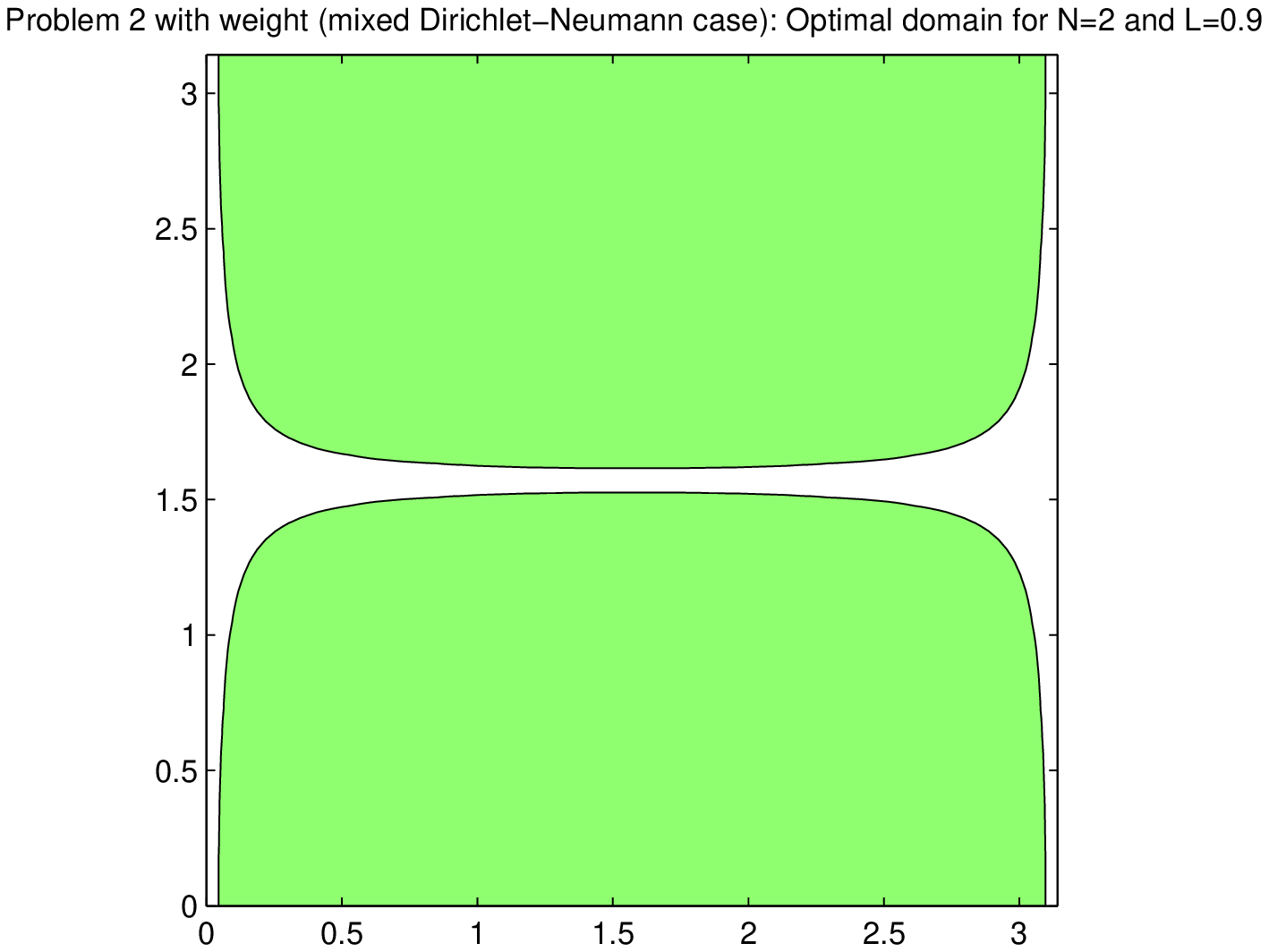}\\
\includegraphics[width=4.9cm]{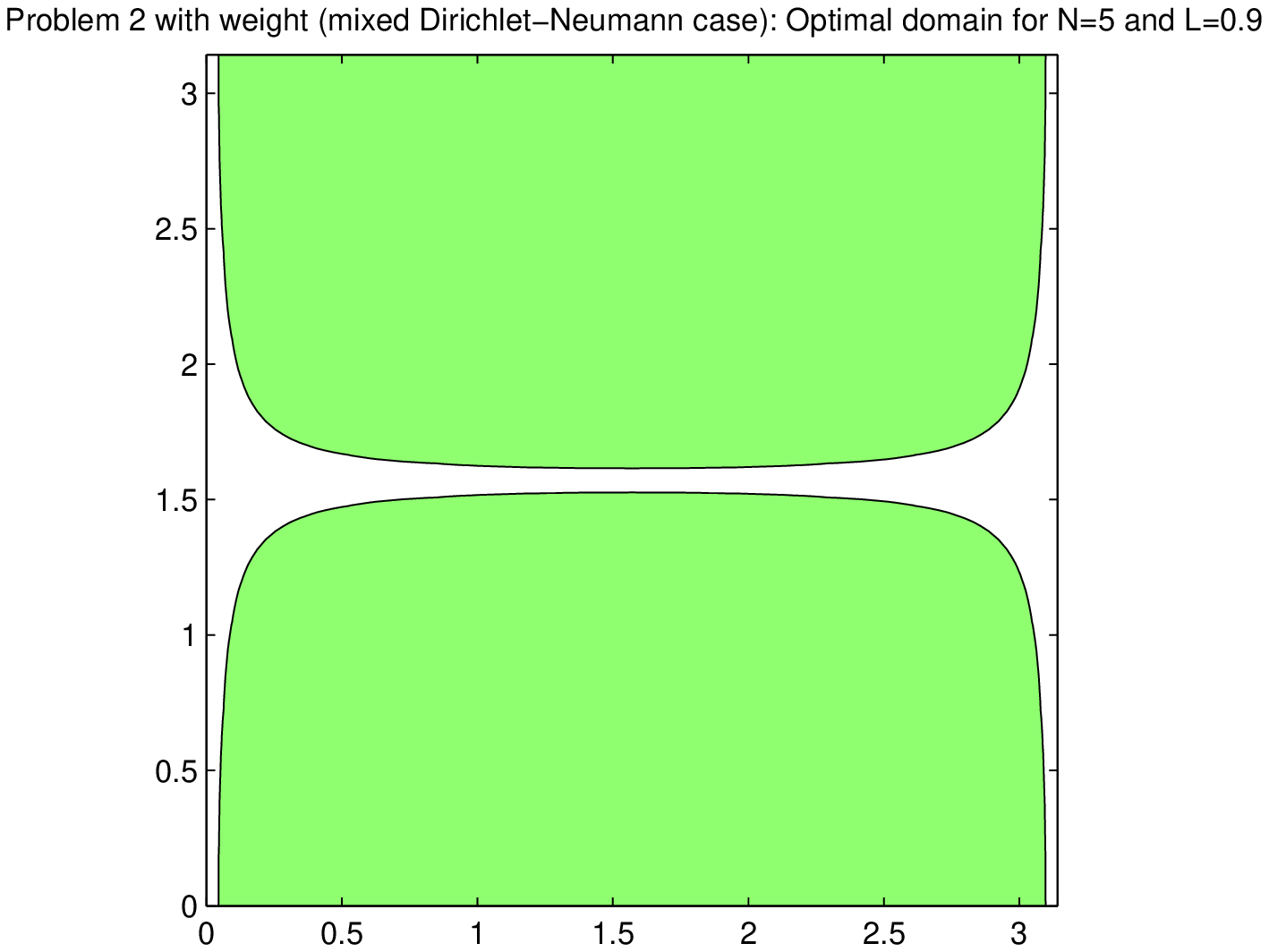}
\includegraphics[width=4.9cm]{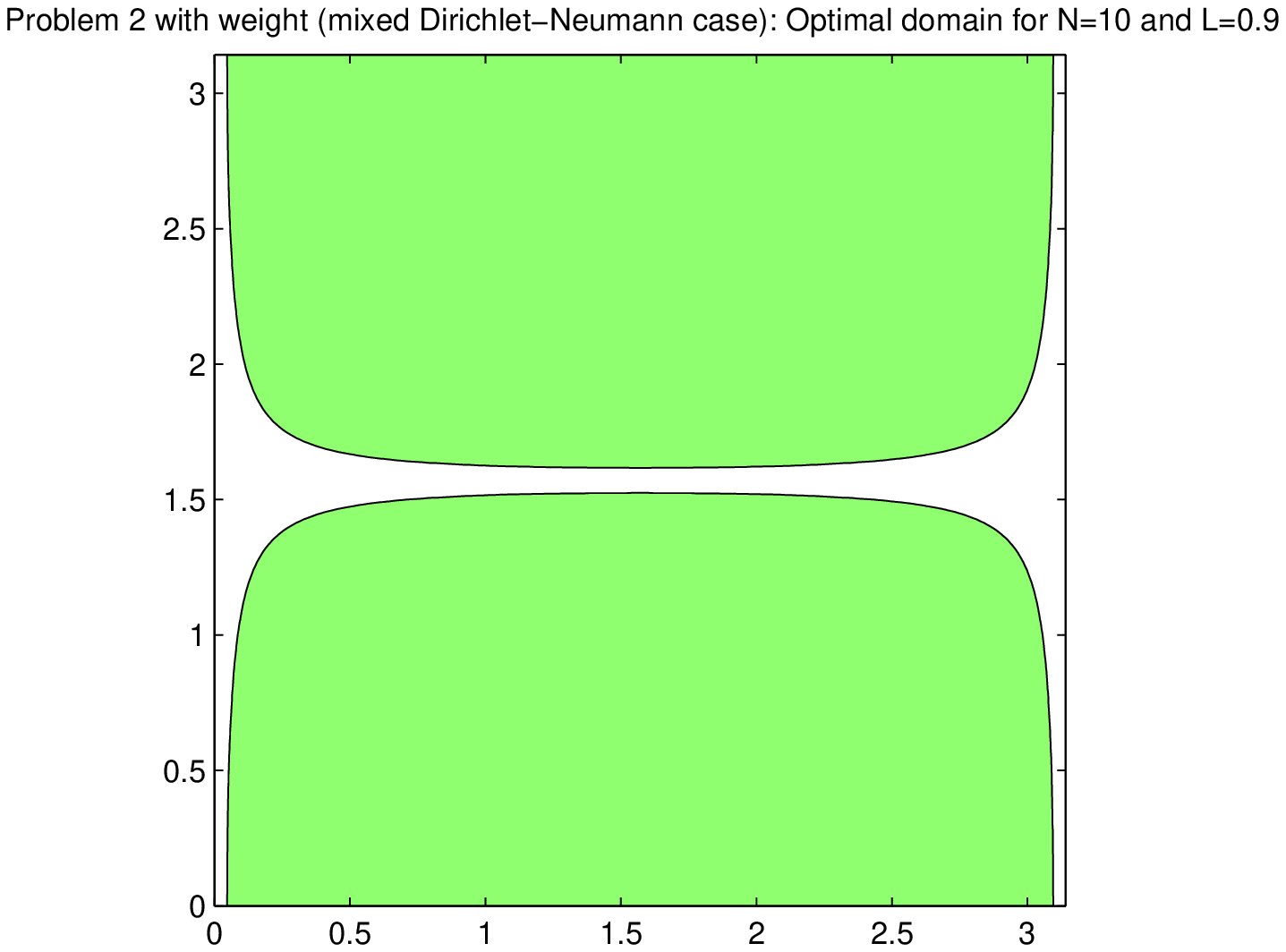}
\caption{On this figure, $\Omega=[0,\pi]^{2}$, with mixed Dirichlet-Neumann boundary conditions, and $L=0.9$. Line 1, from left to right: optimal domain (in green) for $N\in \{1,2\}$. Line 2, from left to right: optimal domain (in green) for $N\in \{5,10\}$}\label{figpb2DNbis}
\end{center}
\end{figure}

\subsection{Optimal location of internal controllers for wave and Schr\"odinger equations}\label{sec6.5}
In this section, we investigate the question of determining the shape and location of the control domain for wave or Schr\"odinger equations that minimizes the $L^2$ norm of the controllers realizing null controllability. In particular, we explain why this optimization problem is exactly equivalent to the problem of maximizing the observability constant. For the sake of simplicity, we will only deal with the wave equation, the Schr\"odinger case being easily adapted from that case. Also, without loss of generality we restrict ourselves to Dirichlet boundary conditions.

Consider the internally controlled wave equation on $\Omega$ with Dirichlet boundary conditions
\begin{equation}\label{waveEqCont}
\begin{array}{ll}
\partial_{tt}y(t,x)-\triangle y(t,x)=h_\omega(t,x) , & (t,x)\in (0,T)\times \Omega,\\
y(t,x)=0, & (t,x)\in [0,T]\times \partial\Omega,\\
y(0,x)=y^0(x), \ \partial_t y(0,x)=y^1(x), & x\in \Omega,
\end{array}
\end{equation}
where $h_\omega$ is a control supported in $[0,T]\times \omega$ and $\omega$ is a measurable subset of $\Omega$.
Note that the Cauchy problem \eqref{waveEqCont} is well posed for all initial data $(y^0,y^1)\in H^1_0(\Omega,\C)\times L^2(\Omega,\C)$ and every $h_\omega\in L^2((0,T)\times\Omega,\C)$, and its solution $y$ belongs to $C^0(0,T;H^1_0(\Omega,\C))\cap C^1(0,T;L^2(\Omega,\C))\cap C^2(0,T;H^{-1}(\Omega,\C))$. The exact null controllability problem settled in these spaces consists of finding a control $h_\omega$ steering the control system \eqref{waveEqCont} to
\begin{equation}\label{controlCondition}
y(T,\cdot)=\partial_t y(T,\cdot)=0.
\end{equation}
It is well known that, for every subset $\omega$ of $\Omega$ of positive measure, the exact null controllability problem is by duality equivalent to the fact that the observability inequality
\begin{equation}\label{ineq1}
C\Vert (\phi^0, \phi^1) \Vert_{L^2(\Omega,\C)\times H^{-1}(\Omega,\C)}^2
\leq \int_0^T\int_\omega |\phi(t,x)|^2 \, dV_g \, dt,
\end{equation}
holds, for all $ (\phi^0, \phi^1)\in L^2(\Omega,\C )\times H^{-1}(\Omega,\C )$, for a positive constant $C$ (only depending on $T$ and $\omega$), where $\phi$ is the (unique) solution of the adjoint system
\begin{equation}\label{AdjointSystem}
\begin{array}{ll}
\partial_{tt}\phi(t,x)-\triangle\phi(t,x)=0,& (t,x)\in (0,T)\times \Omega,\\
\phi(t,x)=0, & (t,x)\in [0,T]\times \partial \Omega,\\
\phi(0,x)=\phi^0(x),\ \partial_{t}\phi(0,x)=\phi^1(x), & x\in \Omega .
\end{array}
\end{equation}
The Hilbert Uniqueness Method (HUM, see \cite{lions2,lions}) provides a way to design the unique control solving the control problem \eqref{waveEqCont}-\eqref{controlCondition} and having moreover a minimal $L^2((0,T)\times\Omega,\C)$ norm. This control is referred to as the HUM control and is characterized as follows. Define the HUM functional $J_\omega$ by
\begin{equation}\label{defJomega}
J_\omega(\phi^0,\phi^1)=\frac{1}{2}\int_0^T\int_\omega\phi(t,x)^2 \, dV_g \, dt - \langle\phi^1,y^0\rangle_{H^{-1},H^1_0} + \langle\phi^0,y^1\rangle_{L^2}.
\end{equation}
The notation $\langle\cdot,\cdot\rangle_{H^{-1},H^1_0}$ stands for the duality bracket between $H^{-1}(\Omega,\C)$ and $H^1_0(\Omega,\C)$, and the notation $\langle\cdot,\cdot\rangle_{L^2}$ stands for the usual scalar product of $L^2(\Omega,\C)$.
If \eqref{ineq1} holds then the functional $J_\omega$ has a unique minimizer (still denoted $(\phi^0,\phi^1)$) in the space $L^2(\Omega,\C)\times H^{-1}(\Omega,\C)$, for all $(y^0,y^1)\in H^1_0(\Omega,\C)\times L^2(\Omega,\C)$. The HUM control $h_\omega$ steering $(y^0,y^1)$ to $(0,0)$ in time $T$ is then given by
\begin{equation}\label{uHUM}
h_\omega(t,x)=\chi_\omega(x)\phi(t,x),
\end{equation}
for almost all $(t,x)\in(0,T)\times \Omega$, where $\phi$ is the solution of \eqref{AdjointSystem} with initial data  $(\phi^0,\phi^1)$ minimizing $J_\omega$. 

The HUM operator $\Gamma_\omega$ is defined by
$$
\Gamma_\omega : \begin{array}[t]{rcl}
H_0^1(\Omega,\C)\times L^2(\Omega,\C) & \longrightarrow & L^2((0,T)\times\Omega,\C)\\
(y^0,y^1) & \longmapsto & h_\omega
\end{array}
$$ 

\begin{quote}
\noindent{\bf Optimal design control problem.}
We investigate the problem of minimizing the norm of the operator $\Gamma_\omega$
\begin{equation}\label{defGammaomega}
\Vert \Gamma_\omega\Vert = \sup \left\{\frac{\Vert h_\omega\Vert_{L^2((0,T)\times\Omega,\C)}}{\Vert (y^0,y^1)\Vert_{H_0^1(\Omega,\C)\times L^2(\Omega,\C)}}\mid (y^0,y^1) \in H_0^1(\Omega,\C)\times L^2(\Omega,\C) \setminus\{(0,0)\} \right\}
\end{equation}
over the set $\mathcal{U}_L$.
\end{quote}

Here, we formulate the optimal design control problem in terms of minimization of the operator norm of $\Gamma_\omega$ in order to discard the dependence with respect to the initial data $(y^0,y^1)$ and improve the  robustness of the cost function.
The next result establishes that the problems \eqref{defJ} and \eqref{defGammaomega} are equivalent, and hence that the approach developed in Sections \ref{sec2} and \ref{solvingpb2obs} is also well adapted to this optimal design control problem. Simultaneously, we generalize \cite{PTZ_HUM} where similar issues were investigated in the one-dimensional case.

\begin{proposition}
Let $T>0$ and let $\omega$ be measurable subset of $\Omega$. If $C_T^{(W)}(\chi_\omega)>0$ then
$$
\Vert \Gamma_\omega\Vert = \frac{1}{C_T^{(W)}(\chi_\omega)},
$$
and if $C_T^{(W)}(\chi_\omega)=0$, then $\Vert \Gamma_\omega\Vert =+\infty$.
\end{proposition}

\begin{proof}
Denote by $\phi_{\omega}$ the adjoint state solution of \eqref{AdjointSystem} whose initial data minimize the functional $J_\omega$. Then $\phi_\omega$ can be expanded as
$$
\phi_\omega(t,x)=\sum_{j=1}^{+\infty}\left(A_j^\omega e^{i\lambda_jt}+B_j^\omega e^{-i\lambda_jt}\right)\phi_j (x),
$$
where the sequences $A=(A_j^\omega)_{j\in\N^*}$ and $B=(B_j^\omega)_{j\in\N^*}$ belong to $\ell^2(\C)$ and are determined in function of the initial data $(\phi^0_\omega,\phi^1_\omega)$ minimizing $J_\omega$. Since $J_\omega$ is convex, the first-order optimality conditions for the problem of minimizing $J_\omega$ over $L^2(\Omega,\C )\times H^{-1}(\Omega,\C )$ are necessary and sufficient. In terms of Fourier coefficients, they are written as
\begin{equation}\label{firstoptimcontrol}
\Lambda_\omega (A,B)=C,
\end{equation}
where the operator $\Lambda_\omega:(\ell^2(\C))^2\rightarrow (\ell^2(\C))^2$ is defined by
$$
\Lambda_\omega (A,B)_j=\int_0^T\int_\omega \sum_{k=1}^{+\infty}(A_ke^{i\lambda_kt}+B_ke^{-i\lambda_kt})\phi_k(x)\phi_j(x)\overline{\left(\begin{array}{c}e^{i\lambda_jt}\\ e^{-i\lambda_jt}\end{array}\right)}\, dV_g\, dt ,
$$
for every $j\in\N^*$, with the notation $\Lambda_\omega (A,B)=(\Lambda_\omega (A,B)_j)_{j\in\N^*}$, and where
$$
C_j=\left(\begin{array}{c}-\langle \phi_j,\phi^1\rangle_{L^2,L^2}\\ \lambda_j\langle \phi_j,\phi^0\rangle_{H^{-1},H_0^1}\end{array}\right) ,
$$
for every $j\in\N^*$. For all $(A,B)\in  [\ell^2(\C)]^2$, one has
$$
\langle \Lambda_\omega (A,B),(A,B) \rangle_{ (\ell^2(\C))^2}=\int_0^T\int_\omega \left|\sum_{k=1}^{+\infty}(A_ke^{i\lambda_kt}+B_ke^{-i\lambda_kt})\phi_k(x)\right|^2\, dV_g\, dt,
$$
and it follows that
$$
C_T^{(W)}(\chi_\omega)\leq \frac{\langle \Lambda_\omega (A,B),(A,B) \rangle_{ (\ell^2(\C))^2}}{\Vert (A,B)\Vert^2_{ (\ell^2(\C))^2}}\leq 2T.
$$
Indeed, we obtain the left-hand side inequality by definition of the observability constant. The right-hand side one is easily obtained, writing that the integral of a nonnegative function over $\omega$ is lower than the integral of the same function over $\Omega$, which permits to use the orthogonality properties of the $\phi_j$'s. By duality, we deduce that $\Lambda_\omega$ is a continuous symmetric invertible operator from $(\ell^2(\C))^2$ to $(\ell^2(\C))^2$. 
Note that
\begin{equation*}
\Vert \Gamma_\omega\Vert =  \sup_{C\in (\ell^2(\R))^2\setminus\{0\}}\frac{\langle \Lambda_{\omega}^{-1}(C),C\rangle _{(\ell^2(\R))^2}}{\Vert C\Vert_{(\ell^2(\R))^2}^2}\\
  =  \sup_{C\in (\ell^2(\R))^2\setminus\{0\}}\frac{\Vert \Lambda_{\omega}^{-1/2}(C)\Vert^2 _{(\ell^2(\R))^2}}{\Vert C\Vert_{(\ell^2(\R))^2}^2},
\end{equation*}
where $ \Lambda_{\omega}^{-1/2}$ denotes the square root of the operator $\Lambda_{\omega}^{-1}$. Setting $\varphi=\Lambda_{\omega}^{-1/2}(C)$, one computes
\begin{eqnarray*}
\Vert \Gamma_\omega\Vert  & = & \sup_{\varphi \in (\ell^2(\R))^2\setminus\{0\}}\frac{\Vert \varphi \Vert_{(\ell^2(\R))^2}^2}{\Vert \Lambda_{\omega}^{1/2}(\varphi)\Vert^2 _{(\ell^2(\R))^2}}
=\frac{1}{\inf \left\{\frac{\Vert \Lambda_{\omega}^{1/2}(\varphi)\Vert^2 _{(\ell^2(\R))^2}}{\Vert \varphi \Vert_{(\ell^2(\R))^2}^2}\ \vert\ \varphi \in (\ell^2(\R))^2\setminus\{0\}\right\}}\\
  & = & \frac{1}{\inf \left\{\frac{\langle \Lambda_{\omega}(\varphi),\varphi\rangle _{(\ell^2(\R))^2}}{\Vert \varphi \Vert_{(\ell^2(\R))^2}^2}\ \vert\ \varphi \in (\ell^2(\R))^2\setminus\{0\}\right\}}=\frac{1}{C_T^{(W)}(\chi_\omega)}.
\end{eqnarray*}
The conclusion follows. 
\end{proof}

It follows from this result that, for the optimal design control problem,
$$
\inf_{\chi_\omega\in\mathcal{U}_L}\Vert \Gamma_\omega\Vert =\left(\displaystyle \sup_{\chi_\omega\in\mathcal{U}_L}C_T^{(W)}(\chi_\omega)\right)^{-1},
$$
and therefore the problem is equivalent to the problem of maximizing the observability constant.
Then, all considerations done in this article can be applied to the optimal design control problem as well.

\appendix
\section{Proof of Proposition \ref{propCantor}} \label{AppendixCantor}
We focus first on   the one-dimensional wave equation.
Whereas the proposition is stated on $[0,\pi]$, we assume hereafter that we are on $[-\pi,\pi]$, in order to facilitate the use of Fourier series.
To avoid any technical problem we assume that we are in the framework of Remark \ref{remTmult2pi}. In particular, we assume that $T$ is an integer multiple of $2\pi$.
According to the characterization of the optimal set in terms of a level set of the function 
$$\varphi(x) = \sum_{j=1}^{+\infty} \lambda_{j}^2\alpha_{jj} \sin^2(jx),$$
with $\alpha_{jj} = p\pi (a_j^2+b_j^2)$, 
and noting that the coefficients $\alpha_{jj}$ are nonnegative and of converging sum,
it suffices to prove the following result.

\begin{proposition}
There exist a measurable open subset $C$ of $[-\pi,\pi]$, of Lebesgue measure $\vert C\vert\in(0,2\pi)$, and a smooth function $f$ on $[-\pi,\pi]$, satisfying the following properties:
\begin{itemize}
\item $C$ is of fractal type, and in particular has an infinite number of connected components;
\item $f(x)>0$ for every $x\in C$, and $f(x)=0$ for every $x\in[-\pi,\pi]\setminus C$;
\item $f$ is even;
\item for every integer $n$,
$$
a_n = \int_{-\pi}^{\pi} f(x)\cos(nx)\, dx > 0;
$$
\item The series $\sum a_n$ is convergent.
\end{itemize}

\end{proposition}

Proposition \ref{propCantor} follows from that result, and the optimal set $\omega$ is then the complement of the fractal set $C$.
By considering cartesian products of this one-dimensional fractal set, it is immediate to generalize the construction to a $n$-dimensional hypercube for the Schr\"odinger equation since the solution remains periodic in this case, which ensures that $G_T$ does not involve any crossed terms.

There are many possible variants of such a construction. We provide hereafter one possible way of proving this result.

\begin{proof}
Let $\alpha\in(0,1/3)$. We assume that $\alpha$ is a rational number, that is, $\alpha=\frac{p}{q}$ where $p$ and $q$ are relatively prime integers, and moreover we assume that $p+q$ is even.
Let us first construct the fractal set $C\subset[-\pi,\pi]$. Since $C$ will be symmetric with respect to $0$, we describe below the construction of $C\cap[0,\pi]$. Set $s_0=0$ and
$$
s_k=\pi-\frac{\pi}{2^k}(\alpha+1)^k,
$$
for every $k\in\N^*$. Around every such point $s_k$, $k\in\N^*$, we define the interval
$$
I_k = \left[ s_k - \frac{\pi}{2^k}\alpha(1-\alpha)^k , s_k + \frac{\pi}{2^k}\alpha(1-\alpha)^k \right]
$$
of length $\vert I_k\vert = \frac{\pi}{2^{k-1}}\alpha(1-\alpha)^k$.

\begin{lemma}\label{lemtechtrain1}
We have the following properties:
\begin{itemize}
\item $\inf I_1>\alpha\pi$;
\item $\sup I_k<\inf I_{k+1}<\pi$ for every $k\in\N^*$.
\end{itemize}
\end{lemma}

\begin{proof}
Since $\alpha<1/3$ it follows that $\inf I_1=\pi-\frac{\pi}{2}(\alpha+1)>\alpha\pi$. For the second property, note that the inequality $\sup I_k<\inf I_{k+1}$ is equivalent to
$$
\alpha(1-\alpha)^{k-1}(3-\alpha)<(\alpha+1)^k,
$$
which holds true for every $k\in\N^*$ since $\alpha(3-\alpha)<\alpha+1$.
\end{proof}

It follows in particular from that lemma that the intervals $I_k$ are two by two disjoint.
Now, we define the set $C$ by
$$
C\cap[0,\pi] = [0,\alpha\pi] \cup \bigcup_{k=1}^{+\infty} I_k.
$$
The resulting set $C$ (symmetric with respect to $0$) is then of fractal type and has an infinite number of connected components (see Figure \ref{figCantor}).

\begin{figure}
\scalebox{0.6} 
{
\begin{pspicture}(0,-7.179625)(23.156,7.145625)
\psline[linewidth=0.032cm,arrowsize=0.05291667cm 2.0,arrowlength=1.4,arrowinset=0.4]{->}(0.0,-3.1703749)(23.14,-3.1903749)
\psline[linewidth=0.032cm,arrowsize=0.05291667cm 2.0,arrowlength=1.4,arrowinset=0.4]{->}(3.98,-6.730375)(3.94,7.1296253)
\psline[linewidth=0.1cm,tbarsize=0.07055555cm 5.0,rbracketlength=0.15]{(-)}(1.0,-3.2103748)(7.04,-3.2103748)
\usefont{T1}{ptm}{m}{n}
\rput(4.821455,-3.5453749){\Large $s_0=0$}
\psdots[dotsize=0.3](4.0,-3.2703748)
\psdots[dotsize=0.3](22.16,-3.1903749)
\psdots[dotsize=0.3](14.02,-3.2103748)
\psdots[dotsize=0.3](18.48,-3.2103748)
\psline[linewidth=0.1cm,tbarsize=0.07055555cm 5.0,rbracketlength=0.15]{(-)}(12.62,-3.2103748)(15.38,-3.2103748)
\psline[linewidth=0.1cm,tbarsize=0.07055555cm 5.0,rbracketlength=0.15]{(-)}(17.82,-3.1903749)(19.12,-3.1903749)
\psline[linewidth=0.04](1.08,-3.2303748)(3.96,5.6496253)(6.94,-3.1703749)
\psline[linewidth=0.04](12.7,-3.2303748)(14.0,2.0896254)(15.3,-3.2103748)
\psline[linewidth=0.04](17.88,-3.2503748)(18.48,0.10962524)(19.06,-3.1903749)
\usefont{T1}{ptm}{m}{n}
\rput(14,-3.7){\Large $s_1$}
\usefont{T1}{ptm}{m}{n}
\rput(18.5,-3.7){\Large $s_2$}
\usefont{T1}{ptm}{m}{n}
\rput(22.2,-3.7){\Large $\pi$}
\usefont{T1}{ptm}{m}{n}
\rput(7.,-3.7){\Large $\alpha\pi$}
\usefont{T1}{ptm}{m}{n}
\rput(1,-3.7){\Large $-\alpha\pi$}
\usefont{T1}{ptm}{m}{n}
\rput(4.7,5.7546253){\Large $b_0$}
\usefont{T1}{ptm}{m}{n}
\rput(14.6,2.2346253){\Large $b_1$}
\usefont{T1}{ptm}{m}{n}
\rput(19.,0.21462524){\Large $b_2$}
\psline[linewidth=0.02cm,arrowsize=0.05291667cm 2.0,arrowlength=1.4,arrowinset=0.4]{<->}(12.7,-4.550375)(15.32,-4.550375)
\psline[linewidth=0.02cm,arrowsize=0.05291667cm 2.0,arrowlength=1.4,arrowinset=0.4]{<->}(17.88,-4.550375)(19.08,-4.5303745)
\usefont{T1}{ptm}{m}{n}
\rput(14.2014555,-4.8853745){\Large $I_1$}
\usefont{T1}{ptm}{m}{n}
\rput(18.621454,-4.8853746){\Large $I_2$}
\end{pspicture} 
}
\caption{Drawing of the function $f$ and of the set $C$}\label{figCantor}
\end{figure}
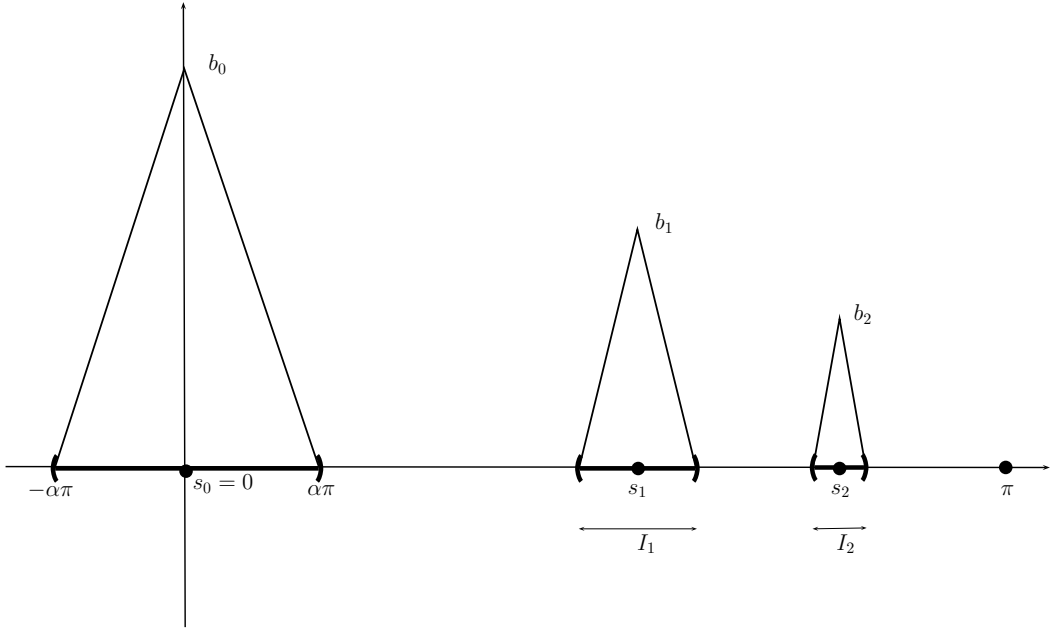

We now define the function $f$ such that $f$ is continuous, piecewise affine, equal to $0$ outside $C$, and such that $f(s_k)=b_k$ for every $k\in\N$, where the $b_k$ are positive real numbers to be chosen (see Figure \ref{figCantor}).

Let us compute the Fourier series of $f$. Since $f$ is even, its sine coefficients are all equal to $0$. In order to compute its cosine coefficients, we will use the following result.

\begin{lemma}\label{lemtech2}
Let $a\in\R$, $\ell>0$ and $b>0$. Let $g$ be the function defined on $\R$ by
\begin{equation*}
g(x) = \left\{ \begin{array}{cl}
\frac{2b}{\ell}(x-a+\frac{\ell}{2}) &\textrm{if}\ a-\frac{\ell}{2}\leq x\leq a,\\
\frac{2b}{\ell}(a+\frac{\ell}{2}-x) &\textrm{if}\ a\leq x\leq a+\frac{\ell}{2},\\
0 & \textrm{otherwise.}
\end{array}\right.
\end{equation*}
In other words, $g$ is a positive triangle of height $b$ above the interval $[a-\frac{\ell}{2},a+\frac{\ell}{2}]$. Then
$$
\int_\R g(x)\cos(nx)\, dx = \frac{4b}{\ell n^2}\cos(na)\left( 1-\cos\frac{n\ell}{2}\right),
$$
for every $n\in\N^*$.
\end{lemma}

It follows from this lemma that
\begin{equation}\label{techint0}
\int_0^{\alpha\pi} f(x)\cos(nx)\, dx = \frac{b_0}{\alpha\pi n^2}(1-\cos(n\alpha\pi)),
\end{equation}
and
\begin{equation}\label{techintk}
\int_{I_k} f(x)\cos(nx)\, dx = \frac{2^{k+1}b_k}{\alpha(1-\alpha)^k\pi n^2} \cos\left(n\pi-\frac{n\pi}{2^k}(\alpha+1)^k\right) \left(1-\cos\left(\frac{n\pi}{2^{k}}\alpha(1-\alpha)^k\right)\right)  ,
\end{equation}
for every $k\in\N^*$.
Note that
\begin{equation}\label{utilintk}
\left\vert \int_{I_k} f(x)\cos(nx)\, dx \right\vert \leq \frac{4b_k}{\alpha\pi n^2} \left( \frac{2}{1-\alpha}\right)^k  ,
\end{equation}
for every $k\in\N^*$.
Formally, the $n^\mathrm{th}$ cosine Fourier coefficient of $f$ is given by
$$
a_n = \int_{-\pi}^{\pi} f(x)\cos(nx)\, dx 
= 2\int_0^{\alpha\pi} f(x)\cos(nx)\, dx + 2 \sum_{k=1}^{+\infty} \int_{I_k} f(x)\cos(nx)\, dx.
$$
Our next task consists of choosing adequately the positive real numbers $b_k$, $k\in\N$, so that the series appearing in the above formal expression of $a_n$ is convergent, $a_n$ is nonnegative, and the series of general term $a_n$ is convergent.

Let us first consider the integral \eqref{techint0} (first peak). It is clearly nonnegative for every $n\in\N^*$, and is positive except whenever $n$ is a multiple of $2q$. Taking advantage of the rationality of $\alpha$, we can moreover derive an estimate from below, as follows. Set
$$
\sigma_0 = \min\{1-\cos(n\frac{p}{q}\pi)\ \vert\ n=1,\ldots,2q-1\}.
$$
One has $\sigma_0>0$, and there holds
\begin{equation}\label{techineq0}
\int_0^{\alpha\pi} f(x)\cos(nx)\, dx \geq \frac{b_0\sigma_0}{\alpha\pi n^2},
\end{equation}
for every $n\in\N^*\setminus(2q\N^*)$.
At this step, assume that
\begin{equation}\label{techrequir1}
b_k\leq \left(\frac{1-\alpha}{2}\right)^k\frac{1}{2^k}\frac{\sigma_0b_0}{8},
\end{equation}
for every $k\in\N^*$ ($b_0>0$ is arbitrary). Under this assumption, using \eqref{utilintk} it follows that the formal expression of $a_n$ above is well defined, and that
$$
\left\vert  \sum_{k=1}^{+\infty} \int_{I_k} f(x)\cos(nx)\, dx \right\vert \leq \frac{1}{2} \frac{b_0\sigma_0}{\alpha\pi n^2} \leq \frac{1}{2} \int_0^{\alpha\pi} f(x)\cos(nx)\, dx ,
$$
for every $n\in\N^*\setminus(2q\N^*)$, ensuring therefore $a_n>0$ for such integers $n$.

If $n=2rq$, with $r\in\N^*$, then the integral \eqref{techint0} vanishes. We then focus on the second peak, that is, on the integral \eqref{techintk} with $k=1$. Since $n=2rq$, its value is
\begin{equation*}
\int_{I_1} f(x)\cos(nx)\, dx = \frac{4b_1}{\alpha(1-\alpha)\pi n^2} \cos\left(2rq\pi-rq\pi(\frac{p}{q}+1)\right) \left(1-\cos\left( rq\pi\frac{p}{q}(1-\frac{p}{q})\right)\right) .
\end{equation*}
Since $p+q$ is even, it follows that $\cos\left(2rq\pi-rq\pi(\frac{p}{q}+1)\right)=1$. Hence, we have
\begin{equation*}
\int_{I_1} f(x)\cos(nx)\, dx = \frac{4b_1}{\alpha(1-\alpha)\pi n^2}  \left(1-\cos\left( r\pi\frac{p}{q}(q-p)\right)\right) \geq 0 .
\end{equation*}
Moreover, since the integers $p$ and $q$ are relatively prime integers and $q-p$ is even, in this last expression one has $\cos( r\pi\frac{p}{q}(q-p))=1$ if and only if $r$ is multiple of $q$, that is, if and only if $n$ is multiple of $2q^2$. As before we derive an estimate from below, setting
$$
\sigma_1 = \min\left\{1-\cos\left( r\pi\frac{p}{q}(q-p)\right)\ \big\vert\ r=1,\ldots,2q-1\right\}.
$$
One has $\sigma_1>0$, and there holds
\begin{equation}\label{techineq1}
\int_{I_1} f(x)\cos(nx)\, dx \geq \frac{4 b_1\sigma_1}{\alpha(1-\alpha)\pi n^2},
\end{equation}
for every $n\in(2q\N^*)\setminus(2q^2\N^*)$.
At this step, additionally to \eqref{techrequir1} assume that
\begin{equation}\label{techrequir2}
b_k\leq \left(\frac{1-\alpha}{2}\right)^{k-1}\frac{1}{2^{k+1}}b_1 \sigma_1,
\end{equation}
for every $k\geq 2$. Under this assumption, using \eqref{utilintk} it follows that
$$
\left\vert  \sum_{k=2}^{+\infty} \int_{I_k} f(x)\cos(nx)\, dx \right\vert \leq \frac{1}{2} \frac{4 b_1\sigma_1}{\alpha(1-\alpha)\pi n^2} \leq \frac{1}{2} \int_{I_1} f(x)\cos(nx)\, dx ,
$$
for every $n\in(2q\N^*)\setminus(2q^2\N^*)$, ensuring therefore $a_n>0$ for such integers $n$.

The construction can be easily iterated. At iteration $m$, assume that $n=2rq^m$, with $r\in\N^*$. Then the integrals over the $m$ first peaks vanish, that is,
$$
\int_0^{\alpha\pi}f(x)\cos(nx)\, dx = \int_{I_k}f(x)\cos(nx)\, dx = 0
$$
for every $k=1,\ldots,m-1$. We then focus on the $(m+1)^\mathrm{th}$ peak, that is, on the integral \eqref{techintk} with $k=m$. Since $n=2rq^m$, its value is
\begin{equation*}
\begin{split}
\int_{I_m} f(x)\cos(nx)\, dx = \frac{2^{m+1}b_m}{\alpha(1-\alpha)^m\pi n^2} 
& \cos\left(2rq^m\pi-\frac{rq^m\pi}{2^{m-1}}\left(\frac{p}{q}+1\right)^m\right) \\
& \qquad\qquad\qquad
\times  \left(1-\cos\left( \frac{rq^m\pi}{2^{m-1}}\frac{p}{q}\left(1-\frac{p}{q}\right)^m\right)\right) .
\end{split}
\end{equation*}
Since $p+q$ is even, it follows that 
$$
\cos\left(2rq^m\pi-\frac{rq^m\pi}{2^{m-1}}\left(\frac{p}{q}+1\right)^m\right) =1,
$$
and hence,
\begin{equation*}
\int_{I_m} f(x)\cos(nx)\, dx = \frac{2^{m+1}b_m}{\alpha(1-\alpha)^m\pi n^2} 
\left(1-\cos\left( \frac{r\pi}{2^{m-1}}\frac{p}{q}\left(q-p\right)^m\right)\right) \geq 0.
\end{equation*}
Moreover, since the integers $p$ and $q$ are relatively prime integers and $q-p$ is even, it follows easily that $q$ and $(\frac{q-p}{2})^m$ are relatively prime integers, and therefore this last expression vanishes if and only if $r$ is multiple of $q$, that is, if and only if $n$ is multiple of $2q^{m+1}$. Setting
$$
\sigma_m = \min\left\{ 1-\cos\left( \frac{r\pi}{2^{m-1}}\frac{p}{q}\left(q-p\right)^m\right) \ \big\vert\ r=1,\ldots,2q-1\right\},
$$
one has $\sigma_m>0$ and 
$$
\int_{I_m} f(x)\cos(nx)\, dx \geq \frac{2^{m+1}b_m\sigma_m}{\alpha(1-\alpha)^m\pi n^2} ,
$$
for every $n\in(2q^m\N^*)\setminus(2q^{m+1}\N^*)$. Additionally to \eqref{techrequir1}, \eqref{techrequir2} and the following iterative assumptions, we assume that
\begin{equation}\label{techrequirm}
b_k\leq  \left(\frac{1-\alpha}{2}\right)^{k-m}\frac{1}{2^{k-m+2}}b_m \sigma_m,
\end{equation}
for every $k\geq m+1$. Under this assumption, using \eqref{utilintk} it follows that
$$
\left\vert  \sum_{k=m+1}^{+\infty} \int_{I_k} f(x)\cos(nx)\, dV_g \right\vert \leq \frac{1}{2} \frac{2^{m+1}b_m\sigma_m}{\alpha(1-\alpha)^m\pi n^2} \leq \frac{1}{2} \int_{I_m} f(x)\cos(nx)\, dV_g ,
$$
for every $n\in(2q^m\N^*)\setminus(2q^{m+1}\N^*)$, ensuring therefore $a_n>0$ for such integers $n$.

The construction of the function $f$ goes in such a way by iteration. By construction, its Fourier cosine coefficients $a_n$ are positive, and moreover, the series
$
\sum_{n=0}^{+\infty} a_n
$
is convergent.
We have thus constructed a function $f$ satisfying all requirements of the statement except the fact that $f$ is smooth.

Let us finally show that, using appropriate convolutions, we can modify $f$ in order to obtain a smooth function keeping all required properties.
Set $f_0 = f_{[-\alpha\pi,\alpha\pi]}$ and $f_k = f_{I_k}$ for every $k\in\N^*$. 
For every $\varepsilon>0$, let $\rho_\varepsilon$ be a real nonnegative function which is even, whose support is $[-\varepsilon,\varepsilon]$, whose integral over $\R$ is equal to $1$, and whose Fourier (cosine) coefficients are all positive. Such a function clearly exists. Indeed, only the last property is not usual, but to ensure this Fourier property it suffices to consider the convolution of any usual bump function with itself.
Then, for every $k\in\N$, consider the (nonnegative) function $\tilde f_k$ defined by the convolution
$\tilde f_k = \rho_{\varepsilon(k)}\star f_k$, where each $\varepsilon(k)$ is chosen small enough so that the supports of all functions $\tilde f_k $ are still disjoint two by two and contained in $[-\pi,\pi]$ as in Lemma \ref{lemtechtrain1}.
Then, we define the function $\tilde f$ as the sum of all functions $\tilde f_k$, and we symmetrize it with respect to $0$. Clearly, every Fourier (cosine) coefficient of $\tilde f$ is the sum of the Fourier (cosine) coefficients of $\tilde f_k$, and thus is positive, and their sum is still convergent. The function $\tilde f$ is smooth and satisfies all requirements of the statement of the proposition. This ends the proof.
\end{proof}

\bigskip
\noindent{\bf Acknowledgment.}The authors thank Nicolas Burq, Antoine Henrot, Luc Hillairet and Zeev Rudnick for very interesting and fruitful discussions.

The first author was partially supported by the ANR project OPTIFORM.\\
The third author was partially supported by Grant MTM2011-29306-C02-00, MICINN, Spain, ERC Advanced Grant FP7-246775 NUMERIWAVES, ESF Research Networking Programme OPTPDE and Grant PI2010-04 of the Basque Government.


\end{document}